\documentclass[a4paper]{amsart}
\usepackage{amscd,amsmath,amssymb,amsfonts,verbatim}
\usepackage[cmtip, all]{xy}
\usepackage{pgf,tikz}
\usepackage{mathrsfs}
\usetikzlibrary{arrows}

\newtheorem{thm}{Theorem}[subsection]
\newtheorem{prop}[thm]{Proposition}
\newtheorem{lem}[thm]{Lemma}
\newtheorem{cor}[thm]{Corollary}
\newtheorem{conj}[thm]{Guess}

\newtheorem{ques}[thm]{Question}
\theoremstyle{definition}

\newtheorem{defn}[thm]{Definition}
\theoremstyle{remark}
\newtheorem{remk}[thm]{Remark}
\newtheorem{remks}[thm]{Remarks}

\newtheorem{exm}[thm]{Example}
\newtheorem{exms}[thm]{Examples}
\newtheorem{notat}[thm]{Notation}
\numberwithin{equation}{subsection}

{\hfill$\square$\end{defn}}
{\hfill$\square$\end{remk}}
{\hfill$\square$\end{remks}}
{\hfill$\square$\end{exm}}
{\hfill$\square$\end{exms}}
{\hfill$\square$\end{notat}}

\newcommand{\CH}{{\rm CH}}
\newcommand{\BGH}{\mathbf{C}\mathbf{H}}

\newcommand{\BGHz}{\mathbf{z}}

\newcommand{\red}{{\rm red}}
\newcommand{\codim}{{\rm codim}}
\newcommand{\Pic}{{\rm Pic}}

\newcommand{\Hom}{{\rm Hom}}

\newcommand{\Spec}{{\rm Spec \,}}
\newcommand{\Spf}{{\rm Spf \,}}

\newcommand{\Sch}{{\operatorname{\mathbf{Sch}}}} 
\newcommand{\QProj}{{\operatorname{\mathbf{QProj}}}}
\newcommand{\Aff}{{\operatorname{\mathbf{Aff}}}}
\newcommand{\QAff}{{\operatorname{\mathbf{QAff}}}}
\newcommand{\Sep}{{\operatorname{\mathbf{Sep}}}}

\newcommand{\hocolim}{\mathop{{\rm hocolim}}}

\newcommand{\op}{{\text{\rm op}}}

\renewcommand{\>}{\rangle}

\newcommand{\Sm}{{\mathbf{Sm}}}

\newcommand{\ds}{{/\kern-3pt/}}

\newcommand{\Supp}{{\operatorname{Supp}}}

\newcommand{\Tor}{{\operatorname{Tor}}}

\newcommand{\colim}{\mathop{\text{\rm colim}}}

\newcommand{\hofib}{{\operatorname{\rm hofib}}}

\newcommand{\sgn}{{\operatorname{\rm sgn}}}

\newcommand{\sAlg}{{\operatorname{\textbf{sAlg}}}}
\newcommand{\un}{\underline}

\renewcommand{\dim}{\text{\rm dim}}
\newcommand{\tuborg}{\left\{\begin{array}{ll}}
\newcommand{\sluttuborg}{\end{array}\right.}

\newcommand{\perf}{{\rm perf}}

\renewcommand{\mod}{ {\rm \ mod \ } }

\begin{document}
\title[Extension of the motivic cohomology]{On extension of the motivic cohomology beyond smooth schemes}
\author{Jinhyun Park}
\address{Department of Mathematical Sciences, KAIST, 291 Daehak-ro Yuseong-gu, Daejeon, 34141, Republic of Korea (South)}
\email{jinhyun@mathsci.kaist.ac.kr; jinhyun@kaist.edu}




\begin{abstract}
We construct an algebraic-cycle based model for the motivic cohomology on the category of schemes of finite type over a field, where schemes may admit arbitrary singularities and may be non-reduced.

We show that our theory is functorial on the category, that it detects nilpotence, and that its restriction to the subcategory of smooth schemes agrees with the pre-existing motivic cohomology theory, which is the higher Chow theory of S. Bloch (Adv. Math., 1986). A few structures and applications are discussed.
\end{abstract}

\maketitle

\setcounter{tocdepth}{1}

\tableofcontents

\section{Introduction}

\subsection{Goal and the summary of central results}
The central goal of this article is to present an algebraic-cycle based functorial model for the motivic cohomology of arbitrary schemes of finite type over an arbitrary field $k$, where the schemes may be arbitrarily singular, non-reduced, or non-separated.

\medskip

In topology, we have the theory of singular cohomology based on topological cycles, which has become an indispensable item of a mathematician's tool box for studies of topological spaces. On the contrary, its analogue in algebraic geometry, the motivic cohomology on $k$-schemes based on algebraic cycles, has been fully developed so far on the subcategory of smooth $k$-schemes only.

\medskip

The main results of this article, summarized below, aim to fill this void:

\begin{thm}\label{thm:main intro}
Let $k$ be an arbitrary field. Let $\Sch_k$ be the category of all $k$-schemes of finite type, that may be arbitrarily singular, non-reduced, or non-separated.

Then for all $Y \in \Sch_k$ and integers $q, n \geq 0$, there exist abelian groups $\BGH^q (Y, n)$ written in boldface, distinct from the higher Chow groups $\CH^q (Y, n)$ of S. Bloch \cite{Bloch HC} in general, that satisfy the following properties:

\begin{enumerate}
\item The association $Y \mapsto \BGH^q (Y, n)$ defines a functor 
$$
\BGH^q (-, n): \Sch_k^{\rm op} \to ({\rm Ab}).
$$

\item The group $\bigoplus_{q, n} \BGH^q (Y, n)$ has a natural bi-graded ring structure, which is also functorial.

\item If $Y$ is smooth and equidimensional, then we have an isomorphism
 $$
 \BGH^q (Y, n) \simeq \CH^q (Y, n),
 $$
 where the group on the right is the higher Chow group of S. Bloch.
 \item For non-reduced schemes, the groups $\BGH^q (Y, n)$ and $\BGH^q (Y_{\red}, n)$ are not identified, in general.
 
 \item The groups $\BGH^q (Y, n)$ satisfy the Zariski descent: for each $Y \in \Sch_k$, there exists an object $\BGHz^q (Y, \bullet)$ in the derived category $\mathcal{D}({\rm Ab} (Y))$ of complexes of abelian sheaves on $Y_{\rm Zar}$ such that
 $$
 \BGH^q (Y, n) \simeq \mathbb{H}_{\rm Zar} ^{-n} (Y, \BGHz^q (Y, \bullet) ),
 $$
 and for each pair of Zariski open subsets $U, V \subset Y$, we have the Mayer-Vietoris long exact sequence
 $$
 \cdots \to \BGH^q (U \cup V, n) \to \BGH^q (U , n) \oplus \BGH^q (V, n) \to \BGH^q (U \cap V, n)  
 $$
 $$
 \hskip5cm \to \BGH^q (U \cup V, n-1) \to \cdots.
$$
\item For each morphism $g: Y_1 \to Y_2$ in $\Sch_k$, there exist the relative groups $\BGH^q (g, n)$ such that we have the long exact sequence
$$
\cdots \to \BGH^q (g, n) \to \BGH^q (Y_1, n) \to \BGH^q (Y_2, n) \to \BGH^q (g, n-1) \to \cdots.
$$
\item For each $Y \in \Sep_k$ in the subcategory of separated $k$-schemes of finite type, there exists the first Chern class homomorphism
$$
c_1: {\rm Pic} (Y) \to \BGH^1 (Y, 0),
$$
which is functorial in $Y$.

\item For each semi-local $k$-scheme $Y$ essentially of finite type, for the group $\BGH^q (Y, n)$ defined as the colimit of the groups over the affine open neighborhoods of the set of all closed points of $Y$, we have the graph homomorphism from the Milnor $K$-theory
$$ 
gr: K_n ^M (Y) \to \BGH^n (Y, n).
$$
When $Y$ is smooth or $Y= \Spec (k[t]/(t^{m+1}))$ for $m \geq 1$ and $|k|\gg 0$, the graph homomorphisms give isomorphisms, by \cite{EVMS}, \cite{Kerz finite}, \cite{Park Tate}. 

 \end{enumerate}
 The groups $\BGH^q (Y, n)$ will be called the \emph{yeni\footnote{\emph{Yeni} means `new' or `renewed' in Turkish. With this word, the author would like to mention that Sinan \"Unver in \.{I}stanbul was once part of this project at an early stage, though later he wanted to leave it to follow his ways as he might have felt this project grew unreasonably ambitious due to the author's insistence. However, the author thanks him for being a good listener even after he left the project.} higher Chow groups}, to distinguish them from the classical higher Chow groups $\CH^q (Y, n)$ of S. Bloch.
\end{thm}

Since the higher Chow groups on the smooth $k$-schemes are the motivic cohomology groups (V. Voevodsky \cite{Voevodsky}), the yeni higher Chow groups $\BGH^q (Y, n)$ for $Y \in \Sch_k$ constructed in this article indeed give a functorial extension of the motivic cohomology theory beyond the smooth schemes that detect the nilpotence of schemes, as the title of this article suggests.

\medskip

The objects $\BGHz^q (Y, \bullet)$ appearing in Theorem \ref{thm:main intro} are somewhat more complicated than one may hope. They are well-defined as an isomorphism class of the derived category $\mathcal{D} ({\rm Ab} (Y))$ of abelian sheaves on $Y_{\rm Zar}$ only, and the association $Y \mapsto \BGHz^q (Y, \bullet)$ is not a sheaf in the strict sense, though it is almost like a sheaf. A precise and appropriate language is that it is a stack. In this terminology, a relevant part of what we prove in this article can be summarized as follows:

\begin{thm}\label{thm:main intro stack}
Let $Y \in \Sch_k$. Let $Op (Y)$ be the category of Zariski open subschemes of $Y$.

Consider the pseudo-functor $\BGHz^q (-,\bullet)$ on $Op(Y)^{\op}$ defined by
$$
U \in Op (Y)^{\op} \mapsto \{ \mbox{the isomorphism class of } \BGHz^q (U, \bullet) \} \mbox{ in } \mathcal{D} ({\rm Ab} (U)),$$
 where the class is considered as a category, all of whose morphisms are isomorphisms. 
 
 Then it satisfies the gluing conditions of Zariski descent data, i.e. $\BGHz^q (-,\bullet)$ is a stack.
\end{thm}

The Zariski descent (Theorem \ref{thm:main intro}-(5)) is a consequence of Theorem \ref{thm:main intro stack}. For this paper, a basic understanding of stacks is enough, for which one can see, e.g. A. Vistoli \cite{Vistoli}.

\medskip

From the statement of Theorem \ref{thm:main intro stack}, those who are familiar with descent theory and stacks may wonder whether we locally define some complexes of sheaves in suitable derived categories, and try to glue them over the intersections of open sets, as done in e.g. \cite[\S 0D65]{stacks} or Beilinson-Bernstein-Deligne \cite[Th\'eor\`eme 3.2.4, p.82]{BBD}. However, we remark that we do \emph{not} follow this way; one important reason is that our locally defined complexes of sheaves are \emph{not} even complexes of $\mathcal{O}_Y$-modules in general. 

The alternative route we choose to follow is a variation of the \v{C}ech machine used in R. Hartshorne \cite[Remark, p.28]{Hartshorne DR}, where he hints how one can construct the algebraic de Rham cohomology for ``non-embeddable" varieties. The machine in \emph{ibid.}~is not immediately applicable to our objects, but we can still modify it suitably to use for our problem. Carrying it out, however, is \emph{not} \emph{mutatis mutandis}: it requires a new version of moving lemmas on complexes of sheaves of cycles on formal schemes, which we carefully develop in \S \ref{sec:first indep}.

\subsection{Some history} One has a few ordinary cohomology theories applicable to singular schemes. For instance, the singular cohomology (when $k= \mathbb{C}$) with the mixed Hodge structures of P. Deligne \cite{Deligne}, the algebraic de Rham cohomology (when ${\rm char} (k) = 0$) of R. Hartshorne \cite{Hartshorne DR}, the intersection cohomology (Goresky-MacPherson \cite{GM}, Beilinson-Bernstein-Deligne \cite{BBD}), and the \emph{cdh}-motivic cohomology of Friedlander-Voevodsky \cite{FV} are functorial on singular schemes. Though they do not distinguish non-reduced schemes from their associated reduced ones. The Chow cohomology of W. Fulton \cite{Fulton Chow} is only partly defined in the homological level $0$ with respect to the Mayer-Vietoris long exact sequence. So, unlike in topology, on the side of algebraic geometry, the situation was yet far from being clear, and one may guess it is probably harder to build a theory that also discerns such nilpotent singularities, without ignoring them.

\medskip

Even though one is interested in studying only smooth $k$-schemes, those smooth ones may still spawn various associated singular, often non-reduced, schemes. For instance, for a morphism $f: X \to Y$ of smooth varieties, the fibers $X_y$ could be singular or non-reduced (see e.g. A. Dubouloz \cite[Example 5]{Dubouloz}). Another class of examples comes from Hilbert schemes of smooth schemes. For example, the Hilbert scheme of $\mathbb{P}^3$ is non-reduced (see e.g. D. Mumford \cite[\S II, p.643]{Mumford}). For many Hilbert schemes of affine or projective spaces, we do not yet know whether they are reduced or not (see Cartwright-Erman-Velasco-Viray \cite[\S 7]{Cartwright}).

So, a construction of a fully functorial cohomology theory that can detect singularities and nilpotence, could potentially help or boost our efforts to understand them, and the author believes that it is a meaningful step forward if it is indeed realized.

\medskip

While singularities and nilpotence could be obstacles, once we know how to embrace them and even live with them (as the ancient stoicism philosophy might teach us), then we may be able to use them to extract new information even from smooth schemes. For instance, the crystalline cohomology of P. Berthelot \cite{Berthelot} can be computed by considering Quillen's higher algebraic $K$-theory \cite{Quillen} of nilpotent thickenings of a smooth scheme, as shown by S. Bloch \cite{Bloch crys}, and more or less equivalently through the $p$-typical de Rham-Witt forms by L. Illusie \cite{Illusie}. Such nilpotent thickenings are also related to deformation theory in general. Thus, if we have a functorial theory that works also for singular schemes, including non-reduced ones, then one may have a platform to deduce a few arithmetic or deformation-theoretic applications.

\medskip

The question is then whether we can construct one such a theory. Since we do have the motivic cohomology on the smooth $k$-schemes, a more specific question should be to find a way to extend this one on the smooth schemes to one such theory on all $k$-schemes of finite type.

\medskip

On smooth schemes, some of the earliest candidates for the motivic cohomology were provided by the Chow groups of W.-L. Chow \cite{Chow} in the 1950s, and their generalizations, the higher Chow groups of S. Bloch \cite{Bloch HC} in the 1980s. More than a decade later, on the category of smooth $k$-schemes, V. Voevodsky \cite{Voevodsky} (see also \cite[Theorem 19.1]{MVW}) proved that they indeed give the right motivic cohomology compatible with his theory of the motivic cohomology. Motivic analogues of the spectral sequences of Atiyah-Hirzebruch (\cite{AH}) converging to the higher algebraic $K$-theory of D. Quillen \cite{Quillen} were also constructed (see \cite{BL}, \cite{FS}, \cite[\S 8]{Levine moving}, \cite{Levine coniveau}) on the smooth $k$-schemes. 

\medskip

When the schemes admit singularities, it is known from a few perspectives that the higher Chow complex of \cite{Bloch HC}, as is, does not provide a good model for the motivic cohomology on singular schemes. For instance, the Chow groups on singular schemes do not form a ring. This problem comes more fundamentally from that pulling back of cycle classes does not work well for general morphisms between singular $k$-schemes. There were attempts (e.g. M. Hanamura \cite{Hanamura}, if the given base field permits) to improve the situation using Hironaka's \cite{Hironaka} resolution of singularities, but this does not still resolve problems for non-reduced schemes. The non-triviality of certain relative algebraic $K$-groups of some non-reduced rings (see S. Bloch \cite{Bloch crys} and W. van der Kallen \cite{vdK}) shows that the higher Chow complex of schemes, which ignores nilpotent thickenings, cannot fit into a conjectural Atiyah-Hirzebruch type spectral sequence for such non-reduced $k$-schemes.

\medskip

Recent years, a few new ideas such as additive higher Chow groups (\cite{BE1}, \cite{BE2}, \cite{KP crys}, \cite{P2}, \cite{R}), and their generalizations via ``cycles with modulus" (see \cite{BS}, \cite{MM1}, \cite{MM2}, \cite{MM3}) emerged and helped in capturing various non-$\mathbb{A}^1$-invariant phenomena on certain singular schemes, especially non-reduced ones. Nevertheless, due to growing technical and philosophical challenges encountered in handling those cycles with modulus over the last decade, at least the author had growing feelings that somewhat genuinely different approaches are probably desirable. 

\medskip

A foray made in \cite{PU Milnor} by the author with S. \"Unver was one such attempt tried from a quite different perspective compared to the previous ones through ``cycles with modulus" conditions. To deal with the problem of constructing motivic cohomology over the scheme $\Spec (k[t]/(t^m))$, that paper considered cycles over the henselian scheme $\Spec (k[[t]])$, put a suitable ``mod $t^m$-equivalence relation" on integral cycles, and compared the cycle class groups with the Milnor $K$-theory of $k[t]/(t^m)$. While the approach there still had various issues to improve, the attempt at least opened up new avenues of possibilities. On this approach, at a social event in Tokyo, Japan, in 2019, Fumiharu Kato suggested the author that it may be better to use the formal scheme $\Spf (k[[t]])$ instead of the scheme $\Spec (k[[t]])$. This idea was tested in the Milnor range again in \cite{Park Tate} by the author, and it indeed gave better results and perspectives (see \S \ref{sec:nilpotence} and Guess \ref{conj:00}).

\medskip

The present article is an outcome of several years' attempts of the author to build a new foundation on motivic cohomology of singular schemes. In short, using algebraic cycles on formal schemes as some of the raw construction materials, in this paper we build groups originating from algebraic cycles, that are contravariant functorial on the category $\Sch_k$ of schemes of finite type over a field $k$, that do not ignore nilpotent singularities, and that coincide with the higher Chow groups of S. Bloch \cite{Bloch HC} on the subcategory $\Sm_k$ of smooth $k$-schemes.

\subsection{Sketch of the construction, part I}\label{sec:preview intro}
We sketch the construction of the model of the article. The construction is, unfortunately and unavoidably, protracted and it requires a few nontrivial steps. After numerous definitions along the way, the main object $\BGH^q (Y, n)$ of this article is finally defined in Definition \ref{defn:Cech final} on p.\pageref{defn:Cech final}. The following diagram summarizes some important objects constructed along the way. The reader may not yet understand what is going on, but the author hopes that after reading the lengthy introduction, this diagram would make a lot more sense. Here, $\Sch_k$ is the category of $k$-schemes of finite type, and $\QAff_k$ is the subcategory of quasi-affine $k$-schemes:
$$
\xymatrix{
 \underset{(3)}{(Y \hookrightarrow {X})} \ar@{..>}[r]  & \underset{(3')}{ \mathcal{M}^q (\widehat{X} \mod Y, \bullet) } \ar@{^{(}->}[d]  \ar@{..>}[dr]  & & \\
\underset{(1)}{z_* (\mathfrak{X}) } \ar@{..>}[r] & \underset{ (2)} {z^q (\widehat{X}, \bullet) } \ar[r]   & \underset{(4)}{ z^q (\widehat{X} \mod Y, \bullet)}    \ar@{..>}[d]    \ar@{..>}[dl]|-{{\rm H}_n}  \ar@{..>}[dr] & \underset{(6)}{ G^q (W, \bullet) } \ar@{..>}[d]  \ar@{_{(}->}[l]  \\
\ar@{}[dr]|-{\begin{matrix} Y \in \QAff_k \uparrow \\ ------------\\ Y \in \Sch_k \downarrow \end{matrix}} &  \underset{(4')}{\CH^q (\widehat{X} \mod Y, n) }  & \underset{(5)}{ \BGHz^q (\widehat{X} \mod Y, \bullet) }  
 \ar@{..>}[dl]|-{{\rm \check{C}ech}} &\underset{ (7)}{ \mathcal{S}^q _{\bullet}} \ar[l] \\
\underset{(8)}{\begin{matrix}\mathcal{U}= \\  \{ (U_i, X_i ) \}_{i \in \Lambda} \end{matrix}}   \ar@{..>}[r]  & \underset{(9)}{ \BGHz^q (\mathcal{U}, \bullet) }   \ar@{..>}[r] |-{\underset{\mathcal{U}}{\hocolim}} \ar@{..>}[d]|-{\mathbb{H}^{-n}_{\rm Zar}} &  \underset{(10)}{ \BGHz^q (Y, \bullet )} \ar@{..>}[r]|-{\mathbb{H}^{-n}_{\rm Zar} } & \underset{ (10')}{\boxed{ \BGH^q (Y, n)}} \\
& \underset{ (9')}{\BGH^q (\mathcal{U}, n) }\ar@{..>}[rru]|-{\underset{\mathcal{U}}{\colim}} & &
}
$$

Here, a broken arrow from a point (A) to a point (B) means, (A) is used to define (B). Hooked arrows are injections. Solid arrows are actual morphisms of complexes. The numbers below indicate the order the objects are defined in the article. From (1) to (7) we have $Y\in \QAff_k$, while from (8) to (10') we have $Y \in \Sch_k$.

As the reader can see, this is complicated. However, an inspiration behind it is simpler and it stems out of the following classical machine.

\medskip

Let $k$ be a field. In extending A. Grothendieck's \cite{Grothendieck DR} algebraic de Rham cohomology of smooth $k$-schemes to singular ones, R. Hartshorne \cite{Hartshorne DR} embedded a given possibly singular $k$-scheme $Y$ (suppose ``embeddable", e.g. quasi-projective) into a smooth $k$-scheme $X$ as a closed subscheme, and took  the completion $\widehat{X}$ along $Y$, i.e. the formal neighborhood of $Y$ in $X$.

The algebraic de Rham cohomology is (\cite[II-\S 1, p.24]{Hartshorne DR}, supposing ${\rm char} (k) = 0$) defined to be the hypercohomology
\begin{equation}\label{eqn:HDR intro}
{\rm H}_{\rm DR} ^n (Y):= \mathbb{H}_{\rm Zar} ^{-n}  (Y, \Omega_{\widehat{X}/k} ^{\bullet})
\end{equation}
of the completion of the de Rham complex $\Omega_{X/k}^{\bullet}$. This cohomology is independent of the choice of $Y \hookrightarrow X$ (see \cite[Theorem (1.4), p.27]{Hartshorne DR}). From this embedding, when $d:= \dim  \ X$, Hartshorne also defined the algebraic de Rham homology (see \cite[II-\S 3, p.37]{Hartshorne DR}) to be the \emph{local} hypercohomology (see SGA II \cite[Expos\'e I]{SGA2})
\begin{equation}\label{eqn:Hart DR homo}
{\rm H}_n ^{\rm DR} (Y):= \mathbb{H}_{Y} ^{2d-n} (X, \Omega_{X/k} ^{\bullet})
\end{equation}
of the de Rham complex $\Omega_{X/k} ^{\bullet}$ with the support in $Y$. When $Y$ is smooth, the algebraic de Rham cohomology and homology coincide (see \cite[Proposition (3.4), p.41]{Hartshorne DR}). 
By construction, Hartshorne's algebraic de Rham cohomology does not distinguish $Y$ from $Y_{\red}$, but on the category of \emph{reduced} $k$-schemes of finite type (also called varieties over $k$), this theory had good properties.

We remark that the idea of using formal neighborhoods $\widehat{X}$ also appeared in Grothendieck's studies of \'etale covers SGA II \cite[Expos\'e X, p.111]{SGA2} as well as his studies of the Lefschetz problem \cite[Expos\'e XII, p.136]{SGA2}, for instance.

\medskip

Going back to the construction of the motivic cohomology on singular schemes, we thus ask: \emph{can we try something similar with algebraic cycles?} 

As a starter, we note interestingly that it has been known for decades that the higher Chow groups of S. Bloch \cite{Bloch HC} admit a description in terms of the local hypercohomology
\begin{equation}\label{eqn:HC local hyper}
\CH^* (Y, n) = \mathbb{H}_Y ^{-n} (X, z^* (\bullet)_X),
\end{equation}
where $z^* (\bullet)_X$ is the Zariski sheafification of the higher Chow complex of $X$ (see part of the proof of S. Bloch \cite[Theorem (9.1)]{Bloch HC}.

A comparison of the local hypercohomology descriptions \eqref{eqn:Hart DR homo} and \eqref{eqn:HC local hyper} suggests that the higher Chow groups should be seen indeed as the motivic \emph{homology}. (N.B. The higher Chow groups as motivic \emph{homology} can also be seen from other perspectives, too. See e.g. B. Totaro \cite[\S 2, p.84]{Totaro Sigma}.) At least their restrictions to the subcategory of smooth $k$-schemes give the motivic cohomology on the smooth schemes.
 
This observation incites an imagination that maybe a hybrid of the machine of Grothendieck and Hartshorne with the world of algebraic cycles of Bloch may conjure up a good candidate for the motivic cohomology of $k$-schemes with singularities. 

\medskip

In this paper, we try to articulate this point and scrutinize this fantasy in mathematically rigorous terms. We do the following first. 
Let $\square_k:= \mathbb{P}^1 \setminus \{ 1 \}$ and $\square^n_k$ be its $n$-fold product over $k$. Let $\square^0_k:= \Spec (k)$. For a while, let $Y$ be a quasi-affine $k$-scheme of finite type. Choose a closed immersion $Y \hookrightarrow X$ into an equidimensional smooth $k$-scheme. Let $\widehat{X}$ be the completion of $X$ along $Y$. We consider cycles on the formal scheme $\widehat{X} \times_k \square^n_k$ that intersect properly with the faces, and as well as $(\widehat{X})_{\red} \times F$ for faces $F$, with respect to the largest ideals of definition of $\widehat{X} \times_k \square_k ^n$ (see Definition \ref{defn:HCG} for details), to form the group $z^q (\widehat{X},n)$. We regard the cycles on the formal scheme $\widehat{X} \times \square^n$ as sorts of ``perturbations" of those on $Y \times \square^n$.

As an attempt, for the complex $z^q (\widehat{X}, \bullet)$ of cycles on the formal scheme $\widehat{X}$, define its higher Chow group $\CH^q (\widehat{X}, n)$. However, an observant reader will note instantly that something is not right; $Y$, $Y_{\red}$, or any ``thickening" $Y^{[n]}$ of $Y$ in $X$, all have the same formal neighborhood $\widehat{X}$, but this group ignores their differences, so this approach can't be the right one we are looking for.

\medskip

An additional idea inspired by \cite{PU Milnor} is to put a suitable equivalence relation on $z^q (\widehat{X}, n)$. Here, we follow the philosophy, but choose somewhat different relations. We define the \emph{mod $Y$-equivalence} relation given by a subgroup $\mathcal{M}^q (\widehat{X}, Y, n)$ and form the quotient group 
\begin{equation}\label{eqn:intro defn cycle}
z^q (\widehat{X} \mod Y, n): = \frac{ z^q (\widehat{X} ,n)}{ \mathcal{M}^q (\widehat{X}, Y, n)}.
\end{equation}

The group $\mathcal{M}^q (\widehat{X}, Y, n)$ is (see Definition \ref{defn:mod Y equiv}) generated by the differences of the associated cycles $[\mathcal{A}_1] - [\mathcal{A}_2]$ for certain pairs $(\mathcal{A}_1, \mathcal{A}_2)$ of admissible coherent sheaves of $\mathcal{O}_{\widehat{X} \times \square^n}$-algebras subject to a constraint that their left derived restrictions to $Y \times \square^n$ are isomorphic to each other as ``rings" in a suitable sense. The embedding of $Y$ into $\widehat{X}$ could be fairly arbitrary in general, and the above notions of derived restrictions and isomorphisms as ``rings" require some careful statements. 

These are done borrowing languages and ideas from the derived algebraic geometry (e.g. J. Lurie \cite{Lurie SAG}, B. To\"en \cite{Toen}).  The above ``rings" are in fact derived rings in this language. See \S \ref{sec:salg}. 
There were a few recent successful applications of the derived algebraic geometry to $K$-theory and cycles (see e.g. Kerz-Strunk-Tamme \cite{KST IM} and Lowry-Sch\"urg \cite{LS}), so it shouldn't surprise the reader that the derived algebraic geometry offers useful languages for this article as well.

\medskip

 The groups in \eqref{eqn:intro defn cycle} over $n \geq 0$ form a complex and its homology groups are denoted by $\CH^q (\widehat{X} \mod Y, n)$ (see Definition \ref{defn:complex}). 
 Here, the superscript $q$ can be seen as a kind of virtual codimension for $Y$.
This group turns out to be much better than $\CH^q (\widehat{X}, n)$, and it is a main \emph{intermediate} object of this article.

\medskip

An interesting feature of the approach in this article is that, unlike a few approaches to singular schemes $Y$ in the literature, we do not attempt to resolve the singularities of $Y$ as in H. Hironaka \cite{Hironaka}, nor to alter them as in A. de Jong \cite{de Jong}: rather we just embrace them as they are. Instead, we consider a \emph{regular} formal neighborhood $\widehat{X}$ around $Y$, which we regard as a kind of an ``{exoskin}" for $Y$, and work with cycles over this regular formal scheme, modulo the mod $Y$-equivalence. 

\medskip

We remark that since $\widehat{X} \times \square_k ^n$ is a regular formal scheme, the coherent sheaves $(\mathcal{A}_1, \mathcal{A}_2)$ considered in the above are also perfect complexes (see Corollary \ref{cor:coh=perfect}). The mod $Y$-equivalence can then be partially interpreted from the perspective of the derived Milnor patching of perfect complexes studied by S. Landsburg \cite{Landsburg Duke} as well. See J. Milnor \cite[\S 2]{Milnor K} for the original Milnor patching for projective modules over rings, i.e. vector bundles over affine schemes.

\subsection{Sketch of the construction, part II}\label{sec:intro indep}
The group $\CH^q (\widehat{X} \mod Y, n)$ in \S \ref{sec:preview intro} is not yet the right final object we seek, though we can process it a bit to deduce the right candidate. 
We solicit some patience of the reader, and suggest to buckle up and be ready for the rocky roads ahead.

\subsubsection{A sheafification}

The story goes as follows: let $Y$ be a quasi-affine $k$-scheme of finite type and let $X, \widehat{X}$ be as before. For each Zariski open subset $U \subset Y$, we have the association $U \mapsto z^q (\widehat{X}|_U \mod U, \bullet)$, where $\widehat{X}|_U$ is the open formal subscheme $(U, \mathcal{O}_{\widehat{X}}|_U)$ of $\widehat{X}$ and the ``$U$" after ``mod" is the open subscheme $(U, \mathcal{O}_Y|_U)$. Its Zariski sheafification on $Y_{\rm Zar}$ is denoted by $\BGHz^q (\widehat{X} \mod Y, \bullet)$ (see Definition \ref{defn:sh_complex}). 
We define the Zariski hypercohomology group
$$
\BGH^q (\widehat{X} \mod Y, n):= \mathbb{H}_{\rm Zar} ^{-n} (Y, \BGHz^q (\widehat{X} \mod Y, \bullet)).
$$

On the other hand, we consider the subgroup $G^q (Y \setminus U, n) \subset z^q (\widehat{X} \mod Y, n)$ of cycles  topologically supported in $(Y \setminus U)\times \square^n$, 
 and define (see Lemma \ref{lem:flasque psh})
$$
\mathcal{S}_n ^q (U):= \frac{z^q (\widehat{X} \mod Y, n)}{G^q (Y \setminus U, n)}.
$$

We show (Proposition \ref{prop:flasque sh}) that this $\mathcal{S}_n ^q$ is a flasque Zariski sheaf on $Y$. This is inspired from S. Bloch \cite[Theorem (3.4), p.278]{Bloch HC}. From this, we deduce that (Theorem \ref{thm:two BGH -1})
$$
\CH^q (\widehat{X} \mod Y, n) = \mathbb{H}_{\rm Zar} ^{-n} (Y, \mathcal{S}_{\bullet} ^q).
$$

We show there is a natural morphism of complexes of Zariski sheaves on $Y$
\begin{equation}\label{eqn:blowup moving intro}
\mathcal{S}_{\bullet} ^q \to \BGHz^q (\widehat{X} \mod Y, \bullet).
\end{equation}
By \eqref{eqn:blowup moving intro}, we deduce a natural homomorphism
\begin{equation}\label{eqn:blowup moving intro 1}
\CH^q (\widehat{X} \mod Y, n) \to \BGH^q (\widehat{X} \mod Y, n).
\end{equation}
We guess \eqref{eqn:blowup moving intro} is a quasi-isomorphism, thus \eqref{eqn:blowup moving intro 1} is an isomorphism: see Appendix, \S \ref{sec:appendix localization}, especially Proposition \ref{prop:flasque rep}, conjecturally based on Guess \ref{conj:localization}. Guess \ref{conj:localization} essentially goes back to the nontrivial ``blowing-up faces of a cube" technique invented by S. Bloch \cite{Bloch moving} for his proof of the localization theorem on higher Chow groups. Its reorganization and generalization are in M. Levine \cite{Levine moving}. Guess \ref{conj:localization} is a version for our cycles on formal schemes. Our construction in the article primarily uses the latter group of \eqref{eqn:blowup moving intro 1}.

\medskip

\subsubsection{The \v{C}ech machine}

The group $\BGH^q (\widehat{X} \mod Y, n)$ defined in the above has a few limitations. For instance, what if $Y$ does not have a closed immersion into a smooth $X$? Is the group independent of the choice of the embedding? 
Do we have some functoriality? 

For instance, at first sight, one might guess that a technique of  ``to define locally and glue" could work. Namely, for two open subsets $U, V \subset Y$ suppose we have some objects $\BGHz (U)$ and $\BGHz (V)$ such that $\BGHz (U)|_{U \cap V} \simeq \BGHz (U \cap V) \simeq \BGHz (V)|_{U \cap V}$ in some derived category. Then can we glue them to obtain a new object, that we might call $\BGHz (U \cup V)$? Unfortunately, some known results on gluing objects in derived categories, e.g. \cite[\S 0D65]{stacks} or \cite[Th\'eo\`eme 3.2.4, p.82]{BBD}, can't be applied here, because our complexes of sheaves of cycles are not even complexes $\mathcal{O}_Y$-modules in general.

\medskip

To answer all these issues, we use a variant of the \v{C}ech machine invented by R. Hartshorne \cite[Remark, p.28]{Hartshorne DR}. Instead of using just open covers of $Y$, he uses systems $\mathcal{U}=\{ (U_i, X_i) \}_{i \in \Lambda}$ of local embeddings, where $\{ U_i \}$ are covers of $Y$ and each $U_i \hookrightarrow X_i$ is a closed immersion into a smooth scheme. 

Aided by our new moving lemma of \S \ref{sec:first indep}, in \S \ref{sec:finite type} we do something similar for complexes of sheaves of cycles of the form $\BGHz^q (\widehat{X}_I \mod U_I, \bullet)$, to define an isomorphism class $\BGHz^q (\mathcal{U}, \bullet)$ of complexes of sheaves in $\mathcal{D} ({\rm Ab} (Y))$. More precisely, for each $I= (i_0, \cdots, i_p) \in \Lambda^{p+1}$, we let $U_I= U_{i_0} \cap \cdots \cap U_{i_p}$ and $X_I = X_{i_0} \times \cdots \times X_{i_p}$ with the diagonal embedding $U_I \hookrightarrow X_I$. Let $\widehat{X}_I$ be the completion of $X_I$ along $U_I$. As a consequence of the moving lemma we obtain the \v{C}ech type complex 
$$
\check{\mathfrak{C}}^q : \check{\mathfrak{C}} ^q (\bullet)^0 \overset{\delta}{\to}  \check{\mathfrak{C}} ^q (\bullet)^1 \overset{\delta}{\to} \check{\mathfrak{C}} ^q (\bullet)^2 \overset{\delta}{ \to}  \cdots ,
$$
where $\check{\mathfrak{C}}^q (\bullet)^p = \prod_{ I \in \Lambda^{p+1}} \BGHz^q (\widehat{X}_I \mod U_I, \bullet)$, and define $\BGHz^q (\mathcal{U}, \bullet) ={\rm Tot} \  \check{\mathfrak{C}} ^q (\bullet)$. With a suitable notion of refinement, which is a bit more general than Hartshorne's, we construct for each $Y \in \Sch_k$,
$$
\BGHz^q (Y, \bullet):= \underset{\mathcal{U}}{\hocolim} \  \BGHz^q (\mathcal{U}, \bullet),
$$
which is well-defined as an isomorphism class of $\mathcal{D} ({\rm Ab} (Y))$. All these require the moving lemma of \S \ref{sec:first indep}.

We define the yeni higher Chow group, our final object, to be the $(-n)$-th Zariski hypercohomology group
$$
\BGH^q (Y, n):= \mathbb{H}_{\rm Zar} ^{-n} (Y, \BGHz^q (Y, \bullet)).
$$
We check some descent date for $\BGHz^q (Y, \bullet)$ so that they form a stack. 

\subsubsection{Consistency and nilpotence}\label{sec:nilpotence}
The above construction of $\BGH^q (Y, n)$ for $Y \in \Sch_k$ has to be consistent with certain objects studied previously. For the first class, when $Y$ is equidimensional and smooth over $k$, we show that our yeni higher Chow group $\BGH^q (Y, n)$ is isomorphic to the classical higher Chow group $\CH^q (Y, n)$ of S. Bloch \cite{Bloch HC} (see Theorem \ref{thm:sm formal}). Thus, indeed our theory is an extension of the classical motivic cohomology on smooth $k$-schemes to $\Sch_k$.

\medskip

On the other hand, for the fat points $Y_m:= \Spec (k[t]/ (t^m))$ over the integers $m \geq 2$, in \cite{Park Tate} we defined the groups $\BGH^q (Y_m, n)$ to be $\CH^q (\Spf (k[[t]]) \mod Y_m, n)$, without taking any colimits over systems of local embeddings. We check that our yeni higher Chow groups here coincide with the definition in \emph{ibid.} See Theorem \ref{thm:consistency k_m}. 

\medskip

One more test is to see whether our theory does detect the nilpotence. This comes from the main theorem of \emph{ibid.} Since our yeni higher Chow group is isomorphic to the group of \cite{Park Tate}, when $|k|\gg 0$, the graph homomorphism
$$K_n ^M (k[t]/(t^m)) \to \BGH^n (Y_m, n)
$$
in the Milnor range is an isomorphism of groups.
Since the Milnor $K$-groups of the non-reduced schemes $Y_m$ differ over different values of $m \geq 2$, so do the groups $\BGH^n (Y_m, n)$.

\subsection{Functoriality and consequences}\label{sec:1.3}
 In \S \ref{sec:Cech3}, we show that the yeni higher Chow groups $\BGH^q (Y, n)$ are contravariant functorial in $Y \in \Sch_k$. The functoriality implies the existence of a cup product structure, as one has the concatenations of cycles and the pull-back along the diagonal morphism $\Delta_Y: Y \to Y \times Y$.

The existence of such a general functorial extension of the higher Chow groups beyond the smooth case has been unknown so far.

For instance, for a local complete intersection (l.c.i.)~morphism between $k$-schemes, we have the associated pull-back on the usual Chow groups (e.g. W. Fulton \cite[\S 6.6, p.113]{Fulton}, or more generally on the algebraic cobordism, M. Levine \cite[\S 6.5.4, p.198]{Levine cobordism}). Even when the given morphism is not l.c.i., if the codomain of a morphism is smooth, then alternatively a version of moving lemma for Chow groups may allow us to define the pull-back. Since every morphism between smooth $k$-schemes is automatically l.c.i., and also a moving lemma holds in this case, this gives a functorial theory on $\Sm_k$. However, a general morphism between singular $k$-schemes may not be l.c.i., nor is there yet a moving lemma that works for higher Chow groups of singular schemes. These circumstances have been technical obstacles so far in studies of algebraic cycles on schemes with singularities. Sometimes Hironaka's resolution of singularities provided a way out for reduced schemes over a good base field, but yet this method was insufficient for non-reduced schemes.

This is one of reasons why the cycle model via formal schemes of this article, which is no longer hindered by technical barriers for functoriality, is a good candidate.  

\medskip

The functoriality engenders some further thoughts. For instance, when $g: Y_1 \to Y_2$ is a morphism, the pull-back $g^*$ is a morphism in a derived category. We define the relative yeni higher Chow complex of $g$ to be the homotopy fiber of $g^*$. In the special case when we have an effective divisor $g: D \hookrightarrow Y$, one may ask whether this relative complex is related to the higher Chow complex of $Y$ with modulus $D$ of Binda-Saito \cite{BS}. 

By restricting our theory onto the subcategory of Artin local $k$-schemes, we naturally deduce various local deformation functors as well, also known as functors of Artin rings by M. Schlessinger \cite{Schlessinger}. Thus the theory of this article offers a ground to study a deformation theory of cycle classes. The author would like to see its potential connection with Bloch-Esnault-Kerz \cite{BEK} in the future.

\subsection{Miscellaneous structures}

The model presented in this paper is a cycle-based functorial extension of the motivic cohomology on $\Sm_k$ to $\Sch_k$ that detects nilpotence, satisfying a few structures. One may still insist to ask why this model is a ``right" one. 

This seems to be a difficult meta-mathematical question to answer fully at this moment. One of ways to justify it is to construct a version of the Atiyah-Hirzebruch type spectral sequence from $\BGH^q (-, n)$ converging to the algebraic $K$-theory, 
\begin{equation}\label{eqn:AH ss}
E^2_{p,q}=\BGH^{q} (Y, p+q) \Rightarrow K_{p+q} (Y),
\end{equation}
generalizing Bloch-Lichtenbaum \cite{BL} and Friedlander-Suslin \cite{FS} to $\Sch_k$. It is far from being completed at this time. Nevertheless, we provide a few properties as evidences. 

Firstly in \S \ref{sec:Chern}, we prove that there are functorial first Chern class maps
$$
c_1: \Pic (Y) \to \BGH^1 (Y, 0)
$$
for $Y $ in the subcategory $\Sep_k$ of separated $k$-schemes of finite type. The reader will see that the Milnor patching description of ${\rm Pic} (Y)$ in terms of ${\rm Pic} (\widehat{X})$ and the mod $Y$-equivalence on cycles are compatible in a sense, and  this compatibility plays its roles in the construction of $c_1$.

Secondly, in \S \ref{sec:preview}, we discuss the connections with the Milnor $K$-theory.

In \S \ref{sec:Milnor2}, we show that there are graph homomorphisms from the Milnor $K$-theory to the yeni higher Chow groups in the Milnor range. We leave one conjectural guess (see Guess \ref{conj:00}) that for semi-local $k$-schemes essentially of finite type, the graph maps should be isomorphisms. 

This conjecture is already known to hold in the smooth case (see Elbaz-Vincent--M\"uller-Stach \cite{EVMS} and M. Kerz \cite{Kerz Gersten} as well as Nesterenko-Suslin \cite{NS} and B. Totaro \cite{Totaro}). As the first non-smooth case, the conjecture is also proven for the scheme $\Spec ( k[t]/(t^m))$ in \cite{Park Tate}, as mentioned earlier in \S \ref{sec:nilpotence}. 

We remark that the computations in \cite{Park Tate} have an exotic (at least for some cycles theorists like the author) flavor of non-archimedean analysis in rigid analytic geometry involving the Tate algebras of J. Tate \cite{Tate}, in addition to the usual dry flavors of algebraic cycles:

\medskip

\begin{center}
\textbf{singular schemes $\longrightarrow$ formal schemes $\longleftarrow$ rigid analytic spaces}
\end{center}
$$
(k[t]/(t^m)) [y_1, \cdots, y_n] \   \longleftarrow \   k[[t]]\{ y_1, \cdots, y_n \} \   \longrightarrow \ \  k(\! ( t ) \!) \{y_1, \cdots, y_n \}.
$$

This calculation may serve as a potential clue on how some computational studies of the yeni higher Chow groups might be approached. A glimpse of such ``adic analytic arguments" can be found in Lemma \ref{lem:rest poly 1}, as well.





\medskip

There are yet lots of questions to ask, attempt, and answer besides those mentioned so far. Here are a few additional ones.

\medskip

For each $Y \in \Sch_k$, in the virtual codimension $1$, we guess that the complex $\BGHz^1 (Y, \bullet)$ is quasi-isomorphic to $\mathcal{O}_Y ^{\times}[-1]$. Its first cohomology is related to the Guess \ref{conj:00} when $n=1$. If this holds, then we can regard $\BGHz^1 (Y, \bullet)$ as the ``motivic sheaf of weight $1$" on $\Sch_k$ denoted either by $\mathbb{Z}(1)$ or $\Gamma (1)$. It was predicted by A. Beilinson \cite{Beilinson Soule}, and on smooth $k$-schemes proven by S. Bloch \cite[Corollary (6.4), p.289]{Bloch HC} using the higher Chow complex.

Another one is on the Grothendieck-Riemann-Roch theorem for singular schemes. In \cite{PP} by the author and P. Pelaez, a new functorial filtration $F ^{\bullet} K_n (Y)$ on the algebraic $K$-group of a singular $Y$ is defined. This also uses various $\widehat{X}$ arising from closed immersions $Y \hookrightarrow X$ into smooth $k$-schemes and a \v{C}ech machine of a cosimplicial space in the homotopy category, analogous to ours here. We can ask whether there is a cycle class map for $Y \in \Sch_k$
$$
 \BGH^q (Y, n) \to gr ^q K_n (Y).
 $$
If this map induces an isomorphism after tensoring with $\mathbb{Q}$, it would be a generalization to singular schemes of the Grothendieck-Riemann-Roch theorem of SGA VI \cite{SGA6} for $n=0$ and of S. Bloch \cite{Bloch HC} for $n>0$, known so far in the smooth case. In \cite{PP}, it is proven that when $Y$ is smooth, the filtration $F^{\bullet}$ coincides with the homotopy coniveau filtration of M. Levine \cite{Levine coniveau}, so that the spectral sequence associated is the desired motivic Atiyah-Hirzebruch spectral sequence in \eqref{eqn:AH ss}. This gives a hope that the filtration and the spectral sequence considered in \cite{PP} may be a good candidate for $Y \in \Sch_k$ in general.

Finally, one may ask whether the yeni higher Chow groups satisfy the pro-\emph{cdh} descent of M. Morrow \cite{Morrow}. The algebraic $K$-theory of noetherian schemes satisfies it (Kerz-Strunk-Tamme \cite{KST IM}), so this might be the case for the groups $\BGH^q (Y, n)$. At least we have the Zariski descent in this article.

\medskip

\textbf{Conventions.}
In this paper $k$ is an arbitrary field, unless said otherwise. 
Let $\Sch_k$ be the category of all $k$-schemes of finite type, not necessarily separated. The subcategory of separated $k$-schemes is denoted by $\Sep_k$. Let $\QProj_k$, $\QAff_k$, $\Aff_k$ be its subcategories of quasi-projective $k$-schemes, quasi-affine $k$-schemes of finite type, and affine $k$-schemes of finite type, respectively. 

Let $\Sm_k$ be the subcategory of $\Sch_k$ of smooth $k$-schemes. 

All noetherian (formal) schemes we consider in this article are assumed to be finite dimensional.

\section{Higher Chow cycles on formal schemes}\label{sec:cycles}

In \S \ref{sec:cycles}, we recall some elementary facts on noetherian formal schemes, and discuss the notion of algebraic cycles on them needed in this article.

\subsection{Some facts on noetherian formal schemes}

Basic references on noetherian formal schemes are EGA I \cite[Ch 0, \S 7, p.60]{EGA1} and \cite[\S 10, p.180]{EGA1}, and  Fujiwara-Kato \cite{FK}. 
We recall some definitions needed in this paper:

\begin{defn} Let $\mathfrak{X}$ be a noetherian formal scheme. 

Recall that the \emph{dimension} of $\mathfrak{X}$ is the supremum of the Krull dimensions of the local rings $\mathcal{O}_{\mathfrak{X}, x}$ over all $x \in |\mathfrak{X}|$. Throughout the article, \emph{all} noetherian formal schemes we consider will be assumed to be finite dimensional. In general, we have $\dim \ \mathfrak{X} \geq \dim \ |\mathfrak{X}|$, where the latter is the dimension of the underlying noetherian topological space $|\mathfrak{X}|$. 

Each open subset $U \subset | \mathfrak{X}|$ gives the open formal subscheme $(U, \mathcal{O}_{\mathfrak{X}} | _U)$, often denoted by $\mathfrak{X}|_U$. 

We say a noetherian affine formal scheme $ \Spf (A)$ is equidimensional if $A$ is equidimensional. We say that $\mathfrak{X}$ is equidimensional if the affine open formal subschemes $\mathfrak{X}|_U$ are equidimensional over all nonempty affine open subsets $U \subset | \mathfrak{X}|$. 

We say $\mathfrak{X}$ is \emph{integral}, if every nonempty affine open formal subscheme $\mathfrak{X}|_U \subset \mathfrak{X}$ is given by $\mathfrak{X}|_U= \Spf (A)$ for some integral domain $A$. 

For a coherent $\mathcal{O}_{\mathfrak{X}}$-module $\mathcal{F}$ on a noetherian formal scheme $\mathfrak{X}$, the \emph{topological support} of $\mathcal{F}$ is defined by
$$
\Supp (\mathcal{F}) = \{ y \in | \mathfrak{X}| \ | \ \mathcal{F}_y \not = 0 \}.
$$

A \emph{closed formal subscheme} $\mathfrak{Y}$ of $\mathfrak{X}$ is a ringed space $ (|\mathfrak{Y}|, \mathcal{O}_{\mathfrak{Y}})$ given by an ideal sheaf $\mathcal{I} \subset \mathcal{O}_{\mathfrak{X}}$, namely, $\mathcal{O}_{\mathfrak{Y}} = \mathcal{O}_{\mathfrak{X}}/ \mathcal{I}$ and $|\mathfrak{Y}|={\rm Supp} (\mathcal{O}_{\mathfrak{Y}})$. We often say that ``$\mathfrak{Y} \subset \mathfrak{X}$ is a closed formal subscheme" (see \cite[Definition 4.3.3, p.327]{FK} or \cite[D\'efinition I-(10.14.2), p.210]{EGA1}).

When $\mathfrak{X}$ is equidimensional and $\mathfrak{Z} \subset \mathfrak{X}$ is an integral closed formal subscheme, we define the codimension to be
$$
{\rm codim}_{\mathfrak{X}} \mathfrak{Z} := \dim \ \mathfrak{X} - \dim \ \mathfrak{Z}.
$$
We can extend this notion of codimension to all closed formal subschemes in the obvious way.

We say $\mathfrak{X}$ is \emph{regular}, if for each point $x \in |\mathfrak{X}|$, the local ring $\mathcal{O}_{\mathfrak{X}, x}$ is a regular local ring. 
\qed
\end{defn}

Some important examples of regular noetherian formal schemes come from the following situation of formal neighborhoods:

\begin{lem}\label{lem:exoskin}
Let $Y$ be a quasi-projective $k$-scheme. 
\begin{enumerate}
\item Then there exists a closed immersion $Y \hookrightarrow X$ into an equidimensional smooth $k$-scheme. 

\item Let $\widehat{X}$ be the completion of $X$ along $Y$. Then $\widehat{X}$ is a regular equidimensional noetherian formal $k$-scheme. 
\end{enumerate}
\end{lem}
\begin{proof}
(1) is apparent. The regularity part of (2) is proved in Park-Pelaez \cite[Lemma 2.3.1]{PP}. That $\widehat{X}$ is equidimensional follows from EGA ${\rm IV}_2$ \cite[Corollaire (7.1.5), p.184]{EGA4-2}. 
\end{proof}

When $Y, X, \widehat{X}$ are as in Lemma \ref{lem:exoskin}, the noetherian formal scheme $\widehat{X}$ is a locally ringed space, where the stalks are local rings defined via the completed rings of fractions (EGA I \cite[Ch 0, (7.6.15), p.74]{EGA1}). A point is that the stalks are regular local rings.

\begin{remk}
Note that when $Y$ is affine in the above, if desired we may choose $X$ to be affine as well. 

However, it doesn't mean that for a given closed immersion $Y \hookrightarrow X$, we can find an affine open $X^0\subset X$ that contains $Y$. There is a counterexample, which the author learned from an answer by Damiano Testa to a question of Timo Sch\"urg on MathOverflow: take $k= \bar{k}$, and take an elliptic curve $E$. Regard $E$ as an abelian group. Take a closed point $p$ of infinite order in the group. Embed $E \subset \mathbb{P}^2$ via the very ample divisor $3 \mathbf{O}$, where $\mathbf{O}\in E$ is the zero of the group. Then $E\setminus \{p \} \subset \mathbb{P}^2 \setminus \{ p \}$ is a closed immersion of an affine curve. This is a counterexample: if there is any affine open $U \subset \mathbb{P}^2 \setminus \{ p \}$ that contains $E\setminus \{p \}$, then the complement $C:=\mathbb{P}^2\setminus U$ is a plane curve that intersects $E$ only at $p$. Since $p$ is of infinite order, this is not possible.

However, when $Y$ is affine (resp. quasi-affine), for any closed immersion $Y \hookrightarrow X$, the formal scheme $\widehat{X}$ is an affine (resp. quasi-affine) formal scheme.
\qed
 \end{remk}

\begin{defn}\label{defn:exo0}
We call the above $\widehat{X}$ an \emph{exoskin of $Y$}. 
\qed
\end{defn}

We know that fiber products exist in the category of formal schemes (see Fujiwara-Kato \cite[Corollary I-1.3.5, p.271]{FK} or EGA I \cite[Proposition (10.7.3), p.193]{EGA1}). The following shows how exoskins behave under the fiber products of formal schemes. 

\begin{lem}\label{lem:prod completion}
For $i=1,2$, let $Y_i$ be quasi-projective $k$-schemes. Choose closed immersions $Y_i \hookrightarrow X_i$ into smooth $k$-schemes. Let $\widehat{X}_i$ be the completion of $X_i$ along $Y_i$. Regard $Y_1 \times Y_2$ as a closed subscheme of $X_1 \times X_2$, and let $\widehat{X_1 \times X_2}$ be the completion of $X_1 \times X_2$ along $Y_1 \times Y_2$. 

Then we have an isomorphism of formal schemes
\begin{equation}\label{eqn:prod completion 0}
 \widehat{X}_1 \times \widehat{X}_2 \simeq \widehat{X_1 \times X_2}.
 \end{equation}
\end{lem}

\begin{proof}

For $i=1,2$ and for scheme points $y_i \in Y_i \subset X_i$, choose affine open neighborhoods $U_i \subset X_i$ containing $y_i$. Since $U_i$ is affine, the closed subscheme $Y_i \cap U_i$ is also affine, and $\widehat{X}_i |_{Y_i \cap U_i} = \widehat{U}_i$. Since the fiber product is defined by patching the local affine formal schemes, the question of establishing the isomorphism \eqref{eqn:prod completion 0} is local. Hence we may assume $Y_i$ and $X_i$ are affine.

For $i=1,2$, let $Y_i = \Spec (B_i)$, $X_i = \Spec (A_i)$ so that we have natural surjections $A_i \to B_i$. Let $I_i \subset A_i$ be the kernels. Let $A_3:= A_1 \otimes_k A_2$, $B_3:= B_1 \otimes_k B_2$. Then $X_1 \times_k X_2 = \Spec (A_3)$ and $Y_1 \times_k Y_2 = \Spec (B_3)$. 

Note that 
\begin{equation}\label{eqn:prod completion 0-}
B_3  = \frac{A_1}{I_1} \otimes_k \frac{A_2}{I_2} \simeq \frac{ A_1 \otimes_k A_2}{ I_1 \otimes_k A_2 + A_1 \otimes_k I_2},
\end{equation}
and let $I_3:= I_1 \otimes_k A_2 + A_1 \otimes_k I_2$, the denominator of the right hand side of \eqref{eqn:prod completion 0-}. Here, for $1 \leq i \leq 3$, we have $\widehat{X}_i = \Spf (\widehat{A}_i)$, where 
$$
\widehat{A}_i:= \varprojlim_n \frac{A_i}{ I_i ^n}.
$$

The product $\widehat{X}_1 \times \widehat{X}_2$ is given by the formal spectrum of the completed tensor product $\widehat{A}_1 \widehat{\otimes}_k \widehat{A}_2$ computed by
\begin{equation}\label{eqn:prod completion 1}
\varprojlim_{m,n} \frac{A_1}{ I_1 ^m} \otimes_k \frac{ A_2}{I_2 ^n} \simeq \varprojlim_{m,n} \frac{ A_1 \otimes_k A_2}{ I_1 ^m \otimes_k A_2 + A_1 \otimes_k I_2 ^n}.
\end{equation}

For any given pair $m, n \geq 1$ of integers, if $N \geq m+n$, by the binomial theorem we have
\begin{equation}\label{eqn:prod completion 2-1}
I_3 ^N = (I_1 \otimes_k A_2 + A_1 \otimes_k I_2)^N \subset I_1 ^m \otimes_k A_2 + A_1 \otimes_k I_2 ^n .
\end{equation}
On the other hand, for any integer $N \geq 1$, we have
\begin{equation}\label{eqn:prod completion 2-2}
I_1 ^N \otimes_k A_2 + A_1 \otimes_k I_2 ^N \subset (I_1 \otimes_k A_2 + A_1 \otimes_k I_2)^N = I_3 ^N.
\end{equation}

By \eqref{eqn:prod completion 2-1} and \eqref{eqn:prod completion 2-2}, we see that the two systems $\{ I_1 ^m \otimes_k A_2 + A_1 \otimes_k I_2 ^n \}_{m, n \geq 1}$ and $\{ I_3 ^N \}_{N \geq 1}$ define the same topology on $A_3$. Hence the ring \eqref{eqn:prod completion 1} is equal to the ring $\varprojlim_N A_3/ I_3 ^N$. This proves the lemma.
\end{proof}

There is a natural notion of the image of a morphism of formal schemes:
\begin{defn}\label{defn:image}
Let $f: \mathfrak{X} \to \mathfrak{Y}$ be a morphism of noetherian formal schemes. For the canonical morphism $f^{\sharp}: \mathcal{O}_{\mathfrak{Y}} \to f_* \mathcal{O}_{\mathfrak{X}}$ of the sheaves of rings on $|\mathfrak{Y}|$, consider the ideal sheaf $\mathcal{I}:= \ker (f^{\sharp}) \subset \mathcal{O}_{\mathfrak{Y}}$. Let $\mathfrak{Z} \subset \mathfrak{Y}$ be the closed formal subscheme defined by $\mathcal{I}$. This $\mathfrak{Z}$ is called the (formal) scheme theoretic image of $f$.

For a nonempty open subset $U \subset |\mathfrak{X}|$, let $g: \mathfrak{X}|_U \hookrightarrow \mathfrak{X}$ be the open immersion of formal schemes. Let $\iota: \mathfrak{Z} \subset \mathfrak{X}|_U$ be a closed formal subscheme. Then the image of the composite $g \circ \iota: \mathfrak{Z} \hookrightarrow \mathfrak{X}|_U \hookrightarrow \mathfrak{X}$ in the above sense is called the Zariski closure of the formal subscheme $\mathfrak{Z}$ in $\mathfrak{X}$.
\qed
\end{defn}

\medskip

Let $\mathfrak{X}$ be a noetherian formal scheme. Recall (see \cite[\S 2.4, 2.5]{Leo AT} or SGA VI \cite[Expos\'e I, \S 4, p.119]{SGA6}) that a complex $\mathcal{E}$ in the derived category $ \mathcal{D} (\mathfrak{X}):= \mathcal{D} (\mathcal{O}_{\mathfrak{X}})$ of $\mathcal{O}_{\mathfrak{X}}$-modules is called a \emph{perfect complex} on $\mathfrak{X}$ if for every point $x \in  | \mathfrak{X} | $, there exist an open neighborhood $U \subset |\mathfrak{X}|$ of $x$ and a bounded complex $\mathcal{F}$ of locally free $\mathcal{O}_{\mathfrak{X}|_U}$-modules of finite type  together with an isomorphism $\mathcal{F} \overset{\sim}{\to} \mathcal{E}|_{U}$ in $\mathcal{D}(\mathfrak{X}|_U)$. We let $\mathcal{D}_{\perf} (\mathfrak{X})$ be the triangulated subcategory of $\mathcal{D}(\mathfrak{X})$ of perfect complexes on $\mathfrak{X}$. We recall that a perfect complex $\mathcal{F}$ is a \emph{strictly perfect complex} if it is a bounded complex of locally free $\mathcal{O}_{\mathfrak{X}}$-modules of finite type.

A complex that is locally quasi-isomorphic to a bounded above complex of locally free sheaves of finite type is called a \emph{pseudo-coherent complex}. Thus a perfect complex is a pseudo-coherent complex. Under the noetherian assumptions, they have coherent cohomology sheaves. They form the triangulated subcategory $\mathcal{D}_{\rm coh} (\mathfrak{X})$ of $\mathcal{D} (\mathfrak{X})$ as well.

\medskip

We recall the following (originally from SGA VI \cite[Expos\'e IV, \S 2.5, pp.280-281]{SGA6} and \cite[Expos\'e I, Corollaire 5.10, p.138]{SGA6}; see also \cite[Lemma 2.2.1]{PP}):

\begin{cor}\label{cor:coh=perfect}
Let $\mathfrak{X}$ be a regular noetherian formal scheme. 

Then the natural inclusion $\mathcal{D}_{\rm perf} (\mathfrak{X}) \hookrightarrow \mathcal{D}_{\rm coh} (\mathfrak{X})$ from the perfect complexes on $\mathfrak{X}$ to the pseudo-coherent complexes with bounded cohomologies, is an equivalence.

In particular, each pseudo-coherent complex on a regular noetherian formal scheme $\mathfrak{X}$ with bounded cohomologies is a perfect complex. 
\end{cor}

Recall (Fujiwara-Kato \cite[Definition 3.1.3, p.300]{FK}):

\begin{defn}\label{defn:aqc}
Let $\mathfrak{X}$ be a noetherian formal scheme. Let $\mathcal{F}$ be an $\mathcal{O}_{\mathfrak{X}}$-module. We say $\mathcal{F}$ is \emph{adically quasi-coherent} (abbreviated a.q.c.), if it is complete with respect to an ideal $\mathcal{I}$ of definition (see EGA I \cite[Proposition (10.5.4), p.187]{EGA1}), and for any open $U \subset | \mathfrak{X}|$, the quotient $(\mathcal{F}|_U) / ( \mathcal{I} \cdot (\mathcal{F}|_U))$ is a quasi-coherent sheaf on the scheme $(U, \mathcal{O}_U/ \mathcal{I})$. \qed
\end{defn}

We have compiled some needed facts on a.q.c sheaves:

\begin{lem}\label{lem:aqc sum}
Let $\mathfrak{X}$ be a noetherian formal scheme. Then: 
\begin{enumerate}
\item The structure sheaf $\mathcal{O}_{\mathfrak{X}}$ is a.q.c.
\item Any ideal of definition is a.q.c.
\item A sheaf $\mathcal{F}$ being a.q.c. is a local condition.
\item A direct sum of two a.q.c.~sheaves is a.q.c.
\item An $\mathcal{O}_{\mathfrak{X}}$-module $\mathcal{F}$ is coherent if and only if it is a.q.c.~of finite type.
\item A perfect complex $\mathcal{F}$ on $\mathfrak{X}$ that is quasi-isomorphic to a strictly perfect complex is quasi-isomorphic to a bounded complex of a.q.c.~sheaves.
\item Let $f: \mathfrak{X} \to \mathfrak{Y}$ be an affine morphism of noetherian formal schemes, and let $\mathcal{F}$ be an a.q.c.~sheaf on $\mathfrak{X}$. Then $R^q f_* \mathcal{F} = 0$ for all $q \geq 1$.
\end{enumerate}
\end{lem}

\begin{proof}
(1), (2) : see Fujiwara-Kato \cite[Proposition 3.1.4, p.300]{FK}. (3) and (4) follow from definition. (5) : see \cite[Exercise I.3.7, p.320]{FK}. 

Note that by (1), (3), (4), all locally free $\mathcal{O}_{\mathfrak{X}}$-modules of finite type are a.q.c. For (6), since a strictly perfect complex is a bounded complex of locally free $\mathcal{O}_{\mathfrak{X}}$-modules of finite type, and each sheaf that appears in the latter complex is a.q.c., we have proven (6).

(7) : see \cite[Theorem 7.1.1-(2), p.397]{FK}. 
\end{proof}

The last property of Lemma \ref{lem:aqc sum} is an analogue of the corresponding well-known result for affine morphisms between noetherian schemes and quasi-coherent sheaves on them, e.g. in EGA II \cite[Corollaire (5.2.2), p.98]{EGA2}.

\subsection{Cycles on formal schemes}\label{sec:cycles fs}

In \S \ref{sec:cycles fs}, we recall and extend the notion of algebraic cycles on noetherian formal schemes from \cite{Park Tate}. While a good part of the notion of cycles on formal schemes is similar to the case of schemes, we need to consider some additional conditions with respect to ideals of definition. 

\medskip

\begin{defn}\label{defn:int cycle}
Let $\mathfrak{X}$ be a noetherian formal scheme. Let $\mathfrak{Y}$ and $ \mathfrak{Z}$ be two closed formal subschemes of $\mathfrak{X}$ given by their respective ideals $\mathcal{I}$ and $\mathcal{J} \subset \mathcal{O}_{\mathfrak{X}}$. 

The \emph{intersection} $\mathfrak{Y} \cap \mathfrak{Z}$ is defined to be the closed formal subscheme of $\mathfrak{X}$ given by the sum  $\mathcal{I} + \mathcal{J}  \subset \mathcal{O}_{\mathfrak{X}}$ of ideals. 

We say the intersection $\mathfrak{Y} \cap \mathfrak{Z}$ on $\mathfrak{X}$ is \emph{proper}, if either $\mathfrak{Y} \cap \mathfrak{Z}= \emptyset$, or $\mathfrak{Y} \cap \mathfrak{Z} \not = \emptyset$ with the right dimension. In case $\mathfrak{X}$ is equidimensional, the latter means 
$$
{\rm codim}_{\mathfrak{X}}  ( \mathfrak{Y} \cap \mathfrak{Z} ) = {\rm codim}_{\mathfrak{X}} \mathfrak{Y} + {\rm codim}_{\mathfrak{X}} \mathfrak{Z}
$$
in terms of the codimensions.\qed
\end{defn}

We don't work with all closed formal subschemes to construct cycles; we will require a few conditions later (see Definitions \ref{defn:formal cycle} and \ref{defn:HCG}). 

Recall that for a noetherian formal scheme $\mathfrak{X}$, there are two associated topological spaces. The first one is the underlying space $|\mathfrak{X}|$ of the ringed space $\mathfrak{X}$, with the Zariski topology. The second one is the sheaf of rings $\mathcal{O}_{\mathfrak{X}}$; since $\mathfrak{X}$ is noetherian, there exists an ideal of definition (see EGA I \cite[Proposition (10.5.4), p.187]{EGA1}), and the collection of all ideals of definition gives the notion of \emph{open ideals}. When $\mathcal{I}$ is a fixed ideal of definition, the collection $\{ \mathcal{I}^n \}_{n \geq 1}$ also generates the same topology, so that an ideal $\mathcal{J} \subset \mathcal{O}_{\mathfrak{X}}$ is open if it contains $\mathcal{I}^n$ for some $n \geq 1$. For any choice of an ideal of definition $\mathcal{I}$, the underlying topological space $|\mathfrak{X}|$ of the formal scheme $\mathfrak{X}$ is equal to the underlying topological space of the \emph{scheme} $\mathfrak{X}_{\mathcal{I}}=(|\mathfrak{X}|, \mathcal{O}_{\mathfrak{X}}/ \mathcal{I})$ defined by $\mathcal{I}$. 

Both of these topologies influence way we deal with cycles on formal schemes. Normally they do not cause confusions because they are defined on different sets.

\medskip

\begin{defn}\label{defn:formal cycle}
Let $\mathfrak{X}$ be a noetherian formal scheme and let $d \geq 0$ be an integer. Let $\un{z}_d (\mathfrak{X})$ be the free abelian group on the set of integral closed formal subschemes $\mathfrak{Z}$ of $\mathfrak{X}$ of dimension $d$. This may contain some undesirable cycles for our purposes, so we call $\un{z}_d (\mathfrak{X})$ the \emph{naive group} of $d$-cycles on $\mathfrak{X}$. 

The group $z_d (\mathfrak{X})$ of $d$-cycles on $\mathfrak{X}$ is the free abelian group on the set of integral closed formal subschemes $\mathfrak{Z}$ of $\mathfrak{X}$ of dimension $d$, that intersect properly with the subscheme $\mathfrak{X}_{\red}$ given by the largest ideal of definition of $\mathfrak{X}$ (EGA I \cite[Proposition (10.5.4), p.187]{EGA1}). Note that $z_d (\mathfrak{X}) \subset \un{z}_d (\mathfrak{X})$.

When $\mathfrak{X}$ is equidimensional of dimension $d_{\mathfrak{X}}$ and $0 \leq q \leq d _{\mathfrak{X}}$ is an integer, we define the group of cycles of codimension $q$ by $z^q (\mathfrak{X}) := z_{d_{\mathfrak{X}} - q} (\mathfrak{X}).$\qed
\end{defn}

While we defined them for general noetherian formal schemes, we mostly work with the affine and quasi-affine ones because then we can define the associated cycles of coherent sheaves on them. We discuss this point in detail in \S \ref{sec:ass cycle}. Recall that a noetherian quasi-affine formal scheme is just an open formal subscheme of a noetherian affine formal scheme.

\begin{exm}\label{exm:223}
Let $A$ be a noetherian $I$-adically complete ring for an ideal $I \subset A$. By EGA ${\rm III}_1$ \cite[Corollaire (5.1.8), p.495]{EGA3-1}, we have a bijection between (integral) closed subschemes of $\Spec (A)$ and (integral) closed formal subschemes of $\Spf (A)$. 
Hence, we can identify $\un{z}_d (\Spf (A)) = z_d (\Spec (A))$, which is the usual group of $d$-cycles on $\Spec (A)$, while $z_d (\Spf (A))$ is its subgroup consisting of the integral closed subschemes of $\Spec (A)$ that intersect properly with $\Spec (A/ I_0)$ for the largest ideal of definition $I_0$.

Suppose that $\mathfrak{X}$ is a noetherian formal scheme, but 
not a scheme. In this case, closed formal subschemes of $\mathfrak{X}$ given by open ideals do not intersect properly with $\mathfrak{X}_{\red}$. Hence the ideal sheaves that give the integral cycles in $z_d (\mathfrak{X})$ are at least non-open ideals of $\mathcal{O}_{\mathfrak{X}}$.
\qed
\end{exm}

\begin{exm}
Suppose the given formal scheme $\mathfrak{X}$ is actually a scheme, so that the zero ideal as well as nilpotent ideals are ideals of definition. Then every integral closed subscheme $\mathfrak{Z} \subset \mathfrak{X}$ intersects properly with the subschemes associated to ideals of definition, including $\mathfrak{X}$, so that $z_d (\mathfrak{X})$ is equal to the usual group of $d$-dimensional cycles on a scheme.\qed
\end{exm}

 On $z^* (\mathfrak{X})$ we have the following Gysin pull-back property:

\begin{lem}\label{lem:basic Gysin 1}
Let $\mathfrak{X}$ be an equidimensional noetherian formal scheme. Let $\mathcal{I}$ be an ideal of definition of $\mathcal{O}_{\mathfrak{X}}$. Let $\iota: \mathfrak{X}_{\mathcal{I}} \hookrightarrow \mathfrak{X}$ be the closed immersion defined by $\mathcal{I}$. 

Then there is the Gysin pull-back $\iota^*: z^* (\mathfrak{X}) \to z^* (\mathfrak{X}_{\mathcal{I}})$, given by $\mathfrak{Z} \mapsto [ \mathfrak{Z} \cap \mathfrak{X}_{\mathcal{I}}]$, where $[ \mathfrak{Z} \cap \mathfrak{X}_{\mathcal{I}}]$ is the cycle associated to the closed subscheme $\mathfrak{Z} \cap \mathfrak{X}_{\mathcal{I}}$ of the scheme $\mathfrak{X}_{\mathcal{I}}$.
\end{lem}

\begin{proof}
Note that for any ideal of definition $\mathcal{I}$ of $\mathfrak{X}$, the closed subscheme $\mathfrak{X}_{\mathcal{I}}$ of $\mathfrak{X}$ given by $\mathcal{I}$ induces a closed immersion $\mathfrak{X}_{\red} \hookrightarrow \mathfrak{X}_{\mathcal{I}}$ such that $|\mathfrak{X}_{\red} | = |\mathfrak{X}_{\mathcal{I}}| = |\mathfrak{X}|$ with $\dim \ \mathfrak{X}_{\red} = \dim \ \mathfrak{X}_{\mathcal{I}}$. The cycles in $z ^* (\mathfrak{X})$ intersect properly with $\mathfrak{X}_{\mathcal{I}}$, too, so we have the lemma. 
\end{proof}

\begin{remk}
In Lemma \ref{lem:basic Gysin 1}, that $\mathfrak{X}_{\mathcal{I}}$ is a noetherian \emph{scheme} is crucial, so that we have the associated cycle $[ \mathfrak{Z} \cap \mathfrak{X}_{\mathcal{I}}]$ on $\mathfrak{X}_{\mathcal{I}}$. 

In general, associated cycles on noetherian \emph{formal} schemes are a bit more complicated than one may hope. We discuss more on this issue in \S \ref{sec:ass cycle} below.
\qed
\end{remk}

\begin{remk}
We noted in Example \ref{exm:223} that cycles in $z_d (\Spf (A))$ are essentially cycles in $z_d (\Spec (A))$ with some extra proper intersection condition with $\Spec (A/I_0)$. One may ask then why we use the formal schemes: at first sight one may think the discussion is unnecessary, because we can work with the cycles in $z_{d, \{ \Spec (A/ I_0) \}} (\Spec (A))$, as is classically written.

We remark here that some other differences do appear later when we consider higher level cycles, i.e. higher Chow cycles on the formal schemes of the form $\mathfrak{X} \times \square^n$ with $\square:= \mathbb{P}^1 \setminus \{ 1 \}$, because restricted formal power series, not polynomials, enter into the picture. See Remark \ref{remk:2.5.2}.
\qed
\end{remk}

\subsection{The associated cycles}\label{sec:ass cycle}
We discuss cycles associated to coherent sheaves on noetherian formal schemes. We define them on noetherian affine and quasi-affine formal schemes only.

\subsubsection{Irreducible components}
A situation quite different from the case of schemes is regarding the notion of \emph{irreducible components} of noetherian formal schemes. For noetherian schemes, the topological irreducible components essentially allow us to define the associated cycles of coherent sheaves. 
They naturally carry the induced reduced integral closed subscheme structures.

However for noetherian \emph{formal} schemes, just resorting to the topological irreducible components of the underlying noetherian topological spaces, does not give a good notion of irreducible components of the formal schemes. At least for noetherian \emph{affine} formal schemes, we can give the following natural notion:

\begin{defn}\label{defn:irred cpnt}
Let $\mathfrak{X}=\Spf (A)$ be a noetherian affine formal scheme. The \emph{irreducible components} of the formal scheme $\mathfrak{X}$ are defined to be the integral closed formal subschemes $\mathfrak{X}_i=\Spf (A/ p_i)$, where $ p_i\subset A$ are the minimal prime ideals. \qed
\end{defn}

\begin{remk}\label{remk:Conrad}
For a general noetherian formal scheme $\mathfrak{X}$, which is not affine, the author does not know whether there is a good notion of its ``irreducible components" even when $\mathfrak{X}$ is defined over a field. 

Ideally, if there is a way to define a ``normalization" of a noetherian formal scheme $\mathfrak{X}$, one may be able to define the irreducible components as follows: for a ``normalization" morphism $\pi: \tilde{\mathfrak{X}} \to \mathfrak{X}$, define the irreducible components of $\mathfrak{X}$ by taking the images (see Definition \ref{defn:image}) of the topological connected components of $\tilde{\mathfrak{X}}$.

For instance, in the special case when $\mathfrak{X}$ is over the valuation ring $R$ of a complete discrete valued field $k$ with a nontrivial non-archimedean norm, B. Conrad \cite{Conrad cpnt} proved that there exists such a normalization. The main focus of \emph{ibid.}~was to define normalizations and irreducible components for rigid analytic spaces, but the proof is given there for formal schemes over $R$ as well. For more stuffs pertaining to this method on formal schemes, one can also read, e.g. J. Nicaise \cite[Definition 2.27 - Lemma 2.29]{Nicaise}. 
\qed
\end{remk}

Due to lack of a suitable normalization procedure for noetherian formal schemes in general, for a while we stick to the affine case to define the associated cycles of coherent sheaves in \S \ref{sec:ass cycle affine}, and then we extend this notion a bit to the quasi-affine case in \S \ref{sec:ass cycle quasi-affine}.

\subsubsection{The case of affine formal schemes}\label{sec:ass cycle affine}

To define the associated cycles of coherent sheaves on noetherian \emph{affine} formal schemes, we use the following notions of localizations and residue fields at integral closed formal subschemes:

\begin{defn}\label{defn:local ring formal}
Let $\mathfrak{X}=\Spf (A) $ be a noetherian affine formal scheme, and let $\mathfrak{Y} \subset \mathfrak{X}$ be a nonempty integral closed formal subscheme. This $\mathfrak{Y}$ is given by a prime ideal $P \subset A$ so that $\mathfrak{Y}= \Spf (B)$ for the integral domain $B= A/P$.

Define the \emph{the local ring of $\mathfrak{X}$ at $\mathfrak{Y}$} to be the localization $\mathcal{O}_{\mathfrak{X}, \mathfrak{Y}} := A_P$, with the maximal ideal $ \mathfrak{M}_{\mathfrak{X}, \mathfrak{Y}}= P \cdot A_P.$

We define the residue field $\kappa (\mathfrak{Y}):= \kappa ( \mathcal{O}_{\mathfrak{X}, \mathfrak{Y}}) =  \mathcal{O}_{\mathfrak{X}, \mathfrak{Y}} / \mathfrak{M}_{\mathfrak{X}, \mathfrak{Y}} = \kappa (P).$\qed
\end{defn}

\medskip

We essentially used the structure of the associated scheme $\Spec (A)$ for the formal scheme $\Spf (A)$ in the above. The integral closed formal subscheme $\mathfrak{Y}$ is not uniquely determined by its topological generic point in $|\mathfrak{Y}| \subset | \mathfrak{X}|$ in general. The local ring $\mathcal{O}_{\mathfrak{X}, \mathfrak{Y}}$ of Definition \ref{defn:local ring formal} is not in general the stalk of the structure sheaf $\mathcal{O}_{\mathfrak{X}}$ at the topological generic point of $\mathfrak{Y}$ either. See Remarks \ref{remk:local ring formal 1} and \ref{remk:local ring formal 2} below.

\begin{remk}\label{remk:local ring formal 1}
The topological generic point $\mathfrak{y}$ of $|\mathfrak{Y}|$ does not uniquely determine $\mathfrak{Y}$ in general. 
For instance, consider $\mathfrak{X} = \Spf (k[[t_1, t_2]])$ with the $(t_1, t_2)$-adic topology on $k[[t_1, t_2]]$. Take integral closed formal subschemes $\mathfrak{Y}_1 = \Spf (k[[t_1, t_2]]/ (t_2))$ and $\mathfrak{Y}_2= \Spf (k[[t_1, t_2]]/ (t_1))$. 

All of $\mathfrak{X}, \mathfrak{Y}_1, \mathfrak{Y}_2$ have the common underlying topological spaces given by the singleton set $\Spec (k)$. Note that the ideals $(t_1)$ and $ (t_2) \subset k[[t_1, t_2]]$ are non-open prime ideals in the $(t_1, t_2)$-adic topology. \qed
\end{remk}

\begin{remk}\label{remk:local ring formal 2}
In the situation of Definition \ref{defn:local ring formal}, we remark that when $|\mathfrak{Y}|$ is irreducible as a topological space, the local ring $\mathcal{O}_{\mathfrak{X}, \mathfrak{y}}$ and the residue field $k(\mathfrak{y})$ at the topological generic point $\mathfrak{y}$ of $|\mathfrak{Y}|$ are different from $\mathcal{O}_{\mathfrak{X}, \mathfrak{Y}}$ and $\kappa (\mathfrak{Y})$ in general.

Consider $\mathfrak{X}=\Spf (k[[t_1, t_2]])$ and $\mathfrak{Y}= \Spf (k[[t_1, t_2]]/ (t_2))$. Let $A:= k[[t_1, t_2]]$ and $P= (t_1)$. Here the local ring at $\mathfrak{Y}$ is $\mathcal{O}_{\mathfrak{X}, \mathfrak{Y}} = A_P= k[[t_1, t_2]]_{P}$ with the maximal ideal $(t_1)A_P$ and its residue field $\kappa (\mathfrak{Y})$ is $A_P/ (t_1)A_P$, which is huge with $\dim_k \ \kappa (\mathfrak{Y})= \infty$.  

On the other hand, the generic point $\mathfrak{y}$ is the singleton $\Spec (k)$ so that the local ring at $\mathfrak{y}$ is $\mathcal{O}_{\mathfrak{X}, \mathfrak{y}} = k[[t_1, t_2]]$ with the maximal ideal $(t_1, t_2)$, and its residue field is $\kappa (\mathfrak{y}) = k[[t_1, t_2]]/ (t_1, t_2) = k$. \qed
\end{remk}

With the notion of local rings in Definition \ref{defn:local ring formal}, we define the associated cycles of coherent $\mathcal{O}_{\mathfrak{X}}$-modules on a noetherian affine formal scheme $\mathfrak{X}= \Spf (A)$, essentially exploiting the associated scheme $\Spec (A)$:

\begin{defn}\label{defn:ass cycle}
Let $\mathfrak{X}=\Spf (A)$ be a noetherian affine formal scheme, and let $\mathcal{F}$ be a coherent $\mathcal{O}_{\mathfrak{X}}$-module. This $\mathcal{F}$ is given by $M^{\Delta}$ for a coherent $A$-module $M$ (see EGA I \cite[Proposition (10.10.2), p.201]{EGA1}).

For an integral closed formal subscheme $\mathfrak{Z} \subset \mathfrak{X}$ given by a prime ideal $P \subset A$ with $\mathfrak{Z} = \Spf (A/P)$, the localization $M_P$ is an $\mathcal{O}_{\mathfrak{X}, \mathfrak{Z}}$-module (see Definition \ref{defn:local ring formal}). Let $\mathcal{F}_{\mathfrak{Z}} := M_P$. Define
$$
n_{\mathfrak{Z}}= n_{\mathfrak{Z}} (\mathcal{F}) :=  \tuborg 0 & \mbox{ if } \mathcal{F}_{\mathfrak{Z}} \mbox{ has an infinite } \mathcal{O}_{\mathfrak{X}, \mathfrak{Z}}\mbox{-length}, \\
\ell_{\mathcal{O}_{\mathfrak{X},\mathfrak{Z}}} (\mathcal{F}_{\mathfrak{Z}}) & \mbox{ if the length is finite.} \sluttuborg
$$

Using this, define the cycle associated to $\mathcal{F}$ to be the formal sum
\begin{equation}\label{eqn:cycle sum}
[ \mathcal{F}] := \sum_{\mathfrak{Z}} n_{\mathfrak{Z}} \cdot \mathfrak{Z} \in \un{z}_* (\mathfrak{X})
\end{equation}
for some integers $n_{\mathfrak{Z}}>0$, where $\mathfrak{Z}$ runs over all integral closed formal subschemes of $\mathfrak{X}$. 

Note that at this moment, the associated cycles could still be defined by open prime ideals as well, so that it is not necessarily in $z_* (\mathfrak{X})$. Recall Definition \ref{defn:formal cycle} for the distinction of these two groups. \qed
\end{defn}

The above Definition \ref{defn:ass cycle} requires us to answer the following question:

\begin{lem}\label{lem:ass cycle}
Let $\mathfrak{X}= \Spf (A) $ be a noetherian affine formal scheme, and let $\mathcal{F}$ be a coherent $\mathcal{O}_{\mathfrak{X}}$-module. Then the sum \eqref{eqn:cycle sum} is finite, so that the cycle $[\mathcal{F}]$ is indeed defined in $\un{z}_* (\mathfrak{X})$.
\end{lem}

\begin{proof}
The coherent sheaf $\mathcal{F}$ is given by $M^{\Delta}$ for a finitely generated $A$-module $M$ (EGA I \cite[Proposition (10.10.2), p.201]{EGA1}). The assertion then follows from the elementary fact in commutative algebra that for a noetherian commutative ring $A$ with unity, the set of associated primes of a finitely generated $A$-module $M$ is a finite set (H. Matsumura \cite[Theorems 6.1, 6.5, pp.38-39]{Matsumura}).
\end{proof}

\begin{remk}
The cycles on a noetherian affine formal scheme $\Spf (A)$ are essentially cycles on the scheme $\Spec (A)$. Thus, we can describe $[\mathcal{F}]$ of \eqref{eqn:cycle sum} also via a Jordan-H\"older type decreasing filtration
$$
0=: \mathcal{F}_{q+1} \subsetneq \mathcal{F}_q \subsetneq \cdots \subsetneq \mathcal{F}_0 = : \mathcal{F},
$$
where each quotient $\mathcal{F}_i / \mathcal{F}_{i+1}$ is isomorphic to some $\mathcal{O}_{\mathfrak{Z}}$ for an integral $\mathfrak{Z}$ (H. Matsumura \cite[Theorem 6.4, p.39]{Matsumura}). The construction of the filtration is done inductively, using irreducible components as in Definition \ref{defn:irred cpnt}. The cycle $[\mathcal{F}]$ is the sum of the above $\mathfrak{Z}$.

This filtration is not unique in general, but the (unordered) collection of successive quotients is. See \cite[Lemma 01YF]{stacks}. \qed
\end{remk}

Definition \ref{defn:ass cycle} readily extends to all perfect complexes on $\mathfrak{X}$:

\begin{cor}\label{cor:perfect ass cycle}
Let $\mathfrak{X}$ be a noetherian affine formal scheme and let $\mathcal{E} \in \mathcal{D}_{\perf} (\mathfrak{X})$ be a perfect complex on $\mathfrak{X}$. Then the association 
$$
\mathcal{E} \mapsto [\mathcal{E}]:= \sum_{i\in \mathbb{Z}} (-1)^i [  \mathcal{H}^i(\mathcal{E})] \in \un{z}_* (\mathfrak{X})
$$
gives a well-defined cycle.
\end{cor}

\begin{proof}
Since $\mathcal{E}$ is a perfect complex, the cohomology sheaves $\mathcal{H}^i (\mathcal{E})$ are all coherent, and there are at most finitely many indices $i \in \mathbb{Z}$ for which $\mathcal{H}^i (\mathcal{E})$ are possibly nonzero. Since these coherent sheaves are well-defined in the quasi-isomorphism class of $\mathcal{E}$ and their associated cycles are also well-defined by Lemma \ref{lem:ass cycle}, the assertion now follows. 
\end{proof}

In the above, the dimension $d$ part will be denoted by $[\mathcal{E}]_d$. If $\mathfrak{X}$ is equidimensional, then we can also define the codimension $q$-cycle $[\mathcal{E}]^q$.

\subsubsection{The case of quasi-affine formal schemes}\label{sec:ass cycle quasi-affine} 
We generalize the discussion of \S \ref{sec:ass cycle affine} to quasi-affine formal schemes.

A noetherian quasi-affine formal scheme is of the form $\mathfrak{X}|_U$ for an open subset $U \subset |\mathfrak{X}|$ of a noetherian affine formal scheme $\mathfrak{X}= \Spf (A)$. We have $|\mathfrak{X}| = | \Spec (A/I_0)|$ for an ideal of definition $I_0 \subset A$. The space $|\mathfrak{X}|$ is a closed subset of $|\Spec (A)|$ and its topology coincides with the induced topology from $|\Spec (A)|$.

Hence for the open $U \subset |\mathfrak{X}|$, there is an open subset $\tilde{U} \subset | \Spec (A)|$ such that $\tilde{U} \cap | \mathfrak{X}| = U$. For some ideal $J \subset A$, we have $\tilde{U} = |\Spec (A)| \setminus V(J)$, where $V(J)$ is the set of prime ideals of $A$ containing $J$.

Since the open subset $\tilde{U} \subset |\Spec (A)|$ is quasi-compact, there exist finitely many affine open subsets $\tilde{U}_i = \Spec (B_i) \subset |\Spec (A)|$ such that $\bigcup_i \tilde{U}_i = \tilde{U}$. For each $i$, let $U_i:= \tilde{U}_i \cap |\mathfrak{X}|$. This is an affine open subset of $|\mathfrak{X}|$. Thus it induces the noetherian affine open formal subschemes $\mathfrak{X}|_{U_i}$ of $\mathfrak{X}|_U$. 

\medskip

Let $\mathcal{F}$ be a coherent $\mathcal{O}_{\mathfrak{X}|_U}$-module. For each $i$, the restriction $\mathcal{F}_{U_i}:= \mathcal{F}|_{U_i}$ is a coherent $\mathcal{O}_{\mathfrak{X}|_{U_i}}$-module on the affine formal scheme $\mathfrak{X}|_{U_i}$. Hence by the affine case in \S \ref{sec:ass cycle affine}, we have the associated cycle $[\mathcal{F}_{U_i}] \in \un{z}_* (\mathfrak{X}|_{U_i})$ given by a formal finite sum of the form $m_{j}  \cdot \mathfrak{Z}_{j}$, where $m_j \in \mathbb{Z}_{>0}$ and $\mathfrak{Z}_{j}$ is an integral closed formal subscheme of $\mathfrak{X}|_{U_i}$ that corresponds to a prime ideal of $B_i$. Note that the closure of $\mathfrak{Z}_j$ in $\mathfrak{X}=\Spf (A)$ uniquely determines an integral cycle on $\mathfrak{X}$ and one on $\mathfrak{X}|_U$ as well. 

One little issue is whether these cycles $[ \mathcal{F}_{U_i}]$ over the indices $i$ are compatible; more precisely, for a different affine open subset $U_{i'}$, write $[\mathcal{F}_{U_{i'}}]= \sum m_j ' \cdot \mathfrak{Z}_j '$. Suppose $\mathfrak{Z}_j | _{U_i} \subset \mathfrak{X} |_{U_i}$ and $\mathfrak{Z}'_j |_{U_{i'}} \subset \mathfrak{X}|_{U_{i'}}$ are both nonempty, and they give the same restriction $\mathfrak{Z}_j|_{U_i \cap U_{i'}} = \mathfrak{Z}_j' |_{U_i \cap U_{i'}}$ on $\mathfrak{X}|_{U_i \cap U_{i'}}$. Then we wonder whether $m_j = m_j'$. 

Since $U_i \cap U_{i'}$ is affine, so is $\mathfrak{X }|_{U_i \cap U_{i'}}$. 
Since $\mathcal{F}_{U_i}|_{U_i \cap U_{i'}} =\mathcal{F}|_{U_i \cap U_{i'}}= \mathcal{F}_{U_{i'}}|_{U_i \cap U_{i'}}$ as $\mathcal{O}_{\mathfrak{X}|_{U_i \cap U_{i'}}}$-modules, under the restriction maps
$$
\un{z}_* (\mathfrak{X}|_{U_i}) \to \un{z}_* (\mathfrak{X}|_{U_i \cap U_{i'}}) \leftarrow \un{z}_* (\mathfrak{X}|_{U_{i'}}),
$$
we have the equalities of cycles $[\mathcal{F}|_{U_i}]|_{U_i \cap U_{i'}} = [ \mathcal{F} |_{U_i \cap U_{i'}}] = [ \mathcal{F}|_{U_{i'}}]|_{U_i \cap U_{i'}}.$ This implies $m_j = m_j' $.

Define the cycle $[\mathcal{F}] \in \un{z}_* (\mathfrak{X}|_U)$ to be the sum of $m  \cdot \mathfrak{Z}$ over all integral closed formal subscheme $\mathfrak{Z}\subset \mathfrak{X}|_U$ such that the restriction $\mathfrak{Z}|_{U_i}$ to $\mathfrak{X}|_{U_i}$ for some $i$, is nonempty and $m \cdot \mathfrak{Z}|_{U_i}$ is a term of $[\mathcal{F}_{U_i}]$. Since each $[\mathcal{F}_{U_i}]$ is a finite sum, and there are only finitely many affine open $U_i$, the sum $[\mathcal{F}]$ is finite. 

\medskip

This discussion generalizes to perfect complexes, too:

\begin{prop}\label{prop:perfect coherent ass quasi-affine}
For a perfect complex $\mathcal{F}$ on a noetherian quasi-affine formal scheme $\mathfrak{X}|_U$, there exists a well-defined associated cycle $[\mathcal{F}] \in \un{z}_* (\mathfrak{X}|_U)$ given by
$$
[\mathcal{F}] = \sum_{j} (-1)^j [\mathcal{H}^j (\mathcal{F})] \in \un{z}_* (\mathfrak{X}|_U).
$$
\end{prop}

\begin{proof}

Because $\mathcal{F}$ is a perfect complex, the cohomology sheaves are well-defined in the quasi-isomorphism class of $\mathcal{F}$, while there are only finitely many possibly nontrivial cohomology sheaves $\mathcal{H}^j (\mathcal{F})$, and each such coherent sheaf has its well-defined associated cycle $[\mathcal{H}^j (\mathcal{F})]$ by the above discussion.
\end{proof}

\subsection{Intersection products}\label{sec:int prod}
Let $\mathfrak{X}$ be a regular noetherian quasi-affine formal scheme and let $\mathfrak{Y}_1, \mathfrak{Y}_2 \subset \mathfrak{X}$ be closed formal subschemes that intersect properly with each other (Definition \ref{defn:int cycle}). Define their intersection product using a Serre-type formula by
\begin{equation}\label{eqn:Serre product}
[\mathfrak{Y}_1] . [\mathfrak{Y}_2] := [\mathcal{O}_{\mathfrak{Y}_1} \otimes_{\mathcal{O}_{\mathfrak{X}}} ^{\mathbf{L}} \mathcal{O}_{\mathfrak{Y}_2}],
\end{equation}
which is the cycle associated to the perfect complex given by the derived tensor product. Here, the regularity of $\mathfrak{X}$ shows that the coherent $\mathcal{O}_{\mathfrak{X}}$-modules $\mathcal{O}_{\mathfrak{Y}_i}$ are perfect complexes on $\mathfrak{X}$, and so is their derived tensor product (Corollary \ref{cor:coh=perfect}).

One may express \eqref{eqn:Serre product} also as
$$
[\mathfrak{Y}_1] . [\mathfrak{Y}_2] :=  \sum_{\mathfrak{Z}} e (\mathfrak{Y}_1 \cap \mathfrak{Y}_2, \mathfrak{Z}) \cdot \mathfrak{Z},
$$
where $\mathfrak{Z}$ runs over all integral closed formal subschemes of $\mathfrak{Y}_1 \cap \mathfrak{Y}_2$, and $e(\mathfrak{Y}_1 \cap \mathfrak{Y}_2, \mathfrak{Z})$ is the intersection number defined to be the alternating sum of the lengths of the higher Tor sheaves of the pair $(\mathcal{O}_{\mathfrak{Y}_1}, \mathcal{O}_{\mathfrak{Y}_2})$ of $\mathcal{O}_{\mathfrak{X}}$-modules at each $\mathfrak{Z}$. If any of the lengths of the Tor sheaves are infinite, this intersection number is defined to be $0$. 

\medskip

The above \eqref{eqn:Serre product} extends to cycles associated to pairs of coherent $\mathcal{O}_{\mathfrak{X}}$-sheaves:

\begin{defn}\label{defn:sheaf proper int}
Let $\mathfrak{X}$ be a regular noetherian quasi-affine formal scheme. Let $\mathcal{F}_1$ and $ \mathcal{F}_2$ be coherent $\mathcal{O}_{\mathfrak{X}}$-modules.

We say that \emph{$\mathcal{F}_1$ and $\mathcal{F}_2$ intersect properly}, if for each pair $(\mathfrak{Z}_1, \mathfrak{Z}_2)$ of the integral cycle components $\mathfrak{Z}_1$ and $\mathfrak{Z}_2$ from the associated cycles $[\mathcal{F}_1]$ and $[\mathcal{F}_2]$ (see Definition \ref{defn:ass cycle}), respectively, they intersect properly with each other in the sense of Definition \ref{defn:int cycle}.\qed
\end{defn}

\begin{lem}\label{lem:sheaf proper int}
Let $\mathfrak{X}$ be a regular noetherian quasi-affine formal scheme. Suppose that coherent $\mathcal{O}_{\mathfrak{X}}$-modules $\mathcal{F}_1$ and $\mathcal{F}_2$ intersect properly. Then we have
\begin{equation}\label{eqn:Serre product2}
[\mathcal{F}_1] . [\mathcal{F}_2] := [ \mathcal{F}_1 \otimes_{\mathcal{O}_{\mathfrak{X}}} ^{\mathbf{L}} \mathcal{F}_2],
\end{equation}
extending \eqref{eqn:Serre product}.
\end{lem}

\begin{proof}
By the constructions of the associated cycles in \S \ref{sec:ass cycle quasi-affine} in the quasi-affine case via the finite affine open cover, it is enough to prove the lemma when $\mathfrak{X}$ is affine.

\medskip

\textbf{Step 1:} 
We first consider the special case when $\mathcal{F}_2= \mathcal{O}_{\mathfrak{Z}}$ for an integral closed formal subscheme $\mathfrak{Z}$. We have a decreasing filtration of coherent $\mathcal{O}_{\mathfrak{X}}$-modules on $\mathcal{F}_1$
\begin{equation}\label{eqn:filtration F_1}
0 = \mathcal{G}_{r+1} \subsetneq \mathcal{G}_{r} \subsetneq \cdots \subsetneq \mathcal{G}_1 \subsetneq \mathcal{G}_0 = \mathcal{F}_1,
\end{equation}
such that $\mathcal{G}_i / \mathcal{G}_{i+1} \simeq \mathcal{O}_{\mathfrak{Z}_i}$ as $\mathcal{O}_{\mathfrak{X}}$-modules for integral closed formal subschemes $\mathfrak{Z}_i$ for $0 \leq i \leq r$. They induce short exact sequences of coherent $\mathcal{O}_{\mathfrak{X}}$-modules
\begin{equation}\label{eqn:ses F_1}
0 \to \mathcal{G}_{i+1} \to \mathcal{G}_i \to \mathcal{O}_{\mathfrak{Z}_i} \to 0.
\end{equation}
Apply the derived tensor $ - \otimes_{\mathcal{O}_{\mathfrak{X}}} ^{\mathbf{L}} \mathcal{F}_2$ (which is exact) to obtain short exact sequences of complexes
\begin{equation}\label{eqn:derived L F_2}
0 \to \mathcal{G}_{i+1} \otimes_{\mathcal{O}_{\mathfrak{X}}} ^{\mathbf{L}} \mathcal{F}_2 \to \mathcal{G}_i  \otimes_{\mathcal{O}_{\mathfrak{X}}} ^{\mathbf{L}} \mathcal{F}_2  \to \mathcal{O}_{\mathfrak{Z}_i}  \otimes_{\mathcal{O}_{\mathfrak{X}}} ^{\mathbf{L}} \mathcal{F}_2  \to 0.
\end{equation}
Taking the associated Tor long exact sequences and taking the alternating sums of the associated lengths, we have the equalities of the associated cycles (\cite[Lemma 02QZ]{stacks})
\begin{equation}\label{eqn:exact sum cycle}
[ \mathcal{G}_i  \otimes_{\mathcal{O}_{\mathfrak{X}}} ^{\mathbf{L}} \mathcal{F}_2  ] = [\mathcal{O}_{\mathfrak{Z}_i}  \otimes_{\mathcal{O}_{\mathfrak{X}}} ^{\mathbf{L}} \mathcal{F}_2] + [ \mathcal{G}_{i+1} \otimes_{\mathcal{O}_{\mathfrak{X}}} ^{\mathbf{L}} \mathcal{F}_2].
\end{equation}
Taking the sum of \eqref{eqn:exact sum cycle} over all $0 \leq i \leq r$, after telescoping cancellations, we obtain (note $\mathcal{G}_0 = \mathcal{F}_1$, $\mathcal{G}_{r+1} =0$, and $\mathcal{F}_2 = \mathcal{O}_{\mathfrak{Z}}$)
$$
[ \mathcal{F}_1  \otimes_{\mathcal{O}_{\mathfrak{X}}} ^{\mathbf{L}}  \mathcal{F}_2] = \sum_{i=0} ^r [ \mathcal{O}_{\mathfrak{Z}_i}  \otimes_{\mathcal{O}_{\mathfrak{X}}} ^{\mathbf{L}}   \mathcal{O}_{\mathfrak{Z}}]= \sum_{i=0} ^r [ \mathfrak{Z}_i] \cdot [\mathfrak{Z}] = [\mathcal{F}_1] \cdot [\mathcal{F}_2] ,
$$
because $[\mathcal{F}_1 ] =  \sum_{i=0} ^r [ \mathfrak{Z}_i] $ from \eqref{eqn:filtration F_1}. This proves \eqref{eqn:Serre product2} in this case.

\medskip

\textbf{Step 2:} In case $\mathcal{F}_1$ is of the form $\mathcal{O}_{\mathfrak{Z}}$, by symmetry we also have the identity \eqref{eqn:Serre product2}.

\medskip

\textbf{Step 3:} In the general case when both $\mathcal{F}_1$ and $\mathcal{F}_2$ are coherent $\mathcal{O}_{\mathfrak{X}}$-modules, we take a filtration of $\mathcal{F}_1$ as in \eqref{eqn:filtration F_1}. This induces short exact sequences of the form \eqref{eqn:ses F_1}. Applying the derived tensor $ - \otimes_{\mathcal{O}_{\mathfrak{X}}} ^{\mathbf{L}} \mathcal{F}_2$, we obtain short exact sequences of complexes \eqref{eqn:derived L F_2}. Taking the Tor long exact sequences and taking the associated lengths, we deduce the equalities of the associated cycles \eqref{eqn:exact sum cycle}. Taking the sum over all $0 \leq i \leq r$, after cancellations, we deduce \eqref{eqn:Serre product2}. 
\end{proof}

\begin{remk}
We may further extend it to pairs involving perfect complexes up to certain extent. As such generalities are not needed in this article, we shrink their discussions.\qed
\end{remk}

\subsection{Special pull-backs and push-forwards}\label{sec:pb pf}

Recall that for (higher) Chow cycles on schemes, we always have pull-backs for flat morphisms and push-forwards for proper morphisms (see S. Bloch \cite{Bloch HC} and W. Fulton \cite{Fulton}). When we work with formal schemes, there are some technical difficulties in general. 

In \S \ref{sec:pb pf}, we discuss some special cases where flat pull-backs and finite push-forwards do exist. While it may sound a bit restrictive at first, this is good enough for our purposes in this paper.

\subsubsection{Special flat pull-backs}\label{sec:special flat pb}

When $f: Y_1 \to Y_2$ is a flat morphism of $k$-schemes, there are natural flat pull-back morphisms $f^*: z^* (Y_2) \to z^* (Y_1)$ of Chow cycles (W. Fulton \cite[\S 1.7, p.18]{Fulton}) and $f^*: z^* (Y_2, n) \to z^* (Y_1, n)$ of higher Chow cycles (S. Bloch \cite[Proposition (1.3)]{Bloch HC}).

When $f: \mathfrak{X} \to \mathfrak{Y}$ is a flat morphism of formal schemes (see Fujiwara-Kato \cite[\S 4.8, p.342]{FK}), unfortunately it is unclear if we always  have $f^*: \un{z}^*  (\mathfrak{Y}) \to \un{z}^* (\mathfrak{X})$. For instance when $f: \Spf (A) \to \Spf (B)$ is flat, in general there is no reason to believe that there exists an associated flat morphism $\tilde{f}: \Spec (A) \to \Spec (B)$, while our cycles are essentially those on $\Spec (A)$ and $\Spec (B)$. 

Nevertheless, a good news is that, we need flat pull-backs for only a few specific types of flat morphisms, and for them the flat pull-backs do exist:

\begin{lem}\label{lem:pre flat pb0}
Consider one of the following types of flat morphisms $f: \mathfrak{Y}_1 \to \mathfrak{Y}_2$ of equidimensional noetherian quasi-affine formal $k$-schemes:
\begin{enumerate}
\item [(I)] For a nonempty open subset $U \subset |\mathfrak{X}|$, the open immersion $f: \mathfrak{X}|_U \hookrightarrow \mathfrak{X}$.
\item [(II)] For an ideal sheaf $\mathcal{I} \subset \mathcal{O}_{\mathfrak{X}}$, let $\mathfrak{X}_{\mathcal{I}}$ be the further completion of $\mathfrak{X}$ by the ideal $\mathcal{I}$. This gives a flat morphism $f: \mathfrak{X}_{\mathcal{I}} \to \mathfrak{X}$ (EGA I \cite[Corollaire (10.8.9), p.197]{EGA1} or H. Matsumura \cite[Theorem 8.8, p.60]{Matsumura}).
\item [(III)] The projection morphism $f: \mathfrak{X}_1 \times_k \mathfrak{X}_2 \to \mathfrak{X}_2$.
\item [(IV)] The morphism ${\rm Id}_{\mathfrak{X}} \times g: \mathfrak{X} \times B_1 \to \mathfrak{X} \times B_2$, for a flat morphism $g: B_1 \to B_2$ of $k$-schemes of finite type.
\end{enumerate}
Then we have the flat pull-back $f^*: \un{z}^q (\mathfrak{Y}_2) \to \un{z} ^q (\mathfrak{Y}_1)$ of the naive cycle groups. 
\end{lem}

\begin{proof}
For an integral cycle $\mathfrak{Z} \in \un{z}^q (\mathfrak{Y}_2)$, define $f^* (\mathfrak{Z}) = [ f^* \mathcal{O}_{\mathfrak{Z}}]$, the cycle associated to the coherent sheaf $f^* \mathcal{O}_{\mathfrak{Z}}$ in the sense of Definition \ref{defn:ass cycle}.

The question is whether $f^*( \mathfrak{Z})$ has the right codimension. 

For (I), it is apparent.

For (II) $\sim$ (IV), the statements are local, so we may assume all formal schemes are affine.

For (II), let $\mathfrak{X} = \Spf (A)$. The ideal sheaf $\mathcal{I}$ is given by an ideal $I \subset A$, and let $\widehat{A}$ be the completion of $A$ along $I$. The natural homomorphism $A \to \widehat{A}$ is flat, so that the morphism $\Spec (\widehat{A}) \to \Spec (A)$ is flat. The assertion then follows from the usual flat pull-backs $z^q (\Spec (A)) \to z^q (\Spec (\widehat{A}))$ of cycles on schemes.

For (III), write $\mathfrak{X}_i=\Spf (A_i) $, where $A_i$ is complete with respect to an ideal $I_i \subset A_i$. The projection morphism $f$ has the associated flat homomorphism $A_2 \to A_1 \widehat{\otimes}_k A_2$ given by $a \mapsto 1 \otimes a$. This gives the flat morphism $\Spec (A_1\widehat{\otimes}_k A_2) \to \Spec (A_2)$. Since $\Spf (A_1 \widehat{\otimes}_k A_2) = \mathfrak{X}_1 \times_k \mathfrak{X}_2$, the assertion then follows from the classical case of flat pull-backs of cycles on schemes.

For (IV), let $\mathfrak{X} = \Spf (A)$. Then ${\rm Id}_{\mathfrak{X}} \times g$ has the associated flat morphism ${\rm Id}_{\Spec (A)} \times g: \Spec (A) \times B_1 \to \Spec (A) \times B_2$, and the assertion follows from the classical case of flat pull-backs of cycles on schemes.
\end{proof}

\begin{lem}\label{lem:pull-back perfect}
Let $f: \mathfrak{X} \to \mathfrak{Y}$ be a flat morphism of equidimensional noetherian quasi-affine formal $k$-schemes, of type ${\rm (I)}$ $\sim$ ${\rm (IV)}$. Let $\mathcal{F}$ be a coherent sheaf on $\mathfrak{Y}$.

 Then we have the equality of cycles
\begin{equation}\label{eqn:pb perfect}
[f^* (\mathcal{F})] = f^* [ \mathcal{F}] \ \mbox{ in } \un{z}^* (\mathfrak{X}).
\end{equation}
\end{lem}

\begin{proof}
As before we may assume $\mathfrak{X}$ and $\mathfrak{Y}$ are both affine.
We prove the lemma in two steps.

\textbf{Step 1:} If $\mathcal{F}= \mathcal{O}_{\mathfrak{Z}}$ for an integral cycle $\mathfrak{Z}$ on $\mathfrak{Y}$, then $[\mathcal{F}] = \mathfrak{Z}$ and $f^*[ \mathfrak{Z}] = [ f^* (\mathcal{O}_{\mathfrak{Z}})]$ so we have the equality \eqref{eqn:pb perfect} by definition. 

\medskip

\textbf{Step 2:} Now suppose that $\mathcal{F}$ is a coherent $\mathcal{O}_{\mathfrak{Y}}$-module. We have a decreasing filtration of coherent $\mathcal{O}_{\mathfrak{Y}}$-submodules on $\mathcal{F}$
$$
0 = \mathcal{G}_{r+1} \subsetneq \mathcal{G}_r \subsetneq \cdots \subsetneq \mathcal{G}_1 \subsetneq \mathcal{G}_0 = \mathcal{F},
$$
such that there are isomorphisms $\mathcal{G}_i / \mathcal{G}_{i+1} \simeq \mathcal{O}_{\mathfrak{Z}_i}$ of $\mathcal{O}_{\mathfrak{Y}}$-modules for some integral closed formal subschemes $\mathfrak{Z}_i$ of $\mathfrak{Y}$ for $0 \leq i \leq r$. They induce short exact sequences of coherent $\mathcal{O}_{\mathfrak{Y}}$-modules
\begin{equation}\label{eqn:stpdAB}
0 \to \mathcal{G}_{i+1} \to \mathcal{G}_i \to \mathcal{O}_{\mathfrak{Z}_i} \to 0.
\end{equation}
Applying the flat pull-backs $f^*$ to \eqref{eqn:stpdAB}, we deduce short exact sequences of coherent $\mathcal{O}_{\mathfrak{X}}$-modules
$$
0 \to f^*( \mathcal{G}_{i+1}) \to f^*( \mathcal{G}_i )\to  f^* (\mathcal{O}_{\mathfrak{Z}_i})  \to 0.
$$
Hence we deduce the equalities of the associated cycles (see \cite[Lemma 02QZ]{stacks})
\begin{equation}\label{eqn:pb coh 1}
[ f^* (\mathcal{G}_i) ] = [ f^* (\mathcal{O}_{\mathfrak{Z}_i})] + [ f^* (\mathcal{G}_{i+1}) ].
\end{equation}

Taking the sum of \eqref{eqn:pb coh 1} over all $0 \leq i \leq r$, after telescoping cancellations, we obtain (note $\mathcal{G}_0 = \mathcal{F}$, $\mathcal{G}_{r+1} = 0$)
$$
[ f^* (\mathcal{F})] = \sum_{i=0} ^r [ f^* (\mathcal{O}_{\mathfrak{Z}_i})] = ^{\dagger} \sum_{i=0}^r f^* [ \mathcal{O}_{\mathfrak{Z}_i}] = f^* \sum_{i=0} ^r [ \mathcal{O}_{\mathfrak{Z}_i}] = f^* [ \mathcal{F}].
$$
where $\dagger$ holds by the Step 1. This proves \eqref{eqn:pb perfect} in this case.
\end{proof}

\begin{remk}
We may be able to extend it to perfect complexes. As this is not needed in this article, we shrink this discussion.\qed
\end{remk}

\subsubsection{Special finite push-forwards}\label{sec:2.7}

Recall that we say that a morphism $f: \mathfrak{X} \to \mathfrak{Y}$ of noetherian formal schemes is (\cite[Definition 4.2.2, p.326]{FK}) \emph{finite} if for any affine open subset $V= \Spf (A)$ of $\mathfrak{Y}$, $f^{-1} (V)$ is affine of the form $\Spf (B)$ for a ring $B$ that is a finitely generated $A$-module. 

\begin{defn}\label{defn:push-forward 1}
For a finite morphism $f: \mathfrak{X} \to \mathfrak{Y}$ of noetherian quasi-affine formal $k$-schemes, the push-forward
$$
f_*: \un{z}_d (\mathfrak{X}) \to \un{z}_d  (\mathfrak{Y})
$$
is defined by sending an integral closed formal subscheme $\mathfrak{Z} \subset \mathfrak{X}$ of dimension $d$ to the $d$-dimensional cycle $[f_* \mathcal{O}_{\mathfrak{Z}}]_d$ associated to the coherent sheaf $f_* \mathcal{O}_{\mathfrak{Z}} $ on $\mathfrak{Y}$ in the sense of Definition \ref{defn:ass cycle}. 
\qed
\end{defn}

When $\mathfrak{X}= \Spf (A)$ and $\mathfrak{Y} = \Spf (B)$, by definition $\un{z}_d (\mathfrak{X}) = z_d (\Spec (A))$ and $\un{z}_d (\mathfrak{Y}) = z_d (\Spec (B))$ and $B \to A$ is a finite homomorphism. Thus $f_*$ is identical to the usual finite push-forward $z_d (\Spec (A)) \to z_d (\Spec (B))$.

\medskip

We also had the following in \cite[\S 2.3]{Park Tate}. It was stated for the affine case, but it generalizes to the quasi-affine case by \S \ref{sec:ass cycle quasi-affine}:

\begin{lem}\label{lem:pf coh cy}
Let $f: \mathfrak{X} \to \mathfrak{Y}$ be a finite morphism of noetherian quasi-affine formal schemes. Let $\mathcal{F}$ be a coherent $\mathcal{O}_{\mathfrak{X}}$-module.  
Then we have the equality of cycles
\begin{equation}\label{eqn:pf coh cy0}
f_* [ \mathcal{F}]= [ f_* \mathcal{F}] \in \un{z}_* (\mathfrak{Y}), 
\end{equation}
where the left hand side is the push-forward (as in Definition \ref{defn:push-forward 1}) of the cycle $[\mathcal{F}]$ (as in Definition \ref{defn:ass cycle}), while the right hand side is the cycle $[f_* \mathcal{F}]$ (as in Definition \ref{defn:ass cycle}) of the push-forward $f_* \mathcal{F}$ of the coherent sheaf.
\end{lem}

\subsection{Higher Chow cycles of quasi-affine formal $k$-schemes}\label{sec:HC}

In \S \ref{sec:HC}, we recall from \cite[Definition 2.2.2]{Park Tate} the definitions of a version of cubical higher Chow cycles for noetherian quasi-affine formal $k$-schemes. See Definition \ref{defn:HCG}. In \emph{ibid.}, it was done in the affine case, but there is no essential difference. The higher Chow complexes for schemes originate from S. Bloch \cite{Bloch HC}.

An important point for formal schemes is that, in addition to the usual general position requirement (\textbf{GP}) with respect to the faces, we have an additional requirement, called (\textbf{SF}) which uses the largest ideal of definition of the formal scheme $\mathfrak{X}$.

\subsubsection{The cycles} \label{subsec:cycles}
For the projective line $\mathbb{P}_k ^1$, we let $\square_k := \mathbb{P}_k ^1 \setminus \{ 1 \}$. Let $\square^0 _k := \Spec (k)$. For $n \geq 1$, let $\square^n_k$ be the $n$-fold self product of $\square_k$ over $k$. Let $(y_1, \cdots, y_n) \in \square^n_k$ be the coordinates.

Define a face $F \subset \square_k ^n$ to be a closed subscheme given by a system of equations of the form $\{ y_{i_1} = \epsilon_1, \cdots, y_{i_s} = \epsilon_s\}$ for an increasing sequence of indices $1 \leq i_1 < \cdots < i_s \leq n$ and $\epsilon_j \in \{ 0, \infty \}$. In case of the empty set of equations, we take $F= \square_k ^n$. 
The codimension one face given by $\{ y_i = \epsilon\}$ is denoted by $F_i ^{\epsilon}$. 

\medskip

Let $\mathfrak{X}$ be an equidimensional noetherian quasi-affine formal $k$-scheme. The fiber product $\mathfrak{X} \times_k \square_k ^n$ is often denote by $\square_{\mathfrak{X}} ^n$.  
A \emph{face} $F_{\mathfrak{X}}$ of $\square_{\mathfrak{X}} ^n$ is given by $\mathfrak{X} \times_k F$ for a face $F \subset \square_k ^n$. Here $\mathfrak{X} \times_k \square_k ^n$ is also equidimensional; see Greco-Salmon \cite[Theorem 7.6-(b), p.35]{GS}. We remark that in case $\mathfrak{X}$ is regular, so is $\mathfrak{X} \times_k \square_k ^n$; see Greco-Salmon \cite[Theorem 8.10, p.39]{GS} or P. Salmon \cite[Th\'eor\`eme 7, p.398]{Salmon}.

\begin{remk}\label{remk:2.5.2}
What does the fiber product $\square_{\mathfrak{X}} ^n = \mathfrak{X}\times_k \square_k ^n$ of formal schemes look like? For simplicity, identify $\square_k^n$ with $\mathbb{A}_k^n$. 

When $\mathfrak{X}=\Spf (A) $ is just a $k$-scheme so that $\Spf (A) = \Spec (A)$, this is the usual fiber product of $k$-schemes. Thus $\square_{\mathfrak{X}} ^n$ is given by the polynomial ring $A[ y_1, \cdots, y_n]$. 

Now suppose $\mathfrak{X}= \Spf (A)$ is not a scheme, where $A$ is $I$-adically complete for an ideal $I \subset A$. The fiber product $\square_{\mathfrak{X}} ^n $ is given by the formal spectrum of the restricted formal power series ring $A\{ y_1, \cdots, y_n \}$; this is the completion
$$
A \{ y_1, \cdots, y_m \} = \varprojlim_m  \ (A/ I^m)[ y_1, \cdots, y_n].
$$

 We have $A[y_1, \cdots, y_n] \subset A \{ y_1, \cdots, y_n \} \subset A [[y_1, \cdots, y_n]]$, and the inclusions are proper in general.

See EGA I \cite[Ch.0, (7.5.1), p.69]{EGA1} for more on the restricted formal power series rings. The notation in Fujiwara-Kato \cite[\S 8.4, p.183]{FK} is $A\! \left< \! \left< y_1, \cdots, y_n \right\> \! \right>$. See Lemma \ref{lem:rest poly 1} as well as \cite{Park Tate} on using them in the context of cycles.
\qed
\end{remk}

We recall (and extend) the following from \cite[\S 2.2]{Park Tate}:

\begin{defn}\label{defn:HCG}

Let $\mathfrak{X}$ be an equidimensional noetherian quasi-affine formal $k$-scheme. Let $\mathfrak{X}_{\red}$ be the closed subscheme given by the largest ideal of definition $\mathcal{I}_0$ of $\mathfrak{X}$
(see EGA I \cite[Proposition (10.5.4), p.187]{EGA1}).

For $n, q \geq 0$, let $\un{z} ^q (\mathfrak{X}, n)$ be the subgroup of $z^q (\square_{\mathfrak{X}} ^n)$ generated by integral closed formal subschemes $\mathfrak{Z} \subset \square_{\mathfrak{X}} ^n$ of codimension $q$, subject to the following conditions: 
\begin{enumerate}
\item [(\textbf{GP})] (General position) The integral formal scheme $\mathfrak{Z}$ intersects properly with $\mathfrak{X} \times F$ for each face $F \subset \square_k ^n$.

\item [(\textbf{SF})] (Special fiber) For each face $F \subset \square_k ^n$, we have
$$
\codim_{ \mathfrak{X}_{\red} \times F} (\mathfrak{Z} \cap (\mathfrak{X}_{\red} \times F)) \geq q.
$$
\end{enumerate}
We may call the cycles in $\un{z}^q (\widehat{X}, n)$ \emph{admissible} for simplicity.\qed
\end{defn}

We have several remarks on the conditions of Definition \ref{defn:HCG}. 

\begin{remk}\label{remk:TS}
In an earlier version of the article, we imposed an additional requirement called the topological support condition that a given integral cycle $\mathfrak{Z}$ and its reduction $Z$ modulo the ideal of definition $\mathcal{I}_0$ have the same underlying topological spaces as ringed spaces. However, as pointed out by Joseph Ayoub, it turned out that the condition holds always for any noetherian formal scheme, so it was not needed. In the article, we implicitly use this elementary observation, if needed.
\qed
\end{remk}

\begin{remk}\label{remk:condition 2-3}
If $\mathfrak{X}$ in Definition \ref{defn:HCG} is a scheme, note that the condition (\textbf{SF}) holds automatically. Our definition thus coincides with that of the usual cubical version of higher Chow cycles in \cite{Bloch HC}.\qed
\end{remk}

\begin{remk}
Let $Z$ be the reduction of $\mathfrak{Z}$ mod $\mathcal{I}_0$, i.e. $Z= \mathfrak{Z} \cap ( \mathfrak{X}_{\red} \times \square^n )$. The condition (\textbf{SF}) of Definition \ref{defn:HCG} implies that $Z$ intersects properly with $\mathfrak{X} _{\red} \times F$ so that we have the Gysin pull-back map
$$
\un{z}^q (\mathfrak{X}, n) \to \un{z} ^q (\mathfrak{X}_{\red}, n), \ \ \mathfrak{Z} \mapsto [Z],
$$
where $[Z]$ is the associated cycle on the scheme $\mathfrak{X}_{\red} \times \square_k ^n$. 
See Lemma \ref{lem:basic Gysin 1} for $n=0$ case.\qed
\end{remk}

\begin{remk}\label{remk:gen ideal defn}
More generally, when the cycles $\mathfrak{Z}$ have the conditions (\textbf{GP}) and (\textbf{SF}) of Definition \ref{defn:HCG}, for any ideal of definition $\mathcal{I}$ of the formal scheme $\mathfrak{X}$ and the closed subscheme $\mathfrak{X}_{\mathcal{I}}$ defined by the ideal, the cycles $\mathfrak{Z}$ also satisfy the corresponding properties with respect to $\mathcal{I}$, with $\mathfrak{X}_{\red}$ replaced by $\mathfrak{X}_{\mathcal{I}}$ and $\mathcal{I}_0$ replaced by $\mathcal{I}$.

Indeed, for each face $F \subset \square^n$, we have $\mathfrak{Z} \cap (\mathfrak{X}_{\red} \times F) \subset \mathfrak{Z} \cap (\mathfrak{X}_{\mathcal{I}} \times F) \subset \mathfrak{Z} \cap (\mathfrak{X} \times F)$ so that
\begin{equation}\label{eqn:gen ideal defn 1}
|\mathfrak{Z} \cap (\mathfrak{X}_{\red} \times F) | \subset | \mathfrak{Z} \cap (\mathfrak{X}_{\mathcal{I}} \times F) | \subset |\mathfrak{Z} \cap (\mathfrak{X} \times F)|,
\end{equation}
but these underlying topological spaces are all equal: if $\mathcal{J}$ and $\mathcal{K}$ are ideal sheaves for $\mathfrak{Z}$ and $\mathfrak{X} \times F$, then $\mathcal{I}_0+ \mathcal{J} + \mathcal{K}$ and $\mathcal{I} + \mathcal{J}+ \mathcal{K}$ are both ideals of definitions of $\mathcal{Z} \cap (\mathfrak{X} \times F)$. In particular, 
$$| \mathfrak{Z} \cap (\mathfrak{X}_{\red} \times F) | =  |\mathfrak{Z} \cap (\mathfrak{X}_{\mathcal{I}} \times F)|.$$


On the other hand, since $\mathfrak{Z} \cap (\mathfrak{X}_{\red} \times F)$ and $\mathfrak{Z} \cap (\mathfrak{X}_{\mathcal{I}} \times F)$ are schemes, their dimensions as schemes are equal to the dimensions of their underlying topological spaces. In particular, we have
\begin{equation}\label{eqn:ideal defn I}
\tuborg
\dim \ \mathfrak{Z} \cap (\mathfrak{X}_{\red} \times F)  = \dim \ \mathfrak{Z} \cap (\mathfrak{X} _{\mathcal{I}} \times F), \\
\dim \ \mathfrak{X}_{\red} \times F  = \dim \ \mathfrak{X} _{\mathcal{I}} \times F,
\sluttuborg
\end{equation}
so that \eqref{eqn:ideal defn I} and (\textbf{SF}) together imply that
\begin{enumerate}
\item [$(\textbf{SF})_{\mathcal{I}}$] (Special fiber) $\codim_{ \mathfrak{X}_{\mathcal{I}} \times F} (\mathfrak{Z} \cap (\mathfrak{X}_{\mathcal{I}} \times F)) \geq q.$
\end{enumerate}
Furthermore \eqref{eqn:ideal defn I} and $ (\textbf{SF})_{\mathcal{I}}$ imply (\textbf{SF}) as well. 

The condition (\textbf{GP}) has nothing to do with the ideal of definition $\mathcal{I}$. Thus we have just shown that 
$$\{ (\textbf{GP}) , (\textbf{SF})  \} \Leftrightarrow \{  (\textbf{GP}), (\textbf{SF})_{\mathcal{I}}  \}.$$
In particular, Definition \ref{defn:HCG} could have been given for any ideal of definition $\mathcal{I}$ of $\mathfrak{X}$.
This fact is often exploited, e.g. see Claim in Lemma \ref{lem:rest poly 1}.
\qed
\end{remk}

\subsubsection{The cycle complex}

We turn the groups of Definition \ref{defn:HCG} into a complex:

\begin{defn}\label{defn:HCG2}
Let $\mathfrak{X}$ be an equidimensional noetherian quasi-affine formal $k$-scheme. 
Let $n \geq 1$. We define the boundary operators $\partial: z^q (\mathfrak{X}, n) \to z^q (\mathfrak{X}, n-1)$ as follows.

 For $1 \leq i \leq n$ and $\epsilon \in \{ 0, \infty \}$, let $\iota_i ^{\epsilon}: \square_{\mathfrak{X}} ^{n-1} \hookrightarrow \square_{\mathfrak{X}} ^n$ be the closed immersion given by the equation $\{ y_i = \epsilon \}$. For each integral cycle $\mathfrak{Z} \in \un{z} ^q (\mathfrak{X}, n)$, define the face $\partial_i ^{\epsilon} (\mathfrak{Z})$ to be the intersection product 
\begin{equation}\label{eqn:codim 1 face defn}
\partial_i ^{\epsilon} (\mathfrak{Z}) := [F_{i,\mathfrak{X}} ^{\epsilon}].[\mathfrak{Z}] = [ \mathbf{L} (\iota_i ^{\epsilon})^* ( \mathcal{O}_{\mathfrak{Z}})],
\end{equation}
as in \eqref{eqn:Serre product}. This is compatible with the codimension $1$ face map considered for the usual higher Chow cycles, cf. \cite{Bloch HC}; indeed note that we have a short exact sequence of $\mathcal{O}_{\square_{\mathfrak{X}} ^n}$-modules
$$
0 \to \mathcal{O}_{\square_{\mathfrak{X}} ^n} \overset{ \times (y_i - \epsilon)}{\longrightarrow} \mathcal{O}_{\square_{\mathfrak{X}} ^n} \to \mathcal{O}_{F_{i, \mathfrak{X}} ^{\epsilon}} \to 0.
$$
Hence the derived pull-back $ \mathbf{L} (\iota_i ^{\epsilon})^* ( \mathcal{O}_{\mathfrak{Z}})$ is quasi-isomorphic to the complex $\mathcal{O}_{\mathfrak{Z}} \overset{ \times (y_i - \epsilon)}{\longrightarrow}  \mathcal{O}_{\mathfrak{Z}}$. Thus the higher homologies ${\rm H}_i (\mathbf{L} (\iota_i ^{\epsilon})^* (\mathcal{O}_{\mathfrak{Z}})) $ with $i>1$ vanish. Since $\mathfrak{Z}$ intersects the face $F_{i, \mathfrak{X}} ^{\epsilon}$ properly by (\textbf{GP}), we have $\ker (  \times (y_i - \epsilon))= 0$, thus ${\rm H}_1 = 0$ as well. Hence only the zero-th homology of $ \mathbf{L} (\iota_i ^{\epsilon})^* ( \mathcal{O}_{\mathfrak{Z}})$ may survive so that $[\mathbf{L} (\iota_i ^{\epsilon})^* (\mathcal{O}_{\mathfrak{Z}})] = [ (\iota_i ^{\epsilon})^* (\mathcal{O}_{\mathfrak{Z}})]$. 

\medskip

Since the conditions (\textbf{GP}), (\textbf{SF}) of Definition \ref{defn:HCG} for $\mathfrak{Z}$ imply the same conditions for the faces of $\mathfrak{Z}$, we have $\partial_i ^{\epsilon} (\mathfrak{Z}) \in \un{z} ^q (\mathfrak{X}, n-1)$.

 Let $\partial:= \sum_{i=1} ^n (-1)^i (\partial_i ^{\infty} - \partial_i ^0)$. By the usual cubical formalism, one checks that $ \partial \circ \partial  = 0$, and we have a cubical abelian group $(\un{n} \mapsto \un{z} ^q (\mathfrak{X}, n))$, where $\un{n}:= \{ 0, 1, \cdots, n \}$. 

Let $z^q (\mathfrak{X}, \bullet)$ be the associated non-degenerate complex, i.e. the quotient complex
\begin{equation}\label{eqn:Deg}
z^q (\mathfrak{X}, \bullet) = \un{z}^q (\mathfrak{X}, \bullet)/ \un{z}^q (\mathfrak{X}, \bullet)_{\rm degn},
\end{equation}
where the subgroup $\un{z}^q (\mathfrak{X}, \bullet)_{\rm degn}$ consists of the degenerate cycles; this group is generated by the pull-backs of cycles (e.g. Lemma \ref{lem:pre flat pb0}-(III)) via the coordinate projections $\square_{\mathfrak{X}} ^n \to \square_{\mathfrak{X}} ^{n-1}$ that forget one of the coordinates. 

Define the homology group $\CH^q (\mathfrak{X}, n):= {\rm H}_n (z^q (\mathfrak{X}, \bullet))$. 
\qed
\end{defn}

\begin{remk}\label{remk:lci}
In the middle of Definition \ref{defn:HCG2}, we saw that $\mathbf{L}(\iota_i ^{\epsilon})^*$ was equal to the usual intersection product $(\iota_i^{\epsilon})^*$ for cycles intersecting the faces properly. We remark that a similar argument holds for l.c.i.~closed immersions $\iota$ so that the usual Gysin pull-backs $\iota^*$ on cycles intersecting properly with the closed (formal) subscheme are equal to $\mathbf{L} \iota^*$.
This is classical, and one can find an argument via Koszul resolutions, see e.g. Yu. Manin \cite[Theorem 2.3, pp.9--12]{Manin}. \qed
\end{remk}

We remind the reader that the group $\CH^q (\mathfrak{X},n)$ is yet just one of the intermediate byproducts toward the eventual definition of our new group $\BGH^q (Y, n)$ for $Y \in \Sch_k$ given in Definition \ref{defn:Cech final}.

\medskip

We mention one variant of Definition \ref{defn:HCG}:

\begin{defn}\label{defn:HCG var}
Let $\mathfrak{X}$ be an equidimensional noetherian quasi-affine formal $k$-scheme. Let $\mathcal{W}$ be a finite set of closed formal $k$-subschemes of $ \mathfrak{X}$.

Let $z^q_{\mathcal{W}} (\mathfrak{X}, n)$ be the subgroup of $z^q (\mathfrak{X}, n)$ generated by the integral closed formal subschemes $\mathfrak{Z} \subset \mathfrak{X}\times_k \square_k ^n$ such that for each $C \in \mathcal{W}$ and each face $F \subset \square_k ^n$, the intersection of $\mathfrak{Z}$ with $C \times F$ is proper, i.e. ${\rm codim}_{C \times F} (\mathfrak{Z} \cap (C \times F)) \geq q$.

The resulting complex $ z^q _{\mathcal{W}} (\mathfrak{X}, \bullet)$ is a subcomplex of $z^q  (\mathfrak{X}, \bullet)$, and its homology is denoted by $\CH^q _{\mathcal{W}} (\mathfrak{X}, n)$.\qed
\end{defn}

\subsubsection{The special flat pull-backs 2}
Some of the special flat pull-backs of Lemma \ref{lem:pre flat pb0} extend to higher Chow cycles on formal schemes:

\begin{lem}\label{lem:pre flat pb}
Let $f: \mathfrak{X} \to \mathfrak{Y}$ be a flat morphism of equidimensional noetherian quasi-affine formal $k$-schemes, which is one of the types ${\rm (I)}$, ${\rm (II)}$, ${\rm (III)}$ in Lemma \ref{lem:pre flat pb0}. Then $f$ induces a natural flat pull-back morphism $f^*: z^q ( \mathfrak{Y}, \bullet) \to z^q (\mathfrak{X}, \bullet)$.
\end{lem}

\begin{proof}
Let $n, q \geq 0$ be integers. 
The induced morphism $f= f_n: \mathfrak{X} \times_k \square_k ^n \to \mathfrak{Y} \times_k \square_k ^n$ is also flat of the same type. 

Let $\mathfrak{Z} \in z^q (\mathfrak{Y}, n)$ be an integral cycle. By Lemma \ref{lem:pre flat pb0}, we know that $f^* (\mathfrak{Z})$ is of codimension $q$ in $\mathfrak{X} \times_k \square_k ^n$. To show $f^*(\mathfrak{Z}) \in z^q (\mathfrak{X}, n)$, we check the conditions (\textbf{GP}), (\textbf{SF}) of Definition \ref{defn:HCG}.

Let $F \subset \square_k ^n$ be a face and let $Z$ be the reduction of $\mathfrak{Z}$ modulo the largest ideal of definition, i.e. $Z= \mathfrak{Z} \cap (\mathfrak{X}_{\red} \times \square_k ^n)$.

\medskip

\textbf{Type (I):} First suppose $f$ is a type (I) flat morphism. In this case $\mathfrak{X}= \mathfrak{Y}|_U$ for a nonempty open subset $U \subset |\mathfrak{Y}|$ and $f^* (\mathfrak{Z}) = \mathfrak{Z}|_U$. Thus 
$$f^* (\mathfrak{Z}) \cap (\mathfrak{X} \times F) = \mathfrak{Z}|_U \cap (\mathfrak{Y}|_U \times F) = (\mathfrak{Z} \cap (\mathfrak{Y} \times F))|_U$$
so that $\dim \ f^*(\mathfrak{Z}) \cap (\mathfrak{X} \times F) \leq \dim \ \mathfrak{Z} \cap (\mathfrak{Y} \times F)$. By the condition (\textbf{GP}) for $\mathfrak{Z}$, we have $\dim \ \mathfrak{Z} \cap (\mathfrak{Y} \times F) \leq \dim \ \mathfrak{Y} \times F -q$, where the latter is equal to $\dim \ \mathfrak{Y}|_U \times F - q = \dim \ \mathfrak{X} \times F - q$ because $U$ is nonempty and $\mathfrak{Y}$ is equidimensional. Hence $ \dim \ f^*(\mathfrak{Z}) \cap (\mathfrak{X} \times F)\leq \dim \ \mathfrak{X} \times F - q$, proving the condition (\textbf{GP}) for $f^* (\mathfrak{Z})$. 

Similarly, we have
$$
f^* (\mathfrak{Z}) \cap ( (\mathfrak{Y}|_U)_{\red} \times F) = (\mathfrak{Z} \cap (\mathfrak{Y} _{\red} \times F))|_U,
$$
and combined with the condition (\textbf{SF}) for $\mathfrak{Z}$, we deduce the condition (\textbf{SF}) for $f^* (\mathfrak{Z})$. 


\medskip

\textbf{Type (II):} This time, suppose $f$ is a type (II) flat morphism, i.e. $f: \mathfrak{X}= \mathfrak{Y}_{\mathcal{I}} \to \mathfrak{Y}$ is given by the further completion with respect to an ideal $\mathcal{I} \subset \mathcal{O}_{\mathfrak{Y}}$ such that $| \mathfrak{X}|=Y$ for a closed subset $Y \subset |\mathfrak{Y}|$. 

We show that the conditions (\textbf{GP}), (\textbf{SF}) for $f^* (\mathfrak{Z})$ follow from those of $\mathfrak{Z}$. Here $f^* (\mathfrak{Z})$ is the completion of $\mathfrak{Z}$, and similarly $f^* (\mathfrak{Z}) \cap (\mathfrak{X} \times F)$ is the completion of $\mathfrak{Z} \cap (\mathfrak{Y} \times F)$. 

Hence for (\textbf{GP}), we have $\dim \ f^* (\mathfrak{Z}) \cap (\mathfrak{X} \times F) \leq \dim \ \mathfrak{Z} \cap (\mathfrak{Y} \times F)$, by the basic fact in commutative algebra that the dimension of the completion is bounded by the dimension of the original (see Greco-Salmon \cite[Theorem 8.9, p.38]{GS} or P. Salmon \cite[\S 2, Proposition 3, p.391]{Salmon}). This shows that the condition (\textbf{GP}) for $\mathfrak{Z}$ implies that for $f^* (\mathfrak{Z})$.

The argument for the condition (\textbf{SF}) for $f^* (\mathfrak{Z})$ is similar.


\medskip

\textbf{Type (III):} Suppose $f$ is a type (III) flat morphism, i.e. $f=pr_2: \mathfrak{X} = \mathfrak{X}' \times \mathfrak{Y} \to \mathfrak{Y}$ is the projection. Then $f^* (\mathfrak{Z}) = \mathfrak{X}' \times \mathfrak{Z}$. 

Again, we verify the conditions (\textbf{GP}), (\textbf{SF}) for $f^* (\mathfrak{Z})$ from those of $\mathfrak{Z}$.

For (\textbf{GP}), we have
\begin{equation}\label{eqn:GP 3}
f^* (\mathfrak{Z}) \cap (\mathfrak{X} \times F) = (\mathfrak{X}'  \times \mathfrak{Z}) \cap (\mathfrak{X}'  \times \mathfrak{Y} \times F)= \mathfrak{X}' \times (\mathfrak{Z} \cap (\mathfrak{Y} \times F))
\end{equation}
so that we deduce that if $\mathfrak{Z}$ has (\textbf{GP}), then so does $f^* (\mathfrak{Z})$. 

For (\textbf{SF}), we have $\mathfrak{X}_{\red} = \mathfrak{X}'_{\red} \times \mathfrak{Y}_{\red}$ so that
\begin{equation}\label{eqn:III-SF}
f^* (\mathfrak{Z}) \cap (\mathfrak{X}_{\red} \times F) = (\mathfrak{X}' \times \mathfrak{Z}) \cap (\mathfrak{X}'_{\red} \times \mathfrak{Y}_{\red} \times F) =  \mathfrak{X}'_{\red} \times  (\mathfrak{Z} \cap (\mathfrak{Y}_{\red} \times F)),
\end{equation}
thus we deduce that if $\mathfrak{Z}$ has (\textbf{SF}), then so does $f^* (\mathfrak{Z})$.

\medskip

Thus for each flat morphism $f$ of type (I), (II), or (III), we have shown that if $\mathfrak{Z} \in z^q (\mathfrak{Y}, n)$, then $f^* (\mathfrak{Z}) \in z^q (\mathfrak{X}, n)$.

\medskip

For the compatibility of $f^*$ with the boundary maps, consider the commutative diagram
$$
\xymatrix{
\mathfrak{X} \times \square^{n-1} \ar[d] ^{f_{n-1}}  \ar@{^{(}->}[r] ^{\iota_{i, \mathfrak{X}} ^{\epsilon}} & \mathfrak{X} \times \square^{n} \ar[d] ^{f_n} \\
\mathfrak{Y} \times \square^{n-1}  \ar@{^{(}->}[r] ^{\iota_{i, \mathfrak{Y}} ^{\epsilon}} & \mathfrak{Y} \times \square^{n}.}
$$
The diagram shows that $\mathbf{L} f_{n-1} ^* \circ \mathbf{L} ( \iota_{i, \mathfrak{Y}} ^{\epsilon} )^* = \mathbf{L} ( \iota_{i, \mathfrak{X}} ^{\epsilon} )^*\circ \mathbf{L} f_n ^*$ (\cite[I-3.6, p.119]{Lipman}). Since $f_n$ and $ f_{n-1}$ are flat, we have $\mathbf{L} f_n ^* = f_n ^*$ and $\mathbf{L} f_{n-1} ^* = f_{n-1} ^*$, while we saw $\mathbf{L} (\iota_{i, \mathfrak{X}} ^{\epsilon})^* = (\iota_{i, \mathfrak{X}} ^{\epsilon})^*$ and $\mathbf{L} (\iota_{i, \mathfrak{Y}} ^{\epsilon})^* = (\iota_{i, \mathfrak{Y}} ^{\epsilon})^*$ in Definition \ref{defn:HCG2}. 
Hence we have $f_{n-1} ^* \circ  ( \iota_{i, \mathfrak{Y}} ^{\epsilon} )^* $ $=$ $ ( \iota_{i, \mathfrak{X}} ^{\epsilon} )^*\circ f_n ^*$, i.e. $f_{n-1}^* \circ \partial_i ^{\epsilon} = \partial_i ^{\epsilon} \circ f_n ^*$. This implies $f^* \circ \partial = \partial \circ f^*$ so that $f^*$ is a morphism of complexes.
\end{proof}

\begin{remk}\label{remk:stalk realize}
Lemma \ref{lem:pre flat pb}-(I) implies in particular that, for an equidimensional noetherian formal scheme $\mathfrak{X}$ and a point $y \in |\mathfrak{X}|$ of the underlying topological space, the colimit
$$ 
\varinjlim_{y \in U} z^q (\mathfrak{X}|_U, \bullet),
$$
over all open neighborhoods of $y$ exists. However, we are not certain whether this is equal to the cycle group over a formal scheme. 
For higher Chow complexes of schemes, such ``continuity" is known, but for formal schemes we do not know whether the conditions in Definition \ref{defn:HCG} permit it. However, this won't cause problems for us in this article. 

Later for semi-local $k$-schemes essentially of finite type, we are going to \emph{define} our cycle class groups via the colimits, so that the continuity holds automatically. See Definition \ref{defn:HCG semi-local}.
\qed
\end{remk}

\subsubsection{Special finite push-forwards 2}\label{sec:sfpf2}

We recall the following again from \cite[\S 2]{Park Tate}. It was stated for the affine case, but it generalizes to the quasi-affine case by \S \ref{sec:ass cycle quasi-affine}:

\begin{lem}\label{lem:pre proper pf}
Let $f: \mathfrak{X} \to \mathfrak{Y}$ be a finite surjective morphism of equidimensional noetherian quasi-affine formal $k$-schemes. 

 Then $f$ induces a natural push-forward morphism $f_*: z^q (\mathfrak{X}, \bullet) \to z^{q} (\mathfrak{Y}, \bullet)$ of complexes of abelian groups.
\end{lem}

\begin{cor}\label{cor:3.5.3-2}
Let $\mathfrak{X}$ be an equidimensional noetherian quasi-affine formal $k$-scheme. Let $k \hookrightarrow k'$ be a finite extension of fields, and let $f: \mathfrak{X}_{k'} \to \mathfrak{X}$ be the base change morphism. Then 
\begin{enumerate}
\item $f$ is finite surjective and flat of the type ${\rm (III)}$, and 
\item for the flat pull-back $f^*: z^q (\mathfrak{X}, \bullet) \to z^q (\mathfrak{X}_{k'}, \bullet)$ and the push-forward $f_* : z^q (\mathfrak{X}_{k'}, \bullet) \to z^q (\mathfrak{X}, \bullet)$, the composite $f_* \circ f^*$ is equal to $[ k':k] \cdot {\rm Id}$.
\end{enumerate}
\end{cor}

\begin{proof}
Here $\mathfrak{X}_{k'} = \mathfrak{X} \times_k \Spec (k')$ so that the base change map $f$ is flat of type (III). This is finite surjective.

The flat pull-back $f^*$ exists by Lemma \ref{lem:pre flat pb} while the finite push-forward $f_*$ exists by Lemma \ref{lem:pre proper pf}. That $f_* \circ f^* = [k' : k] \cdot {\rm Id}$ is apparent.
\end{proof}

\section{Derived rings and the mod equivalences}\label{sec:main mod Y}

Let $Y$ be a quasi-affine $k$-scheme of finite type. Let $Y \hookrightarrow X$ be a closed immersion into an equidimensional smooth $k$-scheme. Let $\widehat{X}$ be the completion of $X$ along $Y$.

\medskip

In \S \ref{sec:main mod Y}, we define the mod $Y$-equivalence relation on the cycle group $z^q (\widehat{X}, n)$ of Definition \ref{defn:HCG} with $\mathfrak{X}= \widehat{X}$. The mod $Y$-equivalence on the cycles comes from the differences of associated cycles of certain pairs of coherent sheaves of algebras, whose derived pull-backs to $Y$ are isomorphic to each other as sheaves of derived rings. This idea uses the philosophy of derived algebraic geometry. This is in part related to (but a bit stronger than) the derived Milnor patching of S. Landsburg \cite{Landsburg Duke}.

We show that the resulting collection $\{ z^q (\widehat{X} \mod Y, n) \}_{n \geq 0}$ forms a complex. Its homology group $\CH^q (\widehat{X} \mod Y, n)$ is one of the central intermediate objects in this article, though there are yet some more steps to climb up.

\subsection{Derived rings}\label{sec:salg}
In \S \ref{sec:salg}, we recall a few basic ideas on derived rings and derived algebraic geometry needed in this paper. Some references are B. To\"en \cite{Toen} and J. Lurie \cite{Lurie SAG}. For a quick essence, one can read Kerz-Strunk-Tamme \cite[\S 2.1]{KST IM}. 
The reader can consult J. Lurie \cite[Ch. 1]{Lurie HTT} for some definitions around $\infty$-categories.

\medskip

\subsubsection{Derived rings} Let $R$ be a commutative ring with unity, and let $A$ and $B$ be $R$-algebras. If we regard $A$ and $B$ as $R$-modules, we have the conventional derived tensor product $A \otimes_R ^{\mathbf{L}} B$ using an $R$-flat resolution of $A$ or $B$. In this process, we lose some information on the ring structures, and sometimes this is not good enough. We want to have \emph{a kind of ring structure} on this derived tensor $A \otimes_R ^{\mathbf{L}} B$, 

A way to do this is to bring all of $A, B, R$ into the bigger category ${\rm sComm}$ of simplicial commutative rings with unity, and to take the homotopy push-out there. A more precise description is given as follows.

Recall that a morphism $R_1 \to R_2$ of simplicial rings is called a \emph{weak-equivalence} if the induced homomorphisms $\pi_i (R_1) \to \pi_i (R_2)$ (where the base points are at $0$) are isomorphisms for all $i \geq 0$. Localizing ${\rm sComm}$ by all weak-equivalences, we obtain \emph{the $\infty$-category of derived rings}
$$
\textbf{sComm} := L ({\rm sComm}),
$$
where $L$ denotes the localization. The push-out of $A \leftarrow R \rightarrow B$ in $\textbf{sComm}$ is the derived ring $A \otimes_R ^{\mathbf{L}} B$. It is represented in the weak-equivalence class by $\tilde{A} \otimes_R B$, where $\tilde{A}$ is a \emph{simplicial resolution} of $A$, and the tensor $\tilde{A} \otimes_R B$ is equipped with the structure of a simplicial ring. We can also regard it as a simplicial $B$-algebra, if desired. See M. Andr\'e \cite[D\'efinition IV-30, p.53]{Andre} for the definition of simplicial resolutions, and see \cite[Th\'eor\`eme IV-44, p.55]{Andre} and \cite[Th\'eor\`eme IX-26, p.128]{Andre} for their existence. S. Iyengar \cite{Iyengar} is another place.

For $i \geq 0$, we have ${\rm \pi}_i (A \otimes_R ^{\mathbf{L}} B) = {\rm H}_i (A \otimes_R ^{\mathbf{L}} B) = \Tor_i ^R (A, B)$. The simplicial ring structure of $A \otimes_R ^{\mathbf{L}} B$ induces a natural graded ring structure on $ \Tor_* ^R (A, B)$. There is even a natural graded $B$-algebra structure $\Tor_* ^R (A, B)$. See \cite[Section 068G]{stacks}.

\medskip

When $R$ is a fixed commutative ring with unity, we can localize the subcategory ${\rm sAlg} (R) \subset {\rm sComm}$ of simplicial $R$-algebras by weak-equivalences to obtain
 $$
 \sAlg (R) := L ( {\rm sAlg} (R)),
 $$
 the $\infty$-category of \emph{derived $R$-algebras}.
 
 When $f: R \to S$ is a homomorphism of commutative rings with unity, we have the induced derived functor
 $$ 
 \mathbf{L} f^*: \sAlg (R) \to \sAlg (S)
 $$
 that sends an $R$-algebra $A$ to the derived $S$-algebra $A \otimes_R ^{\mathbf{L}} S$. 
 
 \medskip
 
For any given (formal) scheme $\mathfrak{X}$, we may promote the $\infty$-categories $\sAlg (R)$ over rings $R$ to the $\infty$-category $\sAlg (\mathcal{O}_{\mathfrak{X}})$ of derived $\mathcal{O}_{\mathfrak{X}}$-algebras. In case $f: \mathfrak{X} \to \mathfrak{Y}$ is a morphism of (formal) schemes, we have the induced derived functor
$$
\mathbf{L} f^* : \sAlg (\mathcal{O}_{\mathfrak{Y}}) \to \sAlg (\mathcal{O}_{\mathfrak{X}}).
$$

\subsubsection{Derived ringed spaces}
Let $Y$ be a noetherian topological space with a finite Krull dimension. 
Recall (e.g. B. To\"en \cite[\S 2.2, pp.184-185]{Toen}, or Kerz-Strunk-Tamme \cite[Definition 2.1]{KST IM}) that a \emph{derived scheme} is a pair $(Y, \mathcal{O})$ equipped with a sheaf (or, more precisely a stack) $\mathcal{O}$ of derived rings, such that the ringed space $(Y, \pi_0 \mathcal{O})$ is a scheme in the usual sense, and for each $i \geq 0$, the sheaf $\pi_i \mathcal{O}$ is a quasi-coherent sheaf of $\pi_0 \mathcal{O}$-modules, or in short, a quasi-coherent $\pi_0 \mathcal{O}$-module. 

For the purpose of this paper, unfortunately the $\infty$-category of derived schemes is a bit narrow, while that of the \emph{derived ringed spaces} is large enough. Recall (B. To\"en \cite[\S 2.2, p.184]{Toen}) that a pair $(Y, \mathcal{O})$ equipped with a sheaf $\mathcal{O}$ of derived rings on $Y$ is called a derived ringed space. They form the $\infty$-category of derived ringed spaces. 

A derived ringed space $(Y, \mathcal{O})$ is a \emph{derived locally ringed space} if the associated truncation $(Y, \pi_0 \mathcal{O})$ is a locally ringed space. 

If we define the $\infty$-category of \emph{derived formal schemes} similarly, then probably the objects we deal with in the paper would lie in there. This particular knowledge is not needed, so we won't pursue the details in this direction.

\subsection{The mod $Y$-equivalence on cycles}\label{sec:mod Y}
We return to the discussions of higher Chow cycles on noetherian quasi-affine formal $k$-schemes.

We specialize to the case when the formal scheme $\mathfrak{X}$ is the formal neighborhood $\widehat{X}$, associated to a closed immersion $Y \hookrightarrow X$ of an quasi-affine $k$-scheme of finite type into an equidimensional smooth $k$-scheme. 

In \S \ref{sec:mod Y}, we define the mod $Y$-equivalence on $z^q (\widehat{X}, n)$. Since $\widehat{X}$ is the disjoint union $\coprod_i \widehat{X}|_{Y_i}$ taken over all connected components $Y_i$ of $Y$ and $z^q (\widehat{X}, n) = \bigoplus_{i} z^q (\widehat{X}|_{Y_i}, n)$, replacing $Y$ by $Y_i$, we may assume that $Y$ is connected.

\begin{defn}\label{defn:mod Y equiv}
Let $Y$ be a connected quasi-affine $k$-scheme of finite type, and let $Y \hookrightarrow X$ be a closed immersion into an equidimensional smooth $k$-scheme. Let $\widehat{X}$ be the completion of $X$ along $Y$. 
Let $\mathcal{W} \subset \widehat{X}$ be a finite set of closed formal subschemes. Recall $\square_{\widehat{X}} ^n$ is regular (see \S \ref{subsec:cycles}).

\begin{enumerate}
\item A perfect complex $\mathcal{F}$ on $\square_{\widehat{X}} ^n$ is said to be \emph{admissible} of codim $q$ with respect to $\mathcal{W}$ if for each $i \in \mathbb{Z}$, the cycle $[\mathcal{H}^i (\mathcal{F})] \in z^q_{\mathcal{W}} (\widehat{X}, n)$ (see Proposition \ref{prop:perfect coherent ass quasi-affine}). We let $\mathcal{D}_{\perf} ^q (\widehat{X}, n, \mathcal{W})$ be the subcategory of $\mathcal{D}_{\perf} (\square_{\widehat{X}}^n)$ generated by admissible perfect complexes of codim $q$ with respect to $\mathcal{W}$.

\item Let $\mathcal{A}$ be a coherent $\mathcal{O}_{\square_{\widehat{X}} ^n}$-algebra. Since $\square_{\widehat{X}} ^n$ is regular, $\mathcal{A}$ is a perfect complex on $\square_{\widehat{X}}^n$ by Corollary \ref{cor:coh=perfect}. 

Let $\mathcal{R}^q (\widehat{X}, n, \mathcal{W})$ be the set of all admissible coherent $\mathcal{O}_{\square_{\widehat{X}} ^n}$-algebras of codim $q$ with respect to $\mathcal{W}$. Any $\mathcal{A}\in \mathcal{R}^q (\widehat{X}, n, \mathcal{W})$ is an a.q.c.~sheaf by Lemma \ref{lem:aqc sum}-(5).

\item Two admissible coherent $\mathcal{O}_{\square_{\widehat{X}} ^n}$-algebras $\mathcal{A}_1, \mathcal{A}_2 \in \mathcal{R}^q (\widehat{X}, n, \mathcal{W})$ are said to be \emph{mod $Y$-equivalent}, if we have an isomorphism 
\begin{equation}\label{eqn:derived int}
\mathcal{A}_1 \otimes_{\mathcal{O}_{\square^n_{\widehat{X}}}}^{\mathbf{L}} \mathcal{O}_{\square_Y ^n} \simeq \mathcal{A}_2 \otimes_{\mathcal{O}_{\square^n_{\widehat{X}}}} ^{\mathbf{L}} \mathcal{O}_{\square_Y ^n}
\end{equation}
in $\sAlg (\mathcal{O}_{\square_Y^n})$. In this case, we write $\mathcal{A}_1 \sim_Y \mathcal{A}_2$.

Let $\mathcal{L}_{\mathcal{W}} ^q (n) = \mathcal{L} ^q (\widehat{X}, Y, n, \mathcal{W})$ be the set of all pairs $(\mathcal{A}_1, \mathcal{A}_2)$ of mod $Y$-equivalent admissible coherent $\mathcal{O}_{\square_{\widehat{X}} ^n}$-algebras $\mathcal{A}_j \in \mathcal{R}^q (\widehat{X}, n, \mathcal{W})$. 

\item Let $\mathcal{M}^q_{\mathcal{W}}(n) = \mathcal{M}^q (\widehat{X}, Y,  n, \mathcal{W})$ be the subgroup of $z^{q}_{ \mathcal{W}} (\widehat{X} ,n)$ generated by the cycles $[\mathcal{A}_1]-[\mathcal{A}_2]$,
where $(\mathcal{A}_1, \mathcal{A}_2)$ runs over all members of $\mathcal{L}_{\mathcal{W}}^q (n)$. 
\end{enumerate}
In case $\mathcal{W}= \emptyset$, we omit $\mathcal{W}$ from the notations of all of the above. \qed
\end{defn}

\begin{remk}
The isomorphism \eqref{eqn:derived int} can be described in the following two additional equivalent ways.

Here is one way: consider two ringed spaces $(|\square_{\widehat{X}}^n|, \mathcal{A}_j)$, $j=1,2$. They give the homotopy fiber products
\begin{equation}\label{eqn:d fiber}
(|\square_{\widehat{X}}^n|, \mathcal{A}_j) \times_{\square_{\widehat{X}} ^n} ^h \square_Y^n,
\end{equation}
 in the $\infty$-category of derived ringed spaces, where note that $|\square_{\widehat{X}}^n|=|\square_Y^n|$. Then the isomorphism \eqref{eqn:derived int} in $\sAlg (\mathcal{O}_{\square_Y^n})$ is equivalent to that the above Cartesian products \eqref{eqn:d fiber} for $j=1,2$ are isomorphic to each other as derived $\mathcal{O}_{\square_Y^n}$-schemes. The Cartesian products \eqref{eqn:d fiber} are given by the derived schemes $(\square_Y^n, \mathcal{A}_j \otimes _{\mathcal{O}_{\square^n_{\widehat{X}}}}^{\mathbf{L}} \mathcal{O}_{\square_Y ^n})$.
 
Another way is the following: they give a graded $\mathcal{O}_{\square_Y^n}$-algebra isomorphism
\begin{equation}\label{eqn:gr Tor}
\Tor_* ^{\mathcal{O}_{\square_{\widehat{X}} ^n}} ( \mathcal{A}_1 , \mathcal{O}_{\square_Y^n}) \simeq \Tor_* ^{\mathcal{O}_{\square_{\widehat{X}} ^n}} ( \mathcal{A}_1 , \mathcal{O}_{\square_Y^n}),
\end{equation}
where the graded algebra structures on the sheaves $\Tor_*$ are given as in \cite[Section 068G]{stacks}.

\medskip

In \cite{Park Tate} by the author, on the generating cycles over $\Spf (k[[t]])$, we used an isomorphism similar to \eqref{eqn:gr Tor}, but in terms of just $\Tor_0$. This is equivalent to \eqref{eqn:gr Tor} in this case, because the cycles over $\Spf (k[[t]])$ had a Tor independence property, so that $\Tor_i = 0$ for $i >0$. See \cite[Lemma 2.2.8, Corollary 2.2.9]{Park Tate}. In general, such Tor independence property is no longer guaranteed, and we need to keep track of the information contained in higher Tor's. However the definition in terms of the graded algebra $\Tor_*$ is technically less versatile to work with. This is why we borrowed  the language of derived rings from the derived algebraic geometry, from which we can also deduce the $\Tor_*$-description as well.
\qed
\end{remk}

\begin{lem}\label{lem:boundary comp}
Let $Y, X, \widehat{X}, \mathcal{W}$ be as in Definition \ref{defn:mod Y equiv}. For each $1 \leq i \leq n$ and $\epsilon \in \{ 0, \infty\}$ we have
$ \partial_i ^{\epsilon} \mathcal{M}^q (\widehat{X} , Y, n, \mathcal{W}) \subset \mathcal{M}^q (\widehat{X}, Y, n-1, \mathcal{W}).$
In particular, the inclusion
\begin{equation}\label{eqn:perfect defn}
\mathcal{M}^q (\widehat{X}, Y,  \bullet, \mathcal{W}) \hookrightarrow z^q_{\mathcal{W}} (\widehat{X}, \bullet)
\end{equation}
 is a morphism of complexes.
\end{lem}

\begin{proof}
Consider the closed immersion $\iota_i ^{\epsilon} : F_{i, \widehat{X}} ^{\epsilon} =\square_{\widehat{X}} ^{n-1} \hookrightarrow \square_{\widehat{X}} ^n$ given by the equation $\{ y_i = \epsilon\}$. 

When $\mathcal{A} \in \mathcal{R}^q (\widehat{X}, n, \mathcal{W})$, by Lemma \ref{lem:sheaf proper int}, we have
\begin{equation}\label{eqn:codim 1 face LDF perf cy}
\partial_i ^{\epsilon} [\mathcal{A}] = [ \mathcal{O}_{F_{i, \widehat{X}} ^{\epsilon}} \otimes_{\mathcal{O}_{\square_{\widehat{X}} ^n}} ^{\mathbf{L}} \mathcal{A}] = [ \mathbf{L} (\iota_i ^{\epsilon})^* (\mathcal{A})] \in z^q _{\mathcal{W}} (\widehat{X}, n-1).
\end{equation}
So, $\mathbf{L}  (\iota_i ^{\epsilon})^* (\mathcal{A}) \in \mathcal{R}^q (\widehat{X}, n-1, \mathcal{W})$.

For a pair $(\mathcal{A}_1, \mathcal{A}_2) \in \mathcal{L}^q (\widehat{X}, Y, n, \mathcal{W})$, the condition \eqref{eqn:derived int} means that we have an isomorphism in $\sAlg (\mathcal{O}_{\square_Y^n})$
\begin{equation}\label{eqn:Y codim 1}
\mathbf{L} (\iota_Y^n) ^* ( \mathcal{A}_1) \simeq \mathbf{L} (\iota_Y ^n)^* ( \mathcal{A}_2),
\end{equation}
 where $\iota_Y^n$ is the closed immersion $\iota_Y^n: \square_{Y}^n \hookrightarrow \square_{\widehat{X}} ^n$ of formal schemes. 
 
 Consider the Cartesian square of closed immersions
\begin{equation}\label{eqn:i_Y i_i}
\xymatrix{
\square_Y ^{n-1} \ar@{^(->}[r] ^{\iota_Y ^{n-1}} \ar@{^(->}[d] ^{\iota_i ^{\epsilon}} & \square_{\widehat{X}} ^{n-1} \ar@{^(->}[d] ^{\iota_i ^{\epsilon}}\\
\square_Y ^n \ar@{^(->}[r] ^{\iota_Y ^n} & \square_{\widehat{X}} ^n.}
\end{equation}

Applying $\mathbf{L}(\iota_i ^{\epsilon})^*$ to \eqref{eqn:Y codim 1}, and using the identity
$\mathbf{L} (\iota_Y ^{n-1})^* \circ  \mathbf{L} (\iota_i ^{\epsilon})^* = \mathbf{L} (\iota_i ^{\epsilon})^* \circ  \mathbf{L} (\iota_Y ^n)^*$, which follows from the diagram \eqref{eqn:i_Y i_i} (\cite[I-3.6, p.119]{Lipman}), 
we deduce an isomorphism in $\sAlg (\mathcal{O}_{\square_{Y}^{n-1}})$
$$
\mathbf{L} (\iota_Y ^{n-1})^* \left( \mathbf{L} ( \iota_i ^{\epsilon})^* (\mathcal{A}_1) \right) \simeq
\mathbf{L} (\iota_Y ^{n-1})^* \left( \mathbf{L} ( \iota_i ^{\epsilon})^* (\mathcal{A}_2) \right),
$$ 
thus
\begin{equation}\label{eqn:LDF codim 1 pair}
(\mathbf{L} (\iota_i ^{\epsilon} )^* (\mathcal{A}_1), \mathbf{L}(\iota_i ^{\epsilon})^* (\mathcal{A}_2)) \in \mathcal{L} ^q ( \widehat{X}, Y, n-1, \mathcal{W}).
\end{equation}

By \eqref{eqn:codim 1 face LDF perf cy}, we have
\begin{equation}\label{eqn:codim 1 perf cy diff}
\partial_i ^{\epsilon}  ( [ \mathcal{A}_1] - [ \mathcal{A}_2] )= [ \mathbf{L}(\iota_i ^{\epsilon})^* (\mathcal{A}_1)] - [ \mathbf{L} (\iota_i ^{\epsilon})^* (\mathcal{A}_2)].
\end{equation}
Combining \eqref{eqn:LDF codim 1 pair} and \eqref{eqn:codim 1 perf cy diff}, we deduce that
$$
\partial _i ^{\epsilon} (\mathcal{M} ^q (\widehat{X}, Y, n, \mathcal{W})) \subset \mathcal{M}^q (\widehat{X}, Y, n-1, \mathcal{W}).
$$
This proves the lemma.
\end{proof}

\begin{defn}\label{defn:complex}
Let $Y, X, \widehat{X}, \mathcal{W}$ be as in Definition \ref{defn:mod Y equiv}. Define 
\begin{equation}\label{eqn:comp00}
z^q _{ \mathcal{W}} (\widehat{X} \mod Y, n) := \frac{z^q _{ \mathcal{W}} (\widehat{X}, n)}{ \mathcal{M}^q (\widehat{X}, Y, n, {\mathcal{W}}) }.
\end{equation}
The equivalence relation on the group $z^q _{\mathcal{W}} (\widehat{X},n)$ defined by the subgroup $\mathcal{M}_{\mathcal{W}}^q (\widehat{X}, Y, \mathcal{W})$, is denoted by $\sim_Y$.

Lemma \ref{lem:boundary comp} shows that the collection of \eqref{eqn:comp00} over $n \geq 0$ defines a complex, which we denote by $z^q _{\mathcal{W}} (\widehat{X} \mod Y, \bullet)$. Its homology is denoted by
\begin{equation}\label{eqn:comp01}
\CH^q _{\mathcal{W}} (\widehat{X} \mod Y, n) := {\rm H}_n (z^q _{ \mathcal{W}} (\widehat{X} \mod Y, \bullet)).
\end{equation}

When $\mathcal{W}= \emptyset$, we drop $\mathcal{W}$ from the notations.\qed
\end{defn}

This group is one step closer to the main object of this article, though not the last object yet. 
The author solicits patience of the reader as there seems no easy royal road at present.

\medskip

Note that the ideal of the closed immersion $Y \hookrightarrow \widehat{X}$ is an ideal of definition of the formal scheme $\widehat{X}$. For general ideals of definition of $\widehat{X}$, we have the following:

\begin{lem}\label{lem:m m'}
Let $Y, X, \widehat{X}$ be as in the above. Let $\mathcal{I}, \mathcal{J}$ be ideals of definition of $\widehat{X}$ such that $\mathcal{J} \subset \mathcal{I}$, and let $\widehat{X}_{\mathcal{I}} \hookrightarrow \widehat{X}_{\mathcal{J}} \hookrightarrow \widehat{X}$ be the induced closed immersions. Then we have 
\begin{equation}\label{eqn:m m' -1}
\mathcal{L}^q (\widehat{X}, \widehat{X}_{\mathcal{J}} , n) \subset \mathcal{L}^q (\widehat{X}, \widehat{X}_{\mathcal{I}}, n).
\end{equation}

In particular, $\mathcal{M} ^q (\widehat{X}, \widehat{X}_{\mathcal{J}} , \bullet)$ is a subcomplex of $\mathcal{M}^q (\widehat{X}, \widehat{X}_{\mathcal{I}}, \bullet).$
\end{lem}

\begin{proof}
Let $\iota_{\mathcal{I}}: \widehat{X}_{\mathcal{I}} \hookrightarrow \widehat{X}$, $\iota_J: \widehat{X}_{\mathcal{J}} \hookrightarrow \widehat{X}$, and $\iota_{\mathcal{I}/\mathcal{J}} : \widehat{X}_{\mathcal{I}} \hookrightarrow \widehat{X}_{\mathcal{J}}$ be the induced closed immersions. Since $\iota_{\mathcal{I}} = \iota_{\mathcal{J}} \circ \iota_{\mathcal{I}/\mathcal{J}}$, we deduce the equality of functors
\begin{equation}\label{eqn:m m' 0}
\mathbf{L} \iota_{\mathcal{I}} ^* = \mathbf{L} \iota_{\mathcal{I}/\mathcal{J} } ^* \circ \mathbf{L} \iota_{\mathcal{J}} ^*: 
\sAlg (\mathcal{O}_{\square_{\widehat{X}} ^n}) \to \sAlg (\mathcal{O}_{\square_{\widehat{X}_{\mathcal{I}}}^n}).
\end{equation}

Let $(\mathcal{A}_1, \mathcal{A}_2) \in \mathcal{L}^q (\widehat{X}, \widehat{X}_{\mathcal{J}}, n)$, so 
we have an isomorphism in $\sAlg (\mathcal{O}_{\square_{\widehat{X}_{\mathcal{J}}}^n})$
\begin{equation}\label{eqn:m m' 1}
\mathbf{L} \iota_{\mathcal{J} } ^* \mathcal{A}_1 \simeq \mathbf{L} \iota_{\mathcal{J}} ^* \mathcal{A}_2.
\end{equation}
Applying $\mathbf{L} \iota_{\mathcal{I}/\mathcal{J}} ^*$ to \eqref{eqn:m m' 1}, we obtain an isomorphism $\mathbf{L} \iota_{\mathcal{I}/\mathcal{J}} ^* \circ \mathbf{L} \iota_{\mathcal{J}} ^* \mathcal{F}_1 \simeq \mathbf{L} \iota_{\mathcal{I}/\mathcal{J}} ^* \circ\mathbf{L} \iota_{\mathcal{J}} ^* \mathcal{F}_2$ in $\sAlg (\mathcal{O}_{\square_{\widehat{X}_{\mathcal{I}}} ^n})$, which is equivalent to having an isomorphism in $\sAlg (\mathcal{O}_{\square_{\widehat{X}_{\mathcal{I}}}^n})$
\begin{equation}\label{eqn:m m' 2}
\mathbf{L} \iota_{\mathcal{I} } ^* \mathcal{A}_1 \simeq \mathbf{L} \iota_{\mathcal{I}} ^* \mathcal{A}_2
\end{equation}
 by \eqref{eqn:m m' 0}. This means $(\mathcal{A}_1, \mathcal{A}_2) \in \mathcal{L} ^q (\widehat{X}, \widehat{X}_{\mathcal{I}}, n)$, proving \eqref{eqn:m m' -1}.
\end{proof}

Here is one observation on the generators of the group $\mathcal{M}^q (\widehat{X}, Y, n)$. 

\begin{lem}\label{lem:modulus simple generator}
Let $\alpha \in \mathcal{M}^q (\widehat{X}, Y, n)$. Then there is a pair $(\mathcal{A}_1, \mathcal{A}_2 ) \in \mathcal{L}^q (\widehat{X}, Y, n)$ such that 
\begin{equation}\label{eqn:mod simple 0}
\alpha = [ \mathcal{A}_1 ] - [ \mathcal{A}_2].
\end{equation}
\end{lem}

\begin{proof}
A cycle $\alpha \in \mathcal{M} ^q (\widehat{X}, Y, n)$ is \emph{a priori} a finite formal sum
\begin{equation}\label{eqn:mod simple 1}
\alpha  = \sum_{i=1} ^N n_i ( [ \mathcal{A}_{i,1}] - [\mathcal{A}_{i, 2}]),
\end{equation}
where $n_i \in \mathbb{Z}$ is a nonzero integer, and $(\mathcal{A}_{i,1}, \mathcal{A}_{i,2}) \in \mathcal{L}^q (\widehat{X}, Y, n)$ for $1 \leq i \leq N$. We may assume each $n_i >0$ for otherwise we may interchange $\mathcal{A}_{i,1}$ and $\mathcal{A}_{i,2}$. Then for $j=1,2$, we consider the finite products of sheaves of rings
\begin{equation}\label{eqn:prod ring sheaf}
\mathcal{A}_j:= \prod_{i=1} ^N \mathcal{A}_{i, j} ^{\times n_i}.
\end{equation}
They are admissible coherent $\mathcal{O}_{\square_{\widehat{X}} ^n}$-algebras such that $(\mathcal{A}_1 , \mathcal{A}_2) \in \mathcal{L}^q (\widehat{X}, Y, n)$. 

Note that \eqref{eqn:prod ring sheaf} gives an isomorphism $\mathcal{A}_j \simeq \bigoplus_{i=1} ^N \mathcal{A}_{i,j} ^{\oplus n_i}$ as coherent $\mathcal{O}_{\square_{\widehat{X}}^n}$-modules, while the associated cycles depend just on the $\mathcal{O}_{\square_{\widehat{X}}^n}$-module structures, not on the $\mathcal{O}_{\square_{\widehat{X}} ^n}$-algebra structures. Hence we have 
$$[\mathcal{A}_j] = \sum_{i=1} ^N n_i [ \mathcal{A}_{i,j}].$$
 This shows that \eqref{eqn:mod simple 1} is equivalent to \eqref{eqn:mod simple 0}.
\end{proof}

We record the following simple observation:

\begin{lem}\label{lem:open pb}
Let $Y$ be a quasi-affine $k$-scheme of finite type and let $Y \hookrightarrow X$ be a closed immersion into an equidimensional smooth $k$-scheme. Let $\widehat{X}$ be the completion of $X$ along $Y$. Then we have:
\begin{enumerate}
\item For any Zariski open subscheme $U \subset X$ containing $Y$, the inclusion $Y\hookrightarrow U$ is a closed immersion, and the map of the completions $\widehat{U} \to \widehat{X}$ of $U$ and $X$ along $Y$, respectively, is the identity map.
\item We have $ z^q   (\widehat{U} , \bullet) = z^q  (\widehat{X} , \bullet)$, $ z^q   (\widehat{U} \mod Y, \bullet) = z^q  (\widehat{X} \mod Y, \bullet)$ and
$$
\CH^q(\widehat{U}, n) = \CH^q  (\widehat{X} , n), \ \  \CH^q(\widehat{U} \mod Y, n) = \CH^q  (\widehat{X} \mod Y, n).
$$
\end{enumerate}
\end{lem}

\begin{proof}
(1) is obvious, and (2) follows immediately from (1) because $\widehat{U} = \widehat{X}$.
\end{proof}

\subsection{Connection to the Milnor patching}\label{sec:Milnor}

In Definitions \ref{defn:mod Y equiv} and \ref{defn:complex}, we used pairs $(\mathcal{A}_1, \mathcal{A}_2)$ of coherent $\mathcal{O}_{\square_{\widehat{X}}^n}$-algebras to define the mod $Y$-equivalence relation on the cycles in $z^q (\widehat{X}, n)$.

In \S \ref{sec:Milnor}, we explain a connection between the mod $Y$-equivalence and the \emph{Milnor patching}, a classical vector bundle patching problem. Since this discussion is not needed for the proofs of the main theorem of the paper, one may skip \S \ref{sec:Milnor} and jump to \S \ref{sec:basic functor} if desired.

We stick to the case when $Y$ is affine. 

\medskip

The classical Milnor patching \cite[\S 2, pp.19--24]{Milnor K} reads as follows: given a Cartesian square of noetherian commutative rings with unity,
\begin{equation}\label{eqn:Milnor K}
\xymatrix{
R \ar[r]^{i_1} \ar[d]^{i_2} & R_1 \ar[d] ^{j_1} \\
R_2 \ar[r] ^{j_2} & R',}
\end{equation}
such that $j_1$ or $j_2$ is surjective, J. Milnor proved (\cite[Theorems 2.1, 2.2]{Milnor K}):

\begin{thm}[Milnor]\label{thm:Milnor K}
In the situation of \eqref{eqn:Milnor K}, let $P_i$ be projective $R_i$-modules for $i=1,2$ such that $P_1 \otimes _{R_1} R' \simeq P_2 \otimes_{R_2} R'$ as $R'$-modules.

Then there exists a projective $R$-module $M$ such that $M \otimes_{R} R_i \simeq P_i$ for $i=1,2$. If $P_1, P_2$ are finitely generated, then so is $M$. Furthermore, every projective $R$-module $M$ is obtained in this way.
\end{thm}

We can rephrase Theorem \ref{thm:Milnor K} geometrically as follows: the Cartesian diagram \eqref{eqn:Milnor K} is equivalent to that $\Spec (R) = \Spec (R_1) \coprod_{\Spec (R')} \Spec (R_2)$. Suppose we have vector bundles $V_i$ on $\Spec (R_i)$ for $i=1,2$, whose restrictions to $\Spec (R')$ are equal to each other. Then there is a vector bundle $V$ on $\Spec (R)$ obtained by patching $V_1$ and $V_2$ along $\Spec (R')$.

\medskip

The Milnor patching does not extend to finitely generated modules (or coherent sheaves) in general, but its derived version for perfect complexes, instead of coherent sheaves, holds. It was studied by S. Landsburg \cite[\S 2]{Landsburg Duke} (or \cite[Theorem 1.4 and Appendix]{Landsburg AJM} for the same proof). 
We present a special case of his result in the following form:

\begin{thm}[Landsburg]\label{thm:Landsburg0}
Consider the co-Cartesian square of noetherian affine schemes, where the morphisms are closed immersions:
\begin{equation}\label{eqn:coCat}
\xymatrix{
Y \ar[r]^{j_1} \ar[d] ^{j_2} & X_1 \ar[d] ^{\iota_1} \\
X_2 \ar[r] ^{\iota_2} & \widetilde{X}.}
\end{equation}
Suppose we are given perfect complexes $\mathcal{F}_i$ on $X_i$ for $i=1,2$ such that there is a quasi-isomorphism $\mathbf{L}j_1 ^* \mathcal{F}_1 \simeq \mathbf{L} j_2 ^* \mathcal{F}_2$ on $Y$. Then 
there is a perfect complex $\widetilde{\mathcal{F}}$ on $\widetilde{X}$ such that $\mathbf{L} \iota_i ^* \widetilde{\mathcal{F}}$ 
are quasi-isomorphic to $\mathcal{F}_i$ for $i=1,2$, respectively. 
Conversely, every perfect complex on $\widetilde{X}$ is obtained in this way.
\end{thm}

\medskip

We extend Theorem \ref{thm:Landsburg0} a bit to a situation involving formal schemes:

\begin{thm}\label{thm:Landsburg}
Let $Y$ be an affine $k$-scheme of finite type. Let $Y \hookrightarrow \mathfrak{X}_{i}$ be closed immersions into noetherian affine formal $k$-schemes for $i = 1,2$ such that we have the equalities of their underlying topological spaces $|Y|= |\mathfrak{X}|$. Then:

\begin{enumerate}
\item They form the following co-Cartesian diagram with $\mathfrak{X}:= \mathfrak{X}_1 \coprod_{Y} \mathfrak{X}_2$ as a noetherian affine formal scheme, where the arrows are closed immersions:
$$
\xymatrix{
Y \ar[r]^{j_1} \ar[d] ^{j_2} & \mathfrak{X}_1 \ar[d] ^{\iota_1} \\
\mathfrak{X}_2 \ar[r] ^{\iota_2} & \mathfrak{X}.}
$$
\item In addition, suppose $\mathcal{F}_1, \mathcal{F}_2$ are perfect complexes on $\mathfrak{X}_1, \mathfrak{X}_2$, respectively, such that there is a quasi-isomorphism $\mathbf{L} j_1 ^* \mathcal{F}_1 \simeq \mathbf{L} j_2 ^* \mathcal{F}_2$ on $Y$. 

Then there is a perfect complex $\widetilde{\mathcal{F}}$ on $\mathfrak{X}$ such that $\mathbf{L} \iota_i ^* \widetilde{\mathcal{F}}$
 are quasi-isomorphic to $\mathcal{F}_{i} $ for $i=1,2$, respectively. 
 Conversely, every perfect complex on $\mathfrak{X}$ is obtained in this way.
\end{enumerate}
\end{thm}

We need the following to prove part of the above:

\begin{thm}[{D. Ferrand \cite[Th\'eor\`eme 5.4]{Ferrand}}]\label{thm:Ferrand pushout}Let $X'$ be a scheme, $Y' \subset X'$ a closed subscheme, and $g: Y' \to Y$ is a finite morphism of schemes. Consider the pushout $X:= X' \coprod_{Y'} Y$ in the category of ringed spaces so that we have the co-Cartesian square
$$
\xymatrix{Y' \ar[r] ^g \ar[d] & Y \ar[d] ^u\\
X' \ar[r] ^f & X.}
$$
Suppose that $X'$ and $Y$ satisfy the following property:
\begin{enumerate}
\item [$(FA)$] Each finite subset of points is contained in an affine open subset.
\end{enumerate}
Then $X$ is a scheme satisfying $(FA)$, the diagram is Cartesian as well, and the morphism $f$ is finite, while $u$ is a closed immersion.
\end{thm}

\begin{proof}[Proof of Theorem \ref{thm:Landsburg}]
Let $\mathcal{I}_i \subset \mathcal{O}_{\mathfrak{X}_i}$ be the ideal of the closed immersion $Y \hookrightarrow \mathfrak{X}_i$. 

(1) For each $n \geq 1$, let $\mathfrak{X}_{i, n}$ be the scheme defined by $\mathcal{I}_i ^n$. We have the natural closed immersion $Y \hookrightarrow \mathfrak{X}_{i, n}$. 

Here the underlying topological spaces of $Y$ and $\mathfrak{X}_{i,n}$ are equal, and being affine schemes, all $\mathfrak{X}_{i,n}$ satisfy the condition $(FA)$ of Theorem \ref{thm:Ferrand pushout}. Hence by this theorem of Ferrand, we have the push-out $\mathfrak{X}_n:= \mathfrak{X}_{1,n} \coprod_Y \mathfrak{X}_{2,n}$ as a noetherian affine scheme. We then have the desired noetherian affine formal scheme $\mathfrak{X}:= \mathfrak{X}_{1} \coprod_Y \mathfrak{X}_2$ as the colimit of the ind-scheme $\{\mathfrak{X}_{1,n} \coprod_Y \mathfrak{X}_{2,n}\}_{n}$, which exists in the category of noetherian formal schemes.

(2) For $n \geq 1$, 
we have morphisms ${\mathfrak{X}}_n \overset{ \alpha_{n} ^{n+1}}{\to} {\mathfrak{X}}_{n+1} \overset{\alpha^{n+1}}{\to}
 {\mathfrak{X}}$.

For $i=1,2$, the given morphism $j_i: Y \to \mathfrak{X}_i$ factors into $Y \overset{ \mathfrak{j}_{i, n}}{\to} \mathfrak{X}_{i, n} \overset{\mathfrak{j} _{i, n} ^{n+1}}{\to} \mathfrak{X}_{i, n+1} \overset{\mathfrak{j}_i ^{n+1}}{\to} \mathfrak{X}_i$. For each $n \geq 1$ and each $i=1,2$, the morphism $\iota_i: \mathfrak{X}_i \to 
{\mathfrak{X}}$ induces $\iota_{i, n}: \mathfrak{X}_{i, n} \to {\mathfrak{X}_n}$. 

They give the following commutative diagram, where each horizontal face is a push-out diagram:
$$
\xymatrix{ 
& Y \ar[dl]_{\mathfrak{j}_{2, n}} \ar[rr]^{ \mathfrak{j}_{1, n}} \ar@{=}[dd] & & \mathfrak{X}_{1,n} \ar[dl] ^{ \iota_{1, n}} \ar[dd]^{\mathfrak{j}_{1, n} ^{n+1}} \ar@/^3pc/@{-->}[dddd]^{\mathfrak{j}_{1} ^n}  \\
\mathfrak{X}_{2,n}  \ar[dd]_{\mathfrak{j}_{2, n} ^{n+1}} \ar[rr]^{ \ \ \ \ \ \ \  \iota_{2, n}} \ar@/_3pc/@{-->}[dddd]_{\mathfrak{j}_{2} ^n}  & & {\mathfrak{X}}_n  \ar[dd] & \\
& Y \ar[dl]_{\mathfrak{j}_{2, n+1}} \ar[rr]^{ \ \ \ \ \ \ \  \mathfrak{j}_{1, n+1}} \ar@{=}[dd]  & & \mathfrak{X}_{1, n+1} \ar[dl] ^{\iota _{1, n+1}} \ar@{-->}[dd]^{\mathfrak{j}_1 ^{n+1}} \\
 \mathfrak{X}_{2, n+1} \ar@{-->}[dd]_{\mathfrak{j}_{2} ^{n+1}} \ar[rr] ^{ \ \ \ \ \ \ \ \ \ \iota_{2, n+1}}& & {\mathfrak{X}}_{n+1} \ar@{-->}[dd] & \\
& Y \ar[dl]_{j_2} \ar[rr]^{ \ \ \ \ \ \  j_1} & & \mathfrak{X}_1 \ar[dl]^{\iota_1} \\
\mathfrak{X}_2 \ar[rr]^{\iota_2} & & {\mathfrak{X}.} &}$$

In the above diagram, the top horizontal
face is the $n$-th level, the middle one is the $(n+1)$-th level, and the bottom is the $\infty$-level. 

For each $n \geq 1$, and for $i=1,2$, define $\mathcal{F}_{i, n}:= \mathbf{L}  (\mathfrak{j}_i ^n)^* \mathcal{F}_i$, which is the left derived pull-back of $\mathcal{F}_i$ to $\mathfrak{X}_{i, n}$. The given condition implies that their left derived pull-backs to $Y$ are quasi-isomorphic to each other. Hence for each level $n \geq 1$, by the derived Milnor patching in Theorem \ref{thm:Landsburg0}, there exists a perfect complex $\widetilde{\mathcal{F}}_n$ on ${\mathfrak{X}}_n$ such that their (left derived) pull-backs to $\mathfrak{X}_{i, n}$ are quasi-isomorphic to $\mathcal{F}_{i, n}$. 

On the other hand, since the diagram above commutes, 
by the derived Milnor patching, 
we deduce that the natural maps 
$$
\mathbf{L} (\alpha_n ^{n+1})^* \widetilde{\mathcal{F}}_{n+1} = \widetilde{\mathcal{F}}_{n+1} \times_{{\mathfrak{X}}_{n+1}}^h  {\mathfrak{X}}_n \overset{\sim}{\to} \widetilde{\mathcal{F}}_n 
$$
are quasi-isomorphisms on $\mathfrak{X}_n$ for $n \geq 1$. Let $\widetilde{\mathcal{F}} := \mathbf{R} \varprojlim_n \widetilde{\mathcal{F}}_n$ on ${\mathfrak{X}}$. Then by \cite[Lemma 0CQG]{stacks} (or B. Bhatt \cite[Lemma 4.2]{Bhatt}), this is a perfect complex on ${\mathfrak{X}}$, and the morphism
$$
\mathbf{L} (\alpha^n)^* \widetilde{\mathcal{F}} = \widetilde{\mathcal{F}} \times ^h _{{\mathfrak{X}}} {\mathfrak{X}}_n \to \widetilde{\mathcal{F}}_n
$$
is a quasi-isomorphism for each $n \geq 1$. One checks that this perfect complex $\widetilde{\mathcal{F}}$ on $\widetilde{\mathfrak{X}}$ satisfies the desired properties.
\end{proof}

\medskip 

We now interpret the mod $Y$-equivalence in terms of the derived Milnor patching in Theorem \ref{thm:Landsburg}. 

Let $Y$ be a affine $k$-scheme of finite type, let $Y \hookrightarrow X$ be a closed immersion into an equidimensional smooth $k$-scheme, and let $\widehat{X}$ be the completion of $X$ along $Y$. In the situation of Theorem \ref{thm:Landsburg}, we take  $\mathfrak{X}_1= \mathfrak{X}_2 = \widehat{X}$, and take $\mathfrak{X}$ to be the push-out $\widehat{X} \coprod_Y \widehat{X}$. We call it the \emph{double} of $\widehat{X}$ along $Y$, and denote it by $D_{\widehat{X}} = D (\widehat{X}, Y)$.  We have closed immersions $\iota_{i} : \widehat{X} \hookrightarrow D_{\widehat{X}}$ for $i= 1, 2$.

Let $(\mathcal{A}_1, \mathcal{A}_2)  \in \mathcal{L}^q (\widehat{X} , Y, n)$. Since $\square_{\widehat{X}}^n$ is regular by Lemma \ref{lem:exoskin}, both $\mathcal{A}_1, \mathcal{A}_2$ are perfect complexes on $\square_{\widehat{X}}^n$ by Corollary \ref{cor:coh=perfect}. Our given assumptions on the pair imply that $\mathbf{L} j_1 ^* \mathcal{A} _1 \simeq \mathbf{L} j_2 ^* \mathcal{A}_2$ in $\mathcal{D}_{\perf} (\square_Y ^n)$. Thus by Theorem \ref{thm:Landsburg}, they patch to give a perfect complex $\tilde{\mathcal{F}}$ on $  \square_{D_{\widehat{X}}} ^n= \square_{\widehat{X}} ^n \coprod _{\square_Y^n} \square_{\widehat{X}} ^n$. Let $\tilde{\mathcal{L}} ^q (D_{\widehat{X}}, n)$ be the set of all perfect complexes $\tilde{\mathcal{F}}$ on $D_{\widehat{X}}$ obtained in this manner from pairs in $\mathcal{L}^q (\widehat{X}, Y, n)$. Abusing notations, we write $\tilde{\mathcal{F}} = (\mathcal{A}_1, \mathcal{A}_2)$.

\medskip

For $i = 1,2$, consider the set maps $\mathbf{L}\iota_{i} ^*: \tilde{ \mathcal{L}} ^q (D_{\widehat{X}}, n) \to \mathcal{D}_{\perf} ^q (\widehat{X}, n)$ that send $\tilde{\mathcal{F}} = (\mathcal{A}_1, \mathcal{A}_2)$ to $\mathbf{L} \iota_{i} ^* (\tilde{\mathcal{F}}) = \mathcal{A}_{i}$. Define the difference map
$$
\tau_{\widehat{X}} (n) =  [ \mathbf{L} \iota_1 ^* ( - ) ]  - [   \mathbf{L} \iota_2 ^* (-) ]:  \tilde{ \mathcal{L}} ^q (D_{\widehat{X}}, n)\to z^q (\widehat{X}, n),
$$
that sends $\tilde{\mathcal{F}} = (\mathcal{A}_1, \mathcal{A}_2)$ to the cycle $[\mathcal{A}_1] - [\mathcal{A}_2]$.

The subgroup 
$$
\left< \tau_{\widehat{X}} (n) (  \tilde{ \mathcal{L}} ^q (D_{\widehat{X}}, n)  ) \right> \subset z^q (\widehat{X}, n)
$$
 generated by the image of $\tau_{\widehat{X}} (n)$ is precisely $\mathcal{M}^q (\widehat{X}, Y, n)$ by definition. Let $\tilde{\tau}_{\widehat{X}} (n)$ be the induced homomorphism $\mathbb{Z} [ \tilde{ \mathcal{L}} ^q (D_{\widehat{X}}, n)]  \to z^q (\widehat{X} , n)$, so we have
$$
{\rm coker} (\tilde{\tau}_{\widehat{X}} (n))= \frac{z^q (\widehat{X}, n)}{\mathcal{M}^q (\widehat{X}, Y, n)}.
$$

\subsection{Special flat pull-backs 3: the mod $Y$-equivalence}\label{sec:basic functor}

In Lemmas \ref{lem:pre flat pb0} and \ref{lem:pre flat pb}, we saw that the flat pull-backs of various cycles exist for a few special types of flat morphisms of noetherian quasi-affine formal $k$-schemes. 

In \S \ref{sec:basic functor}, when the formal schemes are of the form $\widehat{X}$, we show that the special flat pull-backs also respect the mod $Y$-equivalences.

Specifically, the special types of flat morphisms $f$ are the following:

\begin{enumerate}
\item [(I)] The open immersions $f: \widehat{X} |_{U} \hookrightarrow \widehat{X}$ of formal schemes for nonempty open subsets $U \subset Y$.

\item [(II)] The further completion morphisms $f: \widehat{X}' \to \widehat{X}$. 

Here $Y' \hookrightarrow Y \hookrightarrow X$ are closed immersions, with $X$ smooth, $\widehat{X}$ and $\widehat{X}'$ are the completions of $X$ along $Y$ and $Y'$, respectively. The formal scheme $\widehat{X}'$ is also the further completion of the formal scheme $\widehat{X}$ by the ideal sheaf of $Y'$ in $\widehat{X}$.
\item [(III)] The projection morphism $f: \widehat{X}_1 \times \widehat{X}_2 \to \widehat{X}_2$.

Here, $\widehat{X}_i$ is the completion of $X_i$ along $Y_i$ for a closed immersion $Y_i \hookrightarrow X_i$ into an equidimensional smooth $k$-scheme $X_i$. The product $\widehat{X}_1 \times \widehat{X}_2$ is the completion of $X_1 \times X_2$ along $Y_1 \times Y_2$, see Lemma \ref{lem:prod completion}.

\end{enumerate}

\begin{lem}\label{lem:fpb}
For $i=1,2$, let $Y_i$ be quasi-affine $k$-schemes of finite type, and let $Y_i \hookrightarrow X_i$ be closed immersions into equidimensional smooth $k$-schemes. Take the completions $\widehat{X}_i$ of $X_i$ along $Y_i$. 

Suppose that the above objects fit in the commutative diagram:
\begin{equation}\label{eqn:fpb diag}
\xymatrix{
\widehat{X}_1 \ar[r]^{f} & \widehat{X}_2 \\
Y_1 \ar[r]^g \ar@{^(->}[u] & Y_2 \ar@{^(->}[u] 
}
\end{equation}
such that $f$ is a flat $k$-morphism of one of the above types ${\rm (I)}$, ${\rm (II)}$, ${\rm (III)}$, the vertical maps are the induced closed immersions, and the restriction $g= f|_{Y_1} $ maps $Y_1$ into $Y_2$. 

Then the flat pull-back $f^*$ of Lemma \ref{lem:pre flat pb} induces a morphism of complexes 
\begin{equation}\label{eqn:fpb00}
f^*: z^q  (\widehat{X}_2 \mod Y_2, \bullet) \to z^q  (\widehat{X}_1 \mod Y_1, \bullet).
\end{equation}
\end{lem}

\begin{proof}
Under the given assumptions on $f$, by Lemma \ref{lem:pre flat pb}, we have the flat pull-back morphism
\begin{equation}\label{eqn:lala1}
f^*: z^q  (\widehat{X}_2, \bullet) \to z^q (\widehat{X}_1, \bullet).
\end{equation}

It remains to show that for each $n \geq 0$, the pull-back $f^*$ of \eqref{eqn:lala1} maps $\mathcal{M} ^q (\widehat{X}_2, Y_2, n)$ to $\mathcal{M}^q (\widehat{X}_1, Y_1, n)$. By Lemma \ref{lem:pull-back perfect}, this follows if we show that for $(\mathcal{A}_1, \mathcal{A}_2) \in \mathcal{L}^q (\widehat{X}_2, Y_2 ,n)$, we have $(f^{*} ( \mathcal{A}_1), f^{*} (\mathcal{A}_2)) \in \mathcal{L}^q (\widehat{X}_1, Y_1, n)$.

\medskip

\textbf{Claim 1:} \emph{For $\mathcal{A} \in \mathcal{R}^q (\widehat{X}_2, n)$, we have $f^* (\mathcal{A}) \in \mathcal{R}^q (\widehat{X}_1, n)$.}

\medskip

Let $\mathcal{A} \in \mathcal{R}^q (\widehat{X}_2, n)$. 
Since $f$ is flat, we have $\mathbf{L}f^* = f^*$. The sheaf pull-back $f^* \mathcal{A}$ is a coherent $\mathcal{O}_{\square_{\widehat{X}_1} ^n}$-algebra, while by Lemma \ref{lem:pull-back perfect}, we have $[f^* (\mathcal{A})] = f^* [ \mathcal{A}]$. However, the latter belongs to $z^q (\widehat{X}_1 , n)$ by \eqref{eqn:lala1}. Hence $f^*(\mathcal{A}) \in \mathcal{R}^q (\widehat{X}_1, n)$, proving Claim 1.

\medskip

\textbf{Claim 2:} \emph{Let $(\mathcal{A}_1, \mathcal{A}_2) \in \mathcal{L}^q (\widehat{X}_2, Y_2, n)$. Then $(f^* (\mathcal{A}_1), f^* ( \mathcal{A}_2)) \in \mathcal{L}^q (\widehat{X}_1, Y_1, n)$.}

\medskip

Since $\mathcal{A}_j \in \mathcal{R}^q (\widehat{X}_2, n)$, we have $f^* (\mathcal{A}_j ) \in \mathcal{R}^q (\widehat{X}_1, n)$ for $j=1,2$ by Claim 1. Now the given conditions on $(\mathcal{A}_1, \mathcal{A}_2)$ say that we have an isomorphism
\begin{equation}\label{eqn:fp00}
\mathcal{A}_1 \otimes_{\mathcal{O}_{ \square_{\widehat{X}_2} ^n}} ^{\mathbf{L}} \mathcal{O}_{\square_{Y_2} ^n } \simeq \mathcal{A}_2 \otimes_{\mathcal{O}_{\square_{\widehat{X}_2} ^n}} ^{\mathbf{L}} \mathcal{O}_{\square_{Y_2} ^n }
\end{equation}
in $\sAlg (\mathcal{O}_{\square_{Y_2}^n})$. Via the natural morphism $\mathcal{O}_{\square_{Y_2} ^n } \to f_* \mathcal{O}_{\square_{Y_1} ^n}$ of sheaves of rings, applying $ - \otimes_{\mathcal{O}_{\square_{Y_2} ^n }}^{\mathbf{L}} f_* \mathcal{O}_{\square_{Y_1} ^n}$ to \eqref{eqn:fp00}, we obtain an isomorphism
\begin{equation}\label{eqn:fp002}
\mathcal{A}_1 \otimes_{\mathcal{O}_{\square_{\widehat{X}_2} ^n}} ^{\mathbf{L}} f_* \mathcal{O}_{\square_{Y_1} ^n } \simeq \mathcal{A}_2 \otimes_{\mathcal{O}_{\square_{\widehat{X}_2} ^n}}  ^{\mathbf{L}} f_* \mathcal{O}_{\square_{Y_1} ^n}
\end{equation}
in $\sAlg (\mathcal{O}_{\square_{Y_2}^n})$.
Since $f$ is flat, we have $\mathbf{L} f^* = f^*$. For $j=1,2$, we also have
\begin{equation}\label{eqn:fp003}
f^* (\mathcal{A}_j \otimes _{\mathcal{O}_{\square_{\widehat{X}_2} ^n}} ^{\mathbf{L}} f_* \mathcal{O}_{\square_{Y_1} ^n}) = f^* (\mathcal{A}_j)\otimes_{\mathcal{O}_{\square_{\widehat{X}_1} ^n}} ^{\mathbf{L}} f^* f_* \mathcal{O}_{\square_{Y_1} ^n}.
\end{equation}
Note also that from the adjoint functors $(f^*, f_*)$, there is the natural morphism $f^* f_* \mathcal{O}_{\square_{Y_1} ^n} \to \mathcal{O}_{\square_{Y_1} ^n}$ of sheaves of rings. Hence applying $\mathbf{L}f^* = f^*$ to \eqref{eqn:fp002}, using \eqref{eqn:fp003}, and then applying $- \otimes_{ f^* f_* \mathcal{O}_{\square_{Y_1} ^n}} ^{\mathbf{L}} \mathcal{O}_{\square_{Y_1} ^n}$, we obtain the isomorphism
\begin{equation}\label{eqn:fp004}
f^* (\mathcal{A}_1) \otimes_{\mathcal{O}_{\square_{\widehat{X}_1} ^n}} ^{\mathbf{L}} \mathcal{O}_{\square_{Y_1} ^n} \simeq f^* (\mathcal{A}_2) \otimes_{\mathcal{O}_{\square_{\widehat{X}_1} ^n}} ^{\mathbf{L}} \mathcal{O}_{\square_{Y_1} ^n}
\end{equation}
in $\sAlg (\mathcal{O}_{\square_{Y_1}^n})$. 

This means $(f^*( \mathcal{A}_1), f^* ( \mathcal{A}_2)) \in \mathcal{L}^q (\widehat{X}_1, Y_1, n)$, proving the Claim 2, thus the lemma.
\end{proof}

\subsection{Special finite push-forwards 3}\label{sec:special pf}

The purpose of \S \ref{sec:special pf} is to show that push-forward operations based on the finite push-forward of \S \ref{sec:2.7} and \S \ref{sec:sfpf2} respect the mod $Y$-equivalences, under further extra assumptions.  

\begin{prop}\label{prop:pushforward}

Let $g: Y_1 \to Y_2$ be a finite flat surjective morphism of quasi-affine $k$-schemes of finite type. 
Suppose there exists a Cartesian square
\begin{equation}\label{eqn:pf origin}
\xymatrix{
X_1 \ar[r] ^f & X_2 \\
Y_1 \ar@{^(->}[u] \ar[r] ^g & Y_2, \ar@{^(->}[u]}
\end{equation}
where the vertical maps are closed immersions into equidimensional smooth $k$-schemes, and $f$ is finite surjective. Then:
\begin{enumerate}
\item The induced morphism $\widehat{f}: \widehat{X}_1 \to \widehat{X}_2$ of noetherian quasi-affine formal $k$-schemes gives the Cartesian diagram
\begin{equation}\label{eqn:pf star}
\xymatrix{
X_1 \ar[r] ^f & X_2 \\
\widehat{X}_1 \ar@{^(->}[u] \ar[r] ^{\widehat{f}} & \widehat{X}_2.\ar@{^(->}[u]}
\end{equation}

\item The morphism $\widehat{f}$ is finite faithfully flat.

\item We have the induced push-forward morphisms of complexes
\begin{equation}\label{eqn:ex pf eq1}
\tuborg \widehat{f}_*: z^q (\widehat{X}_1, \bullet) \to z^q (\widehat{X}_2, \bullet), \\ 
\widehat{f}_*: z^q (\widehat{X}_1 \mod Y_1, \bullet) \to z^q (\widehat{X}_2 \mod Y_2, \bullet).\sluttuborg
\end{equation}
\end{enumerate}
\end{prop}

\begin{proof}
(1) Since the question is local on $X_2$, it is enough to prove it when $X_2$ is affine. Under this assumption, all of $X_1, X_2, Y_1, Y_2$ as well as $\widehat{X}_1, \widehat{X}_2$ are affine. 

In this case, the Cartesian square \eqref{eqn:pf origin} corresponds to the push-out square of $k$-algebras
\begin{equation}\label{eqn:aaaa1}
\xymatrix{
R_2 \ar[d] \ar[r] & R_1 \ar[d] &  \\
R_2/ I_2 \ar[r] & R_1/ I_1, &  R_1/ I_1 \overset{(\star)}{ \simeq} (R_2/ I_2) \otimes_{R_2} R_1,}
\end{equation}
where $\Spec (R_i) = X_i$ and $\Spec (R_i/I_i) = Y_i$.  Since $(R_2/ I_2) \otimes_{R_2} R_1 \simeq R_1/ ( I_2 R_1 )$, the equality $R_1/I_1 = R_1/ (I_2 R_1)$ coming from $(\star)$ in \eqref{eqn:aaaa1} implies that $I_1 = I_2 R_1$. Thus for each $m \geq 1$, we have $I_1 ^m = I_2 ^m R_1$. This implies that $R_1/ I_1 ^m \simeq (R_2/ I_2^m) \otimes_{R_2} R_1$. Taking $\underset{m}{\varprojlim}$ to this, we deduce $\widehat{R}_1 \simeq \widehat{R}_2 \otimes_{R_2} R_1$, because $R_1$ is a finite $R_2$-module (see H. Matsumura \cite[Theorem 8.7, p.60]{Matsumura}). Thus the diagram
$$
\xymatrix{
R_2 \ar[d] \ar[r] & R_1 \ar[d] \\
\widehat{R}_2 \ar[r] & \widehat{R}_1}
$$ is a push-out square. Hence the diagram \eqref{eqn:pf star} is a Cartesian square, proving (1).

\medskip

(2) Since $f$ is a finite surjective morphism between two smooth schemes, the morphism $f$ is automatically flat by EGA ${\rm IV}_2$ \cite[Proposition (6.1.5), p.136]{EGA4-2}. Thus $f$ is finite faithfully flat. Since the diagram \eqref{eqn:pf star} is Cartesian by (1), the pull-back $\widehat{f}$ is also finite faithfully flat.

\medskip

(3) Since $\widehat{f}$ is finite surjective by (2), we have the induced push-forward morphism
\begin{equation}\label{eqn:bebek_pb0}
\widehat{f}_* : z^q (\widehat{X}_1, \bullet)  \to z^q  (\widehat{X}_2, \bullet)
\end{equation}
by Lemma \ref{lem:pre proper pf}. 
This is the first push-forward map of \eqref{eqn:ex pf eq1}.

\medskip

It remains to show that this $\widehat{f}_*$ respects the mod $Y_i$-equivalences, i.e. 
\begin{equation}\label{eqn:pf-1}
\widehat{f}_* (\mathcal{M} ^q (\widehat{X}_1, Y_1, n) ) \subset \mathcal{M}^q (\widehat{X}_2, Y_2, n)
\end{equation}
for $n \geq 0$.

\medskip

\textbf{Claim 1:} If $\mathcal{A} \in \mathcal{R}^q (\widehat{X}_1, n)$, then $\mathbf{R}\widehat{f}_* \mathcal{A} = \widehat{f}_* \mathcal{A}$, $[\widehat{f}_* \mathcal{A}] = \widehat{f}_* [ \mathcal{A}]$, and $\widehat{f}_* (\mathcal{A}) \in \mathcal{R}^q (\widehat{X}_2, n)$. 

\medskip

Since $\mathcal{A}$ is coherent, it is an a.q.c.~sheaf by Lemma \ref{lem:aqc sum}-(5). Since $\widehat{f}_*$ is finite, thus affine, we have $R^r \widehat{f}_* (\mathcal{A}) =  0$ for $r \geq 1 $ and $\mathbf{R}\widehat{f}_* (\mathcal{A}) = \widehat{f}_* (\mathcal{A})$ by Lemma \ref{lem:aqc sum}-(7). Here, $\widehat{f}_* (\mathcal{A})$ is a coherent $\mathcal{O}_{\square_{\widehat{X}_2}^n}$-algebra, while by Lemma \ref{lem:pf coh cy}, we have $ [ \widehat{f}_* ( \mathcal{A})] =\widehat{f}_* [ \mathcal{A}] $, which is in $z^q (\widehat{X}_2, n)$ by \eqref{eqn:bebek_pb0}. Hence $\widehat{f}_* (\mathcal{A}) \in \mathcal{R} ^q (\widehat{X}_2, n)$. This answers the Claim 1.

\medskip

Now let $(\mathcal{A}_1, \mathcal{A}_2) \in \mathcal{L}^q (\widehat{X}_1, Y_1, n)$, so we have an isomorphism
\begin{equation}\label{eqn:pf00}
\mathcal{A}_1 \otimes_{\mathcal{O}_{\square^n _{\widehat{X}_1}}} ^{\mathbf{L}} \mathcal{O}_{\square^n _{Y_1}} \simeq \mathcal{A}_2 \otimes_{\mathcal{O}_{\square^n _{\widehat{X}_1}}} ^{\mathbf{L}} \mathcal{O}_{\square^n _{Y_1}} 
\end{equation}
in $\sAlg (\mathcal{O}_{\square_{Y_1} ^n})$. We make:

\medskip

\textbf{Claim 2:} \emph{We have an isomorphism in ${\mathbf{sAlg}} (\mathcal{O}_{\square_{Y_1} ^n})$
\begin{equation}\label{eqn:pf01}
\mathcal{A}_1 \otimes _{\mathcal{O}_{\square^n _{\widehat{X}_1}}} ^{\mathbf{L}} \widehat{f}^* \mathcal{O}_{\square^n _{Y_2}} \simeq \mathcal{A}_2 \otimes _{\mathcal{O}_{\square^n _{\widehat{X}_1}}} ^{\mathbf{L}} \widehat{f}^* \mathcal{O}_{\square^n_{Y_2}}.
\end{equation}
}

\medskip

The natural morphism $\widehat{f}^* \mathcal{O}_{\square^n _{Y_2}} \to \mathcal{O}_{\square^n _{Y_1}}$ of sheaves of rings is faithfully flat. Thus, we have the isomorphism \eqref{eqn:pf01} if and only if \eqref{eqn:pf01} tensored with $- \otimes_{ \widehat{f}^* \mathcal{O}_{\square^n_{Y_2}}}  {\mathcal{O}_{\square^n_{Y_1}}} $ is an isomorphism in $\sAlg (\mathcal{O}_{\square_{Y_1}^n})$. But the latter is \eqref{eqn:pf00}, so we have proved the Claim 2.

\medskip

Since $\widehat{f}$ is flat, we have
$\mathbf{L}\widehat{f}^* =\widehat{ f}^*$.
Hence applying $\mathbf{R} \widehat{f}_*$ to each side of \eqref{eqn:pf01}, we obtain 
$$
\mathbf{R} \widehat{f}_* (\mathcal{A}_i) \otimes_{\mathcal{O}_{\square^n _{\widehat{X}_2}}} ^{\mathbf{L}}  \mathcal{O}_{\square^n_{Y_2}} \simeq \mathbf{R} \widehat{f}_* ( \mathcal{A}_i \otimes_{\mathcal{O}_{\square^n _{\widehat{X}_1}}} ^{\mathbf{L}} \widehat{f}^* \mathcal{O}_{\square^n_{Y_2}} )
$$
 for $i=1,2$, by the projection formula (SGA VI \cite[Expos\'e III, Proposition 3.7, p.247]{SGA6}). Thus we deduce an isomorphism
\begin{equation}\label{eqn:star_poof}
\mathbf{R} \widehat{f}_* (\mathcal{A}_1) \otimes_{\mathcal{O}_{\square^n _{\widehat{X}_2}}} ^{\mathbf{L}} \mathcal{O}_{\square^n_{Y_2}} \simeq \mathbf{R} \widehat{f}_* (\mathcal{A}_2) \otimes_{\mathcal{O}_{\square^n _{\widehat{X}_2}}} ^{\mathbf{L}} \mathcal{O}_{\square^n_{Y_2}}
\end{equation}
in $\sAlg(\mathcal{O}_{\square_{Y_2}^n})$. Combining \eqref{eqn:star_poof} with the Claim 1, we deduce 
$$
(\widehat{f}_* (\mathcal{A}_1),  \widehat{f}_* (\mathcal{A}_2)) \in \mathcal{L} ^q (\widehat{X}_2, Y_2, n),
$$
and
$$
\widehat{f}_* ([ \mathcal{A}_1] - [\mathcal{A}_2] ) = [ \widehat{f}_* (\mathcal{A}_1)] - [ \widehat{f}_* (\mathcal{A}_2)] \in \mathcal{M}^q (\widehat{X}_2, Y_2, n),
$$
proving \eqref{eqn:pf-1}. This gives the second push-forward map of \eqref{eqn:ex pf eq1}. This complete the proof of the proposition.
\end{proof}

\begin{remk}
We must mention that for a given finite flat surjective morphism $g: Y_1 \to Y_2$ in $\QAff_k$ (or even in $\Aff_k$), in general we may not always be able to find $f: X_1 \to X_2$ and the embeddings $Y_i \hookrightarrow X_i$, that satisfy the assumptions of Proposition \ref{prop:pushforward}. 

Nevertheless, there are certain cases where we do have them. One of such cases is presented in the following Lemma \ref{lem:finite ext base change 1}. 
More cases are to be discussed in the future in a follow-up work on the push-forward structures.
\qed
\end{remk}

\begin{lem}\label{lem:finite ext base change 1}
Let $Y$ be a quasi-affine $k$-scheme of finite type, and let $Y \hookrightarrow X$ be a closed immersion into an equidimensional smooth  $k$-scheme. Let $\widehat{X}$ be the completion of $X$ along $Y$. Let $k \hookrightarrow k'$ be a finite extension of fields. Then:

\begin{enumerate}
\item $X_{k'}$ is smooth over $k'$ and for the closed immersion $Y_{k'} \hookrightarrow X_{k'}$, the completion of $X_{k'}$ along $Y_{k'}$ is  the base change $\widehat{X}_{k'}$ of $\widehat{X}$ to $k'$. 
\item The diagram
$$
\xymatrix{
X_{k'} \ar[r] ^f & X \\
Y_{k'} \ar@{^(->}[u] \ar[r] ^g & Y\ar@{^(->}[u],}
$$
satisfies the assumptions of Proposition \ref{prop:pushforward}.
\end{enumerate}
 In particular, we have the induced push-forward morphism
 $$
 \widehat{f}_*: z^q (\widehat{X}_{k'} \mod Y_{k'}, \bullet) \to z^q (\widehat{X} \mod Y, \bullet).
 $$
\end{lem}

\begin{proof}
(1) and (2) are apparent. The rest follows from Proposition \ref{prop:pushforward}.
\end{proof}

The following is an analogue of \cite[Corollary (1.2)]{Bloch HC}:

\begin{cor}\label{cor:finite ext base change 2}
Let $Y$ be a quasi-affine $k$-scheme of finite type, and let $Y \hookrightarrow X$ be a closed immersion into an equidimensional smooth $k$-scheme. Let $\widehat{X}$ be the completion of $X$ along $Y$. Let $k \hookrightarrow k'$ be a finite extension of fields. Let $\widehat{f}: \widehat{X}_{k'} \to \widehat{X}$ be the base change.

Then the pull-back $\widehat{f}^*$ and the push-forward $\widehat{f}_*$ exist, and the composite
$$
\widehat{f}_* \circ \widehat{f}^*: z^q (\widehat{X} \mod Y, \bullet) \to z^q (\widehat{X} \mod Y, \bullet)
$$
is equal to $[ k' : k] \cdot {\rm Id}$.
\end{cor}

\begin{proof}
We have $\widehat{f}^*$ and $ \widehat{f}_*$ by Lemmas \ref{lem:fpb} and \ref{lem:finite ext base change 1}. 
That the composite $\widehat{f}_*  \circ \widehat{f}^*$ equals $[k' : k] \cdot {\rm Id}$ is immediate. 
\end{proof}

\section{Zariski sheafifications and flasque sheaves}\label{sec:sheaf}

Let $Y$ be a quasi-affine $k$-scheme of finite type. Let $Y \hookrightarrow X$ be a closed immersion into an equidimensional smooth $k$-scheme. Let $\widehat{X}$ be the completion of $X$ along $Y$.

In \S \ref{sec:sheaf}, from the complex $z^q (\widehat{X} \mod Y, \bullet)$ of abelian groups, we derive two Zariski sheaves $\mathcal{S}_{\bullet}^q$ and $\BGHz^q (\widehat{X} \mod Y, \bullet)$. We simultaneously do the corresponding job for $z^q (\widehat{X}, \bullet)$ as well.

Each $\mathcal{S}_{n}^q (U) $ is defined to be the quotient of $z^q(\widehat{X} \mod Y, n)$ modulo the subgroup $G^q (Y \setminus U, n)$ generated by the cycles whose topological supports lie in $(Y \setminus U) \times \square^n$, while $\BGHz^q (\widehat{X} \mod Y, n)$ is given by the sheafification of $U \mapsto z^q (\widehat{X}|_U \mod Y|_U, n)$ for open $U \subset Y$.

We show that each $\mathcal{S}_n ^q$ is a flasque sheaf, and that there is a natural injective morphism of complexes of sheaves
\begin{equation}\label{eqn:sheaf 00}
\mathcal{S}_{\bullet}^q \to \BGHz^q (\widehat{X} \mod Y, \bullet).
\end{equation}
We prove the equality
$$
\CH^q (\widehat{X}\mod Y, n) \simeq \mathbb{H}_{\rm Zar} ^{-n} (Y, \mathcal{S}_{\bullet}^q),
$$
while we define 
$$
\BGH^q  (\widehat{X} \mod Y, n):= \mathbb{H}_{\rm Zar} ^{-n} (Y, \BGHz^q (\widehat{X} \mod Y, \bullet)).
$$
Thus, we deduce that there is a natural homomorphism of groups
\begin{equation}\label{eqn:sheaf 01}
\CH^q (\widehat{X} \mod Y, n) \to \BGH^q (\widehat{X} \mod Y, n).
\end{equation}

We guess the morphism \eqref{eqn:sheaf 00} is a quasi-isomorphism, equivalently \eqref{eqn:sheaf 01} is an isomorphism, but this isn't proven in this article. Some further remarks on this matter are mentioned in the Appendix, \S \ref{sec:appendix}.


\subsection{Presheaf of cycles and sheafification}\label{sec:presheaf}

Let $Y$ be a quasi-affine $k$-scheme of finite type. Let $Y \hookrightarrow X$ be a closed immersion into an equidimensional smooth $k$-scheme, and let $\widehat{X}$ be the completion of $X$ along $Y$. 

For open subsets $U \subset |Y|$, we can also regard $U$ as the open subscheme $ (U, \mathcal{O}_Y|_U)$ of $Y$. In this case, we have the induced closed immersion $ U \hookrightarrow \widehat{X}|_U$, where $\widehat{X}|_U$ is the quasi-affine open formal subscheme $(U, \mathcal{O}_{\widehat{X}}|_U)$ of $\widehat{X}$. 

Let $\mathcal{W} $ be a finite set of closed formal subschemes of $\widehat{X}$.

\begin{lem}\label{lem:presheaf small}
Let $Y, X, \widehat{X}$, $\mathcal{W}$ be as before. Let $Op (Y)$ be the set of all Zariski open subsets of $Y$. 
Then the association
$$
\tilde{\mathcal{P}}_{\bullet} ^q: U \in Op (Y) \mapsto z^q_{\mathcal{W}|_U}  (\widehat{X}|_U, \bullet)
$$
is a presheaf of complexes of abelian groups on $Y_{\rm Zar}$.
\end{lem}

\begin{proof}
By Lemma \ref{lem:pre flat pb}, we have the pull-back restriction map 
$$
\rho^{U}_V: z^q _{\mathcal{W}_U}(\widehat{X}|_U, \bullet) \to z^q _{\mathcal{W}|_V} (\widehat{X}|_V, \bullet)
$$
 for each pair $V \subset U$ of open subsets in $Op (Y)$. If $W \subset V \subset U$ are open in $Y$, we have $\rho^U_W = \rho ^V_W \circ \rho ^U_V$. 
This gives a presheaf.
\end{proof}

\begin{defn}\label{defn:sh_complex0}
Let $Y, X, \widehat{X}$, $\mathcal{W}$ be as before. Define $\BGHz^q _{\mathcal{W}} (\widehat{X}, \bullet)$ to be the Zariski sheafification on $Y_{\rm Zar}$ of the presheaf $\tilde{\mathcal{P}}_{\bullet} ^q$ of Lemma \ref{lem:presheaf small}.
When $\mathcal{W}= \emptyset$, we drop $\mathcal{W}$ from the notation.\qed
\end{defn}

\begin{lem}\label{lem:mod Y presheaf small}
Let $Y, X, \widehat{X}$, $\mathcal{W}$, $U \in Op (Y)$, and $\widehat{X}|_U$ as before. Then
\begin{equation}\label{eqn:mod Y presheaf small}
\mathcal{P}_{\bullet}^q : U \in Op (Y) \mapsto z^q _{\mathcal{W}|_U} (\widehat{X}|_U \mod U, \bullet)
\end{equation}
is a presheaf of complexes of abelian groups on $Y_{\rm Zar}$, where the last $U$ in ``$\mod U$" is seen as the open subscheme $(U, \mathcal{O}_Y|_U)$ of $Y$.
\end{lem}

\begin{proof}
This time, we can use Lemma \ref{lem:fpb} instead of Lemma \ref{lem:pre flat pb}.
\end{proof}

\begin{defn}\label{defn:sh_complex}
Let $Y, X, \widehat{X}$, $\mathcal{W}$ be as before. Define $\BGHz^q _{\mathcal{W}} (\widehat{X} \mod Y, \bullet)$ to be the Zariski sheafification on $Y_{\rm Zar}$ of the presheaf $\mathcal{P}_{\bullet}^q$ of \eqref{eqn:mod Y presheaf small}. 
When $\mathcal{W}= \emptyset$, we drop $\mathcal{W}$ from the notation.\qed
\end{defn}

\begin{remk}
One may ask whether the inclusions
$$
\BGHz^q_{\mathcal{W}} (\widehat{X}, \bullet) \hookrightarrow \BGHz^q (\widehat{X}, \bullet) \ \ \mbox{ and } \ \ \BGHz^q_{\mathcal{W}} (\widehat{X} \mod Y, \bullet) \hookrightarrow \BGHz^q (\widehat{X} \mod Y, \bullet)
$$
are quasi-isomorphisms. We do not know the answer in general. We have a positive answer in a special case needed in this article. See Theorem \ref{thm:formal moving}.\qed.
\end{remk}

\begin{remk}\label{remk:stalk realize 2}
The stalks of the sheaves in Definitions \ref{defn:sh_complex0} and \ref{defn:sh_complex} are given by the colimits as in Remark \ref{remk:stalk realize}. More precisely, by Lemmas \ref{lem:pre flat pb} and \ref{lem:fpb}, for a scheme point $y \in Y$, we have
$$
\BGHz^q (\widehat{X}, \bullet)_y: = \varinjlim_{y \in U} z^q (\widehat{X}|_U, \bullet),
$$
\begin{equation}\label{eqn:stalk description}
\BGHz^q (\widehat{X} \mod Y, \bullet)_y := \varinjlim_{y \in U} z^q (\widehat{X}|_U \mod U, \bullet),
\end{equation}
where the colimits are taken over all open neighborhoods of $y$ in $Y$, while the last $U$ in \eqref{eqn:stalk description} is the open subscheme $(U, \mathcal{O}_Y|_U)$ of $Y$.

As remarked in Remark \ref{remk:stalk realize}, we do not know whether the stalks can be realized as certain cycle complexes of some formal schemes. Thus, later when we mention semi-local schemes, their cycle class groups will be defined using the colimits. See Definition \ref{defn:HCG semi-local}.\qed
\end{remk}

\subsection{Topological supports and flasque sheaves}\label{sec:desc sheaf}

\subsubsection{The topological supports}

\begin{defn}\label{defn:top support}
Let $\mathfrak{X} $ be a noetherian formal scheme and let $W \subset | \mathfrak{X}|$ be a closed subset. 

We say that an integral cycle $\mathfrak{Z}$ on $\mathfrak{X}$ is \emph{supported in $W$}, if the topological support of $\mathcal{O}_{\mathfrak{Z}}$ is contained in $W$. We say that a cycle $\alpha$ on $\mathfrak{X}$ is \emph{supported in $W$}, if the union of the topological supports of the integral components of $\alpha$ is contained in $W$.\qed
\end{defn}

We concentrate on the following situation: let $Y$ be a quasi-affine $k$-scheme of finite type and let $Y \hookrightarrow X$ be a closed immersion into an equidimensional smooth $k$-scheme. Take $\mathfrak{X}=\widehat{X}$, the completion of $X$ along $Y$.

\begin{defn}
Let $W \subset | Y| = | \widehat{X}|$ be a closed subset, and let $U:=Y \setminus W$.  
Define 
$$
 \tilde{G} ^q  (W, n):= \ker ( z^q (\widehat{X}, n) \to z^q (\widehat{X}|_U, n)),
 $$
for the flat restriction map of Lemma \ref{lem:pre flat pb}. This is the group of cycles in $z^q (\widehat{X}, n)$ supported in $W \times \square^n$.
\qed
\end{defn}

We make the following apparent observation:

\begin{lem}\label{lem:topsupp_rel0}
Let $Y, X, \widehat{X}$ be as before. Let $W_1, W_2 \subset Y$ be closed subsets. 
\begin{enumerate}
\item If $W_1 \subset W_2$, then $\tilde{G}^q (W_1, n) \subset \tilde{G}^q (W_2, n)$.
\item We have $\tilde{G}^q  (W_1, n) \cap \tilde{G}^q  (W_2, n) = \tilde{G}^q  (W_1 \cap W_2, n)$.
\end{enumerate}
\end{lem}

\begin{proof}
This is immediate.
\end{proof}

To go modulo $Y$, we need to check the following. The basic intuition is that the mod $Y$-equivalence does not modify the supports inside $Y \times \square^n$:

\begin{lem}\label{lem:topsupp}
Let $Y, X, \widehat{X}$ be as before. Let $(\mathcal{A}_1, \mathcal{A}_2) \in \mathcal{L}^q (\widehat{X}, Y, n)$. Let $\mathfrak{Z}_i := [\mathcal{A}_i]$ for $i=1,2$. Then $|\mathfrak{Z}_1| = |\mathfrak{Z}_2|$.

In particular, for cycles in $z^q (\widehat{X}, n)$, the notion of topological support is stable under the mod $Y$-equivalence. 
\end{lem}

\begin{proof}
Let $\mathcal{I}_0$ be the largest ideal of definition of $\widehat{X}$ and let $Y_{\red}= \widehat{X}_{\red}$ be the scheme it defines.

The topological support of $Z_i = (\mathfrak{Z}_i$ modulo $\mathcal{I}_0)$ is equal to that of $\mathfrak{Z}_i$, i.e. $|\mathfrak{Z}_i| = |Z_i |$.

Note that $|Z_i|$ is given by the support of $\mathcal{A}_i \otimes_{\mathcal{O}_{\square_{\widehat{X}} ^n}} \mathcal{O}_{\square_{Y_{\red}}^n} = \Tor_0 ^{ \mathcal{O}_{\square_{\widehat{X}} ^n}} ( \mathcal{A}_i, \mathcal{O}_{\square_{Y _{\red}}^n})$. On the other hand, the condition $(\mathcal{A}_1, \mathcal{A}_2) \in \mathcal{L}^q (\widehat{X}, Y, n)$ implies in particular that 
$$
\Tor_0 ^{ \mathcal{O}_{\square_{\widehat{X}} ^n}} ( \mathcal{A}_1, \mathcal{O}_{\square_Y^n}) \simeq \Tor_0 ^{ \mathcal{O}_{\square_{\widehat{X}} ^n}} ( \mathcal{A}_2, \mathcal{O}_{\square_Y^n}),
$$
which in turns implies that
$$
\Tor_0 ^{ \mathcal{O}_{\square_{\widehat{X}} ^n}} ( \mathcal{A}_1, \mathcal{O}_{\square_{Y_{\red}}^n}) \simeq \Tor_0 ^{ \mathcal{O}_{\square_{\widehat{X}} ^n}} ( \mathcal{A}_2, \mathcal{O}_{\square_{Y_{\red}}^n}),
$$
by applying $- \otimes_{\mathcal{O}_{\square_Y^n}} \mathcal{O}_{\square_{Y_{\red}}^n}$. From this, we deduce that $|Z_1| = |Z_2|$.

Hence we have the equalities of the topological supports
$$
|\mathfrak{Z}_1| = |Z_1| = |Z_2| = |\mathfrak{Z}_2|.
$$

The second assertion follows from the first one together with Lemma \ref{lem:modulus simple generator}. 
\end{proof}

\begin{defn}
For a closed subset $W \subset Y$, define $G^q (W, n) \subset z^q (\widehat{X} \mod Y, n)$ to be the subgroup of cycles supported in $W \times \square^n$. This is well-defined by Lemma \ref{lem:topsupp}.\qed
\end{defn}
We have $G^q (W, n) = \tilde{G}^q (W, n) / (\tilde{G}^q (W, n) \cap \mathcal{M} ^q (\widehat{X}, Y, n))$ and the exact sequence
$$
0 \to G^q (W, n) \to z^q  (\widehat{X} \mod Y, n) \to z^q  (\widehat{X}|_U \mod U, n)
$$
for $U:= Y \setminus W$.

\begin{lem}\label{lem:topsupp_rel}
Let $Y, X, \widehat{X}$ be as before. Let $W_1, W_2 \subset Y$ be closed subsets. 
\begin{enumerate}
\item If $W_1 \subset W_2$, then $G^q (W_1, n) \subset G^q (W_2, n)$.
\item We have $G^q  (W_1, n) \cap G^q  (W_2, n) = G^q  (W_1 \cap W_2, n)$.
\end{enumerate}
\end{lem}

\begin{proof}
This is immediate.
\end{proof}

\subsubsection{The quotients by the support subgroups}\label{sec:6.2.2}
Proposition \ref{prop:flasque sh} below is inspired by S. Bloch \cite[Theorem (3.4), p.278]{Bloch HC}. We follow the sketch in \cite[\S 4]{cubical localization} and extend it to our presheaves of cycles on formal schemes.

\begin{lem}\label{lem:flasque psh}
Let $Y, X, \widehat{X}$ be as before. Let $n \geq 0$ be integers. Consider the associations
\begin{equation}\label{eqn:sheaf small}
\tilde{\mathcal{S}}_n = \tilde{\mathcal{S}}_n^q :  U \in Op (Y) \mapsto  \frac{ z^q  ( \widehat{X} , n)}{\tilde{G}^q (Y\setminus U, n )},
\end{equation}
\begin{equation}\label{eqn:mod Y sheaf small}
\mathcal{S}_n = \mathcal{S}_n^q :  U \in Op (Y) \mapsto  \frac{ z^q  ( \widehat{X}  \mod Y, n)}{G^q (Y\setminus U, n )}.
\end{equation}
Then $\tilde{\mathcal{S}}_n$ and $\mathcal{S}_n$ are flasque presheaves on $Y_{\rm Zar}$.
\end{lem}

\begin{proof}
For two Zariski open subsets $V \subset U$ of $Y$, we have closed subsets $Y \setminus U \subset Y \setminus V \subset Y$. They induce the natural inclusions (Lemmas \ref{lem:topsupp_rel0} and \ref{lem:topsupp_rel})
$$
\tuborg \tilde{G}^q (Y \setminus U,n) \subset \tilde{G}^q (Y \setminus V,n) \subset z^q  (\widehat{X}, n), \\
G^q (Y \setminus U,n) \subset G^q (Y \setminus V,n) \subset z^q  (\widehat{X} \mod Y, n).\sluttuborg
$$
These in turn induce the natural surjections
$$
\tuborg
 \rho_{V} ^U: \tilde{\mathcal{S}}_n (U) =  \frac{ z^q ( \widehat{X} , n)}{\tilde{G}^q ( Y \setminus U, n) } \twoheadrightarrow \tilde{\mathcal{S}}_n (V) =  \frac{ z^q  ( \widehat{X} , n)}{\tilde{G}^q ( Y \setminus V, n)}, \\
\rho_{V} ^U: \mathcal{S}_n (U) =  \frac{ z^q ( \widehat{X}  \mod Y, n)}{G^q ( Y \setminus U, n) } \twoheadrightarrow \mathcal{S}_n (V) =  \frac{ z^q  ( \widehat{X}  \mod Y, n)}{G^q ( Y \setminus V, n)}.
\sluttuborg
$$
When $U_3 \subset U_2 \subset U_1$ are open subsets of $Y$, that $\rho_{U_3} ^{U_1} = \rho_{U_3} ^{U_2} \circ \rho_{U_2} ^{U_1}$ is apparent. Hence $\tilde{\mathcal{S}}_n$ and $\mathcal{S}_n$ are flasque presheaves.
\end{proof}

\begin{prop}\label{prop:flasque sh}
The presheaves $\tilde{\mathcal{S}}_n$ and $\mathcal{S}_n$ of Lemma \ref{lem:flasque psh} are flasque sheaves on $Y_{\rm Zar}$. 
\end{prop}

\begin{proof}
By Lemma \ref{lem:flasque psh}, we know that $\tilde{\mathcal{S}}_n$ and $\mathcal{S}_n$ are flasque presheaves. It remains to prove that $\tilde{\mathcal{S}}_n$ and $\mathcal{S}_n$ are sheaves. Since the proof for $\tilde{\mathcal{S}}_n$ is identical to that for $\mathcal{S}_n$, we give the argument for $\mathcal{S}_n$ only.

\medskip

Since $Y$ is quasi-compact, by induction on the number of open sets in a finite open cover of $Y$, one reduces to check that for two open subsets $U, V \subset Y$, the sequence
\begin{equation}\label{eqn:sheafproof}
0 \to \mathcal{S}_n (U \cup V) \overset{\delta_0}{\to} \mathcal{S}_n (U) \oplus \mathcal{S}_n (V) \overset{\delta_1}{\to} \mathcal{S}_n (U \cap V)
\end{equation}
is exact, where $\delta_0 = (\rho_U ^{U \cup V}, \rho_V ^{U \cup V})$ and $\delta_1 (\bar{x}_1, \bar{x}_2) = \rho^U_{U\cap V} (\bar{x}_1) - \rho ^V _{U \cap V} (\bar{x}_2)$. 

\medskip

\textbf{Step 1:} We first show that $\delta_0$ is injective. 

Let $\bar{x} \in \mathcal{S}_n (U \cup V)$. Suppose that $\delta_0 (\bar{x}) = 0$. It means that there is a cycle $x \in z^q (\widehat{X} \mod Y, n)$ representing $\bar{x}$ such that it satisfies 
\begin{equation}\label{eqn:exactness1}
x \in G^q (Y \setminus U,n ) \cap G^q (Y \setminus V, n).
\end{equation}
The intersection of the groups in \eqref{eqn:exactness1} is equal to $G^q ((Y \setminus U) \cap (Y \setminus  V), n)= G^q (Y \setminus (U \cup V), n)$ by Lemma \ref{lem:topsupp_rel}. Hence \eqref{eqn:exactness1} means $\bar{x} = 0$ in $\mathcal{S}_n (U \cup V)$, showing that $\delta_0$ is injective.

\medskip

\textbf{Step 2:} That $\delta_1 \circ \delta_0 = 0$ is apparent so that ${\rm im} (\delta_0) \subset \ker \delta_1$. 

\medskip

\textbf{Step 3:} We prove ${\rm im} (\delta_0) \supset \ker \delta_1$. 

Let $(\bar{x}_1, \bar{x}_2) \in \ker \delta_1$. This means that there are representative cycles $x_1 , x_2 \in z^q  (\widehat{X} \mod Y, n)$ of $\bar{x}_1, \bar{x}_2$ such that they satisfy 
\begin{equation}\label{eqn:exactness2}
x_1 - x_2 \in G^q ( Y \setminus (U \cap V), n) = G^q ( (Y \setminus U) \cup (Y \setminus V), n).
\end{equation}

 Since $x_1 - x_2$ is supported in $ (Y \setminus (U \cap V)) \times \square^n$ by \eqref{eqn:exactness2}, the restriction $x_1|_{U\cap V} - x_2 |_{U \cap V}$ has the empty topological support. Hence the cycles $x_1|_U$ and $x_2|_V$ glue on $\widehat{X}|_{U \cup V} \times \square^n$ so that there exists a cycle $x$ on $\widehat{X}|_{U \cup V} \times \square^n$, such that
 $$
 x|_{U} = x_1|_U, \ \ x|_V = x_2 | _V.
 $$

Let $\tilde{x}$ be the Zariski closure of $x$ in $\widehat{X} \times \square ^n$. By construction, it has no component whose topological support is contained in $(Y \setminus (U \cup V)) \times \square^n$. On the other hand, we don't know whether $\tilde{x} \in z^q (\widehat{X} \mod Y, n)$. We claim it is, and prove it in the following.

\medskip

Let $\tilde{x}_U$ be the Zariski closure of $x_1|_U$ in $\widehat{X} \times \square^n$. By construction, it has no component whose topological support is contained in $ (Y \setminus U) \times \square^n$. On the other hand, it is the closure of an open restriction $x_1|_U$ of already admissible $x_1 \in z^q (\widehat{X} \mod Y, n)$, so we have $\tilde{x}_U \in z^q (\widehat{X} \mod Y, n)$. Similarly, if we define $\tilde{x}_V$ to be the Zariski closure of $x_2 |_V$ in $\widehat{X} \times \square^n$, then it has no component topologically supported in $(Y \setminus V) \times \square^n$, while $\tilde{x}_V \in z^q (\widehat{X} \mod Y, n)$.

Now, by construction, we can write
$$
\tilde{x} = \tilde{x}_U + \Delta_U = \tilde{x}_V + \Delta_V
$$
for some cycles $\Delta_U $ and  $\Delta_V$ on $\widehat{X} \times \square^n$ supported in $(Y \setminus U) \times \square^n$ and $(Y \setminus V)\times \square^n$, respectively. Thus
\begin{equation}\label{eqn:S sheaf 0}
- \Delta_U + \Delta_V = \tilde{x}_U - \tilde{x}_V  \in z^q (\widehat{X} \mod Y, n), 
\end{equation}
though we do not know whether individually $\Delta_U$ or $ \Delta_V$ belongs to $z^q (\widehat{X} \mod Y, n)$.

If there is an integral component $\mathfrak{W}$ of $\Delta_U$ that is \emph{not} admissible, i.e. not in $z^q (\widehat{X} \mod Y, n)$, then by \eqref{eqn:S sheaf 0} there must exist the corresponding $\mathfrak{W}$ in $\Delta_V$ to cancel the $\mathfrak{W}$ of $\Delta_U$. In particular, $\mathfrak{W}$ is topologically supported in $|\Delta_U| \cap |\Delta_V| \subset (Y \setminus (U \cup V)) \times \square^n$. Thus, we may write
\begin{equation}\label{eqn:S sheaf 1}
\Delta_U = \Delta_U ' + \Delta_{U,V},
\end{equation}
where $\Delta_U' \in z^q (\widehat{X} \mod Y, n)$ and $\Delta_{U,V}$ is a cycle on $\widehat{X} \times \square^n$ whose topological support is in $(Y \setminus (U \cup V)) \times \square^n$. Here, in case $\Delta_U '$ has an integral component whose support is contained in $(Y \setminus (U \cup V)) \times \square^n$, then move it into $\Delta_{U,V}$, so that we may assume $\Delta_U' \in z^q (\widehat{X} \mod Y, n)$ and no component of $\Delta_U'$ has the topological support contained in $(Y \setminus (U \cup V)) \times \square^n$.

Combining \eqref{eqn:S sheaf 0} and \eqref{eqn:S sheaf 1}, we now have
$$
\tilde{x} = \tilde{x}_U + \Delta_U ' + \Delta_{U,V},
$$
but by the construction of $\tilde{x}$, it has no component that is topologically supported in $(Y \setminus (U \cup V)) \times \square^n$. Hence we must have $\Delta_{U,V} =0$. In particular, we have $\tilde{x} = \tilde{x}_U + \Delta_U ' \in z^q (\widehat{X} \mod Y, n)$, proving the claim.

\medskip

Let $\bar{x} \in \mathcal{S}_n (U \cup V)= \frac{z^q (\widehat{X} \mod Y, n)}{G^q (Y \setminus (U \cup V), n)}$ be the class of $\tilde{x}$. Because $\tilde{x} = \tilde{x}_U + \Delta_U'$, while $\Delta_U ' \in G^q (Y \setminus U, n)$, we have $\tilde{x}_U \equiv \rho^{U \cup V}_U (\bar{x})$ in $\mathcal{S}_{n} (U)$. By symmetry we have $\tilde{x}_V \equiv \rho ^{U \cup V}_V (\bar{x})$ in $\mathcal{S}_n (V)$. 

Since $\bar{x}_1 \equiv \tilde{x}_U$ in $\mathcal{S}_n (U)$ and $\bar{x}_2 \equiv \tilde{x}_V$ in $\mathcal{S}_n (V)$ by the definition of $\tilde{x}_U$ and $\tilde{x}_V$, we have $\delta_0 (\bar{x}) =(\bar{x}_1, \bar{x}_2)$. This proves the exactness of \eqref{eqn:sheafproof} in the middle. 

\medskip

This shows \eqref{eqn:sheafproof} is exact, finishing the proof.
\end{proof}

\begin{cor}\label{cor:flasque sh}
For each pair of open subsets $U, V \subset Y$, we have the short exact sequences
$$
\tuborg 
0 \to \tilde{ \mathcal{S}}_n (U \cup V) \overset{\delta_0}{\to} \tilde{\mathcal{S}}_n (U) \oplus \tilde{\mathcal{S}}_n (V) \overset{\delta_1}{\to}  \tilde{ \mathcal{S}}_n (U \cap V) \to 0, \\
0 \to \mathcal{S}_n (U \cup V) \overset{\delta_0}{\to} \mathcal{S}_n (U) \oplus \mathcal{S}_n (V) \overset{\delta_1}{\to} \mathcal{S}_n (U \cap V) \to 0.
\sluttuborg
$$
\end{cor}

\begin{proof}
By the exact sequence \eqref{eqn:sheafproof} of Proposition \ref{prop:flasque sh} and the corresponding sequence for $\tilde{\mathcal{S}}_n$, it only remains to check the surjectivity of the map $\delta_1$. But since $\tilde{\mathcal{S}}_n$ and $\mathcal{S}_n$ are flasque (Lemma \ref{lem:flasque psh}), the restriction maps $\tilde{\mathcal{S}}_n (U) \to \tilde{\mathcal{S}}_n (U \cap V)$ and $\mathcal{S}_n (U) \to \mathcal{S}_n (U \cap V)$ are already surjective. Thus so are the maps $\delta_1$.
\end{proof}

\subsection{The morphism $\mathcal{S}_{\bullet} ^q \to \BGHz^q (\widehat{X} \mod Y, \bullet)$}

We discuss how the sheaves $\mathcal{S}_{\bullet}^q$ and $\BGHz^q (\widehat{X} \mod Y, \bullet)$ are related (respectively, $\tilde{\mathcal{S}}_{\bullet} ^q$ and $\BGHz^q (\widehat{X}, \bullet)$).

\begin{lem}\label{lem:flasque natural}
Let $Y$ be a quasi-affine $k$-scheme of finite type. Let $Y \hookrightarrow X$ be a closed immersion into an equidimensional smooth $k$-scheme, and let $\widehat{X}$ be the completion of $X$ along $Y$. 

Then there exist natural injective morphisms of complexes of sheaves on $Y_{\rm Zar}$
\begin{equation}\label{eqn:SP_sh}
\tuborg 
\tilde{\mathcal{S}}_{\bullet}^q  \to \BGHz^q (\widehat{X}, \bullet), \\
\mathcal{S}_{\bullet}^q  \to \BGHz^q (\widehat{X} \mod Y, \bullet).
\sluttuborg
\end{equation}

\end{lem}

\begin{proof}
For each nonempty open subset $U \subset Y$, the sequences
$$
\tuborg 
0 \to \tilde{G}^q (Y \setminus U, n) \to z^q (\widehat{X} , n) \to z^q  (\widehat{X}|_U, n), \\
0 \to G^q (Y \setminus U, n) \to z^q (\widehat{X} \mod Y, n) \to z^q  (\widehat{X}|_U \mod U, n)
\sluttuborg
$$ 
are exact. Hence we deduce natural injective homomorphisms 
\begin{equation}\label{eqn:global qi-local}
\tuborg \tilde{\mathcal{S}}_n ^q (U) = \frac{ z^q (\widehat{X} , n)}{\tilde{G}^q (Y \setminus U, n)} \to \tilde{\mathcal{P}}_n ^q (U) =z^q (\widehat{X}|_U , n), \\
\mathcal{S}_n ^q (U) = \frac{ z^q (\widehat{X} \mod Y, n)}{G^q (Y \setminus U, n)} \to \mathcal{P}_n ^q (U) =z^q (\widehat{X}|_U \mod U, n).\sluttuborg
\end{equation}
They are compatible with the boundary maps $\partial$. So, collecting them over $n \geq 0$, we deduce the injective homomorphisms of complexes of presheaves on $Y$
\begin{equation}\label{eqn:global qi-local2}
\tilde{\mathcal{S}}_{\bullet}^q  \to \tilde{\mathcal{P}}_{\bullet} ^q \ \ \mbox{and} \ \ \mathcal{S}_{\bullet}^q  \to \mathcal{P}_{\bullet} ^q.
\end{equation}
They in turn induce the injective morphisms \eqref{eqn:SP_sh} of complexes of sheaves on $Y_{\rm Zar}$
because $\BGHz^q (\widehat{X}, \bullet)$ and $\BGHz^q (\widehat{X} \mod Y, \bullet)$ are the sheafifications of $\tilde{\mathcal{P}}_{\bullet}^q$ and $\mathcal{P}_{\bullet}^q $, respectively (Definitions \ref{defn:sh_complex0} and \ref{defn:sh_complex}), while $\tilde{\mathcal{S}}_{\bullet} ^q$ and $\mathcal{S}_{\bullet} ^q$ are already sheaves by Proposition \ref{prop:flasque sh}. This proves the lemma.
\end{proof}

\begin{thm}\label{thm:two BGH -1}
Let $Y$ be a quasi-affine $k$-scheme of finite type. Let $Y \hookrightarrow X$ be a closed immersion into an equidimensional smooth $k$-scheme, and let $\widehat{X}$ be the completion of $X$ along $Y$. 

Then we have isomorphisms of abelian groups
\begin{equation}\label{eqn:aaaa2}
\tuborg
\CH^q (\widehat{X}, n) \simeq \mathbb{H}_{\rm Zar} ^{-n} (Y, \tilde{\mathcal{S}}_{\bullet} ^q), \\
\CH^q (\widehat{X}\mod Y, n) \simeq \mathbb{H}_{\rm Zar} ^{-n} (Y, \mathcal{S}_{\bullet} ^q).
\sluttuborg
\end{equation}
In particular, the cycle groups satisfy the Zariski descent.
\end{thm}

\begin{proof}
The complexes $\tilde{\mathcal{S}}_{\bullet} ^q$ and $\mathcal{S}_{\bullet}^q$ are complexes of flasque sheaves (Proposition \ref{prop:flasque sh}). Hence we have
$$
\tuborg
\mathbb{H}_{\rm Zar} ^{-n} (Y, \tilde{\mathcal{S}}_{\bullet} ^q)=  {\rm H}^{-n} \Gamma ( Y, \tilde{\mathcal{S}}_{\bullet} ^q)= {\rm H}_{n} \Gamma ( Y, \tilde{ \mathcal{S}}_{\bullet} ^q), \\
 \mathbb{H}_{\rm Zar} ^{-n} (Y, \mathcal{S}_{\bullet} ^q ) = {\rm H}^{-n} \Gamma ( Y, \mathcal{S}_{\bullet} ^q)= {\rm H}_{n} \Gamma ( Y, \mathcal{S}_{\bullet} ^q).
\sluttuborg
$$

On the other hand, for each $n \geq 0$ by definition
$$
\tuborg 
\Gamma (Y,  \tilde{\mathcal{S}}_n ^q) =  \frac{z^q  (\widehat{X} , n)}{\tilde{G}^q (Y \setminus Y,n)} = z^q (\widehat{X}, n),\\
\Gamma (Y,  \mathcal{S}_n ^q) =  \frac{z^q  (\widehat{X} \mod Y, n)}{G^q (Y \setminus Y,n)} = z^q (\widehat{X} \mod Y, n),
\sluttuborg
$$
so that the $n$-th homology groups of the complexes $\Gamma (Y,  \tilde{\mathcal{S}}_{\bullet}^q)$ and $\Gamma (Y,  \mathcal{S}_{\bullet}^q)$ are precisely $\CH^q (\widehat{X}, n)$ and $\CH^q (\widehat{X} \mod Y, n)$, respectively. This proves \eqref{eqn:aaaa2}.

The last assertion follows by combining the first part with Corollary \ref{cor:flasque sh}.
\end{proof}

We introduce the following notation:

\begin{defn}\label{defn:BGH quasi-affine}
Let $Y$ be a quasi-affine $k$-scheme of finite type. Let $Y \hookrightarrow X$ be a closed immersion into an equidimensional smooth $k$-scheme, and let $\widehat{X}$ be the completion of $X$ along $Y$. Define
$$
\tuborg
\BGH^q (\widehat{X}, n):= \mathbb{H}_{\rm Zar} ^{-n} (Y, \BGHz^q (\widehat{X}, \bullet)), \\
\BGH^q (\widehat{X} \mod Y, n):= \mathbb{H}_{\rm Zar} ^{-n} (Y, \BGHz^q (\widehat{X} \mod Y, \bullet).
\sluttuborg
$$
\end{defn}

By Lemma \ref{lem:flasque natural} and Theorem \ref{thm:two BGH -1}, we have:

\begin{cor}\label{cor:two BGH map}
Let $Y$ be a quasi-affine $k$-scheme of finite type. Let $Y \hookrightarrow X$ be a closed immersion into an equidimensional smooth $k$-scheme, and let $\widehat{X}$ be the completion of $X$ along $Y$. 

Then there exist natural homomorphisms of groups
$$
\tuborg
\CH^q (\widehat{X}, n) \to \BGH^q (\widehat{X}, n), \\
\CH^q (\widehat{X} \mod Y, n) \to \BGH^q (\widehat{X} \mod Y, n).
\sluttuborg
$$
If $|Y|$ is a singleton, the above maps are isomorphisms.
\end{cor}

\begin{proof}
The first assertion follows immediately from Lemma \ref{lem:flasque natural} and Theorem \ref{thm:two BGH -1}. 

For the second part, note that when the underlying topological space $|Y|$ is a singleton, for any nonempty open subset $U \subset |Y|$, $U = |Y|$ and the maps in \eqref{eqn:SP_sh} are surjections, as well.
\end{proof}

We guess \eqref{eqn:SP_sh} are quasi-isomorphisms in general, so that the homomorphisms in Corollary \ref{cor:two BGH map} are isomorphisms. We do not need it in this article. Relevant discussions are made in Appendix \S \ref{sec:appendix}.

\section{A moving lemma for cycles on formal schemes}\label{sec:first indep}

The goal of \S \ref{sec:first indep} is to prove a version of moving lemma for the complexes $\BGHz^q (\widehat{X}, \bullet)$ and $\BGHz^q (\widehat{X} \mod Y, \bullet)$ of sheaves of cycles on formal schemes, and to deduce the existence of pull-backs. 

\medskip

More precisely, we prove the following. Suppose we have closed immersions $Y \hookrightarrow X_1 \hookrightarrow X_2$. We prove there exist zigzags of morphisms of complexes sheaves
$$
\tuborg
\BGHz^q ( \widehat{X}_2, \bullet) \underset{\sim}{\overset{\mathfrak{i}}{\hookleftarrow}} \BGHz^q _{\{ \widehat{X}_1 \}} (\widehat{X}_2, \bullet) \overset{\widehat{\iota}^*}{\to} \BGHz^q (\widehat{X}_1, \bullet),\\
\BGHz^q ( \widehat{X}_2 \mod Y, \bullet) \underset{\sim}{\overset{\mathfrak{i}}{\hookleftarrow}} \BGHz^q _{\{ \widehat{X}_1 \}} (\widehat{X}_2 \mod Y, \bullet) \overset{\widehat{\iota}^*}{\to} \BGHz^q (\widehat{X}_1 \mod Y, \bullet),
\sluttuborg
$$
where the Gysin maps $\widehat{\iota}^*$ are constructed in \S \ref{sec:4.1}, and the inclusions $\mathfrak{i}$ are quasi-isomorphisms. The latter is the moving lemma, Theorem \ref{thm:formal moving}. Thus we have morphisms in $\mathcal{D}^- ({\rm Ab} (Y))$
$$
\tuborg
\BGHz^q (\widehat{X}_2, \bullet) \to \BGHz^q (\widehat{X}_1, \bullet), \\
\BGHz^q (\widehat{X}_2 \mod Y, \bullet) \to \BGHz^q (\widehat{X}_1 \mod Y, \bullet).
\sluttuborg
$$

\medskip

More generally, for $Y_1, Y_2 \in \QAff_k$ with closed immersions $Y \hookrightarrow X_i$ for $i=1,2$, suppose we have morphisms $f : X_1 \to X_2$ and $g: Y_1 \to Y_2$ that form a commutative diagram
$$
\xymatrix{
{X}_1 \ar[r]^{f} & {X}_2 \\ 
Y_1 \ar@{^{(}->}[u] \ar[r] ^g & Y_2. \ar@{^{(}->}[u]
}
$$
Based on Theorem \ref{thm:formal moving}, we prove in Theorem \ref{thm:pull-back moving} that there are pull-back morphisms
$$
\tuborg
\BGHz^q (\widehat{X}_2, \bullet) \to \mathbf{R} g_* \BGHz^q (\widehat{X}_1, \bullet), \\
\BGHz^q (\widehat{X}_2 \mod Y_2, \bullet) \to \mathbf{R}  g_* \BGHz^q (\widehat{X}_1 \mod Y_1, \bullet),\sluttuborg
$$
 in the derived category $\mathcal{D}^- ({\rm Ab} (Y_2))$ of bounded above complexes of sheaves on $(Y_2)_{\rm Zar}$.

\medskip

They will be a basis of our constructions of the yeni higher Chow sheaves in \S \ref{sec:finite type} via a \v{C}ech machine and (homotopy) colimits.

\subsection{A Gysin pull-back}\label{sec:4.1}

We start with the basic case:

\begin{prop}\label{prop:1st indep}
Let $Y$ be a quasi-affine $k$-scheme of finite type. Suppose we have closed immersions $Y \hookrightarrow X_1 \hookrightarrow X_2$, where $X_1$ and $X_2$ are equidimensional smooth $k$-schemes. 
Let $\widehat{X}_i$ be the completion of $X_i$ along $Y$ for $i=1,2$. 
Let $\widehat{\iota}: \widehat{X}_1 \hookrightarrow \widehat{X}_2$ be the induced closed immersion of formal schemes.

Then there exist the Gysin morphisms
\begin{equation}\label{eqn:1st -1}
\tuborg
\widehat{\iota}^*: \BGHz^q _{ \{ \widehat{X}_1 \}} ( \widehat{X}_2 , \bullet) \to \BGHz^q (\widehat{X}_1, \bullet), \\
\widehat{\iota}^*: \BGHz^q _{ \{ \widehat{X}_1 \}} ( \widehat{X}_2 \mod Y, \bullet) \to \BGHz^q (\widehat{X}_1 \mod Y, \bullet).
\sluttuborg
\end{equation}
\end{prop}

\begin{proof}
We first claim that we have the intersection morphism of complexes
\begin{equation}\label{eqn:Gysin pb0}
\widehat{\iota}^*: z^q _{ \{\widehat{X}_1\}} (\widehat{X}_2 , \bullet) \to z^q  (\widehat{X}_1, \bullet)
\end{equation}
defined by sending each integral cycle $\mathfrak{Z}$ to 
$$
\widehat{\iota}^* [\mathfrak{Z}] := [\mathfrak{Z}] . [ \widehat{X}_1 \times_k \square_k ^n] = [ \mathfrak{Z} \cap (\widehat{X}_1 \times _k \square_k ^n)] = [ \mathcal{O}_{\mathfrak{Z}} \otimes_{\mathcal{O}_{\square_{\widehat{X}_2} ^n} }^{\mathbf{L}} \mathcal{O}_{\square_{\widehat{X}_1} ^n}] = [ \mathbf{L} \widehat{\iota}^* ( \mathcal{O}_{\mathfrak{Z}})],
$$
 as given in Lemma \ref{lem:sheaf proper int}. One notes that since $\widehat{\iota}: \widehat{X}_1 \hookrightarrow \widehat{X}_2$ is a local complete intersection (l.c.i.) ($X_1$, $X_2$ being smooth, $X_1 \hookrightarrow X_2$ is l.c.i., and see e.g. \cite[Proposition 09QB]{stacks}), the derived pull-back $\mathbf{L} \widehat{\iota}^*$ is given just by the usual intersection pull-back $\widehat{\iota}^*$ (see Remark \ref{remk:lci}). 
 
 To prove the claim, it is enough to show that for each $n \geq 0$, we have
 $$
 \widehat{\iota}^* \left( z^q _{ \{ \widehat{X}_1 \}} (\widehat{X}_2, n) \right) \subset z^q (\widehat{X}_1, n).
 $$
For that, one needs to check the conditions (\textbf{GP}), (\textbf{SF}) of Definition \ref{defn:HCG}.

\medskip

By Definition \ref{defn:HCG var}, that $\mathfrak{Z} \in z^q _{ \{\widehat{X}_1\}} (\widehat{X}_2 , \bullet) $ means that $\mathfrak{Z}$ intersects properly with $\widehat{X}_1 \times F$ for all faces $F \subset \square_k^n$. Thus indeed $\widehat{\iota} ^* [\mathfrak{Z}]$ satisfies (\textbf{GP}). 

From the closed immersions $Y \hookrightarrow \widehat{X}_1 \hookrightarrow \widehat{X}_2$, we have $Y_{\red} = \widehat{X}_{1, \red} = \widehat{X}_{2, \red}$. Thus, the condition (\textbf{SF}) for $\widehat{\iota}^* [ \mathfrak{Z}]$ is obvious from that of $\mathfrak{Z}$. This proves the claim.

\medskip

Now, for each nonempty open subset $U \subset Y$, the map \eqref{eqn:Gysin pb0} gives 
\begin{equation}\label{eqn:1st -2}
\widehat{\iota}^*_U: z^q _{ \{\widehat{X}_1|_U\}} (\widehat{X}_2|_U , \bullet) \to z^q  (\widehat{X}_1|_U, \bullet).
\end{equation}
 Sheafifying the presheaves, we deduce the first morphism of \eqref{eqn:1st -1}.

\medskip

To construct the second morphism $\widehat{\iota}^*:  \BGHz^q _{ \{ \widehat{X}_1 \}} ( \widehat{X}_2 \mod Y, \bullet) \to  \BGHz^q (\widehat{X}_1 \mod Y, \bullet)$ of \eqref{eqn:1st -1}, we show that we can naturally deduce from \eqref{eqn:1st -2}  the morphism of complexes
\begin{equation}\label{eqn:1st -2-1}
\widehat{\iota}^*_U : z^q _{ \{ \widehat{X}_1|_U\}} ( \widehat{X}_2|_{U} \mod U, \bullet) \to z^q  (\widehat{X}_1 |_{ U} \mod U , \bullet)
\end{equation}
for each nonempty open $U \subset Y $, as we can then sheafify them. Replacing $Y$ by $U$, one reduces to showing it in the case $U= Y$.

\medskip

It is enough to prove that $\widehat{\iota}^* (\mathcal{M} ^q (\widehat{X}_2, Y, n, \{ \widehat{X}_1 \} ))  \subset \mathcal{M}^q (\widehat{X}_1, Y, n)$.

Observe that for $\mathcal{A} \in \mathcal{R}^q  (\widehat{X}_2, n, \{ \widehat{X}_1 \})$ in the sense of Definition \ref{defn:mod Y equiv}, by Lemma \ref{lem:sheaf proper int}, we have
\begin{equation}\label{eqn:Gysin LDF perf cy}
\widehat{\iota}^* [ \mathcal{A}] = [ \mathcal{A} \otimes_{\mathcal{O}_{\square_{\widehat{X}_2}} ^n} ^{\mathbf{L}} \mathcal{O}_{\square_{\widehat{X}_1} ^n}] = [ \mathbf{L} \widehat{\iota}^* (\mathcal{A})].
\end{equation}
Thus we have $\mathbf{L} \widehat{\iota}^* (\mathcal{A}) \in \mathcal{R}^q (\widehat{X}_1, n)$. 

\medskip

Let $\iota_{Y,i} ^n$ be the closed immersion $\iota_{Y,i} ^n: \square_{Y} ^n \hookrightarrow \square_{\widehat{X}_i}^n$ for $i=1,2$. 

For a pair $(\mathcal{A}_1, \mathcal{A}_2) \in \mathcal{L} ^q (\widehat{X}_2, Y, n, \{\widehat{X}_1\})$, we have an isomorphism in $\sAlg ( \mathcal{O}_{\square_Y ^n})$
\begin{equation}\label{eqn:Y_2 Gysin}
\mathbf{L} (\iota_{Y, 2} ^n)^* (\mathcal{A}_1) \simeq \mathbf{L} (\iota_{Y, 2} ^n)^* (\mathcal{A}_2).
\end{equation}
We also have the commutative diagram of closed immersions 
\begin{equation}\label{eqn:iota iota}
\xymatrix{ \square_{\widehat{X}_1}^n \ar@{^(->}[r] ^{\widehat{\iota}} & \square_{\widehat{X}_2} ^n \\
\square_{Y} ^n. \ar@{^(->}[u]^{\iota_{Y,1} ^n} \ar@{^(->}[ru] _{\iota_{Y,2}^n} & }
\end{equation}

The diagram \eqref{eqn:iota iota} implies the identity $ \mathbf{L} (\iota_{Y, 2} ^n)^* = \mathbf{L} (\iota_{Y,1} ^n)^* \circ \mathbf{L} \widehat{\iota}^*$ (\cite[I-3.6, p.119]{Lipman}). Hence from \eqref{eqn:Y_2 Gysin}, we deduce an isomorphism $\mathbf{L} (\iota_{Y,1} ^n)^* \circ \mathbf{L} \widehat{\iota}^* (\mathcal{A}_1) \simeq \mathbf{L} (\iota_{Y,1}^n)^* \circ \mathbf{L} \widehat{\iota}^* (\mathcal{A}_2)$ in $\sAlg (\mathcal{O}_{\square_Y^n})$, i.e.
\begin{equation}\label{eqn:LDF Gysin pair}
(\mathbf{L} \widehat{\iota}^* (\mathcal{A}_1) , \mathbf{L} \widehat{\iota}^* (\mathcal{A}_2)) \in \mathcal{L}^q (\widehat{X}_1, Y, n).
\end{equation}

Combining \eqref{eqn:Gysin LDF perf cy} with \eqref{eqn:LDF Gysin pair}, we have
$$
\widehat{\iota}^* ( [\mathcal{A}_1] - [\mathcal{A}_2] ) = [ \mathbf{L}\widehat{\iota}^* (\mathcal{A}_1) ] - [\mathbf{L}\widehat{\iota}^* (\mathcal{A}_2)].
$$
This implies that $\widehat{\iota}^* (\mathcal{M}^q (\widehat{X}_2, Y, n, \{ \widehat{X}_1 \})) \subset \mathcal{M}^q (\widehat{X}_1, Y, n)$. Now sheafifying the presheaves in \eqref{eqn:1st -2-1} over $U$, we obtain the second morphism of \eqref{eqn:1st -1}. This completes the proof.  
\end{proof}

\subsection{A formal affine moving lemma}\label{sec:formal_move}
We saw in Proposition \ref{prop:1st indep} that there are the Gysin intersection morphisms
$$
\tuborg
\iota: \BGHz^q _{\{ \widehat{X}_1 \}} (\widehat{X}_2 , \bullet) \to \BGHz^q (\widehat{X}_1, \bullet), \\
\iota: \BGHz^q _{\{ \widehat{X}_1 \}} (\widehat{X}_2 \mod Y, \bullet) \to \BGHz^q (\widehat{X}_1 \mod Y,\bullet).
\sluttuborg
$$
We want to know whether $ \BGHz^q  (\widehat{X}_2, \bullet)$ is quasi-isomorphic to its subcomplex $\BGHz^q _{\{ \widehat{X}_1 \}} (\widehat{X}_2, \bullet) $, and similarly for $ \BGHz^q  (\widehat{X}_2 \mod Y, \bullet)$.

\medskip

The goal of \S \ref{sec:formal_move} is to address this issue:

\begin{thm}\label{thm:formal moving}
Let $Y$ be a quasi-affine $k$-scheme of finite type, and let $Y \hookrightarrow X_1 \hookrightarrow X_2$ be closed immersions into equidimensional smooth $k$-schemes $X_1, X_2$. Let $\widehat{X}_i$ be the completion of $X_i$ along $Y$ for $i=1,2$.

 Then the inclusion morphisms of complexes of sheaves on $Y_{\rm Zar}$
\begin{equation}\label{eqn:formal moving}
\tuborg
\mathfrak{i}: \BGHz^q _{\{\widehat{X}_1\}} (\widehat{X}_2 , \bullet) \hookrightarrow \BGHz^q (\widehat{X}_2 , \bullet), \\
\mathfrak{i}: \BGHz^q _{\{\widehat{X}_1\}} (\widehat{X}_2 \mod Y, \bullet) \hookrightarrow \BGHz^q (\widehat{X}_2 \mod Y, \bullet)
\sluttuborg
\end{equation}
are quasi-isomorphisms.
\end{thm}

This is a kind of moving lemma in the context of complexes of sheaves of abelian groups of cycles on formal schemes. 

\medskip

Since $X_1$ and $X_2$ are both smooth, the closed immersion $X_1 \hookrightarrow X_2$ is l.c.i. On the other hand, since the morphisms \eqref{eqn:formal moving} being quasi-isomorphisms is a local statement, we reduce it to small enough open neighborhoods of points of $Y$. We will see that for each $y \in Y$, there is an affine open neighborhood $U$ of $y$ for which $\widehat{X}_1 |_U \hookrightarrow \widehat{X}_2|_U$ comes from a complete intersection. 

\medskip

We first study the case when $X_1 \hookrightarrow X_2$ is a complete intersection, and then we reduce the general case to the complete intersection case. In this complete intersection case, if we write $\widehat{X}_1 = \Spf (A)$, then $\widehat{X}_2= \Spf (A[[\un{t}]])$ for some indeterminates $\un{t}= \{ t_1, \cdots, t_r\}$ (see SGA I \cite[Expos\'e II, Remarque 4.14, p.45]{SGA1}). We want to use a kind of ``general translation"  in the formal power series setting. However, some naive attempts to take ``translation by the generic point" that worked in the classical moving lemma for cycles over affine spaces, do not work well in our case.

\medskip

Here instead, we will use a combination of linear translations of the variables $t_i$ by some ``messy" $c_i$, and discuss how we can choose such $c_i$'s that give us nice properties. The reader may note that in Park-Pelaez \cite[\S 3.3]{PP}, which was concurrently written alongside this article, has a $K$-theoretic analogue of what is studied in this article, and some relevant arguments in \emph{loc.cit.} and some parts of what we prove in this article overlap. To make both papers self-contained, these papers will have their respective needed arguments, although there are some overlaps. 

In the complete intersection case of Theorem \ref{thm:formal moving}, a bit stronger ``unsheafified" version is proven in Theorems \ref{thm:quasi-iso1-pre} and \ref{thm:quasi-iso1}. The general case will be deduced from them.

\subsubsection{Some translations in formal power series}
Let $Y= \Spec (B)$ be a connected affine $k$-scheme of finite type. Suppose we have closed immersions $Y \hookrightarrow X_1 \hookrightarrow X_2$ into equidimensional smooth $k$-schemes $X_1, X_2$, and suppose $X_1 \hookrightarrow X_2$ is a complete intersection.  
Let $r$ be its codimension. If $r=0$, then there is nothing to prove, so suppose $r \geq 1$. Let $\widehat{X}_i$ be the completion of $X_i$ along $Y$.

\medskip

Write $\widehat{X}_1=\Spf (A)$ for an equidimensional regular affine $k$-domain $A$ of finite Krull dimension, not necessarily of finite type over $k$. Let $J \subset A$ be the ideal such that $A/J = B$ and $A$ is $J$-adically complete. We can write $\widehat{X}_2= \Spf (A[[\un{t}]])$ for some indeterminates $\un{t} = \{ t_1, \cdots, t_r\}$.

\medskip

Using the automorphism of $\mathbb{P}_k ^1$ given by $y \mapsto y/ (y-1)$, we again identify $\square_k ^1$ with $\mathbb{A}_k ^1$, where the faces are given by $\{0, 1 \}$. For each $n\geq 0$, let $R_n:= A[[\un{t}]]  \widehat{\otimes}_k k[y_1, \cdots, y_n]= A[[ \un{t}]]\{ y_1, \cdots, y_n \}$ so that $\Spf (R_n) = \square_{\widehat{X}_2} ^n$. This ring $R_n$ is complete with respect to the ideal $(J,  (\un{t}))$.

\medskip

For rings $T \subset U$, and a subset $S$ of $T$, the notation $(S)_{U}$ means the ideal of $U$ generated by the set $S$.

\begin{defn}
 Under the above assumptions, we introduce the following notations:
\begin{enumerate}
\item Let $A_0:= A$, and for $1 \leq i \leq r$, let $A_i:= A[[t_1, \cdots, t_i]]$. We regard $A_i \subset A_{i+1}$ for each $i$ in the apparent way.
\item Let $J_0:= J$, and for $1 \leq i \leq r$, let $J_i:= ( J, t_1, \cdots, t_i)_{A_i}$. The ring $A_i$ is complete with respect to $J_i$. Let $\widehat{X}_{1, i}:= \Spf (A_i)$. Here, $\widehat{X}_{1,0}= \widehat{X}_1$ and $\widehat{X}_{1,r} = \widehat{X}_2$. We also have closed immersions $\widehat{X}_{1, i} \subset \widehat{X}_{1, i+1}$ given by $A_{i+1}/ (t_{i+1}) = A_i$.
\item For $0 \leq i \leq r-1$, we let $(J_i) \subset A_{i+1}$ denote the ideal $(J_i)_{A_{i+1}}$.
\item Let $\mathbb{V}_1:= (J_0)$, and inductively let $\mathbb{V}_{i+1} := \mathbb{V}_i \times (J_i)$ for $1 \leq i \leq r-1$, the cartesian products of sets.
\item For each $\un{c} = (c_1, \cdots, c_r) \in \mathbb{V}_r$, consider the translation automorphism of $A_r=A[[\un{t}]]$ given by $t _i \mapsto t_i + c_i$ for all $1 \leq i \leq r$. Since $c_i \in J_r=(J, t_1, \cdots, t_r)_{A_r}$ for all $i$, and $A_r$ is complete with respect to this ideal, this translation defined by $\un{c}$, is indeed an automorphism of $A_r= A[[\un{t}]]$. It induces the corresponding automorphism of the formal scheme $\widehat{X}_2$
\begin{equation}\label{eqn:psi_c 0}
\psi_{\un{c}}: \widehat{X}_2 \to \widehat{X}_2.
\end{equation}
We use these notations freely in \S \ref{sec:formal_move}.\qed
\end{enumerate}

\end{defn}

For $\un{c} \in \mathbb{V}_r$, the automorphism \eqref{eqn:psi_c 0} in turn induces the automorphism (also denoted by the same notation)
\begin{equation}\label{eqn:psi_c n}
\psi_{\un{c}}:  \widehat{X}_2 \times_k \square_{k} ^n \to \widehat{X}_2 \times_k \square_{k} ^n.
\end{equation}

This is an isomorphism, so we can define the translation pull-back $\psi_{\un{c}} ^*$ on cycles as well. We first mention:

\begin{lem}\label{lem:translation0}
Let $\un{c} \in \mathbb{V}_r$. Let $\mathfrak{Z} \in z^q (\widehat{X}_2, n)$. Then $\psi_{\un{c}} ^* (\mathfrak{Z}) \in z^q (\widehat{X}_2, n)$.
\end{lem}

\begin{proof}
We may assume $\mathfrak{Z}$ is an integral cycle.

 We first prove the condition (\textbf{GP}) for $\psi_{\un{c}} ^* (\mathfrak{Z})$. Let $F\subset \square^n$ be a face. Since $\psi_{\un{c}}$ is an isomorphism, it preserves dimensions. Hence from $\psi_{\un{c}} ^* (\mathfrak{Z}) \cap (\widehat{X}_2 \times F) = \psi_{\un{c}} ^* (\mathfrak{Z} \cap (\widehat{X}_2 \times F))$, we deduce (\textbf{GP}) for $\psi_{\un{c}} ^* (\mathfrak{Z})$ from that of $\mathfrak{Z}$.

The condition (\textbf{SF}) for $\psi_{\un{c}} ^* (\mathfrak{Z})$ holds trivially because all of $t_i, t_{i}+c_i, c_i$ belong to the largest ideal of definition, so that $\psi_{\un{c}} ^* (\mathfrak{Z}) \cap ((\widehat{X}_2)_{\red} \times F) = \mathfrak{Z} \cap  ((\widehat{X}_2)_{\red} \times F)$.
\end{proof}

Given Lemma \ref{lem:translation0}, we ask whether we can further achieve $\psi_{\un{c}} ^* (\mathfrak{Z}) \in z^q _{\{ \widehat{X}_1 \}} (\widehat{X}_2, n)$ for some $\un{c}\in \mathbb{V}_r$. For such arguments, the cardinality of $J$ is important. Here is the trivial case:

\begin{lem}\label{lem:translation fail}
Suppose the ideal $J$ satisfies $|J|< \infty$. Then we have:
\begin{enumerate}
\item $J=0$. 
\item $Y= \widehat{X}_1= X_1$ and $Y$ is smooth over $k$.
\end{enumerate}
\end{lem}

\begin{proof}

Recall $Y= \Spec (B)$ and $\widehat{X}_1 = \Spf (A)$ with $A/J = B$. Write $X_1 = \Spec (A')$ for some smooth $k$-algebra $A'$. Let $I \subset A'$ be the ideal such that $A'/I = B$. By definition, $A= \underset{m}{\varprojlim} \ A' / I^m$, $J= \widehat{I}$, and $A/J = A' /I = B$.

\medskip

(1) Since $J$ is a finite set, the descending chain of ideals 
$$ 
\cdots \subset J^3 \subset J^ 2 \subset J
$$
is stationary, so that there is some $N \geq 1$ such that $J^N = J^{N+1} = \cdots$. 

Since $A$ is $J$-adically complete, we have 
$$
A= \varprojlim_m \ A/ J^m = A/J^N.
$$
 This implies that $J^N = 0$. However, $A$ being a regular $k$-domain, it has no nonzero nilpotent element. Thus $J= 0$, proving (1).
 
 \medskip

(2) That $J=0$ implies $0 = \widehat{I}$, so that $I=0$ and $A= A' = B$. This means $Y = \widehat{X}_1 = X_1$.

 Since $X_1$ is smooth over $k$ and $Y= X_1$, the scheme $Y$ is smooth over $k$.
\end{proof}

\begin{lem}\label{lem:translation fail-1}
Suppose $Y =  X_1$. Then we have:
\begin{enumerate}
\item $\widehat{X}_1 = Y$ and  $(\widehat{X}_2)_{\red} = Y$.
\item $z^q _{\{ \widehat{X}_1\}} (\widehat{X}_2, n) = z^q (\widehat{X}_2, n)$.
\end{enumerate}
\end{lem}

\begin{proof}
Since $Y= X_1$, we have $\widehat{X}_1 = Y$. Since $Y=X_1$ is smooth over $k$, it has no nonzero nilpotent element, and  the closed immersion $Y \hookrightarrow \widehat{X}_2$ is given by the largest ideal of definition of $\widehat{X}_2$. Thus $(\widehat{X}_2)_{\red} = Y$. 

By the special fiber condition (\textbf{SF}) in Definition \ref{defn:HCG}, all cycles in $z^q (\widehat{X}_2, n)$ already intersect $(\widehat{X}_2)_{\red} \times F$ properly for all faces $F \subset \square^n$. But $(\widehat{X}_2)_{\red} \times F= Y \times F= \widehat{X}_1 \times F$. Hence $ z^q _{\{ \widehat{X}_1\}} (\widehat{X}_2, n) = z^q (\widehat{X}_2, n)$.
\end{proof}

Now suppose $J$ is infinite. We want to prove Lemma \ref{lem:translation} below. As said before, part of the arguments here will be repeated to prove part of \cite{PP}. Although there are some overlaps, both of the articles will retain their respective needed arguments to make them self-contained.


\begin{lem}\label{lem:translation PP}
Let $A$ be a noetherian integral domain complete with respect to an ideal $J \subset A$ and suppose $|J|=\infty$. For an indeterminate $t$, consider $A[[t]]$. Let $(J) \subset A[[t]]$ be the ideal $(J)_{A[[t]]}$, which is proper.

Let $I_1, \cdots, I_N \subset A[[t]]$ be prime ideals such that $(J, t) \not \subset I_i$ for all $1 \leq i \leq N$. Then there exists some $c \in (J)$ such that $t+c \not \in I_i$ for all $1 \leq i \leq N$. 
\end{lem}

\begin{proof}
There are a few cases to consider. The easiest case is:

\medskip

\textbf{Case 1:} Suppose $t \not \in I_i$ for all $1 \leq i \leq N$. In this case, we may simply take $c=0$ to get to the conclusion of the lemma.

\medskip

Before we discuss the other cases, we make the following claim:

\medskip

\emph{Claim 1:} \emph{Suppose $t\in I_{i_j}$ for some indices $1 \leq i_1 < \cdots < i_u \leq N$. Then 
\begin{enumerate}
\item $(J) \not \subset \bigcup_{j=1} ^u I_{i_j}$,
\item $(J) \setminus ( (J) \cap (\bigcup_{j=1} ^u I_{i_j}))$ is an infinite set, and
\item for each $c \in (J) \setminus ( (J) \cap (\bigcup_{j=1} ^u I_{i_j}))$, we have $t+c \not \in I_{i_j}$ for $1 \leq j \leq u$.
\end{enumerate}
}

\medskip

For (1), note we are given that $(J, t) \not \subset I_{i_j}$, while $t \in I_{i_j}$ for $1 \leq j \leq u$. Thus we have $(J) \not \subset I_{i_j}$. Hence by the prime avoidance theorem (see Atiyah-MacDonald \cite[Proposition 1.11-i), p.8]{AM}), we deduce (1).

For (2), choose any $c \in (J) \setminus  ((J) \cap (\bigcup_{j=1} ^u I_{i_j}))$. We show that all members of $\{ c, c^2, c^3, \cdots \}$ are distinct from each other, and no member of them belongs to any of $I_{i_j}$. 

For the first part, suppose $c^p = c^{p'}$ for some integers $1 \leq p < p'$. Then $c^p (1- c^{a}) = 0$, where $a:= p'-p$. Since $c \not = 0$ and $A[[t]]$ is an integral domain, we deduce that $1- c^a = 0$, i.e. $c$ is a unit. But $(J)$ is a proper ideal of $A[[t]]$ and $c \in (J)$, so that $c$ cannot be a unit, contradiction. Thus the members of $\{ c, c^2, c^3, \cdots \}$ are distinct from each other, and the set is infinite.

If $c^p \in I_{i_j}$ for some $j$ and $p \geq 1$, then since $I_{i_j}$ is prime, we deduce that $c \in I_{i_j}$, which contradicts our choice of $c$. Hence we proved that the infinite set $\{ c, c^2, c^3, \cdots \}$ is in the set $ (J) \setminus ( (J) \cap (\bigcup_{j=1} ^u I_{i_j}))$, proving (2).

For (3), suppose $t+c \in I_{i_j}$ for some $1 \leq j \leq u$. Since we are given that $t \in I_{i_j}$, we have $(t+c) - t = c \in I_{i_j}$, which contradicts our choice of $c$. Hence $t+c \not \in I_{i_j}$. This proves (3), completing the proof of Claim 1.

\medskip

Returning to the proof of the lemma, consider the following case, opposite to the Case 1:

\medskip

\textbf{Case 2:} Suppose $t \in I_i$ for all $1 \leq i \leq N$. In this case, applying Claim 1, we see that $(J) \setminus ( (J) \cap (\bigcup_{i=1} ^N I_i))$ is an infinite set, such that for any member $c$ of the set, we have $t+c \not \in I_i$ for all $1 \leq i \leq N$. This proves the lemma in this case.

\medskip

The remaining case is the ``mixed case" where we have $t \in I_i$ for some indices $i$, while $t \not \in I_{i'}$ for some other indices $i'$. Before we work out this case, consider the following:

\medskip

\emph{Claim 2:} \emph{For a prime ideal $I \subset A[[t]]$, suppose $t \not \in I$. Let $c \in (J)$.
\begin{enumerate}
\item If $c \in I$, then $t+ c^p \not \in I$ for each integer $p \geq 1$.
\item If $c \not \in I$, then there exists at most one integer $p \geq 1$ such that $t+ c^p \in I$.
\end{enumerate}
}

\medskip

For (1), suppose $c \in I$. Then for $p \geq 1$, we have $c^p \in I$. Thus if $t + c^p \in I $ for some $p$, then $(t+ c^p )- c^p = t \in I$, contradicting our assumption that $t \not \in I$. Thus $t + c^p \not \in I$ for each $p \geq 1$.

For (2), suppose $c \not \in I$. If there are two distinct positive integers $p< p'$ such that both $t+ c^p$ and $ t+ c^{p'} \in I$, then
$$
(t+ c^p) - (t + c^{p'}) = c^p (1- c^a) \in I,
$$
 where $a:= p'-p$. Since $I$ is prime and $c \not \in I$, we have $1- c^a \in I$.

But, $A[[t]]$ is complete with respect to $(J, t)$, and it contains $c$. So we have $1 + c^a + c^{2a} + \cdots \in A[[t]]$. Thus 
$$
1= (1+ c^a + c^{2a} + \cdots) (1- c^a) \in I,
$$
contradicting that $I$ is a prime ideal, so that it is proper in $A[[t]]$. This shows (2).

\medskip

Returning to the proof, now finally we have the last remaining case:

\medskip

\textbf{Case 3:} Suppose we have $t \in I_i$ for some indices $i$, and $t \not \in I_{i'}$ for some other indices $i'$ at the same time. After relabeling them, we may assume that there is a positive integer $1 \leq s <N$ such that we have
$$
\tuborg
t \not \in I_i, & \mbox{ for } 1 \leq i \leq s, \\
t \in I_i, & \mbox{ for } s+1 \leq i \leq N.
\sluttuborg
$$ 

First applying Claim 1 to the prime ideals $I_{s+1}, \cdots, I_N$, we can form the infinite set $\widetilde{J}:= (J) \setminus  ((J) \cap ( \bigcup_{i=s+1} ^N I_i))$.

Choose any $c_0 \in \widetilde{J}$. Then $c_0 ^p \in \widetilde{J}$ for all $p \geq 1$ so that by Claim 1-(3), we have $t+ c_0 ^p \not \in I_i$ for all $s+1 \leq i \leq N$ and all $p \geq 1$. 

On the other hand, for $1 \leq i \leq s$, by Claim 2, there exists a finite subset $B \subset \mathbb{N}$ of the size $|B| \leq s$ such that for all $p \in \mathbb{N} \setminus B$, we have $t+ c_0 ^p \not \in I_i$ for all $1 \leq i \leq s$. 

Thus for $p \in \mathbb{N} \setminus B$, we have just shown that $t + c_0 ^p \not \in I_i$ for all $1 \leq i \leq N$. This proves the lemma in this last Case 3.

This complete the proof of the lemma.
\end{proof}

\begin{lem}\label{lem:translation}
Suppose that $|J|= \infty$. Let $\mathfrak{Z}  \in z^q (\widehat{X}_2, n)$.

Then there exists some $\un{c} \in \mathbb{V}_r$ such that $\psi_{\un{c}} ^* (\mathfrak{Z}) \in z^q _{\{ \widehat{X}_1 \}} (\widehat{X}_2, n)$.
\end{lem}

\begin{proof}
Write $\mathfrak{Z} = \sum_{i=1} ^N m_i \mathfrak{Z}_i \in z^q (\widehat{X}_2, n)$ for some integers $m_i \not = 0$ and finitely many distinct integral cycles $\mathfrak{Z}_i \in z^q (\widehat{X}_2, n)$. 

By Lemma \ref{lem:translation0}, we already have $\psi_{\un{c}} ^* (\mathfrak{Z}) \in z^q (\widehat{X}_2, n)$ for each $\un{c} \in \mathbb{V}_r$. It remains to show that for a suitable $\un{c}\in \mathbb{V}_r$, each cycle $\psi_{\un{c}} ^* (\mathfrak{Z}_i )$ intersects $\widehat{X}_1 \times F$ properly for all faces $F \subset \square_k ^n$. 

\medskip

We prove it by induction on the codimension $r$ of $\widehat{X}_1$ in $\widehat{X}_2$. 

\medskip

\textbf{Step 1:} First suppose $r=1$. Note that $\square_{\widehat{X}_2}^n= \Spf (A[[t_1]]\{ y_1, \cdots, y_n \})$. 

Let $I_i \subset A[[t_1]]\{ y_1, \cdots y_n\}$ be the prime ideal of $\mathfrak{Z}_i$ for $1 \leq i \leq N$.

\medskip

Let $F \subset \square_k^n$ be a face. Inside the space $\widehat{X}_2 \times F$, the integral components of the intersection $\mathfrak{Z}_i \cap (\widehat{X}_2 \times F)$ are given by a finite set of prime ideals $I_{i, j}^F \subset A[[t_1]] \widehat{\otimes}_k \Gamma (F)$, where $\Gamma(F)$ is the coordinate ring of $F$. 
They satisfy $(J, t_1)_{A[[t_1]]} \not \subset I_{i,j} ^F$. 

\medskip

To have proper intersections with $\widehat{X}_1 \times F$ over all $F$ after a translation, we seek some $c \in (J)$ so that $t_1+c \not \in I_{i, j} ^F$ for all $i, j$ and $F$ at the same time. It is achieved if the contracted prime ideals (via $A[[t_1]] \hookrightarrow A[[t_1]] \widehat{\otimes}_k \Gamma (F)$),
$$
I_{i, j} ^{F, 0} := A [[ t_1 ]] \cap I_{i,j} ^F \subset A[[t_1 ]]
$$ 
over all $i,  j , F$ have the property that $t_1 + c \not \in I_{i,j} ^{F, 0}$. Note that we still have $(J, t_1)_{A[[t_1]]} \not \subset I_{i,j} ^{F, 0}$.

Since $\{ I_{i,j} ^{F, 0} \}_{i, j, F}$ is a finite collection of prime ideals of $A[[t_1]]$ such that none of them contains $(J, t_1)_{A[[t_1]]}$, we deduce from Lemma \ref{lem:translation PP} that there exists some $c \in (J)$ such that $t_1 + c \not \in I_{i,j} ^{F, 0}$ for all $i, j, F$.

Thus for such $c \in (J)$, the translated cycle $\psi_{c} ^* (\mathfrak{Z}_i)$ intersects properly with $\widehat{X}_1 \times F$ for all $i$ and $F$, as desired. This proves the lemma when $r=1$.

\medskip

\textbf{Step 2:} Suppose $r \geq 2$ and suppose that the lemma holds when the codimension of $\widehat{X}_1$ in $\widehat{X}_2$ is one of $1, \cdots, r-1$. For the ring $A_{r-1}:= A[[t_1, \cdots, t_{r-1}]]$, which is complete with respect to $J_{r-1}$, we have the formal scheme $\widehat{X}_{1, r-1} = \Spf (A_{r-1})$. We have the closed immersion $\widehat{X}_{1, r-1} \hookrightarrow \widehat{X}_2$ of codimension $1$ defined by $t_r$.

For a given cycle $\mathfrak{Z}$, by the Step 1 applied to $\widehat{X}_{1, r-1}$, there exists some $c_r \in (J_{r-1}) $ such that 
$$
\psi_{t_r \mapsto t_r + c_r} ^* (\mathfrak{Z}) \ \ \mbox{ and } \ \ \widehat{X}_{1, r-1} \times F
$$
intersect properly for all face $F \subset \square_k ^n$.

\medskip

Let $\mathfrak{Z}'$ be the cycle associated to the intersection of $\psi_{t_r \mapsto t_r + c_r} ^* (\mathfrak{Z})$ and $\widehat{X}_{1, r-1} \times \square_k^n$. 
Here $\widehat{X}_1 \subset \widehat{X}_{1, r-1}$ is of codimension $r-1$ and defined by the ideal generated by $t_1, \cdots, t_{r-1}$. Thus by the induction hypothesis, there exists some $\un{c}'= (c_1, \cdots, c_{r-1}) \in \mathbb{V}_{r-1}$ such that the pull-back $\psi_{\un{c}'} ^* (\mathfrak{Z}')$ intersects $\widehat{X}_1 \times F $ properly for all faces $F \subset \square_k ^n$.

Combining with the previously chosen $c_r \in (J_{r-1})$, we have $\un{c}:= (\un{c}', c_r)\in \mathbb{V}_r $ such that $\psi_{\un{c}} ^* (\mathfrak{Z})$ intersects $\widehat{X}_1 \times_k F$ properly for all faces $F \subset \square_k ^n$. 

This proves the lemma by induction.
\end{proof}

In fact, the observant reader notices that the proof given in Lemma \ref{lem:translation} works for a bit more general circumstances. We state them as follows. Since the proof is identical, we omit the arguments.

\begin{lem}\label{lem:translation1}
Suppose that $|J| = \infty$. Let $n_1, n_2, \cdots, n_p, q_1, \cdots, q_p \geq 0 $ be integers, and suppose we have $\mathfrak{Z}_i \in z^{q_i} (\widehat{X}_2, n_i)$ for $1 \leq i \leq p$. 

Then there exists $\un{c} \in \mathbb{V}_r$ such that $\psi_{\un{c}} ^* (\mathfrak{Z}_i) \in z^{q_i} _{\{\widehat{X}_1\}} (\widehat{X}_2, n_i)$ for all $1 \leq i \leq p$.
\end{lem}

\subsubsection{Homotopy}
We want to relate the given cycle $\mathfrak{Z}$ and its translated cycle $\psi_{\un{c}} ^* (\mathfrak{Z})$ of Lemma \ref{lem:translation} by a homotopy. 

For each $\un{c} \in \mathbb{V}_r$, consider the composite
$$
H_{n,\un{c}}:  \Spf (A[[\un{t}]])\times_k \square_k ^{n+1} \overset{\phi_{n+1, \un{c}}}{\to} \Spf (A[[\un{t}]]) \times_k \square_k ^{n+1} \overset{pr}{\to} \Spf (A[[\un{t}]]) \times_k \square_k ^n,
$$
where $\phi_{n+1,\un{c}}$ is given by 
$$
 \tuborg
  \un{t} \mapsto \un{t} + y_{n+1} \un{c}, \\ 
 \mbox{fixing each } a \in A \mbox{ and each of } y_1, \cdots, y_{n+1}, 
 \sluttuborg
 $$
while $pr$ is the projection that ignores the last coordinate $y_{n+1}$. Note that both $\phi_{n+1, \un{c}}$ and $pr$ are morphisms for which their pull-backs of cycles exist.

\medskip

Recall that $\widehat{X}_2= \Spf (A[[\un{t}]])$. For $\mathfrak{Z} \in z^q  (\widehat{X}_2, n)$, and for a choice of $\un{c} \in \mathbb{V}_r$ as in Lemma \ref{lem:translation1} applied to the finite collection of cycles $\partial_i ^{\epsilon}\mathfrak{Z}$ for $1 \leq i \leq n$, $\epsilon \in \{ 0, 1 \}$ and $\mathfrak{Z}$, we have
\begin{eqnarray}\label{eqn:homotopy var1}
& & (-1)^{n-1} H_{n-1, \un{c}} ^* (\partial \mathfrak{Z}) = (-1)^{n-1} H_{n-1, \un{c}} ^* \sum_{i=1} ^n (-1)^i (\partial_i ^{1} - \partial_i ^0) (\mathfrak{Z})
\end{eqnarray}
\begin{eqnarray}
\notag &= & (-1)^{n-1} \sum_{i=1} ^n (-1)^i \phi_{n, \un{c}} ^* pr^* (\partial_i ^1 - \partial_i ^0) ( \mathfrak{Z})\\
\notag & =& (-1)^{n-1} \sum_{i=1} ^n (-1)^i ( \partial_i ^1 - \partial_i ^0) \phi_{n+1,\un{c}} ^* pr^* (\mathfrak{Z}) \\
\notag &= & (-1)^{n-1}  \sum_{i=1} ^n (-1)^i (\partial_i ^1 - \partial_i ^0) H_{n,\un{c}} ^* (\mathfrak{Z}),
\end{eqnarray}
and 
\begin{equation}\label{eqn:homotopy var2}
(-1)^j \partial   H_{n,\un{c}} ^* (\mathfrak{Z}) = (-1)^n \sum_{i=1} ^{n+1} (-1)^i (\partial_i ^1 - \partial_i ^0) H_{n,\un{c}} ^* (\mathfrak{Z}).
\end{equation}
Here, $\partial_{n+1} ^1 H_{n,\un{c}} ^* (\mathfrak{Z}) = \psi_{\un{c}} ^* (\mathfrak{Z})$ and $\partial_{n+1} ^0 H_{n, \un{c}} ^* (\mathfrak{Z}) = \mathfrak{Z}$. Combining them with \eqref{eqn:homotopy var1} and \eqref{eqn:homotopy var2}, we deduce
\begin{equation}\label{eqn:homotopy master}
 (-1)^n \partial H_{n,\un{c}} ^* (\mathfrak{Z}) + (-1)^{n-1} H_{n-1, \un{c}} ^* (\partial \mathfrak{Z}) = \mathfrak{Z} - \psi_{\un{c}} ^* (\mathfrak{Z}).
\end{equation}

\medskip

\begin{lem}\label{lem:homo1 inj}
If $\mathfrak{Z} \in z^q _{\{ \widehat{X}_1\}} (\widehat{X}_2, n)$, then $H_{n, \un{c}} ^* (\mathfrak{Z}) \in z^q _{\{ \widehat{X}_1\}} (\widehat{X}_2, n+1)$. 
\end{lem}

\begin{proof}
Since $H_{n, \un{c}}$ is a composite of a projection $pr$ (a flat morphism of type (III)) and an isomorphism $\phi_{n+1, \un{c}}$, we know that the cycle $H_{n, \un{c}} ^* (\mathfrak{Z})$ has the right codimension $q$. It remains to check the conditions (\textbf{GP}), (\textbf{SF}) of Definition \ref{defn:HCG} for $H_{n, \un{c}} ^* (\mathfrak{Z})$ as well as the conditions in Definition \ref{defn:HCG var} with respect to $\widehat{X}_1 \times F$ over all faces $F \subset \square_{k} ^{n+1}$. Note that the largest ideal of definition of $\widehat{X}_2$ contains $(J, (\un{t}))$. We may assume $\mathfrak{Z}$ is integral. 

\medskip

We first check the condition (\textbf{GP}). Let $F \subset \square ^{n+1}$ be a face. We can write it as $F= F' \times F''$ where $F' \subset \square^{n}$ and $F'' \subset \square^1$ are faces.

If $F'' = \square^1$, then one checks that $H_{n, \un{c}}^* (\mathfrak{Z}) \cap (\widehat{X}_2 \times F' \times \square^1) = H_{n,\un{c}}^* (\mathfrak{Z} \cap (\widehat{X}_2 \times F' ))$. Thus it has the right codimension by the condition (\textbf{GP}) of $\mathfrak{Z}$.

If $F'' = \{ 0 \}$, then $H_{n,\un{c}} ^* (\mathfrak{Z}) \cap (\widehat{X}_2 \times F' \times \{ 0 \}) = \mathfrak{Z} \cap (\widehat{X}_2 \times F')$. Thus it has the right codimension by the condition (\textbf{GP}) for $\mathfrak{Z}$.

If $F'' = \{ 1 \}$, then $H_{n, \un{c}} ^* (\mathfrak{Z}) \cap (\widehat{X}_2 \times F' \times \{ 1 \}) = \psi_{\un{c}} ^* (\mathfrak{Z} \cap (\widehat{X}_2 \times F'))$. Since $\psi_{\un{c}}$ is an isomorphism, by the condition (\textbf{GP}) for $\mathfrak{Z}$, the above cycle also has the right codimension. 

The above discussions prove the condition (\textbf{GP}) for $H_{n, \un{c}} ^* (\mathfrak{Z})$.

\medskip

The condition (\textbf{SF}) for $H_{n, \un{c}} ^* (\mathfrak{Z})$ holds immediately because $H_{n, \un{c}} ^* (\mathfrak{Z}) \cap ( (\widehat{X}_2)_{\red} \times F) $ given by reduction mod $(J, \un{t})$ is just $(\mathfrak{Z} \cap ( (\widehat{X}_2)_{\red} \times F')) \times F''$ because all $t_i$ and $\un{c}$ vanish in the reduction by $(J, \un{t})$, so that they have the right codimensions by the condition (\textbf{SF}) for $\mathfrak{Z}$. 

\medskip



Finally, we check that $H_{n,\un{c}} ^* (\mathfrak{Z}) \cap (\widehat{X}_1 \times F)$ is proper for each face $F= F' \times F''$.

If $F''= \square_k ^1$, then $H_{n,\un{c}} ^* (\mathfrak{Z}) \cap (\widehat{X}_1 \times F' \times \square_k ^1) = H_{n, \un{c}} ^* ( \mathfrak{Z} \cap (\widehat{X}_1 \times F'))$ so that it has the right codimension by the given condition on the intersection $\mathfrak{Z} \cap (\widehat{X}_1 \times F')$.

If $F'' = \{ 0 \}$, then $H_{n, \un{c}} ^* (\mathfrak{Z}) \cap (\widehat{X}_1 \times F' \times \{ 0 \}) = \mathfrak{Z} \cap (\widehat{X}_1 \times F')$, and it has the right codimension by the given condition on $\mathfrak{Z} \cap (\widehat{X}_1 \times F')$.

If $F'' = \{ 1 \}$, then $H_{n, \un{c}} ^* (\mathfrak{Z}) \cap (\widehat{X}_1 \times F' \times \{ 1 \}) = \psi_{\un{c}} ^* (\mathfrak{Z} \cap (\widehat{X}_1 \times F'))$, and it has the right codimension by the given condition on $\mathfrak{Z} \cap (\widehat{X}_1 \times F')$ as well as that $\psi_{\un{c}}$ is an isomorphism.

Thus we checked that $H_{n, \un{c}} ^* (\mathfrak{Z}) \in z_{ \{ \widehat{X}_1 \}} ^q (\widehat{X}_2, n+1)$ as desired.
This proves the lemma.
\end{proof}

The following limited version of moving lemma holds for the higher Chow complexes of certain regular affine formal $k$-schemes:

\begin{thm}\label{thm:quasi-iso1-pre}
Let $Y$ be a connected affine $k$-scheme of finite type, and let $Y \hookrightarrow X_1 \hookrightarrow X_2$ be closed immersions into equidimensional smooth $k$-schemes $X_1$ and $X_2$, such that $X_1 \hookrightarrow X_2$ is a complete intersection. Let $\widehat{X}_i$ be the completion of $X_i$ along $Y$ for $i=1,2$. 

Then the inclusion $\mathfrak{i}: z^q _{ \{ \widehat{X}_1 \}} (\widehat{X}_2, \bullet)  \hookrightarrow z^q  (\widehat{X}_2, \bullet) $
is a quasi-isomorphism.
\end{thm}

\begin{proof} First consider the trivial case when $Y =X_1$. Then by Lemma \ref{lem:translation fail-1}, we have $ z^q _{ \{ \widehat{X}_1 \}} (\widehat{X}_2, \bullet) = z^q  (\widehat{X}_2, \bullet) $, so there is nothing to prove. 

\medskip

We now suppose $Y \subsetneq X_1$. In this case, the ideal $J$ that defines $Y$ in $\widehat{X}_1$ is an infinite set by Lemma \ref{lem:translation fail}. So we can use Lemma \ref{lem:translation}.

\medskip

Let's first prove the surjectivity in homology. For a class $\alpha \in \CH^q (\widehat{X}_2, n)$, choose any cycle representative $\mathfrak{Z} \in z^q (\widehat{X}_2, n)$ such that $\partial \mathfrak{Z}= 0$. 

By Lemma \ref{lem:translation}, there is $\un{c} \in \mathbb{V}_r$ such that $\psi_{\un{c}} ^* (\mathfrak{Z}) \in z^q _{\{ \widehat{X}_1 \}} (\widehat{X}_2, n)$. By \eqref{eqn:homotopy master}, we have
\begin{equation}\label{eqn:quasi-iso1-pre-1}
 (-1)^n \partial H_{n, \un{c}} ^* (\mathfrak{Z}) = \mathfrak{Z} - \psi_{\un{c}} ^* (\mathfrak{Z}).
 \end{equation}
Applying $\partial$ to \eqref{eqn:quasi-iso1-pre-1}, we get $0 = 0 - \partial ( \psi_{\un{c}} ^* (\mathfrak{Z}))$ because $\partial^2 =0$ and $\partial \mathfrak{Z} = 0$. Hence $\partial ( \psi_{\un{c}} ^* (\mathfrak{Z}))= 0$, and by \eqref{eqn:quasi-iso1-pre-1}, $\alpha$ can also be represented by $ \psi_{\un{c}} ^* (\mathfrak{Z})$. But this means $\alpha$ lies in the image of the induced natural homomorphism $\CH_{\{ \widehat{X}_1 \}} ^q (\widehat{X}_2, n) \to \CH^q (\widehat{X}_2, n)$. This proves the surjectivity in homology.

\medskip

We now prove the injectivity in homology. Let 
$$
\alpha \in \ker (\CH_{\{ \widehat{X}_1 \}} ^q (\widehat{X}_2, n) \to \CH^q (\widehat{X}_2, n)).
$$ 
This is represented by a cycle $\mathfrak{Z} \in z^q _{\{\widehat{X}_1\}} (\widehat{X}_2, n)$ such that $\mathfrak{Z} = \partial \mathfrak{W}$ for some $\mathfrak{W} \in z^q (\widehat{X}_2, n+1)$. Since $\partial^2 = 0$, this implies that $\partial \mathfrak{Z} = 0$. 

By Lemma \ref{lem:translation1} applied to $\mathfrak{W}$ and $\mathfrak{Z}$, there exists some $\un{c} \in \mathbb{V}_r$ such that 
\begin{equation}\label{eqn:homo1 inj 00}
\psi_{\un{c}} ^* (\mathfrak{W}) \in z^q _{\{ \widehat{X}_1 \}} (\widehat{X}_2, n+1), \ \ \psi_{\un{c}} ^* (\mathfrak{Z}) = z^q _{\{ \widehat{X}_1 \}} (\widehat{X}_2, n).
\end{equation}
Applying \eqref{eqn:homotopy master} to $ \mathfrak{W}$, we have
\begin{equation}\label{eqn:homotopy_app}
(-1)^{n+1} \partial H_{n+1, \un{c}} ^* (\mathfrak{W}) + (-1)^n H_{n, \un{c}} ^* (\partial \mathfrak{W}) = \mathfrak{W} - \psi_{\un{c}} ^* (\mathfrak{W}).
\end{equation}
Applying $\partial$ to \eqref{eqn:homotopy_app}, we deduce that 
\begin{equation}\label{eqn:homotopy_app1.5}
(-1)^n \partial H_{n, \un{c}} ^* (\partial \mathfrak{W}) = \partial \mathfrak{W} -  \partial ( \psi_{\un{c}} ^* (\mathfrak{W})).
\end{equation}
 Since $\mathfrak{Z}= \partial \mathfrak{W}$, after rearranging the terms, \eqref{eqn:homotopy_app1.5} can be rewritten as
\begin{equation}\label{eqn:homotopy_app2}
\mathfrak{Z} = \partial (\psi_{\un{c}} ^* (\mathfrak{W}) ) + (-1)^{n} \partial H_{n,\un{c}} ^* (\mathfrak{Z}) = \partial  (\psi_{\un{c}} ^* (\mathfrak{W})  + (-1)^n H_{n,\un{c}} ^* (\mathfrak{Z}))
\end{equation}

By \eqref{eqn:homo1 inj 00} and Lemma \ref{lem:homo1 inj}, we have 
\begin{equation}\label{eqn:homotopy_app3}
\psi_{\un{c}} ^* (\mathfrak{W})  + (-1)^n H_{n,\un{c}} ^* (\mathfrak{Z}) \in z^q _{\{ \widehat{X}_1 \}} (\widehat{X}_2, n+1).
\end{equation} Then \eqref{eqn:homotopy_app2} and \eqref{eqn:homotopy_app3} together imply that $\alpha = [\mathfrak{Z}] = 0$ in $ \CH^q _{\{ \widehat{X}_1 \}} (\widehat{X}_2, n) $. This proves the injectivity in homology.

Thus we proved that $\mathfrak{i}$ is a quasi-isomorphism.
\end{proof}

\subsubsection{Mod $Y$-equivalence under translation} We want to repeat the above discussion modulo $Y$. We first check:

\begin{lem}\label{lem:psi mod Y resp}
For $\un{c} \in \mathbb{V}_r$, the map $\psi_{\un{c}} ^*$ respects the mod $Y$-equivalence. 
\end{lem}

\begin{proof}
Suppose $(\mathcal{A}_1, \mathcal{A}_2) \in \mathcal{L}^q (\widehat{X}_2, Y, n)$. It means that we have an isomorphism
\begin{equation}\label{eqn:translation given1}
 \mathcal{A}_1 \otimes _{\mathcal{O}_{\square_{\widehat{X}} ^n}} ^{\mathbf{L}} \mathcal{O}_{\square_Y ^n} \simeq  \mathcal{A}_2 \otimes _{\mathcal{O}_{\square_{\widehat{X}} ^n}} ^{\mathbf{L}} \mathcal{O}_{\square_Y ^n}
 \end{equation}
 in $\sAlg (\mathcal{O}_{\square_Y ^n})$. We apply $\psi_{\un{c}} ^*$ to \eqref{eqn:translation given1}. Since $\psi_{\un{c}}$ sends $t_i$ to $t_i + c_i$, fixing each $a \in A$ and each of $y_{\ell}$, it induces automorphisms of $\square_{\widehat{X}} ^n$ and $\square_Y ^n$, respectively. Hence $\psi_{\un{c}} ^* \mathcal{O}_{\square_{\widehat{X}} ^n} = \mathcal{O}_{\square_{\widehat{X}} ^n}$ and $\psi_{\un{c}} ^* \mathcal{O}_{\square_Y^n} = \mathcal{O}_{\square_Y ^n}$. Thus applying $\psi_{\un{c}} ^*$ to \eqref{eqn:translation given1}, we have an isomorphism
 $$
 (\psi_{\un{c}} ^* \mathcal{A}_1) \otimes _{\mathcal{O}_{\square_{\widehat{X}} ^n}} ^{\mathbf{L}} \mathcal{O}_{\square_Y ^n} \simeq
(\psi_{\un{c}} ^* \mathcal{A}_2) \otimes _{\mathcal{O}_{\square_{\widehat{X}} ^n} } ^{\mathbf{L}} \mathcal{O}_{\square_Y ^n}
$$
in $\sAlg (\mathcal{O}_{\square_Y^n})$. This means $(\psi_{\un{c}} ^{*} (\mathcal{A}_1), \psi_{\un{c}}^{*} (\mathcal{A}_2)) \in \mathcal{L}^q (\widehat{X}_2, Y, n)$.
\end{proof}

\begin{lem}\label{lem:homo mod Y resp}
The homotopy $H_{\bullet, \un{c}}^*$ respects the mod $Y$-equivalence.
\end{lem}

\begin{proof}
Note that $H_{n, \un{c}} ^* = \phi_{n+1, \un{c}} ^* \circ pr^*$. 

Let $(\mathcal{A}_1, \mathcal{A}_2) \in \mathcal{L}^q (\widehat{X}, Y, n)$ so that there is an isomorphism 
\begin{equation}\label{eqn:mod Y homo pr}
\mathcal{A}_1 \otimes_{\mathcal{O}_{\square_{\widehat{X}} ^n}} ^{\mathbf{L}} \mathcal{O}_{\square_Y ^n} \simeq \mathcal{A}_2 \otimes^{\mathbf{L}}_{\mathcal{O}_{\square_{\widehat{X}}^n}} \mathcal{O}_{\square_Y ^n}
\end{equation}
in $\sAlg (\mathcal{O}_{\square_{Y} ^n})$.

We have morphisms of sheaves of rings $pr^* \mathcal{O}_{\square_{\widehat{X}} ^n} \to \mathcal{O}_{\square_{\widehat{X}} ^{n+1}}$ and $pr^* \mathcal{O}_{\square_{Y} ^n} \to \mathcal{O}_{\square_{Y} ^{n+1}}$, and $pr$ is flat (of type (III) as in Lemma \ref{lem:pre flat pb}). Applying $pr^*$ to \eqref{eqn:mod Y homo pr}, we have
$$
pr^* (\mathcal{A}_1) \otimes_{pr^* (\mathcal{O}_{\square_{\widehat{X}} ^n})} ^{\mathbf{L}} pr ^* (\mathcal{O}_{\square_Y ^n}) \simeq  pr^* (\mathcal{A}_2 )\otimes^{\mathbf{L}}_{ pr^* (\mathcal{O}_{\square_{\widehat{X}}^n})}  pr ^* (\mathcal{O}_{\square_Y ^n}),
$$
which implies that 
\begin{equation}\label{eqn:mod Y homo pr -1}
pr^* (\mathcal{A}_1) \otimes_{ \mathcal{O}_{\square_{\widehat{X}} ^{n+1}}} ^{\mathbf{L}} pr ^* (\mathcal{O}_{\square_Y ^n}) \simeq  pr^* (\mathcal{A}_2 )\otimes^{\mathbf{L}}_{  \mathcal{O}_{\square_{\widehat{X}} ^{n+1}}}  pr ^* (\mathcal{O}_{\square_Y ^n}),
\end{equation}
via $pr^* \mathcal{O}_{\square_{\widehat{X}} ^n} \to \mathcal{O}_{\square_{\widehat{X}} ^{n+1}}$. Applying $- \otimes _{pr^* (\mathcal{O}_{\square_Y ^n})} \mathcal{O}_{\square_Y ^{n+1}}$ to \eqref{eqn:mod Y homo pr -1}, we deduce an isomorphism
\begin{equation}\label{eqn:mod Y homo phi}
pr^* (\mathcal{A}_1) \otimes^{\mathbf{L}} _{\mathcal{O}_{\square_{\widehat{X}} ^{n+1}}} \mathcal{O}_{\square_Y ^{n+1}} \simeq pr^* (\mathcal{A}_2) \otimes^{\mathbf{L}}_{\mathcal{O}_{\square_{\widehat{X}}^{n+1}}} \mathcal{O}_{\square_Y ^{n+1}}
\end{equation}
in $\sAlg (\mathcal{O}_{\square_{Y} ^{n+1}})$. 

Since the map $\phi_{n+1, \un{c}}$ sends $\un{t}$ to $\un{t} + y_{n+1} \un{c}$, fixing each $a \in A$ and each of $y_{\ell}$, it induces automorphisms of $\square_{\widehat{X}} ^{n+1}$ and $\square_Y^{n+1}$, respectively. Hence $\phi_{n+1, \un{c}} ^* \mathcal{O}_{\square_{\widehat{X}} ^{n+1}} \simeq \mathcal{O}_{\square_{\widehat{X}} ^{n+1}}$ and $\phi_{n+1, \un{c}} ^* \mathcal{O}_{\square_{Y} ^{n+1}} \simeq \mathcal{O}_{\square_{Y} ^{n+1}}$. Thus, applying $\phi_{n+1, \un{c}} ^*$ to \eqref{eqn:mod Y homo phi}, we deduce an isomorphism 
$$
\phi_{n+1, \un{c}} ^* pr^* (\mathcal{A}_1) \otimes^{\mathbf{L}} _{\mathcal{O}_{\square_{\widehat{X}} ^{n+1}}} \mathcal{O}_{\square_Y ^{n+1} }\simeq \phi_{n+1, \un{c}} ^*  pr^* (\mathcal{A}_2) \otimes^{\mathbf{L}}_{\mathcal{O}_{\square_{\widehat{X}}^{n+1}}} \mathcal{O}_{\square_Y ^{n+1}},
$$
 i.e. an isomorphism
$$
H_{n, \un{c}} ^* (\mathcal{A}_1) \otimes^{\mathbf{L}} _{\mathcal{O}_{\square_{\widehat{X}} ^{n+1}}} \mathcal{O}_{\square_Y ^{n+1} } \simeq H_{n, \un{c}} ^* (\mathcal{A}_2) \otimes^{\mathbf{L}}_{\mathcal{O}_{\square_{\widehat{X}}^{n+1}}} \mathcal{O}_{\square _Y ^{n+1}}
$$
in $\sAlg (\mathcal{O}_{\square_Y ^{n+1}})$. Thus $(H_{n, \un{c}} ^* (\mathcal{A}_1), H_{n, \un{c}} ^* (\mathcal{A}_2)) \in \mathcal{L}^q (\widehat{X}, Y, n+1)$. 
\end{proof}

Lemmas \ref{lem:psi mod Y resp} and \ref{lem:homo mod Y resp} show that we can still use the homotopy identity \eqref{eqn:homotopy master} for the mod $Y$ cycle complexes as well. We now deduce another moving lemma, this time the mod $Y$-version:

\begin{thm}\label{thm:quasi-iso1}
Let $Y$ be a connected affine $k$-scheme of finite type, and let $Y \hookrightarrow X_1 \hookrightarrow X_2$ be closed immersions into equidimensional smooth $k$-schemes $X_1$ and $X_2$, such that $X_1 \hookrightarrow X_2$ is a complete intersection. Let $\widehat{X}_i$ be the completion of $X_i$ along $Y$ for $i=1,2$. 

Then the inclusion
\begin{equation}\label{eqn:mod Y affine moving 0}
\mathfrak{i}: z^q _{\{ \widehat{X}_1 \}} (\widehat{X}_2 \mod Y, \bullet)  \hookrightarrow z^q (\widehat{X}_2 \mod Y, \bullet)
\end{equation}
is a quasi-isomorphism.
\end{thm}

\begin{proof}In the trivial case when $Y= X_1$, by Lemma \ref{lem:translation fail-1} we have $z^q _{\{ \widehat{X}_1\}} (\widehat{X}_2, \bullet) = z^q (\widehat{X}_2, \bullet)$. Thus, the map \eqref{eqn:mod Y affine moving 0} is the identity, and there is nothing to prove.

\medskip

Now suppose $Y \subsetneq X_1$. By Lemma \ref{lem:translation fail}, we have $|J|= \infty$ so that we may use Lemma \ref{lem:translation}.

 The morphism $\mathfrak{i}$ induces the homomorphism for each $n \geq 0$:
 \begin{equation}\label{eqn:induced1}
{\rm H}_n (\mathfrak{i}): \CH^q _{ \{ \widehat{X}_1 \}} (\widehat{X}_2 \mod Y, n) \to \CH^q  (\widehat{X}_2 \mod Y, n).
\end{equation}
It remains to show that \eqref{eqn:induced1} is an isomorphism. 

But once we have Lemmas \ref{lem:translation}, \ref{lem:psi mod Y resp} and \ref{lem:homo mod Y resp}, we can use \eqref{eqn:homotopy master} and the mod $Y$ analogue of Lemma \ref{lem:homo1 inj}. From these, the arguments of the proof that \eqref{eqn:induced1} is an isomorphism are identical to the one given in Theorem \ref{thm:quasi-iso1-pre}. We omit details.
\end{proof}

Note that Theorems \ref{thm:quasi-iso1-pre} and \ref{thm:quasi-iso1} imply the special case of Theorem \ref{thm:formal moving} when $X_1 \hookrightarrow X_2$ is a complete intersection. We now prove it in general:

\medskip

\begin{proof}[Proof of Theorem \ref{thm:formal moving}]
Let $Y$ be a quasi-affine $k$-scheme of finite type, and $Y \hookrightarrow X_1 \hookrightarrow X_2$ are closed immersions into equidimensional smooth $k$-schemes $X_1$ and $X_2$. Let $\widehat{X}_i$ be the completion of $X_i$ along $Y$ for $i=1,2$.

Here, we no longer suppose that $X_1 \hookrightarrow X_2$ is a complete intersection. But still $X_1, X_2$ being smooth, the immersion $X_1 \hookrightarrow X_2$ is l.c.i., so that there is a finite affine open cover $\{ V_j \}$ of $X_2$ such that the induced closed immersions $X_1 \cap V_j \hookrightarrow X_2 \cap V_j$ are complete intersections for all $j$.

Let $U_j:= Y \cap V_j$. Here, each open subset $U_j \subset Y$ of $Y$ is affine, being a closed subset of an affine scheme $V_j$. Thus $\{ U_j \}$ is an affine open cover of $Y$. Regarding $U_j$ as an open subscheme, for the closed immersions $U_j \hookrightarrow X_1 \cap V_j \hookrightarrow X_2 \cap V_j$, the completion of $X_i \cap V_j$ along $U_j$ is identical to $\widehat{X}_i |_{U_j}$. 

In case some $U_j$ is not connected, we can work with each connected component of $U_j$ separately. So, replacing $\{U_j \}$ by the collection of all connected components, we may assume that $\{U_j\}$ is a finite set of connected affine open subschemes of $Y$ that covers $Y$.

\medskip

Returning back to the proof of the theorem, note that the inclusion morphisms \eqref{eqn:formal moving} are quasi-isomorphisms if and only if for each scheme point $y \in Y$, the induced morphisms of the stalks
\begin{equation}\label{eqn:fm-1}
\tuborg
\mathfrak{i}: \BGHz^q _{ \{ \widehat{X}_1 \}} (\widehat{X}_2, \bullet)_y \hookrightarrow \BGHz^q (\widehat{X}_2, \bullet)_y, \\
\mathfrak{i}: \BGHz^q _{ \{ \widehat{X}_1 \}} (\widehat{X}_2 \mod Y, \bullet)_y \hookrightarrow \BGHz^q (\widehat{X}_2 \mod Y, \bullet)_y
\sluttuborg
\end{equation}
are quasi-isomorphisms. Since $\{ U_j \}$ is an open cover of $Y$, for a given $y \in Y$, there is some $j_0$ such that $y \in U_{j_0}$. To prove that each of the morphisms in \eqref{eqn:fm-1} over $y \in Y$ is a quasi-isomorphism, it is enough to prove that the morphisms of complexes
\begin{equation}\label{eqn:fm-U}
\tuborg
\mathfrak{i}: z^q _{ \{ \widehat{X}_1|_{U_{j_0}} \}} (\widehat{X}_2|_{U_{j_0}}, \bullet) \hookrightarrow z^q (\widehat{X}_2|_{U_{j_0}}, \bullet), \\
\mathfrak{i}: z^q _{ \{ \widehat{X}_1|_{U_{j_0}} \}} (\widehat{X}_2|_{U_{j_0}} \mod U_{j_0}, \bullet) \hookrightarrow z^q (\widehat{X}_2|_{U_{j_0}} \mod U_{j_0}, \bullet)
\sluttuborg
\end{equation}
are quasi-isomorphisms. Here $U_{j_0}$ is a connected affine $k$-scheme of finite type, and $\widehat{X}_i |_{U_{j_0}}$ is the completion of $X_i \cap V_{j_0}$ along $U_{j_0}$, while $X_1 \cap V_{j_0} \hookrightarrow X_2 \cap V_{j_0}$ is a complete intersection. So, the morphisms in \eqref{eqn:fm-U} are indeed quasi-isomorphisms by Theorems \ref{thm:quasi-iso1-pre} and \ref{thm:quasi-iso1}, respectively. This completes the proof of the theorem.
\end{proof}

The discussions in \S \ref{sec:4.1} $\sim$ \S \ref{sec:formal_move} summarize as follows:

\begin{thm}\label{thm:sh intersection pb}
Let $Y$ be a quasi-affine $k$-scheme of finite type, and let $Y \hookrightarrow X_1 \overset{\iota}{\hookrightarrow} X_2$ be closed immersions into equidimensional smooth $k$-schemes $X_1, X_2$. Let $\widehat{X}_i$ be the completion of $X_i$ along $Y$ for $i=1,2$. Then the following hold:

\begin{enumerate}
\item We have zigzags of morphisms of complexes of sheaves on $Y_{\rm Zar}$
$$
\tuborg
\BGHz^q ( \widehat{X}_2, \bullet) \overset{\mathfrak{i}}{ \hookleftarrow} \BGHz^q _{\{ \widehat{X}_1 \}} (\widehat{X}_2 , \bullet) \overset{\widehat{\iota}^*}{\to} \BGHz^q (\widehat{X}_1, \bullet),\\
\BGHz^q ( \widehat{X}_2 \mod Y, \bullet) \overset{\mathfrak{i}}{ \hookleftarrow} \BGHz^q _{\{ \widehat{X}_1 \}} (\widehat{X}_2 \mod Y, \bullet) \overset{\widehat{\iota}^*}{\to} \BGHz^q (\widehat{X}_1 \mod Y, \bullet),
\sluttuborg
$$
where the morphisms $\mathfrak{i}$ are quasi-isomorphisms. 

\item They give morphisms in the derived category $\mathcal{D}^- ({\rm Ab} (Y))$ of bounded above complexes of sheaves on $Y_{\rm Zar}$:
$$
\tuborg
\iota^*:  \BGHz^q  (\widehat{X}_2 , \bullet) \to \BGHz^q (\widehat{X}_1, \bullet),\\
\iota^*: \BGHz^q (\widehat{X}_2 \mod Y, \bullet) \to \BGHz^q (\widehat{X}_1 \mod Y, \bullet).
\sluttuborg
$$

\item In particular, we have the induced homomorphisms
$$
\tuborg
\iota^*: \BGH^q (\widehat{X}_2 , n) \to \BGH^q (\widehat{X}_1, n), \\
\iota^*: \BGH^q (\widehat{X}_2 \mod Y, n) \to \BGH^q (\widehat{X}_1 \mod Y, n).
\sluttuborg
$$
\end{enumerate}
\end{thm}

\begin{proof}
(1) follows from Theorem \ref{thm:formal moving} and Proposition \ref{prop:1st indep}. 

(2) follows from (1) because $\mathfrak{i}$ are quasi-isomorphisms by Theorem \ref{thm:formal moving}.

(3) follows from (1) and Definition \ref{defn:BGH quasi-affine}.
\end{proof}

\subsection{General pull-backs up to quasi-isomorphism}\label{sec:pull-back moving}

As an important consequence of the moving lemma (Theorem \ref{thm:formal moving}), in the following Theorem \ref{thm:pull-back moving}, we prove the existence of more general pull-backs for ``algebraizable" morphisms of formal schemes not necessarily composites of the types (I), (II), (III) flat morphisms.

\begin{thm}\label{thm:pull-back moving}
Let $g: Y_1 \to Y_2 \in \QAff_k$ be a morphism. We have the following:
\begin{enumerate}
\item Suppose that there are closed immersions $Y_i \hookrightarrow X_i$ for $i=1,2$ into equidimensional smooth $k$-schemes and a morphism $f : X_1 \to X_2$, such that they form a commutative diagram
\begin{equation}\label{eqn:pull-back moving 1-1}
\xymatrix{
{X}_1 \ar[r]^{f} & {X}_2 \\ 
Y_1 \ar@{^{(}->}[u] \ar[r] ^g & Y_2. \ar@{^{(}->}[u]
}
\end{equation}

Then it induces the pull-back morphisms in $\mathcal{D}^- ({\rm Ab} (Y_2))$: 
\begin{equation}\label{eqn:pull-back moving 1-2}
\tuborg
g^*_{f}: \BGHz^q (\widehat{X}_2, \bullet) \to \mathbf{R} g_* \BGHz^q (\widehat{X}_1, \bullet), \\
g^*_{f} : \BGHz^q (\widehat{X}_2 \mod Y_2, \bullet) \to \mathbf{R}  g_* \BGHz^q (\widehat{X}_1 \mod Y_1, \bullet).
\sluttuborg
\end{equation}

\item The pull-backs in \eqref{eqn:pull-back moving 1-2} are functorial in the following sense: let $g_1: Y_1 \to Y_2$ and $g_2: Y_2 \to Y_3$ be morphisms in $\QAff_k$. Let $Y_i \hookrightarrow X_i$ for $i=1,2,3$ be closed immersions into equidimensional smooth $k$-schemes, and suppose we have morphisms $f_1: X_1\to X_2$ and $f_2: X_2 \to X_3$ such that they form the commutative diagram
\begin{equation}\label{eqn:pull-back moving 2-1}
\xymatrix{
X_1 \ar[r] ^{f_1} & X_2 \ar[r] ^{f_2} & X_3 \\
Y_1 \ar[r] ^{g_1} \ar@{^{(}->}[u] & Y_2 \ar[r] ^{g_2} \ar@{^{(}->}[u]  & Y_3. \ar@{^{(}->}[u]  }
\end{equation}

Then we have the equalities of the pull-back morphisms in $\mathcal{D}^- ({\rm Ab}(Y_2))$:
\begin{equation}\label{eqn:pull-back moving 2-2}
\tuborg
(g_2 \circ g_1)_{f_2 \circ f_1} ^* = (g_1 )_{f_1} ^* \circ  (g_2)_{f_2} ^* : \BGHz^q (\widehat{X}_3, \bullet) \to \mathbf{R} (g_2 \circ g_1)_* \BGHz^q (\widehat{X}_1, \bullet),\\
(g_2 \circ g_1)_{f_2 \circ f_1} ^* = (g_1 )_{f_1} ^* \circ  (g_2)_{f_2} ^* : \BGHz^q (\widehat{X}_3 \mod Y_3, \bullet)
\\ 
\hskip5cm \to \mathbf{R} (g_2 \circ g_1)_* \BGHz^q (\widehat{X}_1 \mod Y_1, \bullet),
\sluttuborg
\end{equation}
where the names of the above pull-back morphisms follow the notations of \eqref{eqn:pull-back moving 1-2}.

\item The pull-backs in \eqref{eqn:pull-back moving 1-2} are functorial in $f$ in the following sense: suppose we have the following commutative diagram
\begin{equation}\label{eqn:pull-back moving 3-1}
\xymatrix{ 
 & X_1 ' \ar[r] ^{f'} \ar[dl] _{h_1} & X_2 ' \ar[dr] ^{h_2} &  \\
 X_1 \ar[rrr] ^f & & &  X_2 & \\
 &  Y_1  \ar@{^{(}->}[uu]  \ar@{^{(}->}[ul] \ar[r] ^g &  Y_2,  \ar@{^{(}->}[uu]  \ar@{^{(}->}[ur] &
}
\end{equation}
where $X_1, X_2, X_1', X_2'$ are all equidimensional smooth $k$-schemes. 

Then it induces the following commutative diagrams in $\mathcal{D}^- ({\rm Ab} (Y_2))$:
$$
\xymatrix{
\BGHz^q (\widehat{X}_2, \bullet) \ar[r] ^{ g_f ^* \ } \ar[d] ^{ ({\rm Id}_{Y_2})_{h_2} ^*}  & \mathbf{R}g_* \BGHz^q (\widehat{X}_1, \bullet) \ar[d] ^{ ({\rm Id}_{Y_1})_{h_1}^*} \\
\BGHz^q (\widehat{X}_2', \bullet) \ar[r] ^{g_{f'} ^* \ } &  \mathbf{R}g_*  \BGHz^q (\widehat{X}_1', \bullet),
}
\hskip0.4cm
\xymatrix{
\BGHz^q (\widehat{X}_2 \mod Y_2, \bullet) \ar[r] ^{ g_f ^* \ \ \  } \ar[d] ^{ ({\rm Id}_{Y_2})_{h_2} ^*}  & \mathbf{R}g_* \BGHz^q (\widehat{X}_1 \mod Y_1, \bullet) \ar[d] ^{ ({\rm Id}_{Y_1})_{h_1}^*} \\
\BGHz^q (\widehat{X}_2' \mod Y_2, \bullet) \ar[r] ^{g_{f'} ^* \ \ \  } &  \mathbf{R}g_*  \BGHz^q (\widehat{X}_1' \mod Y_1, \bullet),
}
$$
where the names of the arrows follow the notational convention in \eqref{eqn:pull-back moving 1-2}.

\end{enumerate}

\end{thm}

\begin{proof}
(1) Let $\widehat{X}_i$ be the completion of $X_i$ along $Y_i$ for $i=1,2$. The diagram \eqref{eqn:pull-back moving 1-1} induces the morphism $\widehat{f}: \widehat{X}_1 \to \widehat{X}_2$.

Note that $\widehat{X}_1 \times \widehat{X}_2$ is the completion of $X_1 \times X_2$ along $Y_1 \times Y_2$ by Lemma \ref{lem:prod completion}. We also have the closed immersions $\iota= gr_f: X_1 \hookrightarrow X_1 \times X_2$, and $Y_1 \overset{gr_g}{\to} Y_1 \times Y_2 \hookrightarrow X_1 \times X_2$. 

Let $\widehat{X}_{12}$ be the completion of $X_1 \times X_2$ along $Y_1$ via the above closed immersions. This $\widehat{X}_{12}$ is equal to the further completion $\alpha: \widehat{X}_{12} \to \widehat{X}_1 \times \widehat{X}_2$ along $Y_1$ as well. Thus we have the commutative diagram
$$
\xymatrix{
\widehat{X}_1 \ar@{^{(}->}[r] ^{\widehat{\iota}}\ar@/^1.5pc/[rrr] ^{\widehat{f}} & \widehat{X}_{12} \ar[r] ^{\alpha \ \ \ } & \widehat{X}_{1} \times \widehat{X}_2 \ar[r] ^{ \ \ \widehat{pr}_2} & \widehat{X}_2 \\
Y_1 \ar@{^{(}->}[u] \ar@{=}[r] & Y_1  \ar@{^{(}->}[u]   \ar@{^{(}->}[r]^{gr_g \ \ } & Y_1 \times Y_2  \ar@{^{(}->}[u]  \ar[r] ^{\ \ pr_2} & Y_2.\ar@{^{(}->}[u] }
$$

In what follows in the proof, all morphisms of complexes of sheaves are morphisms in suitable derived categories of abelian sheaves.

Since the morphisms $\alpha$ and $\widehat{pr}_2$ are type (II) and (II) flat morphisms, respectively, by Lemmas \ref{lem:pre flat pb} and \ref{lem:fpb}, we have the induced flat pull-backs. Thus we do have the associated pull-backs
\begin{equation}\label{eqn:pull-back moving 1-pr}
\tuborg
\widehat{pr}_2 ^*: \BGHz^q (\widehat{X}_2, \bullet) \to \mathbf{R} (pr_2)_* \BGHz^q (\widehat{X}_1 \times \widehat{X}_2, \bullet),\\
\widehat{pr}_2 ^*: \BGHz^q (\widehat{X}_2 \mod Y_2 , \bullet) \to \mathbf{R} (pr_2)_* \BGHz^q (\widehat{X}_1 \times \widehat{X}_2 \mod Y_1 \times Y_2, \bullet),
\sluttuborg
\end{equation}
and
\begin{equation}\label{eqn:pull-back moving 1-alpha}
\tuborg
\alpha^*: \BGHz^q (\widehat{X}_1 \times \widehat{X}_2, \bullet) \to \mathbf{R} (gr_g)_* \BGHz^q (\widehat{X}_{12} , \bullet), \\
\alpha^*: \BGHz^q (\widehat{X}_1 \times \widehat{X}_2 \mod Y_1 \times Y_2, \bullet) \to \mathbf{R} (gr_g)_* \BGHz^q (\widehat{X}_{12} \mod Y_1 , \bullet).
\sluttuborg
\end{equation}
On the other hand, the morphism $\widehat{\iota}$ is deduced from the closed immersion $ X_1 \overset{\iota=gr_f}{\hookrightarrow} X_{12}$ so that by Theorem \ref{thm:sh intersection pb}-(2), we have morphisms
\begin{equation}\label{eqn:pull-back moving 1-iota}
\tuborg
\iota^*: \BGHz^q (\widehat{X}_{12}, \bullet) \to \BGHz^q (\widehat{X}_1, \bullet), \\
\iota^*: \BGHz^q (\widehat{X}_{12} \mod Y_1, \bullet) \to \BGHz^q (\widehat{X}_1 \mod Y_1, \bullet). \\
\sluttuborg
\end{equation}

Composing $\mathbf{R} (pr_2)_* \circ \mathbf{R} (gr_g)_*$ of $\iota^*$ of \eqref{eqn:pull-back moving 1-iota}, $\mathbf{R}(pr_2)_*$ of $\alpha^*$ of \eqref{eqn:pull-back moving 1-alpha}, and \eqref{eqn:pull-back moving 1-pr} in $\mathcal{D}^- ({\rm Ab}(Y_2))$, we deduce \eqref{eqn:pull-back moving 1-2}, because $pr_2 \circ gr_g = g$ so that $\mathbf{R} (pr_2)_* \circ \mathbf{R} (gr_g)_*= \mathbf{R} g_*$. This proves (1). 

We remark that the morphisms in \eqref{eqn:pull-back moving 1-2} can be written as zigzags of concrete morphisms of complexes, but we won't elaborate on this point; keeping track of all of them will make the compositions unbearably complicated as we go further.

\medskip

(2) Let $\widehat{X}_i$ be the completion of $X_i$ along $Y_i$. Following the conventions in the proof of (1), we let $\widehat{X}_{12}$ be the completion of $X_1 \times X_2$ along $Y_1$ via the closed immersion $Y_1 \overset{gr_{g_1}}{\hookrightarrow} Y_1 \times Y_2$, and we similarly define $\widehat{X}_{12}$ and $\widehat{X}_{23}$. 

Then the equalities in \eqref{eqn:pull-back moving 1-2} immediately follow from the commutative diagram
$$
\xymatrix{
\widehat{X}_1 \ar@{^{(}->}[r]  \ar@{^{(}->}[drr]  & \widehat{X}_{12} \ar[r] ^{\alpha_{12} \ \ } & \widehat{X}_1 \times \widehat{X}_2 \ar[r]^{ \ \ pr} & \widehat{X}_2  \ar@{^{(}->}[r] & \widehat{X}_{23} \ar[r] ^{\alpha_{23} \ \ } & \widehat{X}_2 \times \widehat{X}_3 \ar[r] ^{ \ \ pr} & \widehat{X}_3 \\
& & \widehat{X}_{13} \ar[rr] ^{\alpha_{23}} & & \widehat{X}_1 \times \widehat{X}_3. \ar[rru] ^{pr} & &}
$$

\medskip

(3) Once we know the existence of the morphisms in the diagrams, the commutativity is immediate because we have $h_2 \circ f' = f \circ h_1$ by \eqref{eqn:pull-back moving 3-1}.
\end{proof}

\section{A generalization of the \v{C}ech construction}\label{sec:finite type}

When $Y $ is a quasi-affine $k$-scheme $Y$ of finite type, after choosing a closed immersion $Y \hookrightarrow X$ into an equidimensional smooth $k$-scheme, and taking the completion $\widehat{X}$ of $X$ along $Y$, we constructed the complexes of sheaves 
$$
\BGHz^q (\widehat{X}, \bullet) \mbox{ and } \BGHz^q (\widehat{X} \mod Y, \bullet),
$$
and their hypercohomology groups $\BGH^q (\widehat{X}, n)$ and $\BGH^q (\widehat{X} \mod Y, n)$. 
 
 \medskip

Using them as part of the building blocks, the purpose of \S \ref{sec:finite type} is to define their globalized \v{C}ech versions for all $Y \in \Sch_k$. Here, we bring a variant of the \v{C}ech machine of R. Hartshorne \cite[Remark, p.28]{Hartshorne DR}.

\subsection{The \v{C}ech machine}\label{sec:Cech machine}

The construction below was inspired by the two-page long outline in R. Hartshorne \cite[Remark, p.28]{Hartshorne DR} for the algebraic de Rham cohomology. As cycles are generally harder to deal with than differential forms, there are various differences. An analogue for the algebraic $K$-theory is constructed in \cite[\S 4]{PP}.

\subsubsection{Systems of local embeddings and the \v{C}ech complexes}\label{sec:Cech complex sheaf} Recall the following minor variant of the notion of systems of local embeddings of R. Hartshorne \cite[Remark, p.28]{Hartshorne DR}:

\begin{defn}\label{defn:system embedd Hartshorne}
A \emph{system of local embeddings} for a scheme $Y \in \Sch_k$ is a finite set $\mathcal{U}= \{ (U_i, X_i)\}_{i \in \Lambda}$ of pairs, where $\{ U_i\}$ is a quasi-affine open cover of $Y$, and for each $i\in \Lambda$, we have a closed immersion $U_i \hookrightarrow X_i$ into an equidimensional smooth $k$-scheme $X_i$. 
\qed
\end{defn}

In what follows, we use a \v{C}ech-style machine to define a bi-complex of sheaves on $Y$, arising from families of complexes of sheaves over all multi-indices $I \in \Lambda^{p+1}$ and $p \geq 0$, for a given system $\mathcal{U}$ of local embeddings for $Y$. 

\medskip

For $Y \in \Sch_k$, fix a system $\mathcal{U}= \{ (U_i, X_i ) \}_{i \in \Lambda}$ of local embeddings. Let $p, q \geq 0$ be integers. For each $(p+1)$-tuple of indices $I= ( i_0 , \cdots ,i_p ) \in \Lambda^{p+1}$, define the open set $U_{I} := U_{i_0} \cap \cdots \cap U_{i_p}$. Since each $U_i$ is quasi-affine, this $U_I$ is also quasi-affine. Define $X_I:= X_{i_0} \times \cdots \times X_{i_p}$, which is an equidimensional smooth $k$-scheme again. We have the induced diagonal closed immersion $U_I \hookrightarrow X_I$.

Let $\widehat{X}_I$ be the completion of $X_I$ along $U_I$. Consider the complexes of sheaves on $U_I$
$$
\BGHz^q (\widehat{X}_I, \bullet) \ \ \ \mbox{ and } \ \ \ \BGHz^q (\widehat{X}_I \mod U_I, \bullet).
$$

\medskip

For a given $I= ( i_0, \cdots,  i_p)\in \Lambda^{p+1}$, and $0 \leq j \leq p$, let $I_j' := ( i_0 ,\cdots, \widehat{i}_j , \cdots , i_p)\in \Lambda^p$, where ${i}_j$ is omitted from $I$. We have the open immersion $g_I^j: U_I \hookrightarrow U_{I'_j}$, and the projection $X_I \hookrightarrow X_{I'_j}$. They induce the commutative diagram
$$
\xymatrix{
 \widehat{X}_I \ar[r]^{f_{I} ^j} & \widehat{X}_{I'_j} \\
 U_I \ar@{^{(}->}[u] \ar@{^{(}->}[r] ^{g_{I} ^j} & U_{I'_j}, \ar@{^{(}->}[u]}
 $$ 
 where the induced map $f_{I}^j$ is a morphism of quasi-affine formal schemes, which is a composite of flat morphisms of types (I), (II), (III) of Lemmas \ref{lem:pre flat pb} and \ref{lem:fpb}.
(N.B. We give a simple example. Take $I= (1,2)$ and $I_1' = (2)$. We have the associated commutative diagram
 $$
 \xymatrix{
 X_1 \times X_2 \ar[r] & X_1 \times X_2 \ar[r] & X_1 \times X_2 \ar[r] ^{\ \ \ pr} & X_2 \\
 U_1 \cap U_2 \ar@{^{(}->}[u] \ar@{^{(}->}[r] ^{\Delta \ \ \ \ \ \ \ \  } & (U_1 \cap U_2) \times (U_1 \cap U_2) \ar[r]^{ \ \ \ \ \ \ op}  \ar@{^{(}->}[u] & U_1 \times U_2  \ar@{^{(}->}[u] \ar[r] ^{ \ \ \ pr} & U_2, \ar@{^{(}->}[u]}
 $$
 where $\Delta$ is the diagonal embedding, $op$ is the open immersion, and $pr$ are the projections. They induce flat morphisms of formal schemes of the types (II), (I), (III), respectively. The general case is similar.) 

Thus we have the flat pull-back morphisms of complexes of sheaves
 $$
 \tuborg
  \delta_{I,j}:= (f_I^j)^*: \BGHz^q (\widehat{X}_{I_j'}  , \bullet) \to \mathbf{R}  (g_I^j)_* \BGHz^q (\widehat{X}_{I}  , \bullet),\\
 \delta_{I,j}:= (f_I^j)^*: \BGHz^q (\widehat{X}_{I_j'} \mod U_{I'_j} , \bullet) \to \mathbf{R} (g_I^j)_* \BGHz^q (\widehat{X}_{I} \mod U_I , \bullet).
 \sluttuborg
 $$

For each $U_I$, let $\iota=\iota_I: U_I \hookrightarrow Y$ be the open immersion, where we often ignore the subscript $I$ of $\iota_I$ in what follows for notational simplicity. Define 
$$ 
\tuborg
\check{\mathfrak{C}}^q(\mathcal{U}^{\infty})^p:= \underset{ I \in \Lambda ^{p+1}}{\prod} \mathbf{R} \iota_* \BGHz^q (\widehat{X}_I , \bullet), \\
\check{\mathfrak{C}}^q(\mathcal{U})^p:= \underset{ I \in \Lambda ^{p+1}}{\prod} \mathbf{R} \iota_* \BGHz^q (\widehat{X}_I \mod U_I, \bullet).
\sluttuborg
$$
For each $I \in \Lambda^{p+1}$, define
$$
\tuborg
\delta^{p-1} _I:= \overset{p}{\underset{j=0}{\sum}} (-1)^j  \delta_{I,j}: \overset{p}{\underset{j=0}{\prod}} \mathbf{R} \iota_* \BGHz^q (\widehat{X}_{I_j'} , \bullet) \to \mathbf{R} \iota_* \BGHz^q (\widehat{X}_I , \bullet), \\
\delta^{p-1} _I:= \overset{p}{\underset{j=0}{\sum}} (-1)^j  \delta_{I,j}: \overset{p}{\underset{j=0}{\prod}}  \mathbf{R}\iota_* \BGHz^q (\widehat{X}_{I_j'} \mod U_{I_j'}, \bullet) \to \mathbf{R} \iota_* \BGHz^q (\widehat{X}_I \mod U_I, \bullet).
\sluttuborg
$$
Combining them over all $I \in \Lambda^{p+1}$, we deduce the \v{C}ech boundary maps
\begin{equation}\label{eqn:cech boundary map}
\tuborg
\delta^{p-1} :  \check{\mathfrak{C}}^q (\mathcal{U}^{\infty})^{p-1} \to  \check{\mathfrak{C}}^q (\mathcal{U}^{\infty})^p, \\
\delta^{p-1} :  \check{\mathfrak{C}}^q (\mathcal{U})^{p-1} \to  \check{\mathfrak{C}}^q (\mathcal{U})^p.
\sluttuborg
\end{equation}
One checks that $\delta ^p \circ \delta^{p-1} = 0$ by the usual \v{C}ech formalism. Hence we have the double complexes of sheaves (strictly speaking, objects in ${\rm Kom}^+ (\mathcal{D}^- ({\rm Ab} (Y)))$)
$$
\tuborg
\check{\mathfrak{C}} ^q (\mathcal{U}^{\infty}): = \check{\mathfrak{C}} ^q (\mathcal{U}^{\infty})^{\bullet},\\
\check{\mathfrak{C}} ^q (\mathcal{U}): = \check{\mathfrak{C}} ^q (\mathcal{U})^{\bullet}.
\sluttuborg
$$ Here, we regard the homological complexes $\BGHz^q (\widehat{X}_I , \bullet)$ and $\BGHz^q (\widehat{X}_I \mod U_I, \bullet)$ as cohomological complexes with the negative indexing, so that $\check{\mathfrak{C}} ^q (\mathcal{U}^{\infty})$ and $\check{\mathfrak{C}} ^q (\mathcal{U})$ are cohomological double complexes.

 \begin{defn}\label{defn:Cech cx}
 Let $Y \in \Sch_k$. Let $\mathcal{U}$ be a system of local embeddings for $Y$. Define the following objects in $\mathcal{D} ({\rm Ab} (Y))$
 $$
 \tuborg
 {\BGHz} ^q (\mathcal{U}^{\infty}, \bullet):= {\rm Tot}\  \check{\mathfrak{C}}^q (\mathcal{U}^{\infty}), \\
  {\BGHz} ^q (\mathcal{U}, \bullet):= {\rm Tot}\  \check{\mathfrak{C}}^q (\mathcal{U}),
\sluttuborg
 $$
respectively. Define the the Zariski hypercohomology groups

 \begin{equation}\label{eqn:vBGH}
 \tuborg
  \BGH^q (\mathcal{U}^{\infty}, n):= \mathbb{H}_{\rm Zar} ^{-n}  ( Y,  {\BGHz} ^q (\mathcal{U}^{\infty}, \bullet)),\\
 \BGH^q (\mathcal{U}, n):= \mathbb{H}_{\rm Zar} ^{-n}  ( Y,   {\BGHz} ^q (\mathcal{U}, \bullet)),
 \sluttuborg
 \end{equation}
respectively. \qed
 \end{defn}
 
 These groups are not yet the final groups we aim to define. We need one more step: to take (homotopy) colimits over all $\mathcal{U}$. To do so, we need a suitable notion of refinements of systems $\mathcal{U}$.

 \subsubsection{Refinements}\label{sec:refinement}
 We consider the following version of refinements for systems of local embeddings: 

 \begin{defn}\label{defn:refinement system}
 Let $\mathcal{U} = \{ ( U_i, X_i ) \}_{i \in \Lambda}$ and $\mathcal{V}= \{ ( V_j, X_j'  )\}_{j \in \Lambda'}$  be systems of local embeddings for $Y\in \Sch_k$ in the sense of Definition \ref{defn:system embedd Hartshorne}.
 
  Suppose there is a set map $\lambda: \Lambda' \to \Lambda$ such that for each $j\in \Lambda'$,
 \begin{enumerate}
 \item we have the inclusion $V_j \subset U_{\lambda (j)}$, and
 \item there is a morphism $X'_j \to X_{\lambda (j)}$ whose restriction to $V_j$ gives the inclusion $V_j \subset U_{\lambda (j)}$. 
 \end{enumerate}
  Then we say that $\mathcal{V}$ is a \emph{refinement} of $\mathcal{U}$.
 \qed
 \end{defn}

\begin{remk}
We remark that our notion of refinements is more flexible than the original one considered in R. Hartshorne \cite[Remark, p.28]{Hartshorne DR}. There, he assumes an additional requirement that morphisms $X'_j \to X_{\lambda (j)}$ in Definition \ref{defn:refinement system}-(2) are smooth. 

Because we have a somewhat general pull-back property of Theorem \ref{thm:pull-back moving} in suitable derived categories, we do not need to assume any extra hypothesis on the morphisms.
\qed
\end{remk}

In what follows, we will show that our notion of refinements for systems of local embeddings of Definition \ref{defn:refinement system} works well in our \v{C}ech-style construction.  
 
 \begin{lem}\label{lem:common refinement}
 Let $Y \in \Sch_k$. 
 Let $\mathcal{U}  = \{ ( U_i, X_i) \}_{i \in \Lambda}$ and $\mathcal{V}= \{ ( V_j, X'_j )\}_{j \in \Lambda'}$ be two systems of local embeddings for $Y$. 
 
 Then there exists a system $\mathcal{W}$ of local embeddings for $Y$, that is a refinement of both $\mathcal{U}$ and $\mathcal{V}$. 
 \end{lem}
 
 \begin{proof}
 
  Define the set $\Lambda'':= \Lambda \times \Lambda'$. For $(i, j) \in \Lambda''$, let $W_{ij}:= U_i \cap V_j$. Since $U_i, V_j$ are quasi-affine, so is $W_{ij}$. 
   
 We are given a closed immersion $U_i \hookrightarrow X_i$, and $W_{ij}$ is open in $U_i$. The subspace topology on $U_i$ from $X_i$ coincides with the topology of $U_i$. Thus there is an open subscheme $X_{ij} \subset X_i$ with the induced closed immersion $W_{ij} \hookrightarrow X_{ij}$. Similarly, $W_{ij}$ is open in $V_j$, so that there is an open subscheme $X'_{ij} \subset X'_{j}$ with the closed immersion $W_{ij} \hookrightarrow X'_{ij}$.

The closed immersions $W_{ij} \hookrightarrow X_{ij}, X'_{ij}$ induce the diagonal immersion $W_{ij} \hookrightarrow X''_{ij}:= X_{ij} \times_k X'_{ij}$. Let $\mathcal{W}:= \{ (W_{ij}, X''_{ij} ) \}_{(i,j)\in \Lambda''}$. This is a system of local embeddings for $Y$.
 
 For the projection maps $\lambda, \lambda': \Lambda'' = \Lambda \times \Lambda' \to \Lambda, \Lambda'$ of the index sets, respectively, we have the corresponding morphisms
 $$
 \tuborg X''_{ij} = X_{ij} \times X_{ij}' \to X_{ij} \overset{op}{\hookrightarrow} X_i, \\
 X''_{ij}= X_{ij} \times X_{ij} ' \to  X_{ij}' \overset{op}{\hookrightarrow} X_j'.
 \sluttuborg
 $$
 Hence $\mathcal{W}$ is a refinement of both $\mathcal{U}$ and $\mathcal{V}$. 
 \end{proof}

 \begin{lem}\label{lem:induced map proj refine}
 Let $\mathcal{U} =\{ ( U_i, X_i) \}_{i \in \Lambda}$ and $\mathcal{V}= \{ ( V_j, X'_j )\}_{j \in \Lambda'}$ be systems of local embeddings for $Y \in \Sch_k$. Suppose that $\mathcal{V}$ is a refinement of $\mathcal{U}$
 given by a set map $\lambda: \Lambda' \to \Lambda$. 
 
 Then there are associated morphisms 
 $$
 \tuborg
  \lambda^*: {\BGHz}^q (\mathcal{U}^{\infty}, \bullet) \to {\BGHz}^q (\mathcal{V}^{\infty}, \bullet), \\
 \lambda^*: {\BGHz}^q (\mathcal{U}, \bullet) \to {\BGHz}^q (\mathcal{V}, \bullet)
 \sluttuborg
 $$
 in the derived category $\mathcal{D} ({\rm Ab} (Y))$. 
 \end{lem}
 
 \begin{proof}
  We construct morphisms of double complexes of sheaves (strictly speaking, morphisms in ${\rm Kom}^+ (\mathcal{D}^- ({\rm Ab} (Y)))$)
\begin{equation}\label{eqn:proj refine morph}
\tuborg
\lambda^*: \check{\mathfrak{C}} ^q (\mathcal{U}^{\infty}) \to \check{\mathfrak{C}} ^q (\mathcal{V}^{\infty}),\\
\lambda^*: \check{\mathfrak{C}} ^q (\mathcal{U}) \to \check{\mathfrak{C}} ^q (\mathcal{V}).
\sluttuborg
\end{equation}
 
For each integer $p \geq 0$, we want to construct morphisms in $\mathcal{D}^- ({\rm Ab} (Y))$
$$
\tuborg
\check{\mathfrak{C}} ^q (\mathcal{U}^{\infty})^p  \to \check{\mathfrak{C}}^q (\mathcal{V}^{\infty})^p, \\
\check{\mathfrak{C}} ^q (\mathcal{U})^p  \to \check{\mathfrak{C}}^q (\mathcal{V})^p, \ \ \mbox{ i.e.}
\sluttuborg
$$
\begin{equation}\label{eqn:pth refinement}
\tuborg
\underset{I \in \Lambda^{p+1}}{\prod} \mathbf{R} \iota_* \BGHz^q (\widehat{X}_I, \bullet) \to \underset{J \in (\Lambda')^{p+1}}{\prod}  \mathbf{R}  \iota_* \BGHz^q (\widehat{X}_J', \bullet), \\
\underset{I \in \Lambda^{p+1}}{\prod}  \mathbf{R}  \iota_* \BGHz^q (\widehat{X}_I \mod U_I, \bullet) \to \underset{J \in (\Lambda')^{p+1}}{\prod}   \mathbf{R}  \iota_* \BGHz^q (\widehat{X}_J' \mod V_J, \bullet),
\sluttuborg
\end{equation}
where  
$U_I, V_J, \widehat{X}_I, \widehat{X}'_J$ follow the previous notational conventions, and $\iota$ denotes any of the open immersions of $U_I$ and $V_J$ into $Y$.

\medskip

 Let $J= (j_0, \cdots, j_p) \in (\Lambda')^{p+1}$. We continue to follow the notations of \S \ref{sec:Cech complex sheaf}.
 Since $\mathcal{V}$ is a refinement of $\mathcal{U}$, this gives $\lambda (J) := (\lambda (j_0), \cdots, \lambda (j_p)) \in \Lambda^{p+1}$, and we have the commutative diagram
 \begin{equation}\label{eqn:pth refinement 2}
\xymatrix{
{X}_J ' \ar[r] &  {X}_{\lambda (J)}\\
V_J \ar@{^{(}->}[u] \ar@{^{(}->}[r]^{g_J} & U_{\lambda (J)},  \ar@{^{(}->}[u] }
\end{equation}
where the vertical maps are closed immersions into equidimensional smooth $k$-schemes and the bottom horizontal map $g_J$ is the open immersion. By Theorem \ref{thm:pull-back moving} applied to \eqref{eqn:pth refinement 2}, we have the morphisms in $\mathcal{D}^- ({\rm Ab} (Y))$
  \begin{equation}\label{eqn:Cech fpb}
\tuborg
\phi ^{\lambda(J)}_{J}: \mathbf{R} \iota_* \BGHz^q (\widehat{X}_{\lambda (J)} , \bullet) \to \mathbf{R} \iota_* \BGHz ^q (\widehat{X}_{J}' , \bullet), \\
\phi ^{\lambda(J)}_{J}:  \mathbf{R} \iota_* \BGHz^q (\widehat{X}_{\lambda (J)} \mod U_{\lambda (J)}, \bullet) \to \mathbf{R} \iota_* \BGHz ^q (\widehat{X}_{J}' \mod V_J, \bullet).
\sluttuborg
\end{equation}

\medskip

Once we have \eqref{eqn:Cech fpb}, the maps \eqref{eqn:pth refinement} are given by the above data: namely, the $J$-th factor of $\check{\mathfrak{C}} ^q (\mathcal{V}^{\infty})^p$ (resp. $\check{\mathfrak{C}} ^q (\mathcal{V})^p$) in \eqref{eqn:pth refinement} is defined to be the projection from $\check{\mathfrak{C}}^q (\mathcal{U}^{\infty})^p$ (resp. $\check{\mathfrak{C}}^q (\mathcal{U})^p$) to its $\lambda(J)$-th factor, composed with the flat pull-back \eqref{eqn:Cech fpb}. This defines the morphisms \eqref{eqn:pth refinement}. In short, we write \eqref{eqn:pth refinement} as
\begin{equation}\label{eqn:BT convention}
(\lambda^* \omega)( V_J) = \omega (U_{\lambda (J)}).
\end{equation}
(N.B. This notational convention can be found in a few general references, e.g. in Bott-Tu \cite[p.111]{BT}.)

One checks that this is compatible with the \v{C}ech boundary maps \eqref{eqn:cech boundary map}. We leave out the details (cf. \cite[Lemma 10.4.1, p.111]{BT}).
Hence taking the total complexes of \eqref{eqn:proj refine morph}, we obtain the lemma.
\end{proof}

\begin{lem}\label{lem:induced map up to homotopy}
Let $\mathcal{U} =\{ ( U_i, X_i) \}_{i \in \Lambda}$ and $\mathcal{V}= \{ ( V_j, X'_j )\}_{j \in \Lambda'}$ be systems of local embeddings for $Y \in \Sch_k$. Suppose that $\mathcal{V}$ is a refinement of $\mathcal{U}$ for two different set maps $\lambda_1$ and $\lambda_2: \Lambda' \to \Lambda$.

Then the induced morphisms $\lambda_1^*$ and $\lambda_2^*$ in $\mathcal{D} ({\rm Ab} (Y))$
$$
\tuborg
\lambda_1 ^*, \lambda_2 ^*: {\BGHz}^q (\mathcal{U}^{\infty}, \bullet) \to {\BGHz}^q (\mathcal{V}^{\infty}, \bullet), \\
\lambda_1 ^*, \lambda_2 ^*: {\BGHz}^q (\mathcal{U}, \bullet) \to {\BGHz}^q (\mathcal{V}, \bullet)
\sluttuborg
$$
are homotopic to each other, respectively. 
\end{lem}

\begin{proof}
This is standard: as in Bott-Tu \cite[Lemma 10.4.2, p.111]{BT} or Bosch-G\"untzer-Remmert \cite[\S 8.1.3, p.324]{BGR}, we define a homotopy $H: \check{\mathfrak{C}} ^q (\mathcal{U}^{\infty})^p \to \check{\mathfrak{C}} ^q (\mathcal{V}^{\infty})^{p-1}$ (resp. $H: \check{\mathfrak{C}} ^q (\mathcal{U})^p \to \check{\mathfrak{C}} ^q (\mathcal{V})^{p-1}$) by 
$$ 
(H \omega) (V_{( j_0, \cdots, j_{p-1})}) = \sum_{i=0} ^{p-1} (-1)^i \omega ( U_{ (\lambda_1 (j_0), \cdots, \lambda_1 (j_i), \lambda_2 (j_i), \cdots, \lambda_2 (j_{p-1}))}).
$$
For the morphisms $\lambda_1 ^*, \lambda_2 ^*: \check{\mathfrak{C}}^q (\mathcal{U}^{\infty})^{\bullet} \to \check{\mathfrak{C}} ^q (\mathcal{V}^{\infty})^{\bullet}$, one checks that they satisfy $\lambda_1 ^* - \lambda_2 ^* = \delta H + H \delta$. Taking the totals, we obtain the lemma.
\end{proof}

 \subsubsection{The yeni higher Chow groups on $\Sch_k$} 
 We now define the final object in this paper we were after: 
 
 \begin{defn}\label{defn:Cech final}
 Let $Y\in \Sch_k$. Let $\mathcal{U}$ be a system of local embeddings for $Y$ in the sense of Definition \ref{defn:system embedd Hartshorne}. 
 
 For the complexes of sheaves with respect to $\mathcal{U}$, $ \BGHz^q (\mathcal{U}^{\infty}, n)$ and $ \BGHz^q (\mathcal{U}, n)$ of Definition \ref{defn:Cech cx}, we define the yeni higher Chow sheaves of $Y^{\infty}$ and $Y$ to be the homotopy colimits in $\mathcal{D} ({\rm Ab} (Y))$
  $$
 \tuborg
  \BGHz^q (Y^{\infty}, \bullet):= \underset{\mathcal{U}}{\hocolim} \  \BGHz^q (\mathcal{U}^{\infty}, \bullet), \\
 \BGHz^q (Y, \bullet):=  \underset{\mathcal{U}}{\hocolim} \ \BGHz^q (\mathcal{U}, \bullet),
 \sluttuborg
 $$
 over all systems $\mathcal{U}$ of local embeddings for $Y$. Define the yeni higher Chow groups of $Y^{\infty}$ and $Y$ to be their hypercohomology groups 
 \begin{equation}\label{eqn:Cech final 1}
 \tuborg
  \BGH^q (Y^{\infty}, n):= \mathbb{H}_{\rm Zar} ^{-n} (Y, \BGHz^q (Y^{\infty}, \bullet)) = \underset{\mathcal{U}}{\colim} \  \BGH^q (\mathcal{U}^{\infty}, n), \\
 \BGH^q (Y, n):=  \mathbb{H} _{\rm Zar} ^{-n} (Y, \BGHz^q (Y, \bullet)) = \underset{\mathcal{U}}{\colim} \ \BGH^q (\mathcal{U}, n)
 \sluttuborg
 \end{equation}
 where the colimits are taken over all systems $\mathcal{U}$ of local embeddings for $Y$. 
 \qed
 \end{defn}

In particular, by construction, the yeni higher Chow groups in \eqref{eqn:Cech final 1} depend neither on the choices of any finite quasi-affine open cover $\{ U_i \}$ of $Y$, nor on the choices of the closed immersions $U_i \hookrightarrow X_i$ into equidimensional smooth $k$-schemes.

 \subsection{Two consistency questions}
 
 We check that our theory defined in Definition \ref{defn:Cech final} is consistent with the groups defined in \cite{Bloch HC} and \cite{Park Tate}.

\subsubsection{Consistency in the smooth case}\label{sec:1st indep sm}

To see that our theory extends the pre-existing theory of motivic cohomology on smooth $k$-schemes, i.e. higher Chow groups, we ask: ``If $Y$ was equidimensional and smooth over $k$, then are yeni higher Chow groups isomorphic to Bloch's higher Chow groups?" 

We answer it in Theorem \ref{thm:sm formal} below. First recall the following result embedded in the proof of S. Bloch \cite[Theorem (9.1)]{Bloch HC}. 

\begin{thm}\label{thm:Bloch hyper}
Let $Y$ be a quasi-projective $k$-scheme. 

\begin{enumerate} 
\item Let $Y \hookrightarrow X$ be a closed immersion into a smooth $k$-scheme. Let $z^* (\bullet)_X$ be the Zariski sheafification on $X_{\rm Zar}$ of the higher Chow complex of $X$. Then we have 
$$
\CH^* (Y, n) = \mathbb{H}_{Y} ^{-n} (X, z^* (\bullet)_X),
$$
where the left hand side is the higher Chow group of $Y$ of S. Bloch \cite{Bloch HC}, and the right hand side is the local hypercohomology with the support $Y$.
\item In addition, if $Y$ is smooth equidimensional, then $\CH^q (Y, n) = \mathbb{H} ^{-n}_{\rm Zar} (Y, z^q ( \bullet)_Y).$
\end{enumerate}
\end{thm}

\begin{proof}
(1) It readily follows by combining the localization sequence for higher Chow groups (\cite[Theorem (0.1)]{Bloch moving}) with \cite[Theorem (3.4)]{Bloch HC}. (N.B. The localization sequence was first claimed as \cite[Theorem (3.1)]{Bloch HC}, and its proof was fixed in \cite{Bloch moving}. cf. \cite{cubical localization}.)

(2) Since $Y$ is smooth over $k$ itself, take $X= Y$ in (1).
\end{proof}

The assumption that $Y$ is quasi-projective in the above Theorem \ref{thm:Bloch hyper}, can be removed. The point is that the technical localization theorem in \cite{Bloch moving} can be extended to all $k$-schemes of finite type as in \cite{Levine moving}. The argument was given for the simplicial version of the higher Chow groups. In the cubical version, one can also give an argument, see e.g. \cite{cubical localization}. 

From \cite[Theorem 4.1.8]{cubical localization}, we recall have the following consequence of the localization:

\begin{thm}\label{thm:Bloch hyper 1}
Let $Y$ be a $k$-scheme of finite type. Then
$$
\CH_d (Y, n) \simeq \mathbb{H}_{\rm Zar} ^{-n} (Y, \BGHz_d (\bullet)_Y),
$$
where $\BGHz_d (\bullet)_Y$ is the Zariski sheafification of the presheaf $U \in Op (Y) \mapsto \BGHz_d (U, \bullet)$.

\end{thm}

This immediately implies:

\begin{cor}\label{cor:Bloch hyper}
Let $Y$ be an equidimensional smooth $k$-scheme. Then $\CH^q (Y, n) = \mathbb{H} ^{-n}_{\rm Zar} (Y, z^q ( \bullet)_Y).$
\end{cor}

\begin{thm}\label{thm:sm formal}
Let $Y$ be an equidimensional smooth $k$-scheme.

Then we have isomorphisms of the various higher Chow groups: 
\begin{equation}\label{eqn:sm formal}
\BGH^q(Y^{\infty}, n) \simeq \BGH^q (Y, n) \simeq \CH^q (Y, n).
\end{equation}
\end{thm}

\begin{proof}
Since $Y$ is smooth, it is a disjoint union of connected components, where each component is smooth irreducible. Working with each irreducible component separately, we may assume $Y$ is irreducible.

\medskip

Let $\mathcal{U} = \{ (U_i, X_i) \}_{i \in \Lambda}$ be an arbitrary system of local embeddings for $Y$. In particular $|\mathcal{U}|:= \{ U_i \}$ is a finite quasi-affine open cover of $Y$.

Since $Y$ is smooth irreducible, each $U_i$ is smooth irreducible, and the identity map $U_i \to U_i$ can be seen as a closed immersion. Hence $\mathcal{U}_0:= \{ (U_i , U_i) \}_{i \in \Lambda}$ is a system of local embeddings, and it is also a refinement of $\mathcal{U}$.

Recall that when $\mathcal{F}$ is a sheaf, the usual \v{C}ech complex (see \cite[Section 02FR]{stacks})
$$
0\to \mathcal{F} \to \mathfrak{C} (| \mathcal{U}_0|, \mathcal{F})^0  \overset{\delta}{\to} \mathfrak{C} (| \mathcal{U}_0|, \mathcal{F}) ^1  \overset{\delta}{\to} \cdots
$$
is exact (see \cite[Lemma 02FU]{stacks}), so that $\mathcal{F} \to \mathfrak{C} (| \mathcal{U}_0|, \mathcal{F})^{\bullet}$ is a resolution. Applying this to the complex $z^q (\bullet)_Y$ of sheaves, we see that the natural morphism
\begin{equation}\label{eqn:sm formal 00}
z^q (\bullet)_Y \to {\rm Tot} \ \mathfrak{C} ( |\mathcal{U}_0|, z^q (\bullet)_Y)^{\bullet}
\end{equation}
is a quasi-isomorphism. 

On the other hand, for the identity closed immersion $U_i \to U_i$, the completion $\widehat{U}_i$ of $U_i$ along $U_i$ is also $U_i$ itself. Thus, for each $I \in \Lambda^{p+1}$, we have $\widehat{X}_I = X_I = U_I$ and 
$$
\BGHz^q (\widehat{X}_I, \bullet) = \BGHz^q (\widehat{X}_I \mod U_I, \bullet) = z^q (\bullet)_{U_I}.
$$
Hence putting them into their respective \v{C}ech-style complexes, we have
$$
\check{\mathfrak{C}}^q (\mathcal{U}_0 ^{\infty})^{\bullet} = \check{\mathfrak{C}}^q (\mathcal{U}_0 )^{\bullet} = \mathfrak{C}( | \mathcal{U}_0|, z^q (\bullet)_Y)^{\bullet}.
$$


Thus taking the hypercohomology groups of their total complexes (where for $z^q (\bullet)_Y$, using Corollary \ref{cor:Bloch hyper} and \eqref{eqn:sm formal 00}), we deduce
\begin{equation}\label{eqn:sm formal 1}
\BGH^q (\mathcal{U}_0 ^{\infty}, n) \simeq \BGH^q (\mathcal{U}_0, n) \simeq \CH^q (Y, n).
\end{equation}

Here $\mathcal{U}$ was arbitrary, but in \eqref{eqn:sm formal 1}, the group $\CH^q (Y, n)$ is independent of $\mathcal{U}$. Hence the other two groups in \eqref{eqn:sm formal 1} are also independent of $\mathcal{U}$. Since $\mathcal{U}_0$ is a refinement of $\mathcal{U}$, and since the groups in \eqref{eqn:sm formal 1} are all independent of $\mathcal{U}$, we deduce that their colimits over all $\mathcal{U}$ are also independent. This means we have the isomorphisms \eqref{eqn:sm formal}.
\end{proof}

 \subsubsection{Consistency with \cite{Park Tate} for $k[t]/(t^m)$}
 Let $k_m:= k[t]/(t^m)$. In \cite{Park Tate}, for $Y= \Spec (k_m)$, the group $\BGH^q (Y, n)$ was defined to be $\CH^q (\Spf (k[[t]]) \mod Y, n)$ in terms of our notation in Definition \ref{defn:complex}, without taking any colimits. Let's check that this notion coincides with our Definition \ref{defn:Cech final} for the above $Y$:
 
 \begin{thm}\label{thm:consistency k_m}
 Let $Y= \Spec (k_m)$. Then we have 
 $$
 \BGH^q (Y, n) = \CH^q (\Spf (k[[t]]) \mod Y, n),
 $$
 where the left hand side is the group defined as in Definition \ref{defn:Cech final}, and the right hand side is the group in Definition \ref{defn:complex}, which is identical to the group in \cite{Park Tate}.
 \end{thm}
 
 \begin{proof}
 Let $\mathcal{U}$ be an arbitrary system of local embeddings for $Y$. Since $Y= \Spec (k_m)$ is topologically just a singleton space, any member $(U_i, X_i) \in \mathcal{U}$ is of the form $(Y, X)$ for a closed immersion $Y \hookrightarrow X$ into an equidimensional smooth $k$-scheme $X$. Pick any one of them, and define the new singleton system $\mathcal{U}_1:= \{ (Y, X) \}$ of local embeddings for $Y$. It is a refinement of $\mathcal{U}$. After shrinking $X$ and replacing $\mathcal{U}_1$ by the new system if necessary, we may assume $X$ is an integral smooth affine $k$-scheme of finite type.
 
 For the closed immersion $Y \hookrightarrow X$, let $p \in X$ be the $k$-rational point that is the image of the unique closed point of $Y$. 
 
 Since the embedding dimension of $k_m$ is $1$, by successive applications of the version of the Bertini theorems for hypersurfaces containing a given closed subscheme (see Kleiman-Altman \cite{KA} when $k$ is infinite, and see J. Gunther \cite[Theorem 1.1]{Gunther} or F. Wutz \cite[Theorem 2.1]{Wutz} when $k$ is finite. The latter articles for a finite field uses a method from B. Poonen \cite{Poonen}.), there exists a smooth curve $C \subset X$ containing $Y$ as a closed subscheme. This gives the new singleton system $\mathcal{U}_2 := \{ ( Y, C) \}$ for $Y$, and it is again a refinement of $\mathcal{U}_1$. Note that the curve $C$ may not be a rational curve. 
 
 Consider the completion $\widehat{C}$ of $C$ along $Y$. Since $Y$ is a fat point, this $\widehat{C}$ is equal to the completion of $C$ at the closed point $p \in C$. The stalk at $p$ is given by $\widehat{\mathcal{O}}_{C, p} $, which is isomorphic to $k[[t]]$ by the Cohen structure theorem (I. Cohen \cite{Cohen}). Hence $\widehat{C} \simeq \Spf (k[[t]])$. 
 
Since $\mathcal{U}$ was arbitrary, we see that $\BGH^q (\mathcal{U}_2, n) = \CH^q (\Spf (k[[t]]) \mod Y, n)$ is a final object in the inductive system. Thus this is equal to $\BGH^q (Y, n)$. This proves the theorem.
 \end{proof}

\subsection{Zariski descent and stacks}\label{sec:Zariski descent full}
 
When $Y \in \Sch_k$, we constructed the complexes $\BGHz^q (Y^{\infty}, \bullet)$, $\BGHz^q (Y, \bullet)$ in $\mathcal{D} ({\rm Ab}(Y))$, as homotopy colimits over all systems of local embeddings. To double check that the theory is sound regarding taking restrictions to open subsets, one question we check is the following kind of consistency:

\begin{lem}\label{lem:restriction coincidence}
Let $Y \in \Sch_k$. Let $\iota_U: U \subset Y$ be an open subscheme. 

Then we have isomorphisms in $\mathcal{D} ({\rm Ab} (U))$:
\begin{equation}\label{eqn:restriction coincidence 1}
\tuborg
\BGHz^q (Y^{\infty}, \bullet)|_U \overset{\sim}{\to} \BGHz^q (U^{\infty}, \bullet), \\
\BGHz^q (Y, \bullet)|_U \overset{\sim}{\to} \BGHz^q (U, \bullet),
\sluttuborg
\end{equation}
where ${}|_U$ means applying the flat pull-back $\iota_U^*= \mathbf{L} \iota_U^*: \mathcal{D} ({\rm Ab} (Y)) \to \mathcal{D} ({\rm Ab} (U))$.
\end{lem}

\begin{proof}
We introduce a notation: let $S(Y)$ be the collection of all systems of local embeddings for $Y$. This forms a category, where morphisms are given by refinements.

Since the statement holds vacuously when $U=\emptyset$, suppose $U \not = \emptyset$.

Let $\mathcal{U}= \{ (U_i, X_i ) \}_{i \in \Lambda} \in S(Y)$. For each $i \in \Lambda$, the topology on $U_i$ coincides with the subspace topology on $U_i$ induced from $X_i$ via the closed immersion $U_i \hookrightarrow X_i$. Thus, for the open subscheme $U_i \cap U  \subset U_i$, there exists an open subscheme $X_i' \subset X_i$ such that $U_i \cap X_i' = U_i \cap U$ and $U_i \cap U \hookrightarrow X_i'$ is a closed immersion. Here $X_i'$ is again an equidimensional smooth $k$-scheme. Thus, we have the induced restricted system $\mathcal{U}|_U= \{ (U_i \cap U, X'_i) \}_{i \in \Lambda} \in S(U)$.

Since $\BGHz^q (Y^{\infty}, \bullet)$ and $\BGHz^q (U^{\infty}, \bullet)$ (resp. $\BGHz^q (Y, \bullet)$ and $\BGHz^q (U, \bullet)$) are defined as homotopy colimits over $S(Y)^{\op}$ and $S(U)^{\op}$, respectively, the map
$$
S(Y)^{\op} \to S(U)^{\op} 
$$
given by $\mathcal{U} \mapsto \mathcal{U}|_U$ induces the natural morphisms of \eqref{eqn:restriction coincidence 1} in $\mathcal{D} ({\rm Ab}(U))$. To see that the morphisms are isomorphisms, it remains to check that at each scheme point $x \in U$, the morphisms of the stalks in $\mathcal{D} ({\rm Ab})$
\begin{equation}\label{eqn:restriction coincidence 2}
\tuborg
\BGHz^q (Y^{\infty}, \bullet) |_{U, x} = \BGHz^q (Y^{\infty}, \bullet)_x {\to} \BGHz^q (U^{\infty}, \bullet)_x, \\
\BGHz^q (Y, \bullet)|_{U, x} = \BGHz^q (Y, \bullet)_x {\to} \BGHz^q (U, \bullet)_x,
\sluttuborg
\end{equation}
are isomorphisms in $\mathcal{D} ({\rm Ab})$. This is immediate: for any system $\mathcal{V}= \{ (U_i, X_i)\}$ of local embeddings for $U$, for $x \in U$, consider the subset $\mathcal{V}_x$ of the pairs $(U_i, X_i)$ where $x \in U_i$. We may replace it with its refinement where each $U_i$ is quasi-affine. Then one can find a system $\mathcal{U}$ of local embeddings for $Y$, such that the subset $\mathcal{U}_x$, defined similarly, is a refinement of $\mathcal{V}_x$. Hence the morphisms \eqref{eqn:restriction coincidence 2} are isomorphisms in $\mathcal{D} ({\rm Ab})$ as desired. 
\end{proof}

\begin{cor}\label{cor:final stack}
Let $Y\in \Sch_k$. Let $Op (Y)$ be the category of all Zariski open subsets of $Y$. Then the assignment 
$$
U\in Op(Y)^{\op} \mapsto \{ \mbox{the isomorphism class of } \BGHz^q (U, \bullet)\} \mbox{ in } \mathcal{D} ({\rm Ab}(U))
$$
is a stack. The corresponding property for $\BGHz^q (U^{\infty}, \bullet)$ also holds.
\end{cor}

\begin{proof}
We prove the first statement only, as the argument for the second one is identical. 
By induction on the number of open subsets, it is enough to check the descent property for two open subsets $U, V \subset Y$. 

Indeed, by Lemma \ref{lem:restriction coincidence}, we have isomorphisms in $\mathcal{D} ({\rm Ab} (U\cap V))$
$$
\BGHz^q (U, \bullet) |_{U \cap V} \simeq \BGHz^q (U \cap V, \bullet) \simeq \BGHz^q (V, \bullet) |_{U \cap V}.
$$
On the other hand, we have $\BGHz^q (U \cup V, \bullet)$ in $ \mathcal{D} ({\rm Ab} (U \cup V))$ such that there are isomorphisms
$$
\tuborg
\BGHz^q (U \cup V, \bullet)|_{U} \simeq \BGHz^q (U, \bullet), \ \ \mbox{ in } \mathcal{D} ({\rm Ab} (U)),\\
\BGHz^q (U \cup V, \bullet)|_{V} \simeq \BGHz^q (V, \bullet), \ \ \mbox{ in } \mathcal{D} ({\rm Ab} (V)),\\
\sluttuborg
$$
by Lemma \ref{lem:restriction coincidence} again. This proves the corollary.
\end{proof}

\begin{remk}
The above means that, roughly speaking, $\BGHz^q (U \cup V, \bullet)$ is a ``gluing" of $\BGHz^q (U , \bullet)$ and $\BGHz^q (V, \bullet)$ along $\BGHz^q (U \cap V, \bullet)$.
\qed
\end{remk}

\begin{thm}\label{thm:Zariski descent full} 
For $Y \in \Sch_k$, the yeni higher Chow groups of $Y^{\infty}$ and $Y$ satisfy the Zariski descent. 
\end{thm}

\begin{proof}
Let $Y \in \Sch_k$ be a fixed scheme and let $U, V \subset Y$ be open subschemes. We prove the assertion for the yeni higher Chow groups of $Y$ only, as the other case for $Y^{\infty}$ is similar.

For the object $\BGHz^q (Y, \bullet)$ in $\mathcal{D} ({\rm Ab}(Y))$, we have the long exact sequence of the hypercohomology groups
$$
\cdots \to \mathbb{H}^{-n} _{\rm Zar}(U \cup V, \BGHz^q ( Y, \bullet)|_{U \cup V}) \to \mathbb{H}^{-n}_{\rm Zar} (U, \BGHz^q (Y, \bullet) |_{U}) \oplus  \mathbb{H}^{-n} _{\rm Zar} (V, \BGHz^q (Y, \bullet) |_{V})
$$
$$
\hskip5cm  \to \mathbb{H}_{\rm Zar}^{-n} (U \cap V, \BGHz^q (Y, \bullet)_{U \cap V}) \to \cdots.
$$
But by Lemma \ref{lem:restriction coincidence}, we have 
$$
 \mathbb{H}^{-n} _{\rm Zar}(U \cup V, \BGHz^q ( Y, \bullet)|_{U \cup V}) \simeq \mathbb{H}_{\rm Zar} ^{-n} (U \cup V, \BGHz^q (U \cup V, \bullet)),
 $$
 and similarly for all other open subsets $U, V, U \cap V$. Hence by the above long exact sequence and by Definition \ref{defn:Cech final}, we have the long exact sequence
 $$
 \cdots \to \BGH^q (U \cup V, n) \to \BGH^q (U, n) \oplus \BGH^q (V, n) \to \BGH^q (U \cap V, n) 
 $$
 $$
 \hskip5cm \to \BGH^q (U \cup V, n-1) \to \cdots.
 $$
 This proves the theorem.
\end{proof}

 \subsection{The cycle versions}\label{sec:cycle Cech BGH}

 For $Y \in \Sch_k$, recall that in Definition \ref{defn:Cech final}, we defined some hypercohomology groups.
 
 \medskip
 
 In \S \ref{sec:cycle Cech BGH}, we want to give analogous constructions of cycle class groups. In this paper, they play some auxiliary roles only, while we expect that they are isomorphic to the hypercohomology version. See the Appendix, \S \ref{sec:appendix} for some relevant discussions.
 
 
\medskip
 
 Let $Y \in \Sch_k$, and let $\mathcal{U}= \{ (U_i, X_i)\}_{i \in \Lambda}$ be a system of local embeddings for $Y$. Let $q \geq 0$ be a fixed integer. The sketch is largely identical to the discussions in \S \ref{sec:Cech complex sheaf}, where the complexes $\BGHz^q (\widehat{X}_I, \bullet)$ of sheaves are replaced by the complexes $z^q (\widehat{X}_I , \bullet)$ of abelian groups, and similarly $\BGHz^q (\widehat{X}_I \mod U_I, \bullet)$ by $z^q (\widehat{X}_I \mod U_I, \bullet)$.
 
 \medskip
 
Let $p \geq 0$. For each multi-index $I= (i_0, \cdots, \cdots, i_p) \in \Lambda^{p+1}$, let $U_I:= U_{i_0} \cap \cdots \cap U_{i_p}$ and $X_I:= X_{i_0} \times_k \cdots \times _k X_{i_p}$. We have the diagonal closed immersion $U_I \hookrightarrow X_I$. Let $\widehat{X}_I$ be the completion of $X_I$ along $U_I$. 

\medskip
 
For each $0 \leq j\leq p$, let $I_j':= (i_0, \cdots, \widehat{i}_j, \cdots, i_p)$. This gives the open immersion $g_I ^j: U_I \hookrightarrow U_{I'_j}$ and the natural projection $X_I \to X_{I'_j}$ that ignores $X_{i_j}$. This induces $f_I ^j: \widehat{X}_I \to \widehat{X}_{I'_j}$ whose restriction to $U_I$ is $U_{I'_j}$, and it is a composite of flat morphisms of type (I), (II), (III) (see Lemma \ref{lem:fpb}) so that we have the flat pull-back  morphisms
$$
\tuborg
\delta_{I,j} = (f_I ^j)^*: z^q (\widehat{X}_{I'_j}, \bullet) \to z^q (\widehat{X}_I, \bullet), \\
\delta_{I,j} = (f_I ^j)^*: z^q (\widehat{X}_{I'_j} \mod U_{I'_j} \bullet) \to z^q (\widehat{X}_I \mod U_I, \bullet).
\sluttuborg
$$
Let
$$
\tuborg
\check{\rm C}^q (\mathcal{U}^{\infty}) ^p:= \underset{I \in \Lambda ^{p+1}}{\prod} z^q (\widehat{X}_I, \bullet), \\
\check{\rm C}^q (\mathcal{U}) ^p:= \underset{I \in \Lambda ^{p+1}}{\prod} z^q (\widehat{X}_I \mod U_I, \bullet).
\sluttuborg
$$
For each $I \in \Lambda ^{p+1}$, we have
$$
\tuborg
\delta_I ^{p-1} = \overset{p}{ \underset{j=0}{ \sum}}(-1)^j \delta_{I,j} : \overset{p}{ \underset{j=0}{\prod} }z^q (\widehat{X}_{I'_j} , \bullet) \to z^q (\widehat{X}_I, \bullet), \\
\delta_I ^{p-1} = \overset{p}{ \underset{j=0}{ \sum}} (-1)^j \delta_{I,j} : \overset{p}{ \underset{j=0}{\prod} } z^q (\widehat{X}_{I'_j} \mod U_{I'_j} , \bullet) \to z^q (\widehat{X}_I \mod U_I, \bullet).
\sluttuborg
$$
Combining the above over all multi-indices $I \in\Lambda^{p+1}$, we have the \v{C}ech boundary maps
$$
\tuborg
\delta^{p-1} : \check{\rm C}  ^q (\mathcal{U}^{\infty})^{p-1} \to \check{\rm C}^q (\mathcal{U}^{\infty}) ^p, \\
\delta^{p-1} : \check{\rm C}  ^q (\mathcal{U})^{p-1} \to \check{\rm C}^q (\mathcal{U}) ^p.
\sluttuborg
$$
 By the usual \v{C}ech formalism, we have $\delta^p \circ \delta ^{p-1} = 0$ so that we obtain the double complexes $\check{\rm C} ^q(\mathcal{U}^{\infty})= \check{\rm C} ^q(\mathcal{U}^{\infty})^{\bullet}$ and $\check{\rm C} ^q(\mathcal{U})= \check{\rm C} ^q(\mathcal{U})^{\bullet}$ of abelian groups. Here we regard $z^q (\widehat{X}_I, \bullet)$ and $z^q (\widehat{X}_I \mod U_I, \bullet)$ as cohomological complexes with the negative indexing. 

\begin{defn}\label{defn:Cech cycle cx}
Let $Y \in \Sch_k$, and let $\mathcal{U}$ be a system of local embeddings for $Y$. Define the cycle complexes
$$
\tuborg
{z}^q (\mathcal{U}^{\infty}, \bullet) := {\rm Tot}\  \check{\rm C} ^q (\mathcal{U}^{\infty}),\\
{z}^q (\mathcal{U}, \bullet) := {\rm Tot}\  \check{\rm C} ^q (\mathcal{U}).
\sluttuborg
$$
Let $\CH ^q (\mathcal{U}^{\infty}, n):={\rm  H}_{n} ({z}^q (\mathcal{U}^{\infty}, \bullet))$ and $\CH ^q (\mathcal{U}, n):={\rm  H}_{n} ({z}^q (\mathcal{U}, \bullet))$.\qed
\end{defn}

There are natural homomorphisms from the above cycle class groups $\CH ^q (\mathcal{U}^{\infty}, n)$ and $\CH ^q (\mathcal{U}, n)$ to the hypercohomology groups $\BGH^q (\mathcal{U}^{\infty}, n)$ and $\BGH^q (\mathcal{U}, n)$:

\begin{lem}\label{lem:flasque natural cech}
Let $Y\in \Sch_k$ be a $k$-scheme of finite type and let $\mathcal{U} =\{ (U_i, X_i) \}_{i \in \Lambda}$ be a system of local embeddings for $Y$. 

Then there exist natural homomorphisms
\begin{equation}\label{eqn:flasque natural cech 01}
\tuborg
\CH^q ( \mathcal{U}^{\infty}, n) \to \BGH^q (\mathcal{U}^{\infty}, n), \\
\CH^q (\mathcal{U}, n) \to \BGH^q (\mathcal{U}, n).
\sluttuborg
\end{equation}
\end{lem}

\begin{proof}
Denote the sheaves $\tilde{\mathcal{S}}_{\bullet} ^q$ and $\mathcal{S}_{\bullet} ^q$ in Lemma \ref{lem:flasque natural} associated to the triple $(Y, X, \widehat{X})$ by $\tilde{\mathcal{S}}_{\bullet} ^q (\widehat{X})$ and $\mathcal{S}_{\bullet} ^q (\widehat{X} \mod Y)$. Form $\tilde{\mathcal{S}}_{\bullet} ^q (\widehat{X}_I )$ and $\mathcal{S}_{\bullet} ^q (\widehat{X}_I \mod U_I)$ for each multi-index $I \in \Lambda^{p+1}$ for $p \geq 0$, similarly.

We have the natural injective morphisms of complexes of sheaves
$$
\tuborg
\tilde{\mathcal{S}}_{\bullet}^q (\widehat{X}_I ) \to \BGHz^q (\widehat{X}_I, \bullet), \\
\mathcal{S}_{\bullet}^q (\widehat{X} _I \mod U_I ) \to \BGHz^q (\widehat{X}_I \mod U_I, \bullet).
\sluttuborg
$$

We repeat the \v{C}ech construction of  \S \ref{sec:Cech complex sheaf} with the sheaves $\tilde{\mathcal{S}}_{\bullet}^q (\widehat{X}_I )$ over all $I \in \Lambda^{p+1}$ and $p \geq 0$, instead of $\BGHz^q(\widehat{X}_I,\bullet)$ (resp.  $\mathcal{S}_{\bullet}^q (\widehat{X}_I \mod U_I )$ instead of 
 $\BGHz^q(\widehat{X}_I \mod U_I,\bullet)$).

 To distinguish these, let $\check{\mathfrak{C}} (\mathcal{U}^{\infty}, \tilde{\mathcal{S}}^q)$ and $\check{\mathfrak{C}} (\mathcal{U}^{\infty}, \BGHz^q)$ denote the so-obtained double complexes of sheaves, respectively, and similarly, define $\check{\mathfrak{C}} (\mathcal{U}, \mathcal{S}^q)$ and $\check{\mathfrak{C}} (\mathcal{U}, \BGHz^q)$. We thus have the induced morphisms
\begin{equation}\label{eqn:tot two Cech sheaves 01}
\tuborg
{\rm Tot}\  \check{\mathfrak{C}} (\mathcal{U}^{\infty}, \tilde{\mathcal{S}}^q) \to {\rm Tot}\  \check{\mathfrak{C}} (\mathcal{U}^{\infty}, \BGHz^q),\\
{\rm Tot}\  \check{\mathfrak{C}} (\mathcal{U}, \mathcal{S}^q) \to {\rm Tot}\  \check{\mathfrak{C}} (\mathcal{U}, \BGHz^q).
\sluttuborg
\end{equation}

To deduce the homomorphisms \eqref{eqn:flasque natural cech 01} from \eqref{eqn:tot two Cech sheaves 01}, it is enough to identify the hypercohomology groups of the first column total complexes of \eqref{eqn:tot two Cech sheaves 01}. For this, we note that the sheaves $\tilde{\mathcal{S}}_{n} ^q (\widehat{X}_I)$ and $\mathcal{S}_{n} ^q (\widehat{X}_I \mod U_I)$ are flasque by Proposition \ref{prop:flasque sh}. Hence we have
$$
\tuborg
\mathbf{R} \iota_* \tilde{\mathcal{S}}_{\bullet} ^q (\widehat{X}_I) = \iota_*  \tilde{\mathcal{S}}_{\bullet} ^q (\widehat{X}_I) , \\
\mathbf{R} \iota_* {\mathcal{S}}_{\bullet} ^q (\widehat{X}_I) = \iota_*  {\mathcal{S}}_{\bullet} ^q (\widehat{X}_I) , 
\sluttuborg
$$
so that  ${\rm Tot}\  \check{\mathfrak{C}} (\mathcal{U}^{\infty}, \tilde{\mathcal{S}}^q)$ and ${\rm Tot}\  \check{\mathfrak{C}} (\mathcal{U}, \mathcal{S}^q)$ are complexes of flasque sheaves.

Hence we have 
\begin{equation}\label{eqn:tot two Cech sheaves homo comp 01}
\tuborg
\mathbb{H}_{\rm Zar} ^{-n} (Y, {\rm Tot}  \ \check{\mathfrak{C}} (\mathcal{U}^{\infty}, \tilde{\mathcal{S}}^q)) \\
=^{\dagger} {\rm H} ^{-n} \Gamma (Y,  {\rm Tot} \  \check{\mathfrak{C}} (\mathcal{U}^{\infty}, \tilde{\mathcal{S}}^q) )= {\rm H}_n \Gamma (Y, {\rm Tot}\   \check{\mathfrak{C}} (\mathcal{U}^{\infty}, \tilde{\mathcal{S}}^q)), \mbox{ and }\\
\mathbb{H}_{\rm Zar} ^{-n} (Y, {\rm Tot}  \ \check{\mathfrak{C}} (\mathcal{U}, \mathcal{S}^q)) \\
=^{\dagger} {\rm H} ^{-n} \Gamma (Y,  {\rm Tot} \  \check{\mathfrak{C}} (\mathcal{U}, \mathcal{S}^q) )= {\rm H}_n \Gamma (Y, {\rm Tot}\   \check{\mathfrak{C}} (\mathcal{U}, \mathcal{S}^q)),
\sluttuborg
\end{equation}
where $\dagger$ hold because ${\rm Tot} \ \check{\mathfrak{C}} (\mathcal{U}^{\infty}, \tilde{\mathcal{S}}^q)$ and ${\rm Tot} \ \check{\mathfrak{C}} (\mathcal{U}, \mathcal{S}^q)$ are flasque. By the definitions of  $\tilde{\mathcal{S}}^q_{\bullet} (\widehat{X}_I)$ and $\mathcal{S}^q_{\bullet} (\widehat{X}_I \mod U_I)$, we have 
$$
\tuborg
\Gamma (Y, {\rm Tot} \  \check{\mathfrak{C}}  (\mathcal{U}^{\infty}, \tilde{\mathcal{S}}^q))= {\rm Tot}\  \Gamma (Y,  \check{\mathfrak{C}} (\mathcal{U}^{\infty}, \tilde{\mathcal{S}}^q)) = {\rm Tot} \ \check{\rm C}^q (\mathcal{U}^{\infty}) ={z}^q (\mathcal{U}^{\infty}, \bullet), \\
\Gamma (Y, {\rm Tot} \  \check{\mathfrak{C}}  (\mathcal{U}, \mathcal{S}^q))= {\rm Tot}\  \Gamma (Y,  \check{\mathfrak{C}} (\mathcal{U}, \mathcal{S}^q)) = {\rm Tot} \ \check{\rm C}^q (\mathcal{U}) ={z}^q (\mathcal{U}, \bullet).
\sluttuborg
$$
Putting them back into \eqref{eqn:tot two Cech sheaves homo comp 01}, from \eqref{eqn:tot two Cech sheaves 01} we deduce \eqref{eqn:flasque natural cech 01} as desired.
\end{proof}

\section{Functoriality theorem and applications}\label{sec:functoriality}

In \S \ref{sec:finite type}, for each $Y \in \Sch_k$ we defined the yeni higher Chow groups of $Y^{\infty}$ and $Y$ (see Definition \ref{defn:Cech final})
$$
\BGH^q (Y^{\infty}, n) \mbox{ and } \BGH^q (Y, n).
$$

 The goal of \S \ref{sec:functoriality} is to discuss their functoriality and some applications.

In \S \ref{sec:Cech3} we prove that the yeni higher Chow groups are functorial in $Y$. In \S \ref{sec:product}, we show that the groups have the product structures and form bi-graded rings. In \S \ref{sec:relative theory}, we define the relative theory. In \S \ref{sec:local deform}, we discuss some applications to deformation theory.

 \subsection{Functoriality}\label{sec:Cech3}
 
 We first associate homomorphisms of the groups to morphisms $g: Y_1 \to Y_2$ in $\Sch_k$. Since our groups are defined via systems of local embeddings, it is necessary to have a way to relate a given system $\mathcal{U}$ of local embeddings for $Y_2$ to a system $\mathcal{V}$ of local embeddings for $Y_1$ via the morphism $g$: 
 
 \begin{defn}\label{defn:higher refinement}
 Let $g: Y_1 \to Y_2$ be a morphism in $\Sch_k$. Let $\mathcal{U}= \{ (U_i, X_i)\}_{i \in \Lambda}$ and $\mathcal{V}= \{ (V_j, X'_j)\}_{j \in \Lambda'}$ be systems of local embeddings for $Y_2$ and $Y_1$, respectively.
 \begin{enumerate}
 \item We say that \emph{$\mathcal{V}$ is associated to $\mathcal{U}$ by $g$}, or \emph{$(\mathcal{U}, \mathcal{V})$ is a pair associated via $g$}, if there is a set map $\lambda: \Lambda' \to \Lambda$ such that for each $j \in \Lambda'$, we have $g (V_j) \subset U_{\lambda (j)}$, and there is a morphism $f_j: X_j' \to X_{\lambda (j)}$ such that $f_j|_{V_j} = g|_{V_j}$.
 
 \item Suppose $\mathcal{V}$ is associated to $\mathcal{U}$ by $g$ for a set map $\lambda: \Lambda' \to \Lambda$. Let $p \geq 0$ be an integer. Let $I \in \Lambda^{p+1}$ and $J \in (\Lambda')^{p+1}$. We say that \emph{$J$ is associated to $I$ by $g$} or \emph{$(I, J)$ is a pair associated via $g$}, if we have $\lambda (J) = I$. \qed
 \end{enumerate}
  \end{defn}
 
 We remark that when $g={\rm Id}_Y: Y \to Y$ with $Y_1 = Y_2 = Y$, the notion in Definition \ref{defn:higher refinement}-(1) coincides with the notion of refinements in Definition \ref{defn:refinement system}.
 
\begin{lem}\label{lem:system ass functor}
Let $g: Y_1 \to Y_2$ be a morphism in $\Sch_k$. Let $\mathcal{U} = \{ (U_i, X_i) \}_{i \in \Lambda}$ be any system of local embeddings for $Y_2$. 

Then there exists a system $\mathcal{V}$ of local embeddings for $Y_1$ associated to $\mathcal{U}$ by $g$.
\end{lem}

\begin{proof}

For each $i \in \Lambda$, consider the open subset $g^{-1} (U_i) \subset Y_1$. It may not be quasi-affine, but we can find a finite quasi-affine open cover $\{ V_{ij} \}_{j \in \Lambda_i'}$ of $g^{-1} (U_i)$ for a finite set $\Lambda_i'$. Let $\Lambda':= \coprod_{ i \in \Lambda} \Lambda_i '$. Let $\lambda: \Lambda' \to \Lambda$ be the projection set map that sends any $j \in \Lambda_i'$ to $i$.

 For each $V_{ij}$, choose a closed immersion $\iota_{ij}: V_{ij} \hookrightarrow Z_{ij}$ into an equidimensional smooth $k$-scheme. Let $X_{ij}' := Z_{ij} \times X_i$, and consider the a closed immersion
 $$ 
 V_{ij} \overset{ gr_{g_{ij}}}{\hookrightarrow} V_{ij} \times U_i \hookrightarrow Z_{ij} \times  X_i = X_{ij}',
 $$
 where $g_{ij} := g|_{V_{ij}}$. Let $\mathcal{V}:= \{ ( V_{ij}, X'_{ij}) \}_{i,j}$.  This gives a system of local embeddings for $Y_1$.

 For the projection map $f_{ij}: X_{ij}' = Z_{ij} \times X_i \to X_i$, we have $f_{ij}|_{V_{ij}} = g_{ij} = g|_{V_{ij}}$. Thus $\mathcal{V}$ is associated to $\mathcal{U}$ by $g$, as desired.
\end{proof}
 
 \begin{lem}\label{lem:system transitive}
  Let $g_1: Y_1 \to Y_2$ and $g_2: Y_2 \to Y_3$ be morphisms in $\Sch_k$. Let $\mathcal{U}$ (resp. $\mathcal{V}, \mathcal{W}$) be a system of local embeddings for $Y_3$ (resp. $Y_2, Y_1$), such that $(\mathcal{U}, \mathcal{V})$ (resp. $(\mathcal{V}, \mathcal{W})$)  is a pair associated via $g_2$ (resp. $g_1$).

Then $(\mathcal{U}, \mathcal{W})$ is a pair associated via $g_2 \circ g_1$.
 \end{lem}
 
 \begin{proof}
 Easy exercise.
 \end{proof}
 
 The following is based on Theorem \ref{thm:pull-back moving}:
 
 \begin{thm}\label{thm:functoriality general}
 Let $g: Y_1 \to Y_2$ be a morphism in $\Sch_k$. 
 
 \begin{enumerate}
 \item Let $\mathcal{U}$ (resp. $\mathcal{V}$) be a system of local embeddings for $Y_2$ (resp. $ Y_1$), such that $(\mathcal{U}, \mathcal{V})$ is a pair associated via $g$.
 
 Then there exist associated morphisms in $\mathcal{D} ({\rm Ab} (Y_2))$:
 \begin{equation}\label{eqn:functoriality general 1}
 \tuborg 
g_{\mathcal{U}, \mathcal{V}}^*:  \BGHz^q (\mathcal{U}^{\infty}, \bullet) \to \mathbf{R}g_* \BGHz^q (\mathcal{V}^{\infty}, \bullet), \\
g_{\mathcal{U}, \mathcal{V}}^*:  \BGHz^q (\mathcal{U}, \bullet) \to \mathbf{R}g_* \BGHz^q (\mathcal{V}, \bullet).
\sluttuborg
\end{equation}
 
 \item The pull-backs of \eqref{eqn:functoriality general 1} are functorial in the following sense: let $g_1: Y_1 \to Y_2$ and $g_2: Y_2 \to Y_3$ be morphisms in $\Sch_k$. Let $\mathcal{U}$ (resp. $\mathcal{V}, \mathcal{W}$) be a system of local embeddings for $Y_3$ (resp. $Y_2, Y_1$), such that $(\mathcal{U}, \mathcal{V})$ (resp. $(\mathcal{V}, \mathcal{W})$)  is a pair associated via $g_2$ (resp. $g_1$).
 
 Then we have equalities of the pull-back morphisms in $\mathcal{D} ({\rm Ab} (Y_2))$:
 \begin{equation}\label{eqn:functoriality general 2}
\tuborg
(g_2 \circ g_1)_{\mathcal{U}, \mathcal{W}} ^* = (g_1 )_{\mathcal{V}, \mathcal{W}} ^* \circ  (g_2)_{\mathcal{U}, \mathcal{V}} ^* : \BGHz^q (\mathcal{U}^{\infty}, \bullet) \to \mathbf{R} (g_2 \circ g_1)_* \BGHz^q (\mathcal{W}^{\infty}, \bullet),\\
(g_2 \circ g_1)_{\mathcal{U}, \mathcal{W}} ^* = (g_1 )_{\mathcal{V}, \mathcal{W}} ^* \circ  (g_2)_{\mathcal{U}, \mathcal{V}} ^*: \BGHz^q (\mathcal{U}, \bullet) \to \mathbf{R} (g_2 \circ g_1)_* \BGHz^q (\mathcal{W}, \bullet),
\sluttuborg
\end{equation}
where the names of the above pull-backs follow the notations of \eqref{eqn:functoriality general 1}.

\item The morphism of \eqref{eqn:functoriality general 1} is functorial in the pair $(\mathcal{U}, \mathcal{V})$ in the following sense: for $g: Y_1 \to Y_2$, suppose $(\mathcal{U}, \mathcal{V})$ and $(\mathcal{U}', \mathcal{V}')$ are two pairs associated via $g$, such that $\mathcal{U}'$ (resp. $\mathcal{V}'$) is a refinement of $\mathcal{U}$ (resp. $\mathcal{V}$). 

Then we have the following commutative diagrams in $\mathcal{D} ({\rm Ab} (Y_2))$:
$$
\xymatrix{
\BGHz^q (\mathcal{U} ^{\infty}, \bullet) \ar[r] ^{ g_{\mathcal{U}, \mathcal{V}} ^* \ } \ar[d] ^{ ({\rm Id}_{Y_2})_{\mathcal{U}, \mathcal{U}'} ^*}  & \mathbf{R}g_* \BGHz^q (\mathcal{V}^{\infty}, \bullet) \ar[d] ^{ ({\rm Id}_{Y_1})_{\mathcal{V}, \mathcal{V}'}^*} \\
\BGHz^q ((\mathcal{U}')^{\infty}, \bullet) \ar[r] ^{g_{\mathcal{U}', \mathcal{V}'} ^* \ } &  \mathbf{R}g_*  \BGHz^q ((\mathcal{V}')^{\infty}, \bullet),
}
\hskip0.4cm
\xymatrix{
\BGHz^q (\mathcal{U}, \bullet) \ar[r] ^{ g_{\mathcal{U}, \mathcal{V}} ^* \ } \ar[d] ^{ ({\rm Id}_{Y_2})_{\mathcal{U}, \mathcal{U}'} ^*}  & \mathbf{R}g_* \BGHz^q (\mathcal{V}, \bullet) \ar[d] ^{ ({\rm Id}_{Y_1})_{\mathcal{V}, \mathcal{V}'}^*} \\
\BGHz^q (\mathcal{U}', \bullet) \ar[r] ^{g_{\mathcal{U}', \mathcal{V}'} ^* \ } &  \mathbf{R}g_*  \BGHz^q (\mathcal{V}', \bullet),
}
$$
where the names of the arrows follow the notational convention in \eqref{eqn:functoriality general 1}.
 \end{enumerate}
 \end{thm}

\begin{proof} 

(1) We are given a morphism $g: Y_1 \to Y_2$ in $\Sch_k$ and a pair $(\mathcal{U}, \mathcal{V})$ of systems associated via $g$. Write $\mathcal{U}= \{ (U_i, X_i ) \}_{i \in \Lambda}$ and $\mathcal{V}= \{  (V_j, X_j ') \}_{j \in \Lambda'}$. Let $\lambda: \Lambda' \to \Lambda$ be the set function that gives the association via $g$.

For an integer $p \geq 0$, suppose $I \in \Lambda^{p+1}$, $J \in (\Lambda')^{p+1}$ such that $(I, J)$ is associated via $g$, i.e. $I = \lambda (J)$. Then we have a commutative diagram
$$
\xymatrix{
X_J ' \ar[r] ^{f_{I,J}} & X_I \\
V_J \ar[r] ^{g_{I,J}} \ar@{^{(}->}[u] & U_I. \ar@{^{(}->}[u]}
$$
Let $\widehat{X}_I$ (resp. $\widehat{X}_J'$) be the completion of $X_I$ (resp. $X_J'$) along $U_I$ (resp. $V_J$).
By Theorem \ref{thm:pull-back moving}-(1), we have morphisms in $\mathcal{D}^- ({\rm Ab} (U_I))$:
\begin{equation}\label{eqn:functoriality general 1-2}
\tuborg
g_{I,J} ^*: \BGHz^q (\widehat{X}_I, \bullet) \to \mathbf{R} (g_{I,J})_* \BGHz^q (\widehat{X}_J', \bullet),\\
g_{I,J} ^*: \BGHz^q (\widehat{X}_I \mod U_I, \bullet) \to \mathbf{R} (g_{I,J})_* \BGHz^q (\widehat{X}_J' \mod V_J, \bullet).
\sluttuborg
\end{equation}

Taking the right derived push-forwards to $Y_2$ via the open immersions $U_I \hookrightarrow Y_2$, and collecting all of \eqref{eqn:functoriality general 1-2} over all pairs $(I, J)$ associated via $g$ over all $p \geq 0$ as in the previous \v{C}ech construction, we obtain morphisms in $\mathcal{D} ({\rm Ab} (Y_2))$:
\begin{equation}\label{eqn:functoriality general 1-3}
\tuborg
g_{\mathcal{U}, \mathcal{V}} ^*: \check{\mathfrak{C}}^q (\mathcal{U}^{\infty}) \to \mathbf{R} g_* \check{\mathfrak{C}} ^q (\mathcal{V}^{\infty}), \\
g_{\mathcal{U}, \mathcal{V}} ^*: \check{\mathfrak{C}}^q (\mathcal{U}) \to \mathbf{R}  g_* \check{\mathfrak{C}} ^q (\mathcal{V}).
\sluttuborg
\end{equation}
Taking the total complexes, we finally get the morphisms in \eqref{eqn:functoriality general 1} as desired.

(2), (3): Once we know the existence of the morphisms in (1), the arguments are straightforward, similar to the arguments of (2), (3) of Theorem \ref{thm:pull-back moving}. We omit details. 
 \end{proof}

 \begin{lem}\label{lem:refinement ass}
 Let $g: Y_1 \to Y_2$ be a morphism in $\Sch_k$. Let $\mathcal{U}=\{ (U_i, X_i) \}_{i \in \Lambda}$ be a given system of local embeddings for $Y_2$. Let $\mathcal{V} = \{ (V_j, X_j')\}_{j \in \Lambda'}$ be a system for $Y_1$. 
 
 Then there exists a system $\mathcal{V}'$ of local embeddings for $Y_1$, which is a refinement of $\mathcal{V}$, and associated to $\mathcal{U}$ via $g$.
 \end{lem}
 \begin{proof}
 
 We can first choose any system $\mathcal{W}$ of local embeddings for $Y_1$ that is associated to $\mathcal{U}$ via $g$, using Lemma \ref{lem:system ass functor}.
 
 If this is a refinement of $\mathcal{V}$, then with $\mathcal{V}' = \mathcal{W}$, we are done.
 
 If not, then we can find a common refinement $\mathcal{V}'$ of both $\mathcal{V}$ and $\mathcal{W}$ by Lemma \ref{lem:common refinement}. This answers the lemma.
 \end{proof}

Let $g: Y_1 \to Y_2$ be a morphism in $\Sch_k$. Let $(\mathcal{U}, \mathcal{V})$ be a pair of systems associated via $g$. If $\mathcal{V}'$ is a refinement of $\mathcal{V}$, one notes that $(\mathcal{U}, \mathcal{V}')$ is also a pair associated via $g$. Hence for the maps $g_{\mathcal{U}, \mathcal{V}}^*$, we may take the homotopy colimit 
$$g_{\mathcal{U}, \cdot}^*:= \underset{\mathcal{V}}{\hocolim} \ g_{\mathcal{U}, \mathcal{V}}^*.$$

On the other hand, for each refinement system $\mathcal{U}'$ of $\mathcal{U}$ given by a set map $\lambda: \Lambda' \to \Lambda$ of the index sets of the systems, by Lemma \ref{lem:refinement ass} and Theorem \ref{thm:functoriality general}, we have the commutative diagram in $\mathcal{D} ({\rm Ab} (Y_2))$
$$
\xymatrix{ 
\BGHz^q (\mathcal{U}^{\infty}, \bullet) \ar[r] ^{ g_{\mathcal{U}, \cdot} ^*} \ar[d] ^{\lambda^*} & \mathbf{R} g_* \BGHz^q (Y_1^{\infty}, \bullet) \\
\BGHz^q ({\mathcal{U}'} ^{\infty}, \bullet), \ar[ru] ^{ g_{\mathcal{U}', \cdot} ^*} & }
\ \ \ \ 
\xymatrix{ 
\BGHz^q (\mathcal{U}, \bullet) \ar[r] ^{ g_{\mathcal{U}, \cdot} ^*} \ar[d] ^{\lambda^*} & \mathbf{R} g_* \BGHz^q (Y_1, \bullet) \\
\BGHz^q (\mathcal{U}', \bullet). \ar[ru] ^{ g_{\mathcal{U}', \cdot} ^*} & }
$$
Thus we can take the homotopy colimits over $\mathcal{U}$ as well. This gives:

\begin{defn}\label{defn:general pull-back}
Let $g: Y_1 \to Y_2$ be a morphism in $\Sch_k$. 

We have the induced morphisms in $\mathcal{D} ({\rm Ab}(Y_2))$:
$$
\tuborg
g^*:= \underset{\mathcal{U}}{\hocolim} \  \underset{\mathcal{V}}{\hocolim}  \ g_{\mathcal{U}, \mathcal{V}} ^*: \BGHz^q (Y_2^{\infty}, \bullet) \to \mathbf{R} g_*  \BGHz^q (Y_1 ^{\infty}, \bullet), \\
g^*:= \underset{\mathcal{U}}{\hocolim}  \ \underset{\mathcal{V}}{\hocolim}  \ g_{\mathcal{U}, \mathcal{V}} ^*: \BGHz^q (Y_2, \bullet) \to \mathbf{R} g_*  \BGHz^q (Y_1 , \bullet).
\sluttuborg
$$
We will call them \emph{the pull-backs by $g$ on the yeni higher Chow sheaves}.
\qed
\end{defn}

\begin{cor}\label{cor:functoriality general}
 If $g_1: Y_1 \to Y_2$ and $g_2: Y_2 \to Y_3$ are morphisms, then we have the equalities of the morphisms in $\mathcal{D} ({\rm Ab} (Y_3))$:
$$
\tuborg
(g_2 \circ g_1)^* = g_1^* \circ g_2 ^*: \BGHz^q (Y_3 ^{\infty}, \bullet) \to \mathbf{R} (g_2\circ g_1)_* \BGHz^q (Y_1 ^{\infty}, \bullet), \\
(g_2 \circ g_1)^* = g_1^* \circ g_2 ^*: \BGHz^q (Y_3 , \bullet) \to \mathbf{R} (g_2\circ g_1)_* \BGHz^q (Y_1 , \bullet).
\sluttuborg
$$
\end{cor}

\subsection{The product structure}\label{sec:product}

One important application of the functoriality in Definition \ref{defn:general pull-back} and Corollary \ref{cor:functoriality general} is the product structure. We first discuss the concatenation product for a pair of closed immersions $Y_i \hookrightarrow X_i$ for $i=1,2$, consisting of quasi-affine $k$-schemes $Y_i$ of finite type into equidimensional smooth $k$-schemes $X_i$. Here, we implicitly use the fact that for topological rings $A$, the restricted formal power series have canonical isomorphisms (see EGA I \cite[Ch. 0, (7.5.2), p.70]{EGA1})
$$
(A \{ y_1, \cdots, y_r\}) \{ y_{r+1}, \cdots, y_s\} \simeq A \{ y_1, \cdots, y_s \},
$$
or for two topological $k$-algebras $A_1, A_2$, 
$$
(A_1 \{ y_1, \cdots, y_r \}) \widehat{\otimes}_k (A_2 \{ y_{r+1}, \cdots, y_s\}) \simeq (A_1 \widehat{\otimes}_k A_2) \{ y_1, \cdots, y_s\}.
$$

\begin{lem}\label{lem:concatenation}
For $i=1,2$, let $Y_i$ be quasi-affine $k$-scheme of finite type, and let $Y_i \hookrightarrow X_i$ be a closed immersion into an equidimensional smooth $k$-scheme. Let $\widehat{X}_i$ be the completion of $X_i$ along $Y_i$. Let $\widehat{X_1 \times X_2}$ be the completion of $X_1\times X_2$ along $Y_1 \times Y_2$. 
Then:

\begin{enumerate}
\item For any pairs $(q_1, n_1)$ and $(q_2, n_2)$ of nonnegative integers, there exist natural concatenation products
\begin{equation}\label{eqn:concat2}
\tuborg 
\boxtimes: z^{q_1} (\widehat{X}_1 , n_1) \otimes z^{q_2} (\widehat{X}_2, n_2) \to z^{q_1 + z_2} (\widehat{X_1 \times X_2} , n_1+ n_2),\\
\boxtimes: z^{q_1} (\widehat{X}_1 \mod Y_1, n_1) \otimes z^{q_2} (\widehat{X}_2 \mod Y_2, n_2)\\
 \to z^{q_1 + z_2} (\widehat{X_1 \times X_2} \mod Y_1 \times Y_2, n_1+ n_2).
 \sluttuborg
 \end{equation}

They are associative.

\item The boundary operators $\partial$ satisfy the Leibniz rule with respect to $\boxtimes$, i.e.
$$
\partial (\mathfrak{Z}_1 \boxtimes \mathfrak{Z}_2) = (\partial \mathfrak{Z}_1) \boxtimes \mathfrak{Z}_2 + (-1)^{n_1} \mathfrak{Z}_1 \boxtimes (\partial \mathfrak{Z}_2),
$$
 and they induce the concatenation homomorphisms
\begin{equation}\label{eqn:concat2-1}
\tuborg
\boxtimes: \CH^{q_1} (\widehat{X}_1, n_1)  \otimes _{\mathbb{Z}} \CH^{q_2} (\widehat{X}_2, n_2) \to \CH^{q_1+q_2} (\widehat{X_1 \times X_2}, n_1 + n_2), \\
\boxtimes: \CH^{q_1} (\widehat{X}_1 \mod Y_1, n_1)  \otimes _{\mathbb{Z}} \CH^{q_2} (\widehat{X}_2 \mod Y_2, n_2) \\
\hskip3cm \to \CH^{q_1+q_2} (\widehat{X_1 \times X_2} \mod Y_1 \times Y_2, n_1 + n_2).
\sluttuborg
\end{equation}
\end{enumerate}
\end{lem}

\begin{proof}
We implicitly use the identification
$$
 \tau_{n_1, n_2}: (\widehat{X}_1 \times_k \square_k ^{n_1}) \times_k (\widehat{X}_2 \times_k \square_k ^{n_2}) \simeq \widehat{X}_1 \times_k \widehat{X}_2 \times_k \square_k ^{n_1+ n_2},
 $$
given by $(x, t_1, \cdots, t_{n_1}) \times ( x', t_1', \cdots, t_{n_2} ') \mapsto (x, x', t_1, \cdots, t_{n_1}, t_1', \cdots, t_{n_2}').$ 

\medskip

(1) The concatenation homomorphism
\begin{equation}\label{eqn:concat1}
\boxtimes: z^{q_1} (\widehat{X}_1, n_1) \otimes z^{q_2} (\widehat{X}_2, n_2) \to z ^{q_1 + q_2} (\widehat{X}_1 \times \widehat{X}_2, n_1 + n_2)
\end{equation}
is given by sending the pair $(\mathfrak{Z}_1, \mathfrak{Z}_2)$ of integral cycles to the cycle associated to $\mathfrak{Z}_1\times_k \mathfrak{Z}_2$, i.e. the cycle associated to the sheaf $\mathcal{O}_{\mathfrak{Z}_1} \otimes_k \mathcal{O}_{\mathfrak{Z}_2}$, and extending $\mathbb{Z}$-bilinearly. Since $k$ is a field, we have $\otimes_k = \otimes_k ^{\mathbf{L}}$. One checks that $\mathfrak{Z}_1 \boxtimes \mathfrak{Z}_2$ satisfies the conditions (\textbf{GP}), (\textbf{SF}) of Definition \ref{defn:HCG} from the given conditions (\textbf{GP}), (\textbf{SF}) of $\mathfrak{Z}_1$ and $\mathfrak{Z}_2$. We omit details.

\medskip

 Note that the completion of $X_1 \times X_2$ along $Y_1 \times Y_2$ is equal to the fiber product $\widehat{X}_1 \times \widehat{X}_2$ (Lemma \ref{lem:prod completion}), so that $z^{q_1+ q_2} (\widehat{X}_1 \times \widehat{X}_2, n_1 + n_2) = z^{q_1 + q_2} (\widehat{X_1 \times X_2}, n_1 + n_2)$. Thus we have the first one in \eqref{eqn:concat2}.
 
 \medskip

To show that \eqref{eqn:concat1} descends to the second product in \eqref{eqn:concat2} modulo the mod equivalences,
one needs to check that both $\mathcal{M} ^{q_1} (\widehat{X}_1, Y_1, n_1)\boxtimes z^{q_2} (\widehat{X}_2, n_2)$ and $z^{q_1} (\widehat{X}_1, n_1) \boxtimes \mathcal{M} ^{q_2} (\widehat{X}_2, Y_2, n_2)$ belong to $\mathcal{M} ^{q_1 + q_2} (\widehat{X_1 \times X_2}, Y_1 \times Y_2, n_1 + n_2)$. We do it just for the first group, as the argument for the second group is identical by symmetry.

\medskip

Let $(\mathcal{A}_1, \mathcal{A}_2) \in \mathcal{L}^{q_1} (\widehat{X}_1, Y_1, n_1)$, so that we have an isomorphism 
\begin{equation}\label{eqn:concat3}
 \mathcal{A}_1 \otimes _{\mathcal{O}_{\square_{\widehat{X}_1} ^{n_1}}} ^{\mathbf{L}} \mathcal{O}_{\square_{Y_1} ^{n_1}} \simeq \mathcal{A}_2 \otimes _{\mathcal{O}_{\square_{\widehat{X}_1} ^{n_1}}} ^{\mathbf{L}} \mathcal{O}_{\square_{Y_1} ^{n_1}}
\end{equation}
 in $\sAlg (\mathcal{O}_{\square_{Y_1}^{n_1}})$. Since $\mathcal{M}^{q_1} (\widehat{X}_1, Y_1, n_1) \boxtimes z^{q_2} (\widehat{X}_2, n_2)$ is generated by cycles of the form $[\mathcal{A}_1] \boxtimes \mathfrak{Z}' - [ \mathcal{A}_2] \boxtimes \mathfrak{Z}'$ where $\mathfrak{Z}' \in z^{q_2} (\widehat{X}_2, n_2)$ is integral, it remains to prove that we have an isomorphism
\begin{equation}\label{eqn:concat4}
(\mathcal{A}_1 \otimes_k \mathcal{O}_{\mathfrak{Z}'}) \otimes_{ \mathcal{O}_{\square_{\widehat{X}_1\times \widehat{X}_2} ^{n_1+n_2}}} ^{\mathbf{L}} \mathcal{O}_{\square_{Y_1\times Y_2 } ^{n_1+n_2}}\simeq 
(\mathcal{A}_2 \otimes_k \mathcal{O}_{\mathfrak{Z}'}) \otimes_{ \mathcal{O}_{\square_{\widehat{X}_1\times \widehat{X}_2} ^{n_1+n_2}}} ^{\mathbf{L}} \mathcal{O}_{\square_{Y_1\times Y_2 } ^{n_1+n_2}}
\end{equation}
in $\sAlg (\mathcal{O}_{\square_{Y_1\times Y_2} ^{n_1+n_2}}).$

Applying $(-)  \otimes_k (\mathcal{O}_{\mathfrak{Z}'} \otimes_{\mathcal{O}_{\square_{\widehat{X}_2} ^{n_2}}} ^{\mathbf{L}} \mathcal{O}_{\square_{Y_2}^{n_2}})$ to \eqref{eqn:concat3}, we have an isomorphism
\begin{equation}\label{eqn:concat5}
( \mathcal{A}_1 \otimes _{\mathcal{O}_{\square_{\widehat{X}_1} ^{n_1}}} ^{\mathbf{L}} \mathcal{O}_{\square_{Y_1} ^{n_1}} ) \otimes_k (\mathcal{O}_{\mathfrak{Z}'} \otimes_{\mathcal{O}_{\square_{\widehat{X}_2} ^{n_2}}} ^{\mathbf{L}} \mathcal{O}_{\square_{Y_2}^{n_2}})
\end{equation}
$$\simeq ( \mathcal{A}_2 \otimes _{\mathcal{O}_{\square_{\widehat{X}_1} ^{n_1}}} ^{\mathbf{L}} \mathcal{O}_{\square_{Y_1} ^{n_1}} ) \otimes_k (\mathcal{O}_{\mathfrak{Z}'} \otimes_{\mathcal{O}_{\square_{\widehat{X}_2} ^{n_2}}} ^{\mathbf{L}} \mathcal{O}_{\square_{Y_2}^{n_2}})
$$
in $\sAlg (\mathcal{O}_{\square_{Y_1 \times Y_2} ^{n_1+ n_2}})$. Here, $\mathcal{O}_{\square_{\widehat{X}_1\times \widehat{X}_2} ^{n_1+n_2}} \simeq \mathcal{O}_{\square_{\widehat{X}_1} ^{n_1}} \otimes_k \mathcal{O}_{\square_{\widehat{X}_2} ^{n_2}}$ and $\mathcal{O}_{\square_{Y_1 \times Y_2} ^{n_1+n_2}} \simeq \mathcal{O}_{\square_{Y_1} ^{n_1}} \otimes_k \mathcal{O}_{\square_{Y_2} ^{n_2}}$ so that for $i=1,2$, both sides of \eqref{eqn:concat5} satisfy
\begin{equation}\label{eqn:concat6}
(\mathcal{A}_i \otimes_{\mathcal{O}_{\square_{\widehat{X}_1}^{n_1}}} ^{\mathbf{L}} \mathcal{O}_{\square_{Y_1} ^{n_1}} )\otimes_k (\mathcal{O}_{\mathfrak{Z}'} \otimes_{\mathcal{O}_{\square_{\widehat{X}_2} ^{n_2}}} ^{\mathbf{L}} \mathcal{O}_{\square_{Y_2}^{n_2}}).
\end{equation}
$$
\simeq (\mathcal{A}_i \otimes_k \mathcal{O}_{\mathfrak{Z}'}) \otimes_{ \mathcal{O}_{\square_{\widehat{X}_1\times \widehat{X}_2} ^{n_1+n_2}}} ^{\mathbf{L}} \mathcal{O}_{\square_{Y_1\times Y_2 } ^{n_1+n_2}}.
$$
Thus by \eqref{eqn:concat6} the isomorphism \eqref{eqn:concat5} implies the isomorphism \eqref{eqn:concat4}. 

That $\boxtimes$ is associative follows from the associativity of the tensor products. This proves the part (1) of the lemma. 

\medskip

(2) For the Leibniz rule, let $\mathfrak{Z}_j \in z^{q_j} (\widehat{X}_j , n_j)$ be cycles for $j=1,2$. Note that by definition
$$
\partial_i ^{\epsilon} (\mathfrak{Z}_1 \boxtimes \mathfrak{Z}_2) 
=\tuborg (\partial_i ^{\epsilon} \mathfrak{Z}_1) \boxtimes \mathfrak{Z}_2, & \mbox{ for } 1 \leq i \leq n_1, \\
\mathfrak{Z}_1 \boxtimes (\partial_{i- n_1} ^{\epsilon} \mathfrak{Z}_2), & \mbox{ for } n_1 + 1 \leq i \leq n_1 + n_2.\sluttuborg
$$

Hence 
\begin{eqnarray*}
& & \partial (\mathfrak{Z}_1 \boxtimes \mathfrak{Z}_2) = \sum_{i=1} ^{n_1+n_2} (-1)^i (\partial_i ^{\infty} - \partial_i ^0) (\mathfrak{Z}_1 \boxtimes \mathfrak{Z}_2)\\
&=& \sum_{i=1} ^{n_1} (-1)^i (\partial_i ^{\infty} - \partial_i ^0) (\mathfrak{Z}_1 \boxtimes \mathfrak{Z}_2) + \sum_{i= n_1 +1} ^{n_1 + n_2} (-1)^i (\partial_i ^{\infty} - \partial_i ^0) (\mathfrak{Z}_1 \boxtimes \mathfrak{Z}_2)\\
&=& \left( \sum_{i=1} ^{n_1} (-1)^i ( \partial_i ^{\infty} - \partial_i ^0) \mathfrak{Z}_1 \right) \boxtimes \mathfrak{Z}_2 + (-1)^{n_1} \mathfrak{Z}_1 \boxtimes \left( \sum_{i=1}^{n_2} (-1)^i (\partial_i ^{\infty} - \partial_i ^0) \mathfrak{Z}_2 \right) \\
&=& (\partial  \mathfrak{Z}_1) \boxtimes \mathfrak{Z}_2 + (-1)^{n_1} \mathfrak{Z}_1 \boxtimes (\partial \mathfrak{Z}_2)
\end{eqnarray*}
as desired.

\medskip

From this, we deduce the maps in \eqref{eqn:concat2-1}. Indeed, let $\alpha_j \in \CH^{q_j} (\widehat{X}_i, n_j)$ be cycle classes for $j=1,2$. Let $\mathfrak{Z}_j \in z^{q_j} (\widehat{X}_j , n_j)$ be a cycle such that $\partial \mathfrak{Z}_j = 0$, which represents the cycle class $\alpha_j$. By the Leibniz rule, 
$$
\partial (\mathfrak{Z}_1 \boxtimes \mathfrak{Z}_2) = (\partial \mathfrak{Z}_1) \boxtimes \mathfrak{Z}_2 + (-1)^{n_1} \mathfrak{Z}_1 \boxtimes (\partial \mathfrak{Z}_2) = 0 + 0 = 0,
$$
 and in particular $\mathfrak{Z}_1 \boxtimes \mathfrak{Z}_2$ represents a class in $\CH^{q_1 + q_2 } (\widehat{X_1 \times X_2}, n_1 + n_2).$

On the other hand, for $\mathfrak{Z}_j '\in z^{q_j} (\widehat{X}_j, n_j+1)$, by the Leibniz rule we have
$$
 (\partial \mathfrak{Z}_1') \boxtimes \mathfrak{Z}_2 = \partial (\mathfrak{Z}_1' \boxtimes \mathfrak{Z}_2), \ \ \mathfrak{Z}_1 \boxtimes (\partial \mathfrak{Z}_2') = (-1)^{n_1} \partial (\mathfrak{Z}_1 \boxtimes \mathfrak{Z}_2').
 $$
 Thus $\boxtimes$ sends boundaries to boundaries. Hence we have the first morphism of \eqref{eqn:concat2-1}. The proof for the second one is identical.
\end{proof}

\medskip

Using the above discussions in Lemma \ref{lem:concatenation} as a basis, we now discuss the construction of the product structure on the groups for each $Y \in \Sch_k$. We first claim: 
\begin{lem}\label{lem:general concat}
Let $Y \in \Sch_k$. Then there are the concatenation products
\begin{equation}\label{eqn:concat sch}
\tuborg
\boxtimes: \BGH^{q_1} (Y^{\infty}, n_1) \otimes_{\mathbb{Z}} \BGH^{q_2} (Y^{\infty},  n_2) \to \BGH^{q_1 + q_2}((Y \times Y,)^{\infty},  n_1 + n_2),\\
\boxtimes: \BGH^{q_1} (Y, n_1) \otimes_{\mathbb{Z}} \BGH^{q_2} (Y,  n_2) \to \BGH^{q_1 + q_2}(Y \times Y, n_1 + n_2).
\sluttuborg
\end{equation}
\end{lem}

\begin{proof}
Let $\mathcal{U} =\{ (U_i, X_i) \}_{i \in \Lambda}$ be a system of local embeddings of $Y$. Here one checks readily that the set $\mathcal{U}^{\times 2}:= \{ ( U_i \times U_j, X_i \times X_j )\}_{(i,j) \in \Lambda^2}$ gives a system of local embeddings for $Y \times Y$.

\medskip

For each integer $p \geq 0$ and for each pair of multi-indices $I, J \in \Lambda^{p+1}$, by Lemma \ref{lem:concatenation} there is the concatenation products
$$
\tuborg
\boxtimes_{I,J}: \BGHz^{q_1} (\widehat{X}_I , n_1) \otimes_{\mathbb{Z}} \BGHz^{q_2} (\widehat{X}_J , n_2) \to \BGHz^{q_1 + q_2} (\widehat{X_I\times X_J} , n_1 + n_2),\\
\boxtimes_{I,J}: \BGHz^{q_1} (\widehat{X}_I \mod U_I, n_1) \otimes_{\mathbb{Z}} \BGHz^{q_2} (\widehat{X}_J \mod U_J, n_2) \\
\hskip3cm \to \BGHz^{q_1 + q_2} (\widehat{X_I \times X_J} \mod U_I \times U_J, n_1 + n_2).
\sluttuborg
$$
Collecting them over all pairs $(I,J)$ and over all $p \geq 0$, we deduce the concatenation products
\begin{equation}\label{eqn:general concat 1}
\tuborg
\boxtimes:  \check{\mathfrak C} ^{q_1} (\mathcal{U}^{\infty}, n_1) \otimes_{\mathbb{Z}} \check{\mathfrak C}^{q_2} (\mathcal{U}^{\infty}, n_2) \to \check{\mathfrak C} ^{q_1+ q_2} ((\mathcal{U}^{\times 2})^{\infty}, n_1 + n_2),\\
\boxtimes:  \check{\mathfrak C} ^{q_1} (\mathcal{U}, n_1) \otimes_{\mathbb{Z}} \check{\mathfrak C}^{q_2} (\mathcal{U}, n_2) \to \check{\mathfrak C} ^{q_1+ q_2} (\mathcal{U}^{\times 2}, n_1 + n_2).
\sluttuborg
\end{equation}
One checks that $\partial$ and $\boxtimes$ satisfy the Leibniz rule as in Lemma \ref{lem:concatenation}. We deduce
$$
\tuborg
\boxtimes: \BGH^{q_1} (\mathcal{U}^{\infty}, n_1) \otimes_{\mathbb{Z}} \BGH^{q_2} (\mathcal{U}^{\infty}, n_2) \to \BGH^{q_1 + q_2} ( (\mathcal{U} \times \mathcal{U})^{\infty}, n_1+ n_2), \\
\boxtimes: \BGH^{q_1} (\mathcal{U}, n_1) \otimes_{\mathbb{Z}} \BGH^{q_2} (\mathcal{U}, n_2) \to \BGH^{q_1 + q_2} ( \mathcal{U} \times \mathcal{U}, n_1+ n_2).
\sluttuborg
$$

\medskip

Taking the colimit over the systems $\mathcal{U}$ and also over all systems of local embeddings for $Y \times Y$, they give the concatenation products of \eqref{eqn:concat sch}.
\end{proof}

We deduce the following cup product structures:

\begin{thm}\label{thm:general product}
Let $Y\in \Sch_k$. Then we have the cup product maps
$$
\tuborg
 \cup: \BGH^{q_1} (Y^{\infty}, n_1) \otimes_{\mathbb{Z}} \BGH^{q_2} (Y^{\infty}, n_2) \to \BGH ^{q_1 + q_2} (Y^{\infty}, n_1 + n_2), \\
\cup: \BGH^{q_1} (Y, n_1) \otimes_{\mathbb{Z}} \BGH^{q_2} (Y, n_2) \to \BGH ^{q_1 + q_2} (Y, n_1 + n_2),
\sluttuborg
$$
so that
\begin{equation}\label{eqn:bigraded rings}
\bigoplus_{n, q \geq 0} \BGH^q (Y^{\infty}, n), \ \ \ \bigoplus_{n, q \geq 0} \BGH^q (Y, n),
\end{equation}
are bi-graded rings, that are graded-commutative in $n$.

Furthermore, the product structures are functorial in $Y$.
\end{thm}

\begin{proof}
Note that for each $Y \in \Sch_k$, there exists the diagonal morphism $\Delta_Y: Y \to Y \times Y$. When $Y$ is not separated over $k$, this $\Delta_Y$ may not be a closed immersion, but this is not a problem because the functoriality (Definition \ref{defn:general pull-back} and Corollary \ref{cor:functoriality general}) holds for any morphism in $\Sch_k$. Hence by the functoriality, one deduces the pull-back morphisms
\begin{equation}\label{eqn:Cech4-1}
\tuborg
\Delta_Y ^*: \BGH^q ((Y \times Y)^{\infty}, n) \to \BGH^q (Y^{\infty}, n), \\
\Delta_Y ^*: \BGH^q (Y \times Y, n) \to \BGH^q (Y, n).
\sluttuborg
\end{equation}

\medskip

The cup product structure we seek is simply given by composing \eqref{eqn:concat sch} of Lemma \ref{lem:general concat} with \eqref{eqn:Cech4-1}, namely $\cup:=\Delta_Y ^* \circ \boxtimes$. Thus the direct sums in \eqref{eqn:bigraded rings} are bi-graded rings. (The graded-commutativity in $n$ is proven in Theorem \ref{thm:normal permutation} in the Appendix, \S \ref{sec:permutation}.)

The functoriality of the bi-graded ring structures follows from the naturality of the pull-backs in Theorem \ref{thm:pull-back moving}-(3), and repeating the entire constructions. We omit details.
\end{proof}

\subsection{The relative yeni higher Chow groups}\label{sec:relative theory}

Recall that the original higher Chow groups of S. Bloch \cite{Bloch HC} had limited ways to define their relative groups, see, e.g. \cite[p.293]{Bloch HC}. For a relative pair $(Y, Y')$ for a closed immersion $Y' \subset Y$, with $Y$ smooth, under the moving lemma, see e.g. S. Landsburg \cite[Definition 4.1]{Landsburg Ill}. For some $0$-cycles on singular varieties, there is also a version by Levine-Weibel \cite{LW}. 

In contrast, the yeni higher Chow groups naturally have the following kinds of relative groups for all morphisms in $\Sch_k$ and for all dimensions:

 \begin{defn}
 Let $g: Y_1 \to Y_2$ be a morphism in $\Sch_k$. For the pull-backs
 $$
 \tuborg
 g^*: \BGHz^q (Y_2^{\infty}, \bullet) \to \mathbf{R} g_* \BGHz^q (Y_1 ^{\infty}, \bullet), \\
  g^*: \BGHz^q (Y_2^{\infty}, \bullet) \to \mathbf{R} g_* \BGHz^q (Y_1 ^{\infty}, \bullet),
  \sluttuborg
 $$
 we define the \emph{relative yeni higher Chow sheaves} for $g^{\infty}$ and $g$ to be the homotopy fibers in $\mathcal{D} ({\rm Ab} (Y_2))$:
$$
\tuborg
\BGHz^q (g^{\infty}, \bullet):= \hofib (g^*: \BGHz^q (Y_2^{\infty}, \bullet) \to \mathbf{R} g_* \BGHz^q (Y_1 ^{\infty}, \bullet)), \\
\BGHz^q (g, \bullet):= \hofib ( g^*: \BGHz^q (Y_2^{\infty}, \bullet) \to \mathbf{R} g_* \BGHz^q (Y_1 ^{\infty}, \bullet)),
\sluttuborg
$$
respectively. Define
$$
\tuborg
\BGH^q (g^{\infty}, n) := \mathbb{H}_{\rm Zar} ^{-n} (Y_2, \BGHz^q (g^{\infty}, \bullet)),\\
\BGH^q (g, n) := \mathbb{H}_{\rm Zar} ^{-n} (Y_2, \BGHz^q (g, \bullet)).
\sluttuborg
$$
Call them the relative yeni higher Chow groups of $g^{\infty}$ and $g$, respectively.\qed
\end{defn}

We do not study $\BGH^q (g, n)$ in detail, though we leave the following question: 

\begin{ques}{\rm Let $g: D \hookrightarrow Y$ be an effective Cartier divisor on a $k$-scheme of finite type. How is our relative yeni higher Chow group $\BGH^q (g, n)$ related to the higher Chow group $\CH^q (Y|D, n)$ with modulus of Binda-Saito \cite{BS}? \qed }
\end{ques}

\begin{exm}\label{exm:splitcase}
Let $A$ be $k$-algebra of finite type. For an ideal $I \subset A$, suppose that the surjection $A \twoheadrightarrow A/I$ has a splitting homomorphism $A/I \to A$. Then by the functoriality, we have the direct sum decomposition
$$
\BGH^q (A, n) \simeq \BGH^q (A/I, n) \oplus \BGH^q (A \to A/I, n),
$$ 
where $\BGH^q (A, n):= \BGH^q (\Spec (A), n)$, etc.
In particular, $\BGH^q (A \to A/I, n) = \ker ( \BGH^q (A, n) \to \BGH^q (A/I, n))$. 

This applies to the special case when $A= k[[t]]/(t^m)$ and $I= (t)$, where the inclusion $k= A/I \hookrightarrow A$ gives a splitting. 
\qed
\end{exm}

\subsection{Local deformation functor}\label{sec:local deform}

The yeni higher Chow groups of $Y$ may be useful to local deformation theory of the motivic cohomology. This is not the main focus of this article, so we minimize our discussion. The groups for $Y^{\infty}$ are not interesting for non-reduced schemes (see Remark \ref{remk:no orta}), though.

An Artin local $k$-algebra $A$ is finite over $k$, thus it is of finite type over $k$. In particular, we can write $A= k[t_1, \cdots, t_r]/ I$ for an ideal $I$. It can be rephrased as:

\begin{cor}\label{cor:embedding}
Let $A$ be an Artin local $k$-algebra and let $Y= \Spec (A)$. Then there exists a closed immersion $Y \hookrightarrow X$ into the smooth affine $k$-space $X= \mathbb{A}_k ^r$. 
\end{cor}


Recall:

\begin{defn}
A covariant functor $({\rm Art}_k) \to ({\rm Set})$ from the category of Artin local $k$-algebras is called a local deformation functor. \qed
\end{defn}

It was called \emph{a functor of Artin rings} in M. Schlessinger \cite{Schlessinger}.

The association $A \mapsto \Spec (A)$ for $A \in ({\rm Art}_k)^{\rm op}$ allows us to regard $({\rm Art}_k)^{\rm op}$ as a subcategory of the category $\Sch_k$. Restricting the contravariant functor $\BGH^q (-, n)$ on $\Sch_k$ to the subcategory $({\rm Art}_k)^{\rm op}$, we deduce:

\begin{cor}\label{cor:functoriality local}
The association $A \mapsto \BGH^q (A, n)$ defines a covariant functor $\BGH^q (-, n): ({\rm Art}_k) \to ({\rm Ab})$, i.e. it gives a local deformation functor.
\end{cor}

We saw in \cite{Park Tate} and \S \ref{sec:nilpotence} that $K_n ^M (k_m) \simeq \BGH^n (k_m, n)$, so the yeni higher Chow groups are non-constant in the subcollection.
Further studies in this direction may be pursued separately.

\begin{remk}\label{remk:no orta}
The assignment, $A \to \BGH^q (\Spec (A)^{\infty}, n)$, also defines a functor $({\rm Art}_k) \to ({\rm Ab})$, but this is not very useful from the deformation perspective. For instance, all objects of the subcollection $\{ k[t]/(t^m) \}_{m \geq 2}$ have $\widehat{X} = \Spf (k[[t]])$ as a common exoskin, and the groups do not distinguish the objects. 
\qed
\end{remk}

\section{The first Chern class}\label{sec:Chern}

From now on until the end of the article, we work with the yeni higher Chow groups on the subcategory $\Sep_k$ of separated $k$-schemes of finite type. 
 
 \medskip

In \S \ref{sec:Chern}, the goal is to construct the first Chern class map to the yeni Chow group
\begin{equation}\label{eqn:1st Chern BGH0}
c_1 : \Pic (Y) \to \BGH^1 (Y,0)
\end{equation}
for $Y \in \Sep_k$, and we show that it is functorial in $Y$.

This is done in several steps. Roughly the story goes as follows. We will first concentrate on the case when $Y$ is connected and affine. 
We choose a closed immersion $Y \hookrightarrow X$ into a smooth $k$-scheme and consider the exoskin $\widehat{X}$. We can use the ideal class group description (\S \ref{sec:fractional ideal}) of ${\rm Pic} (\widehat{X})$ to define (\S \ref{sec:c1 exo}) a map $c_1: {\rm Pic} (\widehat{X}) \to \CH^1 (\widehat{X}, 0)$. 

On the other hand, ${\rm Pic} (Y)$ is expressed as a quotient of ${\rm Pic} (\widehat{X})$ via the Milnor patching (\S \ref{sec:Pic affine}), 
$$
\Pic (Y) = \Pic (\widehat{X}) / \tau (\Pic (D_{\widehat{X}})),
$$
where $D_{\widehat{X}} = \widehat{X} \coprod_Y \widehat{X}$, and we show that (\S \ref{sec:c1 affine}) the relation of the Milnor patching is compatible with the mod $Y$-equivalence on the cycle group. This gives $c_1$ on $\Aff_k$.

In \S \ref{sec:c1 general}, we globalize it to all $Y \in \Sep_k$ using the \v{C}ech machine of \S \ref{sec:cycle Cech BGH}, based on the construction in the affine case in \S \ref{sec:c1 affine}.

\subsection{Invertible fractional ideals}\label{sec:fractional ideal}

For affine schemes, it is well-known that the invertible fractional ideals and ideal class groups of the relevant rings provide convenient means to describe the Picard groups. There are several general references. We mention H. Matsumura \cite[Ch. 11, p.80]{Matsumura}, for instance.

\medskip

Let $A$ be an integral domain and let $K$ be its field of fractions. Recall that a \emph{fractional ideal of $A$} is a nonzero $A$-submodule $I \subset K$ such that there is a nonzero element $r \in A$ for which $r I \subset A$. For a fractional ideal $I$, we let $I^{-1}:= \{ r \in K \ | \ r I \subset A \}$. We say $I$ is a \emph{invertible} if $I I ^{-1} = A$. 

We suppose $A$ is noetherian so that $I$ is automatically a finitely generated $A$-module.

A fractional ideal that is generated by a single nonzero element of $K$ is called a \emph{principal fractional ideal} of $A$. A principal fractional ideal is invertible. The set of invertible fractional ideals form an abelian group with respect to the product of the fractional ideals, and the principal fractional ideals form a subgroup. The group
$$
{\rm Cl} (A):= \frac{ \{ \mbox{The invertible fractional ideals of } A\} }{ \{ \mbox{The principal fractional ideals of } A\}}
$$
is called the \emph{ideal class group} of $A$. In particular, in the group ${\rm Cl}(A)$ every invertible fractional ideal is equivalent to an (integral) ideal of $A$ that is invertible.

Recall a fractional ideal $I$ is invertible if and only if $I$ is a projective $A$-module if and only if for each maximal ideal $P \subset A$, the fractional ideal $I_P= I A_P \subset A_P$ is principal (\cite[Theorem 11.3, p.80]{Matsumura}). 
Furthermore, each class of the group ${\rm Cl} (A)$ has a representative given by a nonzero (integral) ideal $I \subset A$, that is a projective $A$-module.

\begin{prop}\label{prop:various class group}
Let $\mathfrak{X}= \Spf (A)$ be an integral regular noetherian affine formal scheme for a regular integral domain $A$.

Then we have an isomorphism.
$$
\Pic (\mathfrak{X}) \simeq  {\rm Cl} (A).
$$
\end{prop}

\begin{proof}
Since $\mathfrak{X}= \Spf (A)$ is affine, we have ${\rm Pic} (\Spf (A))= {\rm Pic} (\Spec (A))$ by EGA I \cite[Proposition (10.10.2)-(iii), p.201]{EGA1}.

On the other hand, under the integrality and regularity of $A$, it is a classical result that ${\rm Pic} (\Spec (A)) = {\rm Cl} (A)$. Thus we have ${\rm Pic} (\mathfrak{X}) \simeq {\rm Cl} (A)$ as desired.
\end{proof}

\subsection{$c_1$ over the exoskins $\widehat{X}$}\label{sec:c1 exo}

Let $Y$ be a connected affine $k$-scheme of finite type. Choose a closed immersion $Y \hookrightarrow X$ into an equidimensional smooth $k$-scheme, and take the completion $\widehat{X}$ of $X$ along $Y$. The formal scheme $\widehat{X}$ is given as $\Spf (A)$ for some regular noetherian integral $k$-domain $A$ of finite Krull dimension, and ${\rm Pic} (\widehat{X}) = {\rm Cl} (A)$ by Proposition \ref{prop:various class group}.

Let $J \subset A$ be the ideal of $A$ that defines $Y$ as a closed subscheme in $\widehat{X}$. It also defines the closed immersion $Y \subset \Spec (A)$ as well. Here, $A$ is $J$-adically complete.

Let $\mathcal{V} \in {\rm Cl} (A)$. It is represented by an invertible fractional ideal $I \subset K= {\rm Frac} (A)$. Since there is a nonzero $r \in K$ with $r I \subset A$, after replacing $I$ by $rI$ in the equivalence class in ${\rm Cl} (A)$ we may assume $I \subset A$ is a projective ideal. Furthermore, we can find $I \subset A$ so that it intersects properly with $J$. We now define:

\begin{defn}\label{defn:c1 exo}
Let $\widehat{X}$ be as the above. Let $\mathcal{V} \in {\rm Pic} (\widehat{X})= {\rm Cl} (A)$. Choose its representative invertible integral projective ideal $I \subset A$ such that $I$ intersects $J$ properly. Define
$$c_1: {\rm Pic} (\widehat{X}) \to \CH^1 (\widehat{X}, 0)$$
by sending the equivalence class of $\mathcal{V}$ to the associated cycle $[\widetilde{A/I}] \in z^1 (\widehat{X}, 0)$ given by the coherent $\mathcal{O}_{\widehat{X}}$-module $\widetilde{A/I}$. For notational convenience, in what follows we may just write $[A/I]$ for $[\widetilde{A/I}]$.
\qed
\end{defn}

One checks that $c_1$ is a homomorphism of groups through the classical identification of ${\rm Cl} (A)$ with the Picard group and resort to the standard fact, e.g. \cite[Proposition 2.5-(e), p.41]{Fulton}.

\medskip

The map $c_1$ of Definition \ref{defn:c1 exo} has the following compatibility:

\begin{lem}\label{lem:c1 exo commute}
Let $Y_1$ and $Y_2$ be connected affine $k$-schemes of finite. Suppose we have a commutative diagram
$$
\xymatrix{
Y_1 \ar[d] ^g  \ar@{^{(}->}[r] & X_1 \ar[d] ^f \\
Y_2  \ar@{^{(}->}[r] & X_2,}
$$
where the horizontal maps are closed immersions into equidimensional smooth $k$-schemes. Let $\widehat{X}_i$ be the completion of $X_i$ along $Y_i$ for $i=1,2$, and let $\widehat{f}: \widehat{X}_1 \to \widehat{X}_2$ be the induced morphism of formal schemes.

Then the following diagram
\begin{equation}\label{eqn:c1 exo comm 0}
\xymatrix{
\Pic (\widehat{X}_2) \ar[d] ^{\widehat{f}^*} \ar[rr] ^{c_1} & & \CH^1 (\widehat{X}_2, 0) \ar[d] ^{\widehat{f}^*} \\
\Pic (\widehat{X}_1) \ar[rr] ^{c_1} & & \CH^1 (\widehat{X}_1, 0)}
\end{equation}
commutes, where $\widehat{f}^*$ is the associated pull-back given in Theorem \ref{thm:pull-back moving}.
\end{lem}

\begin{proof}
When $\widehat{X}_i = \Spf (A_i)$, we have ${\rm Pic} (\widehat{X}_i) = {\rm Pic} (\Spec (A_i))$, while the cycles on the right hand side were defined using integral closed subschemes of $\Spec (A_i)$ as well. The commutativity thus follows from the case of schemes, where this is known.
\end{proof}

\subsection{$\Pic (Y)$ via the Milnor patching}\label{sec:Pic affine}

We want to use $c_1: \Pic (\widehat{X}) \to \CH^1 (\widehat{X}, 0)$ of Definition \ref{defn:c1 exo} to construct $c_1: \Pic (Y) \to \BGH^1 (Y, 0)$ eventually. To facilitate this, in \S \ref{sec:Pic affine} we study a relationship between $\Pic (Y)$ and $\Pic (\widehat{X})$. 

\medskip

The following is a Milnor patching deduced from Theorem \ref{thm:Milnor K} in \S \ref{sec:Milnor} (or J. Milnor \cite{Milnor K}), with EGA I \cite[Proposition (10.10.2), p.201]{EGA1}. Note that $D_{\widehat{X}}$ exists as a formal scheme by Theorem \ref{thm:Landsburg}-(1).

\begin{lem}\label{lem:1st Chern Milnor patch}
Let $Y$ be a connected affine $k$-scheme of finite type. Let $Y \hookrightarrow X$ be a closed immersion into an equidimensional smooth $k$-scheme, and let $\widehat{X}$ be the completion of $X$ along $Y$. Consider the push-out diagram of closed immersions
\begin{equation}\label{eqn:1st Chern Milnor push}
\xymatrix{
Y \ar[d] ^{j_1}\ar[r] ^{j_2} & \widehat{X} \ar[d] ^{ \iota_1}\\
\widehat{X} \ar[r] ^{\iota_2} & D_{\widehat{X}},}
\end{equation}
where $D_{\widehat{X}} = \widehat{X} \coprod_Y \widehat{X}$. Then for each pair $(\mathcal{L}_1, \mathcal{L}_2)$ of locally free sheaves of rank $1$ on $\widehat{X}$ such that $j_1 ^* \mathcal{L}_1 \simeq j_2 ^* \mathcal{L}_2$, there exists a locally free sheaf $\tilde{\mathcal{L}}$ of rank $1$ on $D_{\widehat{X}}$ such that $\iota_i^* \tilde{\mathcal{L}} \simeq \mathcal{L}_i$ for $i=1,2$. Furthermore, each locally free sheaf of rank $1$ on $D_{\widehat{X}}$ is obtained in this way.
\end{lem}

\begin{lem}\label{lem:Pic Milnor}
In the situation of Lemma \ref{lem:1st Chern Milnor patch}, we have an exact sequence
\begin{equation}\label{eqn:Pic Milnor}
 \Pic (D_{\widehat{X}}) \overset{ (\iota_1 ^*, \iota_2 ^*)}{\to} \Pic (\widehat{X}) \oplus \Pic (\widehat{X}) \overset{ j_1 ^* -j_2 ^*}{\to} \Pic (Y) \to 0,
\end{equation}
where $j_1^*- j_2 ^*$ is written additively. 
\end{lem}

\begin{proof}
By Milnor's theorem (see H. Bass \cite[Theorem (5.3), pp.481--482]{Bass}, cf. J. Milnor \cite{Milnor K}), we have an exact sequence
\begin{equation}\label{eqn:Pic Milnor 0}
\Pic (D_{\widehat{X}}) \overset{ (\iota_1 ^*, \iota_2 ^*)}{\to} \Pic (\widehat{X}) \oplus \Pic (\widehat{X}) \overset{ j_1 ^* -j_2 ^*}{\to} \Pic (Y).
\end{equation}

We prove the surjectivity of the last map. Let $\mathcal{I}$ be the ideal of $Y$ in $\widehat{X}$, and let $\widehat{X}_m$ be the scheme given by the ideal $\mathcal{I}^m$ for $m \geq 1$. Since $Y$ is affine, by R. Hartshorne \cite[Exercise II-9.6-(d), p.200]{Hartshorne}, the system $\{ \Gamma (\widehat{X}_m, \mathcal{O}_{\widehat{X}_m}) \}_{m \geq 1}$ satisfies the Mittag-Leffler condition. Now by \cite[Exercise II-9.6-(c), p.200]{Hartshorne}, for each $m \geq 1$, each locally free sheaf of rank $1$ on $\widehat{X}_m$ lifts to a locally free sheaf of rank $1$ on $\widehat{X}_{m+1}$. Thus we deduce that the map
$$
\Pic (\widehat{X}) \to \Pic (\widehat{X}_1) = \Pic (Y)
$$
is surjective. Thus, the last map of \eqref{eqn:Pic Milnor 0} is surjective.
\end{proof}

\begin{prop}\label{prop:Milnor line}
Under the above, let $\tau: \Pic (D_{\widehat{X}}) \to \Pic (\widehat{X})$ be the map $\iota_1 ^* - \iota_2 ^*$, written additively.

Then we have an isomorphism ${\rm coker} (\tau) = \Pic (\widehat{X}) / \tau (\Pic (D_{\widehat{X}})) \overset{\simeq}{\to} \Pic (Y)$.  
\end{prop}

\begin{proof}
Since $D_{\widehat{X}}$ is the push-out $\widehat{X} \coprod_{Y} \widehat{X}$, it induces a unique morphism $g: D_{\widehat{X}} \to \widehat{X}$ that extends the two copies of the identity maps $\widehat{X} \to \widehat{X}$. This induces $\Delta:= g^*: \Pic (\widehat{X}) \to \Pic (D_{\widehat{X}})$. Combined with Lemma \ref{lem:Pic Milnor}, they form the commutative diagram
$$
\xymatrix{
0 \ar[r] & \Pic (\widehat{X}) \ar[d] ^{\bar{\Delta}} \ar[r] ^{\rm Id} & \Pic (\widehat{X}) \ar[d] ^{ ({\rm Id}, {\rm Id})} \ar[r] & 0 \ar[r] \ar[d] & 0 \\
0 \ar[r] & \overline{ \Pic } (D_{\widehat{X}}) \ar[r] & \Pic (\widehat{X}) \oplus \Pic (\widehat{X}) \ar[r]^{ \ \ \ \ \ \  j_1 ^* - j_2 ^*} & \Pic (Y) \ar[r] & 0,}
$$
where $\overline{\Pic} (D_{\widehat{X}}) := {\rm im}  (\iota_1 ^*, \iota_2 ^*)$, and $\bar{\Delta}: = (\iota_1 ^*, \iota_2 ^*) \circ \Delta$. 
The snake lemma induces the short exact sequence
\begin{equation}\label{eqn:Milnor line 1}
0 \to {\rm coker} (\bar{\Delta}) \to {\rm coker} ({\rm Id}, {\rm Id}) \to \Pic (Y) \to 0.
\end{equation}
Here, the middle term of \eqref{eqn:Milnor line 1} is isomorphic to $\Pic (\widehat{X})$ by sending the pair $(\alpha, \beta)$ to $\alpha - \beta$, written additively. Under this identification, the image of ${\rm coker} (\bar{\Delta}) \to \Pic (\widehat{X})$ is equal to $\tau \Pic (D_{\widehat{X}})$. This proves the proposition.
\end{proof}

\subsection{The first Chern class: the affine case}\label{sec:c1 affine}

We return to the construction of the first Chern class map \eqref{eqn:1st Chern BGH0} when $Y$ is affine. We will show that the Milnor patching description of $\Pic (Y)$ and the mod $Y$-equivalence for the yeni Chow group are compatible and this compatibility naturally induces the construction of $c_1$ over $Y$.
We may assume that $Y$ is connected, as we can work with individual connected components separately.

For $\Pic (Y)$ of \eqref{eqn:1st Chern BGH0}, in Proposition \ref{prop:Milnor line} we saw $\Pic (Y) = \Pic (\widehat{X}) / \tau \Pic (D_{\widehat{X}})$ via the Milnor patching. 
On the other hand, by definition
$$
\CH^1 (\widehat{X} \mod Y,0) = \CH^1 (\widehat{X},0)/ \overline{\mathcal{M}}^1 (\widehat{X} , Y,0),
$$
where $\overline{\mathcal{M}} ^1 (\widehat{X}, Y, 0)$ is the image under the composite $\mathcal{M}^1 (\widehat{X}, Y, 0) \hookrightarrow z^1 (\widehat{X}, 0) \twoheadrightarrow \CH^1 (\widehat{X}, 0)$, and there is a natural composite homomorphism
$$
\CH^1 (\widehat{X} \mod Y, 0) \to \BGH^1 (\widehat{X} \mod Y, 0) \to \BGH^1 (Y, 0).
$$
 Hence to construct the map \eqref{eqn:1st Chern BGH0} when $Y$ is connected and affine, we may construct a homomorphism
\begin{equation}\label{eqn:1st Chern BGH}
c_1: {\rm Pic} (Y) = \frac{\Pic (\widehat{X})  }{ \tau \Pic (D_{\widehat{X}})} \to \frac{ \CH^1 (\widehat{X},0)}{ \overline{\mathcal{M} } ^1 (\widehat{X}, Y,0)}= \CH^1 (\widehat{X} \mod Y, 0),
\end{equation}
and to check that this is compatible with the maps $c_1$ constructed for different choices of the closed immersions $Y\hookrightarrow X'$.

\begin{prop}\label{prop:1st Chern affine}
Let $Y$ be a connected affine $k$-scheme of finite type, and let $Y \hookrightarrow X$ and $\widehat{X}$ be as before. Let  $c_1: \Pic (\widehat{X}) \to \CH^1 (\widehat{X}, 0)$ be the map in Definition \ref{defn:c1 exo}. Then we have the following:
\begin{enumerate}
\item Then $c_1 (\tau \Pic (D_{\widehat{X}})) \subset \mathcal{M}^1 (\widehat{X}, Y, 0)$ so that we have the induced homomorphism \eqref{eqn:1st Chern BGH}.
\item If we have a different closed immersion $Y \hookrightarrow X'$ with a morphism $f: X' \to X$ under $Y$, we deduce the commutative diagram
\begin{equation}\label{eqn:1st Chern affine 1}
\xymatrix{
{\rm Pic} (Y) \ar@{=}[r] \ar@{=}[rd]& \frac{\Pic (\widehat{X})  }{ \tau \Pic (D_{\widehat{X}})}\ar@{=}[d]  \ar[r] ^{c_1 \ \ \ \ }  & \CH^1 (\widehat{X} \mod Y, 0) \ar[d] ^{f^*} \\
 & \frac{\Pic (\widehat{X}')  }{ \tau \Pic (D_{\widehat{X}'})}   \ar[r] ^{c_1 \ \ \ \ }  &\CH^1 (\widehat{X}' \mod Y, 0).
 }
 \end{equation}
\end{enumerate}
\end{prop}

\begin{proof}
Since $Y$ is connected, the regular noetherian affine formal $k$-scheme $\widehat{X}$ is given as $\Spf (A)$ for a regular integral domain $A$. Let $Y= \Spec (B)$. Let $J \subset A$ be the ideal such that $A/J = B$ and $A$ is $J$-adically complete.

\medskip

To show that $c_1 (\tau \Pic (D_{\widehat{X}})) \subset \mathcal{M}^1 (\widehat{X}, Y, 0)$, we need to check that given two invertible sheaves $\mathcal{V}_1, \mathcal{V}_2 \in \Pic (\widehat{X})$ such that $\mathcal{V}_1 |_Y \simeq \mathcal{V}_2 |_Y$, we have $c_1 (\mathcal{V}_1)- c_1 (\mathcal{V}_2) \in \mathcal{M}^1 (\widehat{X}, Y, 0)$. 

We use the description of $\Pic (\widehat{X})$ in terms of the ideal class group ${\rm Cl} (A)$ of the invertible fractional ideals in \S \ref{sec:fractional ideal}, and translate the question in terms of commutative and homological algebra.

We can represent $\mathcal{V}_i$ by invertible fractional ideals. We can find representatives given by projective (integral) ideals $I_i \subset A$ for $i=1,2$ such that $I_i$ intersect properly with $J$. The condition $\mathcal{V}_1 |_Y \simeq \mathcal{V}_2|_Y$ reads $I_1 \otimes_A A/J = I_2 \otimes_A A/J$. This means the equality of the ideals
\begin{equation}\label{eqn:c1 ideal equal}
(I_1 + J)/ J = (I_2 + J)/ J \ \ \mbox{ in } A/ J.
\end{equation}

For the quotient ring $A/ I_i$, the short exact sequence
\begin{equation}\label{eqn:c1 proj resol}
0 \to I_i \to A \to A/ I_i \to 0
\end{equation}
gives a projective resolution $I_i \hookrightarrow A$ of the $A$-module $A/ I_i$ for $i=1,2$, because $I_i$ is a projective $A$-module (H. Matsumura \cite[Theorem 11.3, p.80]{Matsumura}).

Since $c_1 (\mathcal{V}_i)$ of Definition \ref{defn:c1 exo} is represented by the cycle $[A/I_i]$ on $\widehat{X}$ associated to the ring $A/ I_i$, to show that $c_1 (\mathcal{V}_1) - c_1 (\mathcal{V}_2) = [A/ I_1] - [A/I_2] \in \mathcal{M}^1 (\widehat{X}, Y,0)$, it remains to prove:
\medskip

\textbf{Claim:} \emph{We have an isomorphism
$$ 
A/ I_1 \otimes_A ^{\mathbf{L}} A/ J \simeq A/I_2 \otimes_A ^{\mathbf{L}} A/J
$$
as derived rings in $ \mathbf{sAlg} (A/J)$.}

\medskip

To prove the Claim, observe first:

\medskip

\textbf{Subclaim:} \emph{For all $j \geq 1$ and $i=1,2$, we have $\Tor_{j} ^A (A/I_i, A/J) = 0$.}

\medskip

By \eqref{eqn:c1 proj resol}, the two term complex $I_i \hookrightarrow A $ is a projective resolution of the $A$-module $A/I_i$ of length $1$. Thus we have  $\Tor_{j} ^A (A/I_i, A/J) = 0$ for $j \geq 2$ already. 

When $j=1$, we tensor the injection $I_i \hookrightarrow A $ with $(-)\otimes_A A/J$ to get $I_i \otimes_A A/J \to A \otimes_A A/J$, which is $(I_i +J)/ J \to A/ J$. This is still injective. Thus its kernel $\Tor_{1} ^A (A/I_i, A/J) = 0$. This proves the Subclaim.

\medskip

Since all higher Tor groups for $j>0$ are $0$ by the Subclaim, the statement of the Claim is equivalent to the isomorphism of the $\Tor_0$'s as algebras, i.e. the usual tensor products of algebras $ A/ I_1 \otimes_A  A/ J$ and $ A/I_2 \otimes_A A/J$ are isomorphic to each other as $(A/J)$-algebras. This follows from the equality of the ideals in \eqref{eqn:c1 ideal equal}. Hence the Claim follows, and this proves (1).

\medskip

Once we have the commutative diagram \eqref{eqn:c1 exo comm 0} of Lemma \ref{lem:c1 exo commute} with $Y_1= Y_2 = Y$, $X_1= X$, $X_2= X'$ as well as the induced map $f^*: \CH^1 (\widehat{X} \mod Y, 0) \to \CH^1 (\widehat{X}' \mod Y, 0)$ by Theorem \ref{thm:pull-back moving}, the commutativity of the right square in \eqref{eqn:1st Chern affine 1} follows. The equalities of the left triangle of \eqref{eqn:1st Chern affine 1} were proven in Proposition \ref{prop:Milnor line}. This proves (2).
\end{proof}

\begin{prop}\label{prop:Chern natural}

Let $Y_1$ and $Y_2$ be connected affine $k$-schemes of finite type. Suppose we have a commutative diagram
$$
\xymatrix{
Y_1 \ar[d] ^g  \ar@{^{(}->}[r] & X_1 \ar[d] ^f \\
Y_2  \ar@{^{(}->}[r] & X_2,}
$$
where the horizontal maps are closed immersions into equidimensional smooth $k$-schemes. Let $\widehat{X}_i$ be the completion of $X_i$ along $Y_i$ for $i=1,2$, and let $\widehat{f}: \widehat{X}_1 \to \widehat{X}_2$ be the induced morphism of formal schemes.

Then the following diagram
\begin{equation}\label{eqn:Chern natural 1}
\xymatrix{
\Pic (Y_2) \ar[d] ^{g^*} \ar[rr] ^{c_1 \ \ \ \ } & & \CH^1 (\widehat{X}_2 \mod Y_1, 0) \ar[d] ^{\widehat{f}^*} \\
\Pic (Y_1) \ar[rr] ^{c_1 \ \ \ \ } & & \CH^1 (\widehat{X}_1 \mod Y_2, 0)}
\end{equation}
commutes, where $\widehat{f}^*$ is the associated pull-back given in Theorem \ref{thm:pull-back moving}.
\end{prop}

\begin{proof}
Once we have the horizontal maps $c_1$ of \eqref{eqn:Chern natural 1}, as proven in Proposition \ref{prop:1st Chern affine}, the commutativity of the diagram follows from Lemma \ref{lem:c1 exo commute}.
\end{proof}

\subsection{The first Chern class: the general case}\label{sec:c1 general}

Now let $Y$ be a connected separated $k$-scheme of finite type in $\Sep_k$. The basic idea is to exploit the above constructions given for the affine schemes, and to glue them through the cycle version of the \v{C}ech construction of \S \ref{sec:cycle Cech BGH}.

One may possibly cast a doubt on whether this approach might work, with the following sort of reasoning: for instance, for a given line bundle $\mathcal{M}$, one can find an affine open cover $\{ U_i \}$ of $Y$ such that each affine open set is small enough to have the trivial bundle $\mathcal{M}|_{U_i}$ on $U_i$. Then how do we construct $c_1 (\mathcal{M})$ out of the vanishing first Chern classes on the affine open subsets?

There is a fallacy though: for each fixed $\mathcal{M}$, we may have a fine enough affine open cover for which $\mathcal{M}$ is locally trivial, but we are not fixing a line bundle $\mathcal{M}$ first: we choose an affine open cover of $Y$ first, and then there is no reason to believe that an arbitrary line bundle is trivial over each of the given affine open subsets in the cover.

\medskip

Returning back to the construction of the first Chern class, for a given connected separated $k$-scheme $Y$ of finite type, choose a system $\mathcal{U} = \{ (U_i, X_i) \}_{i \in \Lambda}$ of local embeddings for $Y$ (see Definition \ref{defn:system embedd Hartshorne}) such that each open set $U_i$ is affine. The subcollection of such systems with affine open subsets is cofinal in the collection $S(Y)^{\op}$ of all systems. Here let $U_{ij}:= U_i \cap U_j$ and $X_{ij} := X_i \times X_j$ with the diagonal embedding $U_{ij} \hookrightarrow X_{ij}$ as we did in \S \ref{sec:Cech machine}. Since $Y$ is separated over $k$, each $U_{ij}$ is affine open again. In this situation, we want to construct
\begin{equation}\label{eqn:c1 general 1}
c_1: {\rm Pic} (Y) \to \BGH^1 (\mathcal{U}, 0).
\end{equation}
Once we do  it, via the natural map $\BGH^1 (\mathcal{U}, 0) \to \BGH^1 (Y, 0)$, we deduce the first Chern class $c_1$ we want.

In terms of the notations of \S \ref{sec:cycle Cech BGH}, consider the following sub-bi-complex of $\check{\rm C}^1 (\mathcal{U})$, concentrated in the cohomological bi-degrees $(0,0)$, $(1,0)$, $(1, -1)$, $(0,-1)$ respectively:
\begin{equation}\label{eqn:subbicx}
\xymatrix{
 \underset{i}{\prod} z^1 (\widehat{X}_i \mod U_i, 0)  \ar[r] ^{\delta_0 \ \ } & \underset{ i, j}{ \prod} z^1 (\widehat{X}_{ij} \mod U_{ij}, 0) \\
 \underset{i}{\prod} z ^1 ( \widehat{X}_i \mod U_i, 1)  \ar[u] ^{\partial} \ar[r] ^{\delta_1 \ \ } &  \underset{i,j}{ \prod} z^1 (\widehat{X}_{ij} \mod U_{ij}, 1), \ar[u] ^{\partial}
}
\end{equation}
where the vertical arrows $\partial$ are the cycle boundary maps and the horizontal arrows $\delta_0, \delta_1$ are part of the \v{C}ech boundary maps. 

Let $C^{\bullet}= C^{\bullet} (\mathcal{U})$ be the total complex of the bi-complex \eqref{eqn:subbicx}, and let $T^{\bullet}:= {\rm Tot}  \ \check{\rm C}^1 (\mathcal{U})$. 
The inclusion $C^{\bullet} \hookrightarrow T^{\bullet}$ of complexes induces 
\begin{equation}\label{eqn:subbicx1.5}
{\rm H}^0 (C^{\bullet}) \to {\rm H} ^0 (T^{\bullet}).
\end{equation}

On the other hand, taking the cokernels of the vertical maps of \eqref{eqn:subbicx}, we get the cohomological complex concentrated in degrees $0, 1$:
\begin{equation}\label{eqn:subbicx2}
D^{\bullet}:= D^{\bullet} (\mathcal{U}): \prod_i \CH^1 (\widehat{X}_i \mod U_i, 0) \overset{\delta_{\CH}}{\longrightarrow} \prod_{i,j} \CH^1 (\widehat{X}_{ij} \mod U_{ij}, 0).
\end{equation}

\textbf{Claim:} There is a natural homomorphism
\begin{equation}\label{eqn:subbicx3}
{\rm H}^0 (D^{\bullet}) = \ker \delta_{\CH} \to  {\rm H}^0 (C^{\bullet}).
\end{equation}

Indeed, if $a \in \ker \delta_{\CH}$, then it has a cycle representative $\alpha \in \prod_i z^1 (\widehat{X}_i \mod U_i, 0)$ such that 
\begin{equation}\label{eqn:subbicx4}
\delta_0 ( \alpha) = \partial (\beta)
\end{equation}
for some $\beta \in \prod_{i,j} z^1 (\widehat{X}_{ij} \mod U_{ij}, 1)$. The equality \eqref{eqn:subbicx4} implies that $(\alpha, \beta) \in {\rm Z}^0 (C^{\bullet})$, as $C^{\bullet}$ is the total complex of \eqref{eqn:subbicx}. Hence to define the map \eqref{eqn:subbicx3}, we send $a$ to the equivalence class $\overline{(\alpha, \beta)} \in {\rm H}^0 (C^{\bullet})$. The image of the left vertical arrow $\partial$ in \eqref{eqn:subbicx} is killed in ${\rm H}^0 (C^{\bullet})$ by definition, so that we deduce the map \eqref{eqn:subbicx3}, proving the Claim.

\medskip

 Hence we have the induced composite homomorphism
\begin{equation}\label{eqn:cech pic}
{\rm H}^0 (D^{\bullet}) \to {\rm H}^0 (C^{\bullet}) \to {\rm H}^0 (T^{\bullet}) = {\rm C}{\rm H}^1  (\mathcal{U}, 0) \to \BGH^1 (\mathcal{U}, 0) \to  \BGH^1 (Y, 0)
\end{equation}
by \eqref{eqn:subbicx1.5}, \eqref{eqn:subbicx3}, and Lemma \ref{lem:flasque natural cech}. Thus, to define $c_1 (\mathcal{M}) \in  \BGH^1 (Y, 0)$ for $\mathcal{M} \in \Pic (Y)$, 
we may construct a class in ${\rm H}^0 (D^{\bullet})$.

\medskip

Consider the following commutative diagram
\begin{equation}\label{eqn:pic Chow Cech}
\xymatrix{
\underset{i}{\prod} \Pic (U_i) \ar[r] ^{\delta_\Pic} \ar[d] ^{c_1} & \underset{i,j}{\prod} \Pic (U_{ij}) \ar[d] ^{c_1} \\
\underset{i}{\prod} \CH^1 (\widehat{X}_i \mod U_i, 0) \ar[r] ^{\delta_{\CH} \ \ } & \underset{i,j}{ \prod} \CH^1 (\widehat{X}_{ij} \mod U_{ij}, 0),
}
\end{equation}
where the vertical maps $c_1$ are given by the affine case in Proposition \ref{prop:1st Chern affine}-(1) because all of $U_i$, $U_{ij}$ are affine and the digram commutes by Proposition \ref{prop:Chern natural}. The complex $D^{\bullet}$ in \eqref{eqn:subbicx2} is equal to the bottom complex of \eqref{eqn:pic Chow Cech}.

For a given $\mathcal{M} \in \Pic (Y)$, consider the collection of restrictions $\mathcal{M}_i := \mathcal{M}|_{U_i}\in \Pic (U_i)$ over all $i \in \Lambda$. They satisfy $\mathcal{M}_i |_{U_{ij}} = \mathcal{M}_j |_{U_{ij}}$ for all $i, j \in \Lambda$, so that the collection $\{ \mathcal{M}_i \}_{i \in \Lambda}$ defines a member in $\ker \delta_{\Pic}$. By the diagram \eqref{eqn:pic Chow Cech}, the member in $\ker \delta_{\Pic}$ induces a member in $\ker \delta_{\CH} = {\rm H}^0 (D^{\bullet})$. Hence via \eqref{eqn:cech pic}, this defines a class in $\CH^1 (\mathcal{U}, 0)$, thus in $\BGH^1 (Y, 0)$ via the natural map. This gives a map $c_1: \Pic (Y) \to \BGH^1 (Y, 0)$ of \eqref{eqn:1st Chern BGH0}, which \emph{a priori} depends on $\mathcal{U}$.

\medskip

One needs to check that the above $c_1$ (temporarily call it $c_1 ^{\mathcal{U}}$) is independent of the choice of the system $\mathcal{U}$ of local embeddings. Let $\mathcal{V}$ be another system of local embeddings, whose underlying cover $|\mathcal{V}|$ of $Y$ consists of affine open subsets. By Lemma \ref{lem:common refinement}, we find a common refinement $\mathcal{W}$ of both $\mathcal{U}$ and $\mathcal{V}$. 

Thus we reduce to checking the independence when $\mathcal{V}= \{ (V_{i'} , X_{i'} ') \}_{i' \in \Lambda'}$ is a refinement of $\mathcal{U} = \{ (U_i, X_i) \}_{i \in \Lambda}$ via a set map $\lambda: \Lambda' \to \Lambda$. Recall for each $i' \in \Lambda'$, we have $V_{i'} \subset U_{\lambda (i')}$ and a morphism $X'_{i'} \to X_{\lambda (i')}$ that restricts to $V_{i'} \hookrightarrow U_{\lambda (i')}$. By Proposition \ref{prop:Chern natural} applied to this inclusion morphism of affine schemes for each $i' \in \Lambda'$, we have the commutative diagram
\begin{equation}\label{eqn:c1 aff refine}
\xymatrix{
\Pic (U_{\lambda (i')}) \ar[d] \ar[rr]^{c_1 \ \ \ \ \ } &  & \CH^1 ( \widehat{X}_{\lambda (i')} \mod U_{\lambda (i')}, 0) \ar[d] \\
\Pic (V_{i'}) \ar[rr]^{c_1 \ \ \ \ \ } & & \CH^1 (\widehat{X}_{i'} \mod V_{i'}, 0),}
\end{equation}
and similarly for $V_{i'j'} \subset U_{\lambda(i') \lambda (j')}$ as well.

Recall we also had the commutative diagrams \eqref{eqn:pic Chow Cech} for $\mathcal{U}$ and $\mathcal{V}$, respectively. 
Hence we deduce the following commutative diagram
$$
\xymatrix{
& \underset{i}{\prod} \Pic (U_i) \ar[dl]_{\lambda^*} \ar[dd] _{ {\begin{matrix} c_1  \\ \\ \\ \\ \end{matrix}} } \ar[rr]^{\delta_{\Pic} ^{\mathcal{U}}} & &\underset{i,j}{ \prod} \Pic ( U_{ij}) \ar[dl]_{\lambda^*} \ar[dd]^{c_1 }\\
\underset{i'}{\prod} \Pic (V_{i'}) \ar[dd] ^{c_1} \ar[rr] _{ \ \ \ \ \  \ \ \ \delta_{\Pic} ^{\mathcal{V}}} & & \underset{i',j'}{\prod} \Pic (V_{i'j'}) \ar[dd] _{{\begin{matrix} c_1  \\ \\ \\ \\ \end{matrix}} } & \\
& \underset{i}{ \prod} \CH^1 (\widehat{X}_i  / U_i) \ar[dl]_{\lambda^*} \ar[rr] ^{ \ \ \ \ \ \ \ \ \ \delta_{\CH} ^{\mathcal{U}}} & & \underset{i,j}{ \prod} \CH^1 (\widehat{X}_{ij}  / U_{ij}) \ar[dl]_{\lambda^*} \\
\underset{i'}{\prod} \CH^1 (\widehat{X}_{i'} ' / V_{i'}) \ar[rr] ^{\delta_{\CH} ^{\mathcal{V}}} & &\underset{i',j'}{ \prod} \CH^1 (\widehat{X}'_{i'j'}  / V_{i'j'}),& }
$$
where to save the space, we used the notations like $\CH^1 (A/ B)$ in the diagram to mean $\CH^1 (A \mod B, 0)$.
 
The above diagram induces the commutative diagram
$$
\xymatrix{
\Pic (Y) \ar[r] \ar[dr] & {\rm H}^0 (D^{\bullet} (\mathcal{U})) \ar[d]_{\lambda^*}  \ar[r]  & \CH^1 (\mathcal{U}, 0) \ar[d]^{\lambda^*} \ar[r] & \BGH^1 (Y, 0)\\
& {\rm H}^0 (D^{\bullet} (\mathcal{V}) ) \ar[r] & \CH^1 (\mathcal{V} , 0), \ar[ru] & }
$$
where the top composite via $\CH^1 (\mathcal{U}, 0)$ is $c_1 ^{\mathcal{U}}$ and the bottom composite via $\CH^1 (\mathcal{V}, 0)$ is $c_1 ^{\mathcal{V}}$. This commutativity shows that after taking the colimit to $\BGH^1 (Y, 0)$, $c_1$ is independent of the choice of $\mathcal{U}$, indeed.

\medskip

When $Y$ has a multiple number of connected components, we can work with the individual connected components and take their direct sum. This gives $c_1= c_1 (Y): \Pic (Y) \to \BGH^1 (Y, 0)$ for each $Y \in \Sep_k$. 

\medskip

We now have:

\begin{thm}
The first Chern class $c_1$ gives a natural transformation
$$
c_1: \Pic (-) \longrightarrow \BGH^1 (-, 0)
$$
of the contravariant functors on $\Sep_k$. In other words, there is a functorial first Chern class $c_1$ on $\Sep_k$ for the yeni Chow groups.
\end{thm}

\begin{proof}
We saw that for each $Y \in \Sep_k$, we have a well-defined $c_1 (Y): \Pic (Y) \to \BGH^1 (Y, 0)$. It remains to show that for any morphism $g: Y_1 \to Y_2$ in $\Sep_k$, the following diagram commutes
\begin{equation}\label{eqn:c1 funct0}
\xymatrix{
\Pic (Y_2) \ar[d] ^{g^*} \ar[rr] ^{c_1 (Y_2)} & & \BGH^1 (Y_2, 0) \ar[d] ^{g^*} \\
\Pic (Y_1) \ar[rr] ^{c_1 (Y_1)  } & & \BGH^1 (Y_1, 0),}
\end{equation}
where the left vertical $g^*$ is the pull-back of invertible sheaves and the right vertical $g^*$ is the pull-back of Theorem \ref{thm:functoriality general}.

Since both $c_1$ and $g^*$ are constructed using the \v{C}ech machine, we are reduced to checking the commutativity in the case  when $Y_1$ and $Y_2$ are both affine. But this is known by Proposition \ref{prop:Chern natural}.
\end{proof}

Given the first Chern class maps, we may wonder whether we can construct the higher Chern classes $c_i: K_0 (Y) \to \BGH^i (Y,0)$, following the sketches in SGA VI \cite{SGA6} and A. Grothendieck \cite[\S 2, \S 3]{Grothendieck BSMF}. We remark that, in \emph{loc.cits.}, Grothendieck assumed that $Y$ is smooth, but that assumption seems needed just to make sure that the classical Chow groups $\CH^* (Y)$ form a ring. This fails when $Y$ has singularities for the classical Chow groups. However, if we replace the classical Chow groups by the yeni Chow groups $\BGH^* (Y,0)$, it is a ring (see Theorem \ref{thm:general product}) regardless of the singularities of $Y$.

\section{Connections to the Milnor $K$-theory}\label{sec:preview}

In \S \ref{sec:preview}, we discuss a relationship between the Milnor $K$-theory and the yeni higher Chow groups in the Milnor range, and present some conjectural guesses as well as their implications.

\subsection{Milnor $K$-theory and graph homomorphisms}\label{sec:Milnor2}

In \S \ref{sec:Milnor2}, we work with the Milnor range $\BGH^n (A, n)$ with $q=n$ where $A$ is a $k$-algebra of finite type.

\subsubsection{The graph homomorphism}\label{subsec:Milnor}

We first show that there exists the graph homomorphism from the Milnor $K$-theory. Recall that the Milnor $K$-ring $K_*^M (A)$ is the graded tensor algebra $T_{\mathbb{Z}} ^* (A^{\times})$ of the units $A^{\times}$ over $\mathbb{Z}$ modded out by the two sided ideal generated by elements of the form $a \otimes (1-a)$ for $a \in A$ such that $a, 1-a \in A^{\times}$. The degree $n$ part is the $n$-th Milnor $K$-group $K_n ^M (A)$.

\medskip

We have a ring $R$ and its quotient $R/I$ by an ideal $I \subset R$, the group homomorphism $R ^{\times} \to (R/I)^{\times}$ is not surjective in general, but there are some cases where the map is indeed surjective. The following is well-known and easy to prove (see Hartshorne-Polini \cite[Lemma 5.2]{HP} and see also C. Weibel \cite[Exercise I-1.12-(iv), p.7]{Weibel}):

\begin{lem}\label{lem:unit surj}
Let $R$ be a ring and $I \subset R$ be an ideal. Assume either $R$ is a local ring or $R$ is $I$-adically complete. Then the group homomorphism $R^{\times} \to (R/I)^{\times}$ is surjective.
\end{lem}

\begin{thm}\label{thm:graph general}
Let $Y= \Spec (A)$ be a $k$-algebra of finite type. Then:
\begin{enumerate}
\item For any closed immersion $Y \hookrightarrow X$ into an equidimensional smooth $k$-scheme and the resulting completion $\widehat{X} = \Spf (\widehat{R})$ of $X$ along $Y$, there is the graph homomorphism
\begin{equation}\label{eqn:Milnor-1}
gr_X: K_n ^M (Y) \to \CH^n (\widehat{X} \mod Y, n).
\end{equation}
\item Suppose there is another closed immersion $Y \hookrightarrow X'$into an equidimensional smooth $k$-scheme such that there is a morphism $f: X' \to X$ under $Y$. Let $\widehat{X}'$ be the completion of $X'$ along $Y$. Then the following diagram commutes:
 \begin{equation}\label{eqn:Milnor3}
 \xymatrix{
 K_n ^M (Y) \ar[r]^{gr_{X} \ \ \ \ }  \ar[rd] _{gr _{X'} \ \ \ \ }  & \CH^n (\widehat{X} \mod Y, n) \ar[d] ^{f^*} \\
 & \CH^n (\widehat{X}' \mod Y, n).}
 \end{equation}
\end{enumerate}
\end{thm}

\begin{proof}
(1) We may assume $Y= \Spec (A)$ is connected. There is an ideal $\widehat{I} \subset \widehat{R}$ such that $\widehat{R}/ \widehat{I} = A$ and $\widehat{R}$ is $\widehat{I}$-adically complete.

\medskip

For $\widehat{X}:=\Spf (\widehat{R})$, we first construct a homomorphism
\begin{equation}\label{eqn:Milnor0}
gr_{\widehat{R}} : K_n ^M (\widehat{R}) \to \CH^n (\Spf (\widehat{R}), n)
\end{equation}
as follows. We use the argument of \cite[Lemma 2.1]{EVMS}, which did the job for $\Spec (\widehat{R})$. The map is given by sending the Milnor symbol $\{ a_1, \cdots, a_n \}$ (where $a_i \in \widehat{R}^{\times}$ such that $1-a_i \in \widehat{R}^{\times}$) to the integral closed formal subscheme $\Gamma_{(a_1, \cdots, a_n)}$ of $\Spf (\widehat{R}) \times_k \square_k ^n$ given by the set of polynomials
$$
 \{y_1 - a_1 , \cdots, y_n - a_n\}.
$$ 

Indeed, for any codimension $1$ face $F= \{ y_i = \epsilon \} \subset \square_k^n$ for $\epsilon \in \{ 0, \infty \}$, we have $\Gamma_{ (a_1, \cdots, a_n ) } \cap (\widehat{X} \times F) = \emptyset$ because $a_i \in \widehat{R}^{\times}$. Thus 
$$
\Gamma_{(a_1, \cdots, a_n) } \in \ker \left( \partial : z^n (\Spf (\widehat{R}), n) \to z^{n} (\Spf (\widehat{R}), n-1) \right).
$$

That this map kills the Steinberg relations modulo boundaries follows exactly as in \emph{loc.cit.}, and 
we obtain the homomorphism \eqref{eqn:Milnor0}. We shrink details.

\medskip

The important point is to show that the map \eqref{eqn:Milnor0} descends to give the map \eqref{eqn:Milnor-1}. 

For any Milnor symbol $\{ b_1, \cdots, b_n \} \in K_n ^M (A) = K_n ^M (\widehat{R}/ \widehat{I})$, with $b_i \in A^{\times}$ such that $1-b_i \in A^{\times}$, choose liftings $a_1, \cdots, a_n \in \widehat{R}^{\times}$ of $b_1, \cdots, b_n$ (which is possible by Lemma \ref{lem:unit surj}), and send the Milnor symbol $\{ a_1, \cdots, a_n \}\in K_n ^M (\widehat{R})$ to the cycle class of $\Gamma_{(a_1, \cdots, a_n)}$ in $\CH^n (\Spf (\widehat{R}) \mod Y, n)$. 

\medskip

To prove that this map is well-defined, choose another sequence of liftings $a_1 ', \cdots a_n ' \in \widehat{R}^{\times}$ of $b_1, \cdots, b_n$. Here we have $a_i  - a'_i \in \widehat{I}$ for $1 \leq i \leq n$. It remains to show that
\begin{equation}\label{eqn:graph mod Y}
\mathfrak{Z}:= \Gamma_{(a_1, \cdots, a_n)} \sim_{Y} \Gamma_{(a_1 ', \cdots, a_n ')}=:\mathfrak{Z}',
\end{equation}
i.e. they are mod $Y$-equivalent.

\medskip

We claim that $(\mathcal{O}_{\mathfrak{Z}}, \mathcal{O}_{\mathfrak{Z}'}) \in \mathcal{L}^n (\widehat{X}, Y, n)$, i.e. that there is an isomorphism
\begin{equation}\label{eqn:Koszul0}
\mathcal{O}_{\mathfrak{Z}} \otimes_{\mathcal{O}_{\square^n _{\widehat{X}}}} ^{\mathbf{L}} \mathcal{O}_{\square^n_Y} \simeq \mathcal{O}_{\mathfrak{Z}'}  \otimes_{\mathcal{O}_{\square^n _{\widehat{X}}}} ^{\mathbf{L}} \mathcal{O}_{\square^n_Y}
\end{equation}
in $\sAlg (\mathcal{O}_{\square_Y^n})$. 

\medskip

Here, $\mathcal{O}_{\mathfrak{Z}}$ is defined by the explicit sequence of linear polynomials
$$
 y_1 - a_1,  \ \ \cdots, \ \  y_n - a_n  \ \ \in  \widehat{R} \{ y_1, \cdots, y_n \}
 $$
and they give a regular sequence. Hence we can use an explicit simplicial resolution given by the Koszul derived ring of the regular sequence (see, for instance in \cite[4.12$\sim$4.15]{Iyengar}. Or in terms of the Koszul derived (formal) scheme as in \cite[Definition 3.1]{KST IM}), we have a simplicial resolution
$$
\mathcal{K}_{\bullet} (y_{\cdot}- a_{\cdot}) \to \mathcal{O}_{\mathfrak{Z}} \to 0,
$$
where 
$$\mathcal{K}_{\bullet} (y_{\cdot}- a_{\cdot}) \simeq \frac{\mathcal{O}_{\square_{\widehat{X}}^n}}{ (y_1 - a_1)} \otimes ^{\mathbf{L}} _{\mathcal{O}_{\square_{\widehat{X}}^n}} \cdots  \otimes ^{\mathbf{L}} _{\mathcal{O}_{\square_{\widehat{X}}^n}} \frac{\mathcal{O}_{\square_{\widehat{X}}^n}}{ (y_n - a_n)}
$$
 in $\sAlg (\mathcal{O}_{\square_{\widehat{X}}^n})$. Similarly, we have a simplicial resolution $\mathcal{K}_{\bullet} (y_{\cdot}- a'_{\cdot}) \to \mathcal{O}_{\mathfrak{Z}'} \to 0$ of $\mathcal{O}_{\mathfrak{Z}'}$ as well.

Having chosen simplicial resolutions of $\mathcal{O}_{\mathfrak{Z}}$ and $\mathcal{O}_{\mathfrak{Z}'}$, the statement \eqref{eqn:Koszul0} is equivalent to that we have an isomorphism 

\begin{equation}\label{eqn:Koszul}
 \mathcal{K}_{\bullet}   (y_{\cdot} - a_{\cdot}) \otimes _{\mathcal{O}_{\square^n _{\widehat{X}}}} \mathcal{O}_{\square^n _Y} \simeq \mathcal{K}_{\bullet} ( y_{\cdot} - a'_{\cdot}) \otimes _{\mathcal{O}_{\square^n _{\widehat{X}}}} \mathcal{O}_{\square^n _Y}
 \end{equation}
 in $\sAlg (\mathcal{O}_{\square_Y ^n})$.

Since each $\mathcal{O}_{\square_{\widehat{X}}^n}/ (y_i - a_i)$ is quasi-isomorphic to the locally free resolution $\mathcal{O}_{\square_{\widehat{X}}^n} \overset{y_i - a_i}{\longrightarrow} \mathcal{O}_{\square_{\widehat{X}}^n}$, and since the sequence $\{y_1 - \bar{a}_1, \cdots, y_n - \bar{a}_n \}$ in $A[y_1, \cdots, y_n] = \widehat{R}\{ y_1, \cdots, y_n \} / (I)$ is again a regular sequence, tensoring with $\mathcal{O}_{\square_Y^n}$ over $\mathcal{O}_{\square_{\widehat{X}}^n}$ as did in \eqref{eqn:Koszul} gives also simplicial resolutions of $\mathcal{O}_{\bar{\mathfrak{Z}}}$ and $ \mathcal{O}_{\bar{\mathfrak{Z}}'}$. Here, $\bar{\mathfrak{Z}}$ and $\bar{\mathfrak{Z}}'$ are the mod $I$-reductions of $\mathfrak{Z}$ and $\mathfrak{Z}'$, and they are equal in $\square_Y ^n$ because $\bar{a}_j = \bar{a}_j '$ for all $1 \leq j \leq n$. 
 
 In other words, the derived rings $\mathcal{K}_{\bullet} (y_{\cdot} - \bar{a}_{\cdot})$ and $ \mathcal{K}_{\bullet} ( y _{\cdot} - \bar{a}'_{\cdot})$ give simplicial resolutions of $\mathcal{O}_{\bar{\mathfrak{Z}}}$ and $ \mathcal{O}_{\bar{\mathfrak{Z}}'}$, while we have the equalities $\mathcal{O}_{\bar{\mathfrak{Z}}}= \mathcal{O}_{\bar{\mathfrak{Z}}'}$ and
   \begin{equation}\label{eqn:Koszul2}
 \mathcal{K}_{\bullet}  (y_{\cdot} - \bar{a}_{\cdot}) = \mathcal{K}_{\bullet} ( y _{\cdot} - \bar{a}'_{\cdot}).
 \end{equation}
 In particular, this is an isomorphism in $\sAlg (\mathcal{O}_{\square_Y^n})$, i.e. \eqref{eqn:Koszul} is an isomorphism. This in turn gives an isomorphism \eqref{eqn:Koszul0}, proving the Claim.

 \medskip
 
 Hence, $[\mathfrak{Z} ] - [\mathfrak{Z}'] = [\mathcal{O}_{\mathfrak{Z}}] - [ \mathcal{O}_{\mathfrak{Z}'}] \in \mathcal{M}^n (\widehat{X}, Y, n)$, proving \eqref{eqn:graph mod Y}. This shows that we have the map \eqref{eqn:Milnor-1}.

 \medskip
 
 (2) We now suppose there is another closed immersion $Y \hookrightarrow X'$ with $\widehat{X}' = \Spf (\widehat{R}')$, we assume there is a morphism $f: X' \to X$ under $Y$. This induces the morphism $\widehat{f}: \widehat{X}' \to \widehat{X}$ under $Y$, and there corresponds a ring homomorphism $\widehat{R} \to \widehat{R}'$ as well. By Theorem \ref{thm:pull-back moving}, we know that there exists a pull-back $f^*: \CH^n (\widehat{X} \mod Y, n) \to \CH^n (\widehat{X}' \mod Y, n)$. Here $\widehat{R}'$ is $\widehat{I}'$-adically complete and $\widehat{R}'/\widehat{I}' = A$.
 
 Since the graph maps send Milnor symbols to graph cycles, we immediately have the commutative diagram
 \begin{equation}\label{eqn:Milnor2}
 \xymatrix{
 K_n ^M (\widehat{R}) \ar[d] \ar[r]  & \CH^n (\widehat{X} \mod Y, n) \ar[d] ^{f^*} \\
 K_n ^M (\widehat{R}') \ar[r] & \CH^n (\widehat{X}' \mod Y, n).}
 \end{equation}
 The existence of the maps \eqref{eqn:Milnor-1} for $\widehat{R}$ and $\widehat{R}'$ together with the diagram \eqref{eqn:Milnor2} implies the diagram \eqref{eqn:Milnor3}, because $\widehat{R}'/\widehat{I}' = \widehat{R}/\widehat{I} = A$.
 \end{proof}

\subsubsection{For semi-local $k$-schemes}

The map \eqref{eqn:Milnor-1} in Theorem \ref{thm:graph general} is not in general an isomorphism even when $Y= \Spec (A)$ is smooth over $k$. However when $A$ is a \emph{semi-local} $k$-algebra essentially of finite type, i.e. it is obtained by localizing a finite type $k$-algebra at a finite set of scheme points, then we guess that it is an isomorphism. 

One little problem in the above assertion is that, so far we did \emph{not} define $\CH^q (\widehat{X} \mod Y, n)$ or $\BGH^q (Y, n)$ for such semi-local $k$-schemes $Y$ essentially of finite type.

\medskip

Recall from Remarks \ref{remk:stalk realize} and \ref{remk:stalk realize 2} that for a given point $y \in |Y| = |\widehat{X}|$, when we consider the colimit 
$$
\varinjlim_{y \in U} z^q (\widehat{X} |_U \mod U, \bullet),
$$
we do not know whether it can be realized as a cycle complex of a suitable formal scheme. Likewise, we do not know whether the colimit of the cycle class groups
$$
\varinjlim_{y \in U} \CH^q (\widehat{X} |_U \mod U, n)
$$
can be realized as a cycle class group of a concrete geometric object, either.

But, maybe this is unnecessary for our purposes. In this article, we try the following more practical approach for semi-local $k$-schemes essentially of finite type:

\begin{defn}\label{defn:HCG semi-local}
Let $Y$ be a semi-local $k$-scheme essentially of finite type obtained by localizing an affine $k$-scheme $\tilde{Y}$ of finite type at a finite set $\Sigma \subset \tilde{Y}$ of scheme points. Let's write it as $\tilde{Y}_{\Sigma} =Y$.

Define the yeni higher Chow group of $Y$ to be the (double) colimit
\begin{equation}\label{eqn:BGH semi-local}
\BGH^q (Y, n):= \varinjlim_{\Sigma \subset U \subset \tilde{Y}} \BGH^q (U, n),
\end{equation}
where the last colimit is taken over all open subschemes $\mathcal{U}$ of $\tilde{Y}$ that contain $\Sigma$. Note that this colimit exists because $\BGH^q (-, n)$ is a contravariant functor on $\Sch_k$, and the category of abelian groups is cocomplete.

Because the motivic cohomology on smooth schemes (i.e. higher Chow groups on smooth schemes) is continuous in that it satisfies \eqref{eqn:BGH semi-local} (see e.g. M. Kerz \cite[Lemma 1.0.8, p.9]{Kerz thesis}), our definition is consistent with this known fact.
\qed
\end{defn}

We may equally describe \eqref{eqn:BGH semi-local} as follows, too. Let $S(\tilde{Y})$ be the collection of systems of local embeddings for $\tilde{Y}$. Let $S (\tilde{Y})_{\Sigma}$ be the subcollection of $\mathcal{U} \in S(\tilde{Y})$ such that for each $(U_i, X_i) \in \mathcal{U}$, we have $\Sigma \subset U_i$. Then we have
\begin{equation}\label{eqn:BGH semi-local 2}
\BGH^q (Y, n) = \varinjlim_{\mathcal{U} \in S (\tilde{Y})_{\Sigma} ^{\op}} \CH^q (\mathcal{U}, n).
\end{equation}

\medskip

This construction is functorial on the category $({\rm Loc}_k)$ of semi-local $k$-schemes essentially of finite type. This category $({\rm Loc}_k)$ is defined as follows. The objects are semi-local $k$-schemes essentially of finite type. The morphisms are, roughly speaking, the local $k$-morphisms induced from those in $\Aff_k$.

More precisely, let $Y_1$ and $Y_2$ be semi-local $k$-schemes essentially of finite type. Let's say that a morphism $g: Z_1 \to Z_2$ in $\Aff_k $ is \emph{relative to $(Y_1, Y_2)$}, if (1) for $i=1,2$, there exist nonempty finite subsets $\Sigma_i \subset Z_i$ such that ${Z_i}_{\Sigma_i} \simeq Y_i$ as $k$-schemes, with the maximal ideals of both sides correspond to each other, and (2)  for each $p \in \Sigma_1$, we have $g(p) \in \Sigma_2$. Note that a morphism $g: Z_1 \to Z_2$ relative to $(Y_1, Y_2)$ induces a local $k$-morphism $\bar{g}: Y_1 \to Y_2$, i.e. it sends the closed points to closed points. In this case, let's say \emph{$g$ is a spreading of $\bar{g}$.}

For two morphisms $g: Z_1 \to Z_2$ and $g': Z_1' \to Z_2'$ in $\Aff_k$, both relative to $(Y_1, Y_2)$, we say that $g \equiv g' $ rel $(Y_1, Y_2)$, if their induced local $k$-morphisms $\bar{g} , \bar{g}' : Y_1 \to Y_2$ are equal. One checks that $\equiv$ is an equivalence relation. We define an explicit morphism in $({\rm Loc}_k)$ from $Y_1$ to $Y_2$ to be the set of equivalence classes of the morphisms in $\Aff_k$ relative to $(Y_1, Y_2)$. 

One checks that $({\rm Loc}_k)$ is a category with the above notions of objects and morphisms. Since the yeni higher Chow theory  $\BGH^q ( -, n)$ is a contravariant functor on $\Sch_k$, by taking the colimits, we deduce the contravariant functor $\BGH^q (- , n)$ on $({\rm Loc}_k)$.

\medskip

Since the Milnor $K$-theory enjoys the continuity property (see M. Kerz \cite[Proposition 6, p.177]{Kerz finite}) like the one for $\BGH^q (-, n)$ in \eqref{eqn:BGH semi-local} and \eqref{eqn:BGH semi-local 2}, Theorem \ref{thm:graph general} implies:

\begin{cor}\label{cor:graph semi-local}
Let $A$ be a semi-local $k$-algebra essentially of finite type. Then there is the graph homomorphism
\begin{equation}\label{eqn:Milnor-1-semi-local}
gr_A: K_n ^M (A) \to \BGH^n (A, n).
\end{equation}
\end{cor}

On this graph map, we make the following conjectural guess:

\begin{conj}\label{conj:00}
Let $k$ be a field with $|k| \gg 0$. Then for each semi-local $k$-algebra $A$ essentially of finite type, the graph homomorphism \eqref{eqn:Milnor-1-semi-local} is an isomorphism.
\end{conj}

The condition $|k| \gg 0$ in Guess \ref{conj:00} is intended so as to have the equality $K^M_n (A) = \widehat{K}^M_n (A)$ of the usual Milnor $K$-theory with the improved Milnor $K$-theory of Gabber-Kerz. See M. Kerz \cite[Proposition 10-(5), p.181]{Kerz finite}.

\bigskip

When $A$ in Guess \ref{conj:00} is smooth over $k$, by Theorem \ref{thm:sm formal}, we have $\BGH^n (A, n) = \CH^n (A, n)$, and the validity of Guess \ref{conj:00} is known by Elbaz-Vincent--M\"uller-Stach \cite{EVMS} and M. Kerz \cite{Kerz Gersten}. This includes the case of fields by Nesterenko-Suslin \cite{NS} and B. Totaro \cite{Totaro}. Thus the prediction made by the Guess \ref{conj:00} is new only when $A$ is \emph{not} smooth over $k$.

 In \cite{Park Tate}, Guess \ref{conj:00} is proven for $A= k_m= k[t]/(t^m)$ with $m \geq 1$. This provides an evidence for the Guess in the non-smooth case. Note that $A=k_m$ is an Artin local $k$-algebra, which is already of finite type. So we don't need to take the colimit as in \eqref{eqn:BGH semi-local} to define $\BGH^n (A, n)$.

\subsection{The case $n=1$}\label{sec:n=1}
We do not yet know the validity of Guess \ref{conj:00} in general. In \S \ref{sec:n=1}, we present some attempts on the special case when $n=1$. 

\subsubsection{Representability}

We first remark the following:

\begin{thm}\label{thm:n=1 representable}
Assume Guess \ref{conj:00} holds for $n=1$. Then the covariant functor $\BGH^1 (- , 1) : ({\rm Loc}_k)\to ({\rm Ab})$ from the category of semi-local $k$-schemes essentially of finite type is bijective to $\Hom_{k}(-, \mathbb{G}_m)$ at each object, where $\mathbb{G}_m:= \Spec (k[t, t^{-1}])$.
\end{thm}

Loosely, we can say $\BGH^1 (-, 1)$ is ``representable" by $\mathbb{G}_m$, though $\mathbb{G}_m$ is not in $({\rm Loc}_k)$.

\begin{proof}
By Guess \ref{conj:00}, for $A \in ({\rm Loc}_k)$, we can identify $\BGH^1 (A, 1)$ with $K_1 ^M (A) = A^{\times}$. The functor $\Hom_{\Sch_k} (- , \mathbb{G}_m)$ given by $\mathbb{G}_m$ satisfies
$$
\Hom_{k} (\Spec (A), \mathbb{G}_m) = \Hom_{{\rm Alg}_k} (k[t, t^{-1}], A).
$$
Since $\Hom_{{\rm Alg}_k} (k[t, t^{-1}], A)= A^{\times}$, this proves the assertion.
\end{proof}

\begin{remk}
When $n \geq 2$, the author does not expect the Milnor $K$-theory $K_n ^M (-)$ is representable by a finite dimensional $k$-scheme. When $n=2$, this seems related to the representability question of the Brauer groups. \qed
\end{remk}

\subsubsection{The generators for $n=1$} 

For the rest of \S \ref{sec:n=1}, we test Guess \ref{conj:00} when $n=1$ and $A$ is an Artin local $k$-algebra whose residue field is $k$. Here, $A$ is of finite type over $k$, so that we do not need to take the additional colimit of Definition \ref{defn:HCG semi-local} to define its yeni higher Chow groups.

\medskip

Let $Y:= \Spec (A)$. Topologically, this $Y$ is just a singleton. Choose a closed immersion $Y \hookrightarrow X$ into an equidimensional smooth $k$-scheme. Completing $X$ along $Y$, we get $\widehat{X} = \Spf (\widehat{R})$ for a regular local $k$-domain $\widehat{R}$ of finite Krull dimension, where $\widehat{R}$ is not necessarily of finite type over $k$. It is complete with respect to an ideal $\widehat{I} \subset \widehat{R}$ such that $\widehat{R}/\widehat{I} = A$.

We have the following commutative diagram
\begin{equation}\label{eqn:n=1 diag}
\xymatrix{
K_1 ^M (\widehat{R}) \ar@{>>}[d] \ar[r] ^{gr_{\widehat{R}} \ \ \ \ \ } & \CH^1 (\Spf (\widehat{R}), 1) \ar[d] \\
K_1 ^M (A) \ar[r] ^{\tilde{gr}_A \ \ \ \  \ \ \ \ \ \ } & \CH^1 (\Spf (\widehat{R}) \mod Y, 1),}
\end{equation}
where the left vertical map is surjective by Lemma \ref{lem:unit surj}. 
An interesting first observation is:

\begin{lem}\label{lem:rest poly 1}
Let $\mathfrak{Z} \in z^1 (\Spf (\widehat{R}), 1)$ be an integral cycle. 

Then it is given by a ring $\widehat{S}= \widehat{R}\{ y_1 \}/ (p(y_1))$ for a prime element $p(y_1) \in \widehat{R}\{ y_1 \}$, and this $p(y_1)$ can be chosen to be a polynomial in $y_1$, whose leading coefficient and the constant term in $y_1$ are units in $\widehat{R}^{\times}$.

In particular, $\mathfrak{Z}$ is a complete intersection and  we can find a monic polynomial that defines it. 
\end{lem}

\begin{proof}
The cycle $\mathfrak{Z}$ is of codimension $1$ so that it is given by a height $1$ prime ideal $P \subset \widehat{R} \{ y_1 \}$. Since $\widehat{R} \{ y_1 \}$ is a UFD (see P. Salmon \cite[5. Corollaire 1, p.395]{Salmon} or P. Samuel \cite[Ch.2, Theorem 3.2, p.52]{Samuel}), there is an element $p(y_1) \in \widehat{R} \{y_1\}$ such that $P= (p(y_1))$ by H. Matsumura \cite[Theorem 20.1, p.161]{Matsumura}. 

This $p(y_1)$ is \emph{a priori} a restricted formal power series in $y_1$ with the coefficients in $\widehat{R}$. We make:

\medskip

\textbf{Claim:} \emph{This $p(y_1)$ is in fact a polynomial in $y_1$.}

\medskip

Write 
$$
p = \sum_{i =0} ^{\infty} \alpha_i y_1 ^i \in \widehat{R}\{ y_1 \},
$$
for some $\alpha_i \in \widehat{R}$ such that as $i \to \infty$, we have $\alpha_i \to 0$ in the $\widehat{I}$-adic topology on $\widehat{R}$. This means that for each integer $m \geq 1$, there exists an integer $r \geq 1$ such that whenever $i \geq r$, we have $\alpha_i \in \widehat{I}^m$.

For each $m \geq 1$, take the reduction $Z|_m$ of $\mathfrak{Z}$ mod $\widehat{I}^m$, i.e.
$$
Z|_m:= \mathfrak{Z} \cap (\Spec ( \widehat{R}/ \widehat{I}^m) \times \square^1).
$$
This is a closed subscheme in $\Spec (\widehat{R}/\widehat{I}^m) \times \square_k ^1$ defined by the polynomial 
$$
p_m (y_1):=( p(y_1) \mod \widehat{I}^m) \in (\widehat{R}/\widehat{I}^m)[y_1].
$$
Let $d_m:= \deg_{y_1} p_m(y_1)$.

The condition (\textbf{SF}) of Definition \ref{defn:HCG} implies that $\mathfrak{Z}$ intersects with all faces of $\Spec (\widehat{R}/\widehat{I}^m) \times \square_k ^1$ properly (see Remark \ref{remk:gen ideal defn}). Thus, for each codimension $1$ face $F = \{ y_1 = \epsilon \} \subset \square_k ^1$ with $\epsilon \in \{ 0, \infty \}$, we have
\begin{equation}\label{eqn:n=1 proper intersection}
\codim _{ \Spec (\widehat{R}/ \widehat{I}^m) \times F} \ (  \mathfrak{Z} \cap (\Spec (\widehat{R}/ \widehat{I}^m) \times F)) \geq 1.
\end{equation}
Since $\widehat{R}/\widehat{I}^m$ is Artinian and $\dim \ F = 0$, we have $\dim \ \Spec (\widehat{R}/ \widehat{I}^m) \times F = 0$. Hence \eqref{eqn:n=1 proper intersection} means $\mathfrak{Z} \cap (\Spec (\widehat{R}/ \widehat{I}^m) \times F) = \emptyset$. 

Taking $F= \{y_1 = \infty\}$, we see that the coefficient of the highest $y_1$-degree term of $p_m (y_1)$ is a unit in $(\widehat{R}/\widehat{I}^m)^{\times}$. Since $d_m = \deg_{y_1} p_m (y_1)$, this means $\bar{\alpha}_{d_{m}} \in (\widehat{R}/\widehat{I}^m)^{\times}$. Thus $\alpha_{d_m} \cdot \beta = 1 + x$ for some $\beta \in \widehat{R}$ and $x \in \widehat{I}^m$. However, $\widehat{I}^m$ is contained in the unique maximal ideal of $\widehat{R}$, thus in the Jacobson radical of $\widehat{R}$. Hence $1+x \in \widehat{R}^{\times}$, thus $\alpha_{d_m} \in \widehat{R}^{\times}$, as well.

After applying the above over each $m \geq 1$, we have a sequence of non-decreasing positive integers $d_1 \leq d_2 \leq d_3 \leq \cdots$ such that
$$
\alpha_{d_1},  \alpha_{d_2}, \alpha_{d_3}, \cdots \in \widehat{R}^{\times}.
$$

\medskip

\textbf{Subclaim:} \emph{ The sequence $d_1 \leq d_2  \leq \cdots$ is stationary, i.e. for some $N \geq 1$, we have $d_N = d_{N+1} = \cdots.$}

\medskip

Toward contradiction, suppose not, so there is a strictly increasing sequence $i_1 < i_2 < \cdots $ of indices in $\mathbb{N}$ such that the associated subsequence $d_{i_1} < d_{i_2} < d_{i_3} < \cdots$ in $\mathbb{N}$ increases strictly.

Here the subsequence $\{\alpha_{d_{i_1}}, \alpha_{d_{i_2}} , \cdots \}$ consists of units in $\widehat{R}^{\times}$, so that this subsequence does not converge to $0$ in the $\widehat{I}$-adic topology. But, this violates the given assumption that as $i \to \infty$, we have $\alpha_i \to 0$ in the $\widehat{I}$-adic topology of $\widehat{R}$. This is a contradiction. Hence the sequence $d_1 \leq d_2 \leq \cdots$ must be stationary, proving the Subclaim.

\medskip

The Subclaim says that $p(y_1) \in \widehat{R}\{ y_1 \}$ is in fact a polynomial in $y_1$ of degree $d_N$ (proving the Claim), and its leading coefficient is in $\widehat{R}^{\times}$. 

\medskip

It remains to see that the constant term of $p(y_1)$ in $y_1$ is a unit. But, this follows from the proper intersection condition with $\Spec (\widehat{R}/ \widehat{I}) \times \{ y_1 = 0 \}$, which is the empty intersection. This proves the lemma.
\end{proof}

\begin{cor}\label{cor:n=1 finite}
Let $\mathfrak{Z} = \Spf (\widehat{S}) \in z^1 (\Spf (\widehat{R}), 1)$ be an integral cycle. Then $\widehat{S}$ is finite over $\widehat{R}$.
\end{cor}

\begin{proof}
By Lemma \ref{lem:rest poly 1}, for a monic \emph{polynomial} $p(y_1) \in \widehat{R} \{ y_1 \}$, we have $\widehat{S} = \widehat{R} \{ y_1 \}/ ( p (y_1))$. Let $d:= \deg_{y_1} p(y_1)$. Then for any integer $d' \geq d$, we have $ y_1 ^{ d'} \in {\rm Span}_{\widehat{R}}  \{ 1, y_1, \cdots, y_1 ^{d-1} \}$ in $\widehat{R} \{ y_1 \}/ (p (y_1))$. In particular, $\widehat{S}$ is a finite $\widehat{R}$-module.
\end{proof}

\begin{cor}\label{cor:n=1 v surj}
Let $\mathfrak{Z} \in z^1 (\Spf (\widehat{R}), 1)$ be an integral cycle. Then $\partial_1 ^{\epsilon} \mathfrak{Z}=0$ for $\epsilon \in \{ 0, \infty \}$. In particular $z^1 (\Spf (\widehat{R}), 1)= \ker (\partial: z^1 (\Spf (\widehat{R}), 1) \to z^1 (\Spf (\widehat{R}), 0))$. 
\end{cor}

\begin{proof}
We emulate the argument of \cite[Lemma 2.21, p.1005]{KP sfs}.  Let $\widehat{X}= \Spf (\widehat{R})$. Let $F \subsetneq \square_k ^1$ be a proper face, so that $\dim \ F = 0$. We claim that $\mathfrak{Z} \cap (\widehat{X} \times F) = \emptyset$.

Toward contradiction, suppose $\mathfrak{Z} \cap (\widehat{X} \times F) \not = \emptyset$. By Corollary \ref{cor:n=1 finite}, the morphism $\mathfrak{Z} \to \widehat{X}$ is finite, so that the composite
$$
\mathfrak{Z} \cap (\widehat{X} \times F) \hookrightarrow \mathfrak{Z} \to \widehat{X}
$$
is also finite. Since $\mathfrak{Z} \cap (\widehat{X} \times F) \not = \emptyset$, its image is closed in $\widehat{X}$. Since $\widehat{X}$ has the unique closed point, call it $\mathfrak{m} \in | \widehat{X}|$, we deduce that $\mathfrak{Z} \cap (\mathfrak{m} \times F) \not = \emptyset$.

But $\mathfrak{m}= |\widehat{X}_{\red}| = |Y_{\red}|$, so by the special fiber condition (\textbf{SF}), the intersection $\mathfrak{Z} \cap (\mathfrak{m} \times F)$ is proper, i.e its codimension in $\mathfrak{m} \times F$ is $\geq 1$. However, $\mathfrak{m} \times F$ is of dimension $0$, so that we cannot have a nonempty codimension $1$ closed subset. This is a contradiction. Thus $\mathfrak{Z} \cap (\widehat{X} \times F ) = \emptyset$, proving the first assertion.

The remaining assertions follow immediately from the first one.
\end{proof}

\begin{cor}\label{cor:n=1 free}
Let $\mathfrak{Z}= \Spf (\widehat{S}) \in z^1 (\Spf (\widehat{R}), 1)$ be an integral cycle. Then $\widehat{S}$ is a free $\widehat{R}$-module of finite rank. In particular, there is the norm map $N: \widehat{S}^{\times} \to \widehat{R}^{\times}$.
\end{cor}

\begin{proof}
We saw that $\widehat{S}$ is a complete intersection ring in Lemma \ref{lem:rest poly 1}. Hence it is Cohen-Macaulay (see Bruns-Herzog \cite[Proposition 3.1.20, p.96]{BH} or H. Matsumura \cite[Theorems 18.1, 21.3, p.141, p.171]{Matsumura}). We also saw in Corollary \ref{cor:n=1 finite} that it is a finite $\widehat{R}$-module. Thus ${\rm depth} (\widehat{S}) = \dim \ \widehat{S} = \dim \ \widehat{R}$, and it shows $\widehat{S}$ is a maximal Cohen-Macaulay $\widehat{R}$-module.

Since $\widehat{R}$ is a regular local ring, this implies that $\widehat{S}$ is a free $\widehat{R}$-module by \cite[Lemma 00NT]{stacks}, proving the first assertion.

For the norm map, let $a \in \widehat{S}^{\times}$. Since $\widehat{S}$ is a free $\widehat{R}$-module of finite rank, the left multiplication $L_a: \widehat{S} \to \widehat{S}$ by $a$ defines an $\widehat{R}$-linear homomorphism of free $\widehat{R}$-modules of finite rank. We define the norm $N (a):= \det (L_a)$. Then standard linear algebra arguments show that this is independent of the choice of an $\widehat{R}$-basis of $\widehat{S}$.
\end{proof}

The above norm $N: \widehat{S}^{\times} \to \widehat{R}^{\times}$ in Corollary \ref{cor:n=1 free} can be also computed concretely in the following direct way:

\begin{lem}\label{lem:n=1 surj}
Let $\mathfrak{Z} \in z^1 (\Spf (\widehat{R}), 1)$ be an integral cycle given by a monic polynomial
\begin{equation}\label{eqn:n=1 surj 0}
p(y_1) = y^d + a_{d-1} y^{d-1} + \cdots + a_1 y + a_0, \ \ a_i \in \widehat{R},
\end{equation}
where $a_0 \in \widehat{R}^{\times}$, as in Lemma \ref{lem:rest poly 1}. 

Then modulo the boundary of a cycle in $z^1 (\Spf (\widehat{R}), 2)$, we have the equivalence
$$
\mathfrak{Z} \equiv  [ (-1)^d a_0 ] \ \ \mbox{ in } \CH^1 (\Spf (\widehat{R}), 1)
$$
where the latter is given by the polynomial $y_1 - (-1)^d a_0 \in \widehat{R} \{ y_1 \}$.
\end{lem}

\begin{proof}
We emulate part of B. Totaro \cite[p.186]{Totaro}. Consider the polynomial in $y_1$ and $ y_2$:
$$
Q (y_1, y_2):=p (y_1) - (y_1 -1)^{d-1} (y_1 - (-1)^d a_0) y_2 \in \widehat{R} \{ y_1, y_2 \}.
$$
Let $\mathfrak{W}$ be the cycle defined by $\widehat{R}\{ y_1, y_2 \}/ (Q (y_1, y_2))$. By direct computations, we have
$$
\tuborg 
\partial_1 ^0 \mathfrak{W} = [ \{a_0 (1 - y_2)\}] = [ \{ y_2 -1 \}] = \emptyset, \\
\partial_1 ^{\infty} \mathfrak{W} = [ \{ 1- y_2 \}] = \emptyset, \\
\partial_2 ^{0} \mathfrak{W} = [ \{ p (y_1)\} ] = \mathfrak{Z}, \\
\partial_2 ^{\infty} \mathfrak{W} = [ \{ (y_1 -1)^{d-1} \cdot (y_1 - (-1)^d a_0 ) \}] = [ (-1)^d a_0].\sluttuborg
$$
Using this, one checks that $\mathfrak{W} \in z^1 (\Spf (\widehat{R}), 2)$ as well as that
$$- \mathfrak{Z} + [ (-1)^d a_0] = \partial  (\mathfrak{W}),$$
proving the lemma.
\end{proof}

\begin{cor}
The graph map $gr_{\widehat{R}}: K_1 ^M (\widehat{R}) \to \CH^1 (\Spf (\widehat{R}), 1)$ is surjective. In particular, for an Artin local $k$-algebra $A$, the graph map $gr_A: K_1 ^M (A) \to \BGH^1 (A, 1)$ is surjective.
\end{cor}

\begin{proof}
Since $K_1 ^M (\widehat{R}) = \widehat{R}^{\times}$, the surjectivity of $gr_{\widehat{R}}$ follows from Lemma \ref{lem:n=1 surj}. 

For the surjectivity of $gr_A$, one notes that each class of $\BGH^1 (A, 1)$ is represented by some cycle in $\CH^1 (\Spf (\widehat{R}) \mod Y, 1)$ associated to some embedding $Y \hookrightarrow X$, with $\widehat{X} = \Spf (\widehat{R})$. By Corollary \ref{cor:n=1 v surj}, the map $\CH^q (\Spf (\widehat{R}), 1) \to \CH^1 ( \Spf (\widehat{R}) \mod Y, 1)$ is surjective. So, the map $\tilde{gr}_A$ of the diagram \eqref{eqn:n=1 diag} is surjective. This implies $gr_A$ is surjective. 
\end{proof}

The author guesses that $gr_A: K_1 ^M (A) \to \BGH^1 (A, 1)$ is injective as well, but it wasn't yet done. 

\section{Appendix}\label{sec:appendix}

In \S \ref{sec:normalization}, we prove ``the normalization theorem" for some cubical complexes arising in this article. In \S \ref{sec:appendix localization}, we present a conjectural guess stated as Guess \ref{conj:localization}, which is related to the localization theorem in the usual higher Chow theory, and we discuss its conjectural consequences. While it looks similar, this Guess does not produce a localization type theorem, so it might make sense to call it the ``pseudo-localization". 

\subsection{Normalization of some cubical groups}\label{sec:normalization}

The purpose of \S \ref{sec:normalization} is to prove Theorem \ref{thm:normalization}, that we call the \emph{normalization theorem}, for a few different collections of cubical cycle complexes relevant to those considered in this article. 

\subsubsection{Normalization}\label{sec:n4.1}
The word ``normalization" here, which is unrelated to the same word in Remark \ref{remk:Conrad}, originates from the Dold-Kan correspondence on simplicial abelian groups that the homology groups of the complex $A_{\bullet}$ associated to a simplicial abelian group are equal to the homology groups of the normalized subcomplex $A_{\bullet, N}$, where $A_{n, N}$ consists of the elements $x \in A_n$ all of whose faces except the last one are zero. See J. P. May \cite[Theorem 22.1, p.94]{May}.

The author does not know whether the ``normalization theorem" holds for a cubical abelian group in general, but at least M. Levine \cite[Lemma 1.6]{Levine SM} had identified a sufficient condition that an \emph{extended} cubical abelian group always has this property. Recall (\cite[\S 1]{Levine SM}) that an extended cubical abelian group is a functor $\textbf{ECube}^{\rm op} \to ({\rm Ab})$, where $\textbf{ECube}$ is the smallest symmetric monoidal subcategory of the category of sets containing the category $\textbf{Cube}$, and the morphism $\mu: \un{2} \to \un{1}$.


\medskip

One of the examples where the normalization theorem holds is the following: let $Y$ be a $k$-scheme of finite type. For the cubical higher Chow complex $z^q (Y, \bullet)$, let $z^q _N (Y, n) \subset z^q (Y, n)$ be the subgroup generated by cycles $Z$ such that $\partial_i ^0 (Z) = 0$ for all $1 \leq i \leq n $ and $\partial_i ^{\infty} (Z) = 0$ for $2 \leq i \leq n$. With the ``last" face $\partial_1 ^{\infty}$, this gives the \emph{normalized subcomplex} $(z_N ^q (Y, \bullet), \partial_1 ^{\infty}) \hookrightarrow (z^q (Y, \bullet), \partial)$, and this inclusion is a quasi-isomorphism. One can find a reference in the published literature at M. Li \cite[Theorem 2.6]{Li}, which is based on an argument of Bloch's notes \cite[Theorem 4.4.2]{Bloch note}.

For a pair $(Y, D)$ consisting of a smooth $k$-scheme and an effective Cartier divisor, similar results are known for the higher Chow cycles $z^q (Y|D, \bullet)$ with modulus. See \cite[Theorem 3.2]{KP DM}.

We study cases relevant to this paper in \S \ref{sec:4.2} below.

\subsubsection{The normalization theorems}\label{sec:4.2}

Let $Y$ be a quasi-affine $k$-scheme of finite type and let $Y \hookrightarrow X$ be a closed immersion into an equidimensional smooth $k$-scheme. Let $\widehat{X}$ be the completion of $X$ along $Y$. Let $U \subset Y$ be a nonempty open subset and let $\widehat{X}|_U \subset \widehat{X}$ be the corresponding quasi-affine open formal subscheme.

\begin{defn}\label{defn:normal cx}
Consider the following types of complexes and their normalized subcomplexes:
\begin{enumerate}
\item The subcomplex $z^q _N (\widehat{X}, \bullet) \subset z^q (\widehat{X} , \bullet)$.
\item The subcomplex $z^q _N (\widehat{X} \mod Y, n) \subset z^q (\widehat{X} \mod Y, n)$.
\item For the restriction map $\rho ^Y _U (\bullet): z^q (\widehat{X}, \bullet) \to z^q (\widehat{X}|_U, \bullet)$, the subcomplex $\mathcal{C}_{\bullet, N}$ of the quotient complex
$$
\mathcal{C}_{\bullet}:= {\rm coker} (\rho^Y_U)= \frac{ z^q (\widehat{X}|_U, \bullet)}{ {\rm im} \ (\rho ^Y _U (\bullet))}.
$$
\item For the restriction map $\rho ^Y _U (\bullet): z^q (\widehat{X} \mod Y, \bullet) \to z^q (\widehat{X}|_U \mod U, \bullet)$, where the last $U$ is seen as the open subscheme of $Y$, the subcomplex $\mathcal{C}_{\bullet, N}'$ of the quotient complex
$$
\mathcal{C}_{\bullet}':=  {\rm coker} (\rho^Y_U)=  \frac{ z^q (\widehat{X}|_U \mod U, \bullet)}{ {\rm im} \ (\rho ^Y _U (\bullet))}.
$$
\end{enumerate}
For instance, for (1), $z^q _N (\widehat{X}, n) \subset z^q (\widehat{X} , n)$ is the subgroup of cycles $\mathfrak{Z}$ such that $\partial _i ^{0} (\mathfrak{Z}) = 0$ for all $1 \leq i \leq n$ and $\partial_i ^{\infty} (\mathfrak{Z}) = 0$ for $2 \leq i \leq n$ in $z^q (\widehat{X}, n-1)$. One checks that $\partial_1 ^{\infty} \circ \partial_1 ^{\infty} = 0$. 

For (2)$\sim$(4), we have the similar membership conditions.
\qed
\end{defn}

\medskip

The main result of \S \ref{sec:normalization} is the following:

\begin{thm}\label{thm:normalization}
Under the above notations, the complexes 
$$
z^q (\widehat{X}, \bullet),  \ \ z^q (\widehat{X} \mod Y, \bullet), \ \ \mathcal{C}_{\bullet}, \ \ \mathcal{C}_{\bullet}'
$$
are extended cubical abelian groups, and the corresponding inclusions of complexes
\begin{enumerate}
\item $z^q _N (\widehat{X}, \bullet) \hookrightarrow z^q (\widehat{X}, \bullet)$,
\item $z^q _N (\widehat{X} \mod Y, \bullet) \hookrightarrow z^q (\widehat{X} \mod Y, \bullet)$,
\item $\mathcal{C}_{\bullet, N} \hookrightarrow \mathcal{C}_{\bullet}$,
\item $\mathcal{C}_{\bullet, N}' \hookrightarrow \mathcal{C}_{\bullet}'$
\end{enumerate}
are quasi-isomorphisms.
\end{thm}

\begin{proof}
For simplicity, via the automorphism $\psi: \mathbb{P}_k ^1 \to \mathbb{P}_k ^1$, $ y \mapsto \frac{y}{y-1}$, we identify $(\square, \{ 0, \infty \})$ with $(\square_{\psi}, \{ 0, 1 \})$, where $\square_{\psi} := \mathbb{A}^1$. On $\square_{\psi}^n$, the faces are given by intersecting the hyperplanes $\{ y_i = \epsilon_i\}$, where $\epsilon_i \in \{ 0, 1 \}$, and we may assume the cycles in $z^q (\widehat{X}, n)$ are on $\widehat{X} \times \square_{\psi}^n$, satisfying the corresponding conditions as in Definition \ref{defn:HCG}, except the faces are given as the above.

\medskip

Consider the morphism $ \mu: \square_{\psi}^2 \to \square_{\psi}^1$ given by $(y_1, y_2)\mapsto y_1y_2$. This corresponds to the morphism $\mu: \un{2} \to \un{1}$ of \textbf{ECube}. This induces various similar morphisms e.g. for $1 \leq i <  n+1$
$$
\mu_{i} : \square_{\psi}^{n+1} \to \square_{\psi}^n,
$$
that send $(y_i, y_{i+1}) \mapsto y_i y_{i+1}$, while $y_{i'} \mapsto y_{i'}$ for $i' \not = i, i+1$.

Since the complexes already come from cubical abelian groups, to prove that they are extended cubical abelian groups, we need to show that for a cycle $\mathfrak{Z} \in z^q (\widehat{X}, n)$ (resp. $z^q (\widehat{X} \mod Y, n)$, $\mathcal{C}_n$, $\mathcal{C}_n'$), we have $\mu^*_{i} (\mathfrak{Z}) \in z^q (\widehat{X}, n+1)$ (resp. $z^q (\widehat{X} \mod Y, n+1)$, $\mathcal{C}_{n+1}$, $\mathcal{C}_{n+1}'$), where $\mu^* (\mathfrak{Z})$ is given as the type (IV) flat pull-back in Lemma \ref{lem:pre flat pb0}.

\medskip

(1): For $z^q (\widehat{X}, n)$, this is proven separately in Lemma \ref{lem:pull-back mu 1} below.

\medskip

(2): For $z^q (\widehat{X} \mod Y, n)$, given that (1) holds, it follows from that $\mu ^*_{i}$ respects the mod $Y$-equivalence. The latter holds because the morphism $\mu_{i}$ does not disturb anything on the factor $\widehat{X}$ of $\widehat{X} \times \square_{\psi}^{\bullet}$.

\medskip

(3), (4): For $\mathcal{C}_n$ and $\mathcal{C}_n'$, the assertions follow from (1) and (2), and the commutative diagrams
$$
\xymatrix{ 
z^q (\widehat{X}, n+1) \ar[r] ^{\rho^Y_U} & z^q (\widehat{X}|_U, n+1) \\
z^q (\widehat{X}, n) \ar[r] ^{\rho^Y_U} \ar[u] ^{\mu ^*_{i}} & z^q (\widehat{X}|_U, n), \ar[u] ^{\mu ^*_{i}}}
\hskip0.3cm
\xymatrix{ 
z^q (\widehat{X} \mod Y, n+1) \ar[r] ^{\rho^Y_U} & z^q (\widehat{X}|_U \mod U, n+1) \\
z^q (\widehat{X} \mod Y, n) \ar[r] ^{\rho^Y_U} \ar[u] ^{\mu ^*_{i}} & z^q (\widehat{X}|_U \mod U, n), \ar[u] ^{\mu ^*_{i}}}
$$
of flat pull-backs of type (IV).

As remarked in \S \ref{sec:n4.1}, once the complexes are extended cubical abelian groups, the second part of the theorem follows by M. Levine \cite[Lemma 1.6]{Levine SM}. 
\end{proof}

We used the following in the proof of Theorem \ref{thm:normalization}. 

\begin{lem}\label{lem:pull-back mu 1}
Via the automorphism $\psi: \mathbb{P}^1 \to \mathbb{P}^1$, $y \mapsto \frac{y}{y-1}$, we identify $(\square, \{ 0, \infty\})$ with $(\square_{\psi}= \mathbb{A}^1, \{ 0, 1 \})$. Suppose we defined $z^q (\widehat{X}, n)$ using the corresponding faces of $\square_{\psi}^n$. Let $1 \leq i < j \leq n+1$.
 
Consider $\mu: \square_{\psi} ^2 \to \square_{\psi}^1$ given by $(y, y') \mapsto y y'$, which induces the type ${\rm (IV)}$ flat morphisms (in the sense of Lemma \ref{lem:pre flat pb0}) for $1 \leq i < n+1$
$$
\mu_{i}: \widehat{X} \times \square_{\psi} ^{n+1} \to \widehat{X} \times \square_{\psi}^n
$$
that send $(y_i, y_{i+1})\mapsto y_i y_{i+1}$ and $y_{i'} \mapsto y_{i'}$ for $i' \not = i, i+1$.

Then for each $\mathfrak{Z} \in z^q (\widehat{X}, n)$, the pull-back $\mu^*_{i} (\mathfrak{Z}) \in \un{z}^q (\widehat{X} \times \square_{\psi}^{n+1})$ belongs to $z^q (\widehat{X}, n+1)$. 
\end{lem}

\begin{proof}
We may assume that $\mathfrak{Z}$ is integral. It satisfies the properties (\textbf{GP}), (\textbf{SF}) of Definition \ref{defn:HCG}. The question is whether $\mu^*_{i} (\mathfrak{Z})$ also enjoys the three properties.

Without loss of generality, we may assume $i=n$ so that we consider 
$$\tilde{\mu}^n:= {\rm Id}_{\square_{\psi}^{n-1}} \times \mu: \square_{\psi}^{n+1} \to \square_{\psi}^n,
$$
given by $(y_1, \cdots, y_n, y_{n+1}) \mapsto (y_1, \cdots, y_{n-1}, y_n y_{n+1})$.

\medskip

Let's first inspect (\textbf{GP}). We check the case of the codimension $1$ faces first.

\textbf{Case 1:} Suppose $1 \leq i \leq n-1$. The face map $\partial_i ^{\epsilon}$ and $\mu$ do not interact at all so that $\partial_i ^{\epsilon}( \tilde{\mu}^n)^* (\mathfrak{Z}) = (\tilde{\mu}^{n-1})^* ( \partial_i ^{\epsilon} (\mathfrak{Z}))$, and it has the right codimension. 

\textbf{Case 2:} Suppose $i=n$. When $\epsilon=1$, we have the composite 
$$
\square_{\psi} ^{n} \overset{\iota_{n} ^1}{\hookrightarrow} \square_{\psi}^{n+1} \overset{\tilde{\mu}^n}{\to} \square_{\psi}^n, \ \ (y_1, \cdots, y_n) \mapsto (y_1, \cdots, y_{n-1}, 1, y_n) \mapsto (y_1, \cdots, y_n),
$$
which is the identity. Thus $\partial_n ^1 ( \tilde{\mu}^n)^* (\mathfrak{Z}) = \mathfrak{Z}$, which also has the right codimension. 

When $\epsilon=0$, we have the composite
$$
\square_{\psi} ^n \overset{\iota_{n}^0}{\hookrightarrow} \square_{\psi}^{n+1} \overset{\tilde{\mu}^n}{\to} \square_{\psi}^n, \ \ 
(y_1, \cdots, y_n) \mapsto (y_1, \cdots, y_{n-1}, 0, y_n) \mapsto (y_1, \cdots, y_{n-1}, 0),
$$
and it is equal to the composite
$$
\square_{\psi}^n \overset{pr_n}{\to} \square_{\psi}^{n-1} \overset{\iota_n^0}{\hookrightarrow } \square_{\psi}^n.
$$
Thus $\partial_n ^0 (\tilde{\mu}^n)^* (\mathfrak{Z}) = pr_n ^* (\partial_n ^0 (\mathfrak{Z}))$. Since this is a degenerate cycle, it is $0$ in $z^q (\widehat{X}, n)$.

\textbf{Case 3:} The case when $i=n+1$ is identical to $i=n$ by symmetry, because $\mu(y, y') = \mu (y', y) = y y'$.

\medskip

For higher codimension faces, we repeat and combine the above codimension $1$ cases, and we deduce that for any face $F \subset \square_{\psi}^{n+1}$, the intersection of $(\tilde{\mu}^n)^* (\mathfrak{Z})$ with $\widehat{X} \times F$ is proper.

\medskip

Let's check the condition (\textbf{SF}) for $\mu^* (\mathfrak{Z})$. 

Note first that by the condition (\textbf{SF}) of $\mathfrak{Z}$, the intersection $Z:=\mathfrak{Z} \cap (\widehat{X}_{\red} \times \square_{\psi}^n)$ has the codimension $q$ in $\widehat{X}_{\red} \times \square_{\psi}^n$. Since this intersection is a cycle on the scheme $\widehat{X}_{\red} \times \square_{\psi}^n$, and $\tilde{\mu}^n: \widehat{X}_{\red} \times \square_{\psi}^{n+1} \to \widehat{X}_{\red} \times \square_{\psi}^n$ is flat, we can take the usual flat pull-back $(\tilde{\mu}^n)^* (Z)$ for cycles on schemes. This is of codimension $q$ in $\widehat{X}_{\red} \times \square_{\psi}^{n+1}$. Since 
$$
(\tilde{\mu}^n)^* (Z) = (\tilde{\mu}^n)^* (\mathfrak{Z} \cap (\widehat{X}_{\red} \times \square_{\psi}^n))= (\tilde{\mu}^n)^* (\mathfrak{Z}) \cap (\widehat{X}_{\red} \times \square_{\psi}^{n+1}),
$$
we see that $ (\tilde{\mu}^n)^* (\mathfrak{Z}) \cap (\widehat{X}_{\red} \times \square_{\psi}^{n+1})$ has the codimension $q$ in $\widehat{X}_{\red} \times \square_{\psi}^{n+1}$.

Once we know it, we can inspect the intersections with the faces $\widehat{X}_{\red} \times F$, as we did in checking (\textbf{GP}). The same argument shows that the intersection $(\tilde{\mu}^n)^* (Z) \cap (\widehat{X}_{\red} \times F)$ is proper for each face $F$. Thus $(\tilde{\mu}^n)^*(\mathfrak{Z})$ satisfies (\textbf{SF}). 

This shows that $(\tilde{\mu}^n)^* (\mathfrak{Z}) \in z^q (\widehat{X}, n+1)$ as desired.
\end{proof}

\subsubsection{An application}\label{sec:permutation}
For each permutation $g \in \mathfrak{S}_n$, we have the induced action $g* (y_1, \cdots, y_n) = (y_{g (1)}, \cdots, y_{g (n)})$ for $(y_1, \cdots, y_n) \in \square^n$. This induces the natural actions of $\mathfrak{S}_n$ on $z^q (\widehat{X}, n)$ and $z^q (\widehat{X} \mod Y, n)$ as well. 
An application of Theorem \ref{thm:normalization} is the following, essentially from \cite[Lemma 5.3]{KP DM}. Theorem \ref{thm:normal permutation} is used in this article only for the graded commutativity of the product structure in \S \ref{sec:product}.

\begin{thm}\label{thm:normal permutation}
Let $Y$ be a quasi-affine $k$-scheme of finite type and let $Y \hookrightarrow X$ be a closed immersion into an equidimensional smooth $k$-scheme. Let $\widehat{X}$ be the completion of $X$ along $Y$.

Let $\mathfrak{Z} \in z^q (\widehat{X} , n)$ (resp. $\in z^q (\widehat{X} \mod Y , n)$) be a cycle such that $\partial_i ^{\epsilon} (\mathfrak{Z}) = 0$ for all $1 \leq i \leq n$ and $\epsilon \in \{ 0, \infty \}$. Let $g \in \mathfrak{S}_n$ be a permutation. 

Then we have the equivalence of the normalized cycles
\begin{equation}\label{eqn:normal perm}
g * \mathfrak{Z}  \equiv \sgn (g) \circ \mathfrak{Z}
\end{equation}
 modulo the boundary of a cycle from $z^q (\widehat{X}, n+1)$ (resp. $z^q (\widehat{X} \mod Y, n+1)$).
\end{thm}

\begin{proof}
The proof is almost identical to \cite[Lemma 5.3]{KP DM} (cf. \cite[Lemma 3.16]{KP2}). We sketch it following \cite{KP DM}. 

We first consider the case when $g \in \mathfrak{S}_n$ is the transposition of the form $\tau= (p, p+1)$ for some $1 \leq p \leq n-1$. We do it for $p=1$ only, i.e. $\tau= (1, 2)$. The other cases are similar. 

Consider the rational map $\nu: \widehat{X} \times \square ^{n} \times \square ^1 \to \widehat{X} \times \square^{n+1}$ given by 
$$
(x, y_1, \cdots, y_n)\times (y) \mapsto \left( x, y, y_2,  \frac{ y - y_1}{y-1}, y_3, \cdots, y_n\right).
$$
 For a given normalized cycle $\mathfrak{Z}$, we let $\gamma_{\mathfrak{Z}} ^\tau$ be the cycle given by the Zariski closure of $\nu (\mathfrak{Z} \times \square^1)$. 

One checks that $\partial_1 ^0 (\gamma_{\mathfrak{Z}} ^{\tau}) = \tau* \mathfrak{Z}$, $\partial_1 ^{\infty} (\gamma_{\mathfrak{Z}} ^{\tau}) = 0$, $\partial_3 ^0 (\gamma_{\mathfrak{Z}} ^{\tau}) =\mathfrak{Z}, $ and $\partial_3 ^{\infty} (\gamma_{\mathfrak{Z}} ^{\tau})= 0$, while $\partial_i ^{\epsilon} (\gamma_{\mathfrak{Z}} ^{\tau}) = 0$ for all $i \not = 1, 3$ and $\epsilon \in \{ 0, \infty\}$. Hence $\partial ( \gamma_{\mathfrak{Z}} ^{\tau}) = \tau * \mathfrak{Z} +  \mathfrak{Z}$ as desired because $\sgn (\tau) = -1$. This answers the question for $g=\tau$ is a transposition of the form $(p, p+1)$.

\medskip

Now let $g \in \mathfrak{S}_n$ be any permutation. By a basic result from the group theory, one can express $g = \tau_{r} \tau _{r-1} \cdots \tau_2 \tau_1$ for some transpositions $t_i$ of the form $(p, p+1)$ considered before. We let $g_0 := {\rm Id}$, $g_{\ell} = \tau_{\ell} \circ \tau_{\ell-1} \cdots \tau_1$ for $1 \leq \ell \leq r$. For each such $\ell$, by the first step we have
$$
(-1)^{\ell-1} g_{\ell-1} * \mathfrak{Z} + (-1)^{\ell -1} \tau_{\ell} * (g_{\ell-1} * \mathfrak{Z} )  = \partial ( (-1)^{\ell -1} \gamma_{g_{\ell-1} * \mathfrak{Z}} ^{\tau_{\ell}  }) .
$$ 
Since $\tau_{\ell} g_{\ell-1} = g_{\ell}$, by taking the sum of the above over all $1 \leq \ell \leq r$, after cancellations we obtain 
\begin{equation}\label{eqn:normal perm 0}
\mathfrak{Z} + (-1)^{r-1} g * \mathfrak{Z}= \partial (\gamma_{\mathfrak{Z}} ^{g}),
\end{equation}
 where 
 $$
 \gamma_{\mathfrak{Z}} ^g:= \sum_{\ell=1} ^{r} (-1)^{\ell -1} \gamma_{g_{\ell -1}*\mathfrak{Z}} ^{\tau_{\ell}}.
 $$
  Since $(-1)^r = \sgn (g)$, \eqref{eqn:normal perm 0} implies \eqref{eqn:normal perm}.
\end{proof}

\subsection{A conjectural pseudo-localization}\label{sec:appendix localization}
In \S \ref{sec:appendix localization}, we discuss a conjectural guess presented as Guess \ref{conj:localization}, that resembles the proof of the localization theorem of S. Bloch \cite{Bloch moving} and M. Levine \cite{Levine moving}. 

The author was yet unable to prove it, though we discuss a few conjectural consequences of it. 

\begin{conj}\label{conj:localization}
Let $Y$ be a connected quasi-affine $k$-scheme of finite type. Let $Y \hookrightarrow X$ be a closed immersion into an equidimensional smooth $k$-scheme, and let $\widehat{X}$ be the completion of $X$ along $Y$. Let $U \subset |Y|$ be a nonempty open subset.

Consider the restriction morphisms 
\begin{equation}\label{eqn:localization}
 \tuborg z^q (\widehat{X} , \bullet) \to z^q (\widehat{X}|_U, \bullet),\\
 z^q (\widehat{X} \mod Y, \bullet) \to z^q (\widehat{X}|_U \mod U, \bullet), \sluttuborg
\end{equation}
where $\widehat{X}|_U$ is the quasi-affine open formal subscheme $(U, \mathcal{O}_{\widehat{X}}|_U)$ of $\widehat{X}$, and the last $U$ in \eqref{eqn:localization} is seen as the quasi-affine open subscheme $(U, \mathcal{O}_Y|_U)$ of $Y$.

Then their respective cokernels $\mathcal{C}_{\bullet}$ and $\mathcal{C}_{\bullet} '$ are acyclic.
\end{conj}

\begin{remk}\label{remk:localization artin}
In case $Y$ is an Artin local $k$-algebra, the above Guess \ref{conj:localization} holds trivially because any nonempty open $U \subset |Y|$ contains all of $|Y|$ already. \qed
\end{remk}

Conjectural consequences of Guess \ref{conj:localization} (together with Proposition \ref{prop:flasque sh}) are the following:

\begin{prop}\label{prop:flasque rep}
Assume Guess \ref{conj:localization}. Let $Y$ be a quasi-affine $k$-scheme of finite type. Let $Y \hookrightarrow X$ be a closed immersion into an equidimensional smooth $k$-scheme and let $\widehat{X}$ be the completion of $X$ along $Y$.

Then the natural injective morphisms of complexes of sheaves on $Y_{\rm Zar}$ of Lemma \ref{lem:flasque natural}
\begin{equation}\label{eqn:SP_sh 01}
\tuborg 
\tilde{\mathcal{S}}_{\bullet}^q  \to \BGHz^q (\widehat{X}, \bullet), \\
\mathcal{S}_{\bullet}^q  \to \BGHz^q (\widehat{X} \mod Y, \bullet),
\sluttuborg
\end{equation}
are quasi-isomorphisms.
\end{prop}

\begin{proof}
If $Y$ is not connected, then we may work with the individual connected components of $Y$. So, we may now assume $Y$ is connected.

Recall that in the proof of Lemma \ref{lem:flasque natural}, we had the injective homomorphisms of complexes of presheaves on $Y$
\begin{equation}\label{eqn:global qi-local2 01}
\tilde{\mathcal{S}}_{\bullet}^q  \to \tilde{\mathcal{P}}_{\bullet} ^q \ \ \mbox{and} \ \ \mathcal{S}_{\bullet}^q  \to \mathcal{P}_{\bullet} ^q,
\end{equation}
whose sheafifications give the injective morphisms of \eqref{eqn:SP_sh 01}.

By Guess \ref{conj:localization}, the cokernels of \eqref{eqn:global qi-local2 01} are acyclic over each nonempty open subset $U \subset Y$. Hence the morphisms \eqref{eqn:global qi-local2 01} are quasi-isomorphisms over each nonempty open $U \subset Y$. These quasi-isomorphisms induce quasi-isomorphisms of the morphisms of the stalks of \eqref{eqn:SP_sh 01} at all scheme points $y \in Y$. Thus, the morphisms \eqref{eqn:SP_sh 01} are quasi-isomorphisms of complexes of sheaves. 
\end{proof}

The above offers descriptions of the hypercohomology groups $\BGH^q (\widehat{X}, n)$ and $\BGH^q (\widehat{X}\mod Y, n)$ in Definition \ref{defn:BGH quasi-affine} in terms of the concrete cycle class groups $\CH^q (\widehat{X}, n)$ and $\CH^q (\widehat{X} \mod Y, n)$ in Definitions \ref{defn:HCG2} and \ref{defn:complex}:

\begin{thm}\label{thm:two BGH}
Assume Guess \ref{conj:localization}. Let $Y$ be a quasi-affine $k$-scheme of finite type. Let $Y \hookrightarrow X$ be a closed immersion into an equidimensional smooth $k$-scheme and let $\widehat{X}$ be the completion of $X$ along $Y$. 

Then we have isomorphisms
\begin{equation}\label{eqn:semi-loc case}
\tuborg
  \CH^q  (\widehat{X}, n)= \BGH^q (\widehat{X}, n), \\
 \CH^q  (\widehat{X} \mod Y, n) =  \BGH^q (\widehat{X} \mod Y, n).
\sluttuborg
\end{equation}
\end{thm}

\begin{proof}
By Theorem \ref{thm:two BGH -1}, we had
$$
\tuborg
 \mathbb{H}_{\rm Zar} ^{-n} (Y, \tilde{\mathcal{S}}_{\bullet}^q)= \CH^q  (\widehat{X}, n), \\
\mathbb{H}_{\rm Zar} ^{-n} (Y, \mathcal{S}_{\bullet}^q)= \CH^q  (\widehat{X} \mod Y, n).
\sluttuborg
$$

Since the morphisms $\tilde{\mathcal{S}}_{\bullet}^q  \to \BGHz^q (\widehat{X}, \bullet)$ and $\mathcal{S}_{\bullet}^q  \to \BGHz^q (\widehat{X} \mod Y, \bullet)$ in \eqref{eqn:SP_sh} are quasi-isomorphisms (Proposition \ref{prop:flasque rep}), we now deduce \eqref{eqn:semi-loc case} as desired.
\end{proof}

\medskip


\begin{thm}\label{thm:two Cech}
Assume Guess \ref{conj:localization}. Let $Y \in \Sch_k$, and let $\mathcal{U} = \{ (U_i, X_i)\}_{i \in \Lambda}$ be a system of local embeddings for $Y$. 

Then the groups in Definitions \ref{defn:Cech cx} and \ref{defn:Cech cycle cx} are isomorphic to each other, respectively:
\begin{equation}\label{eqn:two Cech main}
\tuborg
\CH^q (\mathcal{U}^{\infty}, n) = \BGH^q (\mathcal{U}^{\infty}, n) , \\
\CH^q (\mathcal{U}, n) = \BGH^q (\mathcal{U}, n).
\sluttuborg
\end{equation}
\end{thm}

\begin{proof}
By Lemma \ref{lem:flasque natural cech}, and in terms of the notations there, we have the natural morphisms of complexes of sheaves
$$
\tuborg
\tilde{\mathcal{S}}_{\bullet}^q (\widehat{X}_I ) \to \BGHz^q (\widehat{X}_I, \bullet), \\
\mathcal{S}_{\bullet}^q (\widehat{X} _I \mod U_I ) \to \BGHz^q (\widehat{X}_I \mod U_I, \bullet),
\sluttuborg
$$
that are quasi-isomorphisms by Proposition \ref{prop:flasque rep} under Guess \ref{conj:localization}. Hence the morphisms
\begin{equation}\label{eqn:two Cech sheaves}
\tuborg
\underset{{I \in \Lambda^{p+1}} }{\prod}  \mathbf{R} \iota_* \tilde{ \mathcal{S}}_{\bullet} ^q (\widehat{X}_I) \to \underset{{I \in \Lambda^{p+1}} }{\prod} \mathbf{R}  \iota_* \BGHz^q (\widehat{X}_I, \bullet),\\
\underset{{I \in \Lambda^{p+1}} }{\prod}  \mathbf{R}  \iota_* \mathcal{S}_{\bullet} ^q (\widehat{X}_I \mod U_I) \to \underset{{I \in \Lambda^{p+1}} }{\prod}  \mathbf{R}  \iota_* \BGHz^q (\widehat{X}_I \mod U_I, \bullet)
\sluttuborg
\end{equation}
are isomorphisms in $\mathcal{D}^- ({\rm Ab} (Y))$ as well, where $\iota: U_I \hookrightarrow Y$ are the open immersions.

\medskip

We repeat the \v{C}ech construction of  \S \ref{sec:Cech complex sheaf} with the sheaves $\tilde{\mathcal{S}}_{\bullet}^q (\widehat{X}_I )$ over all $I \in \Lambda^{p+1}$ and $p \geq 0$, instead of $\BGHz^q(\widehat{X}_I,\bullet)$ (resp.  $\mathcal{S}_{\bullet}^q (\widehat{X}_I \mod U_I )$ instead of 
 $\BGHz^q(\widehat{X}_I \mod U_I,\bullet)$).

 As did before in Lemma \ref{lem:flasque natural cech}, let $\check{\mathfrak{C}} (\mathcal{U}^{\infty}, \tilde{\mathcal{S}}^q)$ and $\check{\mathfrak{C}} (\mathcal{U}^{\infty}, \BGHz^q)$ denote the so-obtained double complexes of sheaves, respectively, and similarly, define $\check{\mathfrak{C}} (\mathcal{U}, \mathcal{S}^q)$ and $\check{\mathfrak{C}} (\mathcal{U}, \BGHz^q)$. Since \eqref{eqn:two Cech sheaves} are isomorphisms in $\mathcal{D}^- ({\rm Ab} (Y))$ for $p \geq 0$, the induced morphisms
\begin{equation}\label{eqn:tot two Cech sheaves}
\tuborg
{\rm Tot}\  \check{\mathfrak{C}} (\mathcal{U}^{\infty}, \tilde{\mathcal{S}}^q) \to {\rm Tot}\  \check{\mathfrak{C}} (\mathcal{U}^{\infty}, \BGHz^q),\\
{\rm Tot}\  \check{\mathfrak{C}} (\mathcal{U}, \mathcal{S}^q) \to {\rm Tot}\  \check{\mathfrak{C}} (\mathcal{U}, \BGHz^q)
\sluttuborg
\end{equation}
are also isomorphisms in $\mathcal{D} ({\rm Ab}(Y))$. In addition, we note that the sheaves $\tilde{\mathcal{S}}_{n} ^q (\widehat{X}_I)$ and $\mathcal{S}_{n} ^q (\widehat{X}_I \mod U_I)$ are flasque by Proposition \ref{prop:flasque sh}. Hence we have
$$
\tuborg
\mathbf{R} \iota_* \tilde{\mathcal{S}}_{\bullet} ^q (\widehat{X}_I) = \iota_*  \tilde{\mathcal{S}}_{\bullet} ^q (\widehat{X}_I) , \\
\mathbf{R} \iota_* {\mathcal{S}}_{\bullet} ^q (\widehat{X}_I) = \iota_*  {\mathcal{S}}_{\bullet} ^q (\widehat{X}_I) , 
\sluttuborg
$$
so that  ${\rm Tot}\  \check{\mathfrak{C}} (\mathcal{U}^{\infty}, \tilde{\mathcal{S}}^q)$ and ${\rm Tot}\  \check{\mathfrak{C}} (\mathcal{U}, \mathcal{S}^q)$ are complexes of flasque sheaves. 

We saw in Lemma \ref{lem:flasque natural cech} that 
\begin{equation}\label{eqn:tot two Cech 02}
\tuborg
\CH^q (\mathcal{U}^{\infty}, n) = \mathbb{H}_{\rm Zar} ^{-n} (Y, {\rm Tot}  \ \check{\mathfrak{C}} (\mathcal{U}^{\infty}, \tilde{\mathcal{S}}^q)), \\
\CH^q (\mathcal{U}, n) = \mathbb{H}_{\rm Zar} ^{-n} (Y, {\rm Tot}  \ \check{\mathfrak{C}} (\mathcal{U}, \mathcal{S}^q)).
\sluttuborg
\end{equation}

Since we have
\begin{equation}\label{eqn:tot two Cech sheaves homo comp}
\tuborg
\BGH^q (\mathcal{U}^{\infty}, n) = \mathbb{H}_{\rm Zar} ^{-n} (Y, {\rm Tot}\  \check{\mathfrak{C}} (\mathcal{U}^{\infty}, \BGHz^q)) \simeq ^{\dagger} \mathbb{H}_{\rm Zar} ^{-n} (Y, {\rm Tot}  \ \check{\mathfrak{C}} (\mathcal{U}^{\infty}, \tilde{\mathcal{S}}^q)) \\
\BGH^q (\mathcal{U}, n) = \mathbb{H}_{\rm Zar} ^{-n} (Y, {\rm Tot}\  \check{\mathfrak{C}} (\mathcal{U}, \BGHz^q)) \simeq ^{\dagger} \mathbb{H}_{\rm Zar} ^{-n} (Y, {\rm Tot}  \ \check{\mathfrak{C}} (\mathcal{U}, \mathcal{S}^q)),
\sluttuborg
\end{equation}
where $\dagger$ hold because \eqref{eqn:tot two Cech sheaves} are quasi-isomorphisms, combined with \eqref{eqn:tot two Cech 02}, we deduce \eqref{eqn:two Cech main}. 
\end{proof}

 \medskip

\noindent\emph{Acknowledgments.} 
During the past several years of this work, JP benefited directly or indirectly from conversations and correspondences with numerous people. Specifically he thanks Donu Arapura, Aravind Asok, Joseph Ayoub, Federico Binda, Thomas Geisser, Rob de Jeu, Bruno Kahn, Steven Landsburg, Marc Levine, Leovigildo Alonso Tarr\'io, Fumiharu Kato, Shane Kelly, Youngsu Kim, Pablo Pelaez, and Takao Yamazaki. JP thanks Chang-Yeon Chough and Minhyong Kim for inspiring conversations around derived algebraic geometry, that influenced the author's perspective on the project.

JP would also like to mention that part of the inspirations of the work were conceived while he was visiting Sinan \"Unver in Berlin and in \.{I}stanbul during the past collaborations. For some time, \"Unver was a Humboldt fellow at the workgroup Esnault at Freie Universit\"at Berlin. JP wishes to thank H\'el\`ene Esnault, Kay R\"ulling, and Sinan \"Unver for their hospitality during his visits.

JP wishes to thank Spencer Bloch for suggesting him the paper of Hartshorne on algebraic de Rham cohomology almost two decades ago, from which some new ideas sprang up.
JP thanks Damy for being the source of joy of life and Seungmok Yi for helping him in enduring the hard process of writing this long article.

During this work, JP was supported by the National Research Foundation of Korea (NRF) grant (2018R1A2B6002287) funded by the Korean government (Ministry of Science and ICT).


\begin{thebibliography}{99}



\bibitem{Leo AT} L. Alonso Tarr\'io, A. Jerem\'ias L\'opez, M. P\'erez Rodr\'iguez, and M. J. Vale Gonsalves, {\sl On the existence of a compact generator on the derived category of a noetherian formal scheme\/}, Appl. Categor. Struct, \textbf{19}, (2011), 865--877.


\bibitem{Andre} M. Andr\'e, {\sl Homologie des alg\`ebres commutatives\/}, Grundlehren Math. Wissen. \textbf{206}, Springer-Verlag, 1974.



\bibitem{AH} M. F. Atiyah and F. Hirzebruch, {\sl Vector bundles and homogeneous spaces\/}, Proc. Sympos. Pure Math. \textbf{3}, 1961. Amer. Math. Soc. Providence, R.I. pp.7--38. 

\bibitem{AM} M. F. Atiyah and I. G. MacDonald, {\sl Introduction to commutative algebra\/}, Addison-Wesley, 1969. ix+128 pp.


\bibitem{Bass} H. Bass, {\sl Algebraic $K$-theory\/}, W. A. Benjamin, Inc., New York-Amsterdam, 1968, xx+762 pp.


\bibitem{Beilinson Soule} A. Beilinson, {\sl Letter to C. Soul\'e, dated Nov. 1, 1982}, article 0694 of $K$-theory preprint server at the University of Illinois, Urbana-Champaign.

\bibitem{BBD} A. Beilinson, J. Bernstein, P. Deligne, and O. Gabber, {\sl Faisceaux pervers (2nd Edition)\/}, Ast\'erisque, \textbf{100}, (2018), Soc. Math. France.




\bibitem{Berthelot} P. Berthelot, {\sl Cohomologie cristalline des sch\'emas de caract\'eristique $p>0$\/}, Lect. Notes in Math., \textbf{407}, Springer-Verlag, Berlin-New York, 1974, 604 pp. \


\bibitem{SGA6} P. Berthelot, A. Grothendieck, and L. Illusie, {\sl Th\'eorie des intersections et th\'eor\`eme de Riemann-Roch\/}  (SGA6), S\'eminaire de G\'eom\'etrie Alg\'ebrique du Bois-Marie 1966-67, Lect. Notes in Math., \textbf{225}, (1971), Springer-Verlag.



\bibitem{Bhatt} B. Bhatt, {\sl Algebraization and Tannaka duality\/}, Cambridge J. Math., \textbf{4}, (2016), no. 4, 403--461.


\bibitem{BS} F. Binda and S. Saito, {\sl Relative cycles with moduli and regulator maps\/}, J. Inst. Math. Jussieu, \textbf{18}, (2019), no. 6, 1233--1293.


\bibitem{Bloch crys} S. Bloch, {\sl Algebraic K-theory and crystalline cohomology\/}, Publ. Math. de l'Inst. Hautes \'Etudes Sci., \textbf{47}, (1977), 187--268.

\bibitem{Bloch HC} S. Bloch, {\sl Algebraic cycles and higher $K$-theory\/}, Adv. Math., \textbf{61}, (1986), no. 3, 267--304.

\bibitem{Bloch moving} S. Bloch, {\sl The moving lemma for higher Chow groups\/}, J. Algebraic Geom., \textbf{3}, (1994), 537--568.

\bibitem{Bloch note} S. Bloch, {\sl Some notes on elementary properties of higher chow groups, including functoriality properties and cubical chow groups}, an online note available at {\tt http://www.math.uchicago.edu/\~{}bloch/cubical\_{}chow.pdf}

\bibitem{BE1} S. Bloch and H. Esnault, {\sl An additive version of higher Chow groups\/}, Ann. Scient. de l'{\'E}c. Norm. Sup., \textbf{36}, (2003), 463--477. 


\bibitem{BE2} S. Bloch and H. Esnault, {\sl The additive dilogarithm\/}, Documenta Math. \textbf{Extra Vol.} for Kazuya Kato's fiftieth birthday conference, (2003), 131--155.


\bibitem{BEK} S. Bloch, H. Esnault, and M. Kerz, {\sl $p$-adic deformation of algebraic cycle classes\/}, Invent. Math., \textbf{195}, (2014), 673--722.

\bibitem{BL} S. Bloch and S. Lichtenbaum, {\sl A spectral sequence for motivic cohomology\/},  dated March 3, 1995, preprint, article 62 of $K$-theory preprint server at the University of Illinois, Urbana-Champaign.


\bibitem{BGR} S. Bosch, U. G\"untzer, and R. Remmert, {\sl Non-Archimedean analysis\/}, Grundlehren Math. Wissen., \textbf{261}, Springer-Verlag, 1984.


\bibitem{BT} R. Bott and L. Tu, {\sl Differential forms in Algebraic Topology\/}, Grad. Texts in Math., \textbf{82}, Springer-Verlag, New York, 1982.



\bibitem{BH} W. Bruns and J. Herzog, {\sl Cohen-Macaulay rings\/}, 2nd Ed., Cambridge studies in advanced math., \textbf{39}, Cambridge Univ. Press, Cambridge, 1998. 453pp. 

\bibitem{Cartwright} D. A. Cartwright, D. Erman, M. Velasco, and B. Viray, {\sl Hilbert schemes of $8$ points\/}, Algebra \& Number Theory, \textbf{3}, (2009), no. 7, 763--795.

\bibitem{Chow} W.-L. Chow, {\sl On equivalence classes of cycles in an algebraic variety\/}, Ann. Math., \textbf{64}, (1956), no. 3, 450--479.

\bibitem{Cohen} I. S. Cohen, {\sl On the structure and ideal theory of complete local rings\/}, Trans. Amer. Math. Soc., \textbf{59}, (1946), 54--106.

\bibitem{Conrad cpnt} B. Conrad, {\sl Irreducible components of rigid spaces\/}, Ann. Inst. Fourier, Grenoble, \textbf{49}, (1999), no. 2, 473--541.



\bibitem{de Jong} A. J. de Jong, {\sl Smoothness, semi-stability and alterations\/}, Publ. Math. de l'Inst. Hautes \'Etudes Sci., \textbf{83}, (1996), 51--93.

\bibitem{Deligne} P. Deligne, {\sl Th\'eorie de Hodge : II\/}, Publ. Math. de l'Inst. Hautes \'Etudes Sci., \textbf{40}, (1971), 5--57.



\bibitem{Dubouloz} A. Dubouloz, {\sl $\mathbb{A}^1$-cylinders over smooth $\mathbb{A}^1$-fibered affine surfaces\/}, Proceeding of Kinosaki Algebraic Geometry symposium 2019, arXiv:2002.12838.


\bibitem{EVMS} P. Elbaz-Vincent and S. M\"uller-Stach, {\sl Milnor $K$-theory of rings, higher Chow groups and applications\/}, Invent. Math., \textbf{148}, (2002), no. 1, 177--206.


\bibitem{Ferrand} D. Ferrand, {\sl Conducteur, descente et pincement\/}, Bull. Soc. Math. France, \textbf{131}, (2003), no. 4, 553--585.





\bibitem{FS} E. Friedlander and A. Suslin, {\sl The spectral sequence relating algebraic $K$-theory to motivic cohomology\/}, Ann. Sci. \'Ecole Norm. Sup. $4^e$, \textbf{35}, (2002), no. 6, 773-875.

\bibitem{FV} E. Friedlander and V. Voevodsky, {\sl Bivariant cycle cohomology\/} in Cycles, Transfers and Motivic Homology Theories, Ann. Math. Stud., \textbf{143}, Princeton Univ. Press, Princeton, NJ, 2000, 138--187.

\bibitem{FK} K. Fujiwara and F. Kato, {\sl Foundations of Rigid Geometry I,\/}, EMS Monographs in Mathematics, European Math. Soc., Z\"urich, 2018.

\bibitem{Fulton Chow} W. Fulton, {\sl Rational equivalence on singular varieties\/}, Publ. Math. de l'Inst. Hautes \'Etudes Sci., \textbf{45}, (1975), 147--167.


\bibitem{Fulton} W. Fulton, {\sl Intersection theory, Second Edition\/}, Ergebnisse der Math. und ihrer Grenzgebiete.  Springer-Verlag, Berlin, 1998. xiv+470 pp.






\bibitem{GM} M. Goresky and R MacPherson, {\sl Intersection homology theory\/}, Topology, \textbf{19}, (1980), no. 2, 135--162.

\bibitem{GS} S. Greco and P. Salmon, {\sl Topics in $\mathfrak{m}$-adic Topologies\/}, Ergebnisse der Math. und ihrer Grenzgebiete, \textbf{58}, Springer-Verlag, New York, Heidelberg, Berlin, 1971.



\bibitem{Grothendieck BSMF} A. Grothendieck, {\sl La th\'eorie des classes de Chern\/}, Bull. Soc. Math. France, \textbf{86}, (1958), 137--154.


\bibitem{EGA1} A. Grothendieck, {\sl \'El\'ements de G\'eom\'etrie Alg\'ebrique I, Le langage des sch\'emas.\/} Publ. Math. de l'Inst. Hautes \'Etudes Sci., \textbf{4}, (1960), 228pp.

\bibitem{EGA2} A. Grothendieck, {\sl \'El\'elements de G\'eom\'etrie Alg\'ebrique II, \'Etude globale \'el\'ementaire de quelques classes de morphismes.\/}, Publ. Math. de l'Inst. Hautes \'Etudes Sci., \textbf{8}, (1961), 5--222.


\bibitem{EGA3-1} A. Grothendieck, {\sl \'El\'ements de G\'eom\'etrie Alg\'ebrique III, \'Etude cohomologique des faisceaux coh\'erent, Premi\`ere partie\/}, Publ. Math. de l'Inst. Hautes \'Etudes Sci., \textbf{11}, (1961), 5--167.

\bibitem{EGA4-2} A. Grothendieck, {\sl \'El\'ements de G\'eom\'etrie Alg\'ebrique IV, \'Etude locale des sch\'emas et des morphismes de sch\'emas, Seconde partie\/}, Publ. Math. de l'Inst. Hautes \'Etudes Sci., \textbf{24}, (1965), 5--231.


\bibitem{Grothendieck DR} A. Grothendieck, {\sl On the de Rham cohomology of algebraic varieties\/}, Publ. Math. de l'Inst. Hautes \'Etudes Sci., \textbf{29}, (1966), 95--103.


\bibitem{SGA2} A. Grothendieck, {\sl Cohomologie locale des faisceaux coh\'erents et th\'eor\`emes de Lefschetz locaux et globaux\/} (SGA 2), Augment\'e d'un expos\'e par Mich\`ele Raynaud. S\'eminaire de G\'eom\'etrie Alg\'ebrique du Bois-Marie, 1962. Advanced Studies in Pure Mathematics, Vol. 2. North-Holland Publishing Co., Amsterdam; Masson \& Cie, \'Editeur, Paris, 1968, vii+287 pp.

\bibitem{SGA1} A. Grothendieck, {\sl Rev\^etements \'etales et groupe fondamental (SGA1)\/}, S\'eminaire de G\'eom\'etrie Alg\'ebrique du Bois-Marie, 1960-61, Lect. Notes in Math., \textbf{224}, (1971), Berlin, New York, Springer-Verlag, xxii+447.

\bibitem{Gunther} J. Gunther, {\sl Random hypersurfaces and embedding curves in surfaces over finite fields\/}, J. Pure Appl. Alg., \textbf{221}, (2017), 89--97.




\bibitem{Hanamura} M. Hanamura, {\sl Homological and cohomological motives of algebraic varieties\/}, Invent. Math., \textbf{142}, (2000), 319--349.

\bibitem{Hartshorne DR} R. Hartshorne, {\sl On the De Rham cohomology of algebraic varieties\/}, Publ. Math. de l'Inst. Hautes \'Etudes Sci., \textbf{45}, (1975), 5--99.

\bibitem{Hartshorne} R. Hartshorne, {\sl Algebraic Geometry\/}, Grad. Texts in Math., \textbf{52}, Springer-Verlag, New York-Heidelberg, 1977. xvi+496 pp.

\bibitem{HP} R. Hartshorne and C. Polini, {\sl Divisor class groups of singular surfaces\/}, Trans. Amer. Math. Soc., \textbf{367}, (2015), no. 9, 6357--6385.


\bibitem{KA} S. Kleiman and A. Altman, {\sl Bertini theorems for hypersurface sections containing a subscheme\/}, Comm. Alg., \textbf{7}, (1979), no. 8, 775--790.











\bibitem{Hironaka} H. Hironaka, {\sl Resolution of singularities of an algebraic variety over a field of characteristic zero. I\/}, Ann. Math., \textbf{79}, (1964), 109--203, and part II, 205--326.

\bibitem{Illusie} L. Illusie, {\sl Complexe de de Rham-Witt et cohomologie cristalline\/}, Ann. Sci. l'\'Ecole Norm. Sup. $4^e$, \textbf{12}, (1979), 501--661.




\bibitem{Iyengar} S. Iyengar, {\sl Andr\'e-Quillen homology of commutative algebras\/}, in Interactions between homotopy theory and algebra, Contempo. Math., \textbf{436}, Amer. Math. Soc., Providence, RI, 2007. pp.203--234.

\bibitem{MM1} B. Kahn, H. Miyazaki, S. Saito, and T. Yamazaki, {\sl Motives with modulus, I: Modulus sheaves with transfers for non-proper modulus pairs\/}, \'Epijournal G\'eom. Alg\'ebr., \textbf{5}, (2021), Article Nr. 1.

\bibitem{MM2} B. Kahn, H. Miyazaki, S. Saito, and T. Yamazaki, {\sl Motives with modulus, II: Modulus sheaves with transfers for proper modulus pairs\/}, \'Epijournal G\'eom. Alg\'ebr., \textbf{5}, (2021), Article Nr. 2.

\bibitem{MM3} B. Kahn, H. Miyazaki, S. Saito, and T. Yamazaki, {\sl Motives with modulus, III: The categories of motives\/}, preprint, arXiv:2011.11859, to appear in Ann. $K$-theory.




\bibitem{Kerz thesis} M. Kerz, {\sl Milnor $K$-theory of local rings\/}, Ph.D. thesis, University of Regensburg, 2008.

\bibitem{Kerz Gersten} M. Kerz, {\sl The Gersten conjecture for Milnor $K$-theory\/}, Invent. Math., \textbf{175}, (2009), no. 1, 1--33.

\bibitem{Kerz finite} M. Kerz, {\sl Milnor $K$-theory of local rings with finite residue fields\/}, J. Algebraic Geom., \textbf{19}, (2010), 173--191.

\bibitem{KST IM} M. Kerz, F. Strunk, and G. Tamme, {\sl Algebraic $K$-theory and descent for blow-ups\/}, Invent. math., \textbf{211}, (2018), 523--577.


\bibitem{KP2} A. Krishna and J. Park, {\sl DGA-structure on additive higher Chow groups\/}, Internat. Math. Res. Notices, \textbf{2015}, no. 1, 1--54. 


\bibitem{KP DM} A. Krishna and J. Park, {\sl On additive higher Chow groups of affine schemes\/}, Documenta Math., \textbf{21}, (2016), 49--89.

\bibitem{KP sfs} A. Krishna and J. Park, {\sl A moving lemma for relative $0$-cycles\/}, Algebra \& Number Theory, \textbf{14}, (2020), no. 4, 991--1054.

\bibitem{KP crys} A. Krishna and J. Park, {\sl de Rham-Witt sheaves via algebraic cycles (with an appendix by Kay R\"ulling)\/}, Compositio Math., \textbf{158}, (2021), no. 10, 2089--2132.




\bibitem{Landsburg AJM} S. E. Landsburg, {\sl Relative cycles and algebraic $K$-theory\/}, Amer. J. Math., \textbf{111}, (1989), 599--632.
\bibitem{Landsburg Duke} S. E. Landsburg, {\sl $K$-theory and patching for categories of complexes\/}, Duke Math. J., \textbf{62}, (1991), no. 2, 359--384.

\bibitem{Landsburg Ill} S. E. Landsburg, {\sl Relative Chow groups\/}, Ill. J. Math., \textbf{35}, (1991), no. 4, 618--641.


\bibitem{Levine moving} M. Levine, {\sl Techniques of localization in the theory of algebraic cycles\/}, J. Algebraic Geom., \textbf{10}, (2001), 299--363.

\bibitem{Levine coniveau} M. Levine, {\sl The homotopy coniveau tower\/}, J. Topology, \textbf{1}, (2008), 217--267. 

\bibitem{Levine SM} M. Levine, {\sl Smooth motives\/}, in Motives and algebraic cycles, Fields Inst. Commun., \textbf{56}, Amer. Math. Soc., Providence, RI, 2009. pp.175--231.


\bibitem{Levine cobordism} M. Levine and F. Morel, {\sl Algebraic Cobordism\/}, Springer Monographs in Math., Springer, Berlin, 2007. xii+244pp.


\bibitem{LW} M. Levine and C. Weibel, {\sl Zero cycles and complete intersections on singular varieties\/}, J. Reine Angew. Math., \textbf{359}, (1985), 106--120.


\bibitem{Li} M. Li, {\sl Integral regulators for higher Chow complexes\/}, SIGMA \textbf{14}, (2018), 118, 12 pages, Contribution to the Special Issue on Modular Forms and String Theory in honor of Noriko Yui.

\bibitem{Lipman} J. Lipman and M. Hashimoto, {\sl Foundations of Grothendieck duality for diagrams of schemes\/}, Lect. Notes in Math., \textbf{1960}, Springer-Verlag, Berlin Heidelberg, 2009.


\bibitem{LS} P. Lowry and T. Sch\"urg, {\sl Derived algebraic cobordism\/}, J. Inst. Math. Jussieu, \textbf{15}, (2016), no. 2, 407--443.

\bibitem{Lurie HTT} J. Lurie, {\sl Higher Topos Theory\/}, Ann. Math. Studies, \textbf{170}, Princeton Univ. Press, Princeton, NJ, 2009. xviii+925 pp.

\bibitem{Lurie SAG} J. Lurie, {\sl Spectral Algebraic Geometry\/}, preprint, version February 18, 2018, available on the homepage of Jacob Lurie at Harvard and IAS.

\bibitem{Manin} Yu. I. Manin, {\sl Lectures on the $K$-functor in algebraic geometry\/}, Russ. Math. Surveys, \textbf{24}, (1969), 1--89.



\bibitem{Matsumura} H. Matsumura, {\sl Commutative ring theory\/}, Cambridge studies in advanced math. \textbf{8}, Cambridge Univ. Press, Cambridge, 1986, xiv+320pp.

\bibitem{May} J. P. May, {\sl Simplicial objects in algebraic topology\/}, Chicago Lect. in Math., The Univ. of Chicago Press, Chicago, IL, USA, 1967.

\bibitem{MVW} C. Mazza, V. Voevodsky, and C. Weibel, {\sl Lecture notes on motivic cohomology\/}, Clay Mathematics Monographs, \textbf{2}. American Mathematical Society, Providence, RI; Clay Mathematics Institute, Cambridge, MA, 2006. xiv+216 pp.



\bibitem{Milnor K} J. Milnor, {\sl Introduction to algebraic $K$-theory\/}, Annals Math. Studies \textbf{72}, Princeton Univ. Press, Princeton, NJ, 1971.

\bibitem{Morrow} M. Morrow, {\sl Pro unitality and pro excision in algebraic $K$-theory and cyclic homology\/}, J. reine angew. Math., \textbf{736}, (2018), 95--139.

\bibitem{Mumford} D. Mumford, {\sl Further pathologies in algebraic geometry\/}, Amer. J. Math., \textbf{84}, (1962), no. 4, 642--648.

\bibitem{NS} Yu. P. Nesterenko and A. A. Suslin, {\sl Homology of the general linear group over a local ring, and Milnor's $K$-theory\/}, (Russian) Izv. Akad. Nauk SSSR Ser. Mat., \textbf{53}, (1989), no. 1, 121--146; English translation: Math. USSR-Izv., \textbf{34}, (1990), no. 1, 121--145.

\bibitem{Nicaise} J. Nicaise, {\sl A trace formula for rigid varieties, and motivic Weil generating series for formal schemes\/}, Math. Ann., \textbf{343}, (2009), 285--349.


\bibitem{P2} J. Park, {\sl Regulators on additive higher Chow groups\/}, Amer. J. Math., \textbf{131}, (2009), no. 1, 257--276.


\bibitem{cubical localization} J. Park, {\sl On localization for cubical higher Chow groups\/}, preprint, arXiv:2108.13561, 2021. A version of it is accepted to appear in Tohoku Math. J.

\bibitem{Park Tate} J. Park, {\sl Motivic cohomology of fat points in Milnor range via formal and rigid geometries\/}, preprint, arXiv:2108.13563, 2021.


\bibitem{PP} J. Park and P. Pelaez, {\sl On the coniveau filtration on algebraic $K$-theory of singular schemes\/}, preprint, 2021.

\bibitem{PU Milnor} J. Park and S. \"Unver, {\sl Motivic cohomology of fat points in Milnor range\/}, Documenta Math., \textbf{23}, (2018), 759--798.

\bibitem{Poonen} B. Poonen, {\sl Bertini theorems over finite fields\/}, Ann. Math., \textbf{160}, (2004), no. 3, 1099--1127.


\bibitem{Quillen} D. Quillen, {\sl Higher algebraic $K$-theory. I\/}, in Algebraic $K$-theory, I: Higher $K$-theories (Proc. Conf., Battelle Memorial Inst., Seattle, Wash., 1972), Lect. Notes in Math., \textbf{341}, (1973), Berlin, New York, Springer-Verlag, pp.85--147.



\bibitem{R} K. R\"ulling, {\sl The generalized de Rham-Witt complex over a field is a complex of zero-cycles\/}, J. Algebraic Geom., \textbf{16}, (2007), no. 1, 109--169.


\bibitem{Salmon} P. Salmon, {\sl Sur les s\'eries formelles restreintes\/}, Bull. Soc. Math. France, \textbf{92}, (1964), 385--410.

\bibitem{Samuel} P. Samuel, {\sl Lectures on Unique Factorization Domains\/}, Notes by M. Pavman Murthy, Tata Inst. Fund. Res. Lect. on Math., \textbf{30}, Bombay 1964 ii+84+iii pp.

\bibitem{Schlessinger} M. Schlessinger, {\sl Functors of Artin rings\/}, Trans. Amer. Math. Soc., \textbf{130}, (1968), 208--222.






\bibitem{stacks} The Stacks Project at {\tt{https://stacks.math.columbia.edu}}.



\bibitem{Tate} J. Tate, {\sl Rigid analytic spaces\/}, Invent. Math., \textbf{12}, (1971), 257--289.





\bibitem{Toen} B. To\"en, {\sl Derived algebraic geometry\/}, Eur. Math. Soc. Surv. Math. Sci., \textbf{1}, (2014), 153--240. 


\bibitem{Totaro} B. Totaro, {\sl Milnor $K$-theory is the simplest part of algebraic K-theory\/}, $K$-Theory, \textbf{6}, (1992), no. 2, 177--189. 


\bibitem{Totaro Sigma} B. Totaro, {\sl Chow groups, Chow cohomology, and linear varieties\/}, Forum Math. Sigma, \textbf{2}, (2014), Paper No. e17, 25 pp. 


\bibitem{vdK} W. van der Kallen, {\sl Le $K_2$ des nombres duaux\/}, C. R. Acad. Sci. Paris S\'er. A-B, \textbf{273}, (1971), A1204--A1207.

\bibitem{Vistoli} A. Vistoli, {\sl Grothendieck topologies, fibered categories and descent theory\/}, in Fundamental Algebraic Geometry, Math. Surveys Monogr., \textbf{123}, Providence, RI, Amer. Math. Soc., pp.1--104.

\bibitem{Voevodsky} V. Voevodsky, {\sl Motivic cohomology groups are isomorphic to higher Chow groups in any characteristic\/}, Internat. Math. Res. Notices, \textbf{2002}, no. 7, 351--355. 


\bibitem{Weibel} C. Weibel, {\sl The $K$-book: An Introduction to Algebraic $K$-theory\/}, Graduate Studies in Math., \textbf{145}, Amer. Math. Soc., Providence, RI, 2013. xii+618pp.

\bibitem{Wutz} F. Wutz, {\sl Bertini theorems for hypersurface sections containing a subscheme over finite fields\/}, Ph.D. thesis, Universit\"{a}t Regensburg, 2014.

\end{thebibliography}
\end{document}